\documentclass[a4paper,10pt]{amsart}
\usepackage{amsmath}

\usepackage{comment}
\usepackage[font+=Large, labelformat=empty, skip=10pt]{subcaption}

\usepackage{amssymb,amsmath,amsthm,ascmac,array,bm}
\usepackage{graphicx}
\usepackage{multirow}
\usepackage{enumitem}
\usepackage{float} 

\usepackage{xcolor}
\usepackage{mathtools}

\usepackage[colorlinks=true,citecolor=blue,linkcolor=blue]{hyperref}

\usepackage{empheq}

\newtheorem{thm}{Theorem}[section]

\newtheorem{lem}[thm]{Lemma}

\newtheorem{cor}[thm]{Corollary}

\newtheorem{rem}[thm]{Remark}

\numberwithin{equation}{section}
\allowdisplaybreaks

\newcommand{\mcf}{\mathcal{F}}

\newcommand{\mci}{\mathcal{I}}
\newcommand{\mcl}{\mathcal{L}}

\newcommand{\mbbh}{\mathbb{H}}

\newcommand{\mbbr}{\mathbb{R}}

\newcommand{\mbbzp}{\mathbb{Z}_{+}}

\newcommand{\al}{\alpha} \newcommand{\be}{\beta} 
\newcommand{\ep}{\epsilon} \newcommand{\ve}{\varepsilon}
\newcommand{\vp}{\varphi} 
\newcommand{\del}{\delta} \newcommand{\D}{\Delta} 
\newcommand{\sig}{\sigma}  
\newcommand{\gam}{\gamma} 
\newcommand{\lam}{\lambda} 
\newcommand{\p}{\partial}
\newcommand{\ra}{\rangle} \newcommand{\la}{\langle}
\newcommand{\wc}{\Rightarrow} 
\newcommand{\cil}{\xrightarrow{\mcl}} 
\newcommand{\cip}{\xrightarrow{p}} 

\def\ds#1{\displaystyle{#1}}

\def\nn{\nonumber}

\newcommand{\pr}{P_{\theta}} \newcommand{\E}{E_{\theta}}

\newcommand{\sgn}{{\rm sgn}}
\def\diag{{\rm diag}}

\def\sumj{\sum_{j=1}^{n}}

\def\lp{L\'evy process}

\def\cadlag{c\`adl\`ag}
\def\ssou{ssOU}

\newcommand{\tz}{\theta_{0}}
\newcommand{\mz}{\mu_{0}}
\newcommand{\tes}{\hat{\theta}_{n}}
\newcommand{\aes}{\hat{\alpha}_{n}}
\newcommand{\bes}{\hat{\beta}_{n}}

\newcommand{\mes}{\hat{\mu}_{n}}
\newcommand{\ses}{\hat{\sigma}_{n}}
\newcommand{\les}{\hat{\lambda}_{n}}

\newcommand{\AllHistFigure}[1]{%
\begin{figure}[p]
\centering
\caption{
Histograms of parameter estimators (Euler-QMLE), $\alpha=#1$.
Each row corresponds to a parameter, and columns correspond to
$n=500,1000,2000$.
}
\label{fig:hist-all-alpha-#1}

\textbf{$\lambda$}\par\vspace{1mm}
\begin{minipage}{0.32\linewidth}\centering
\includegraphics[width=\linewidth]{fig/hist_Euler_MLE_lambda_#1_5_timescale.pdf}\\[-1mm]
{\small $n=500$}
\end{minipage}\hfill
\begin{minipage}{0.32\linewidth}\centering
\includegraphics[width=\linewidth]{fig/hist_Euler_MLE_lambda_#1_10_timescale.pdf}\\[-1mm]
{\small $n=1000$}
\end{minipage}\hfill
\begin{minipage}{0.32\linewidth}\centering
\includegraphics[width=\linewidth]{fig/hist_Euler_MLE_lambda_#1_20_timescale.pdf}\\[-1mm]
{\small $n=2000$}
\end{minipage}

\vspace{2mm}

\textbf{$\mu$}\par\vspace{1mm}
\begin{minipage}{0.32\linewidth}\centering
\includegraphics[width=\linewidth]{fig/hist_Euler_MLE_mu_#1_5_timescale.pdf}
\end{minipage}\hfill
\begin{minipage}{0.32\linewidth}\centering
\includegraphics[width=\linewidth]{fig/hist_Euler_MLE_mu_#1_10_timescale.pdf}
\end{minipage}\hfill
\begin{minipage}{0.32\linewidth}\centering
\includegraphics[width=\linewidth]{fig/hist_Euler_MLE_mu_#1_20_timescale.pdf}
\end{minipage}

\vspace{2mm}

\textbf{$\alpha$}\par\vspace{1mm}
\begin{minipage}{0.32\linewidth}\centering
\includegraphics[width=\linewidth]{fig/hist_Euler_MLE_alpha_#1_5_timescale.pdf}
\end{minipage}\hfill
\begin{minipage}{0.32\linewidth}\centering
\includegraphics[width=\linewidth]{fig/hist_Euler_MLE_alpha_#1_10_timescale.pdf}
\end{minipage}\hfill
\begin{minipage}{0.32\linewidth}\centering
\includegraphics[width=\linewidth]{fig/hist_Euler_MLE_alpha_#1_20_timescale.pdf}
\end{minipage}

\vspace{2mm}

\textbf{$\sigma$}\par\vspace{1mm}
\begin{minipage}{0.32\linewidth}\centering
\includegraphics[width=\linewidth]{fig/hist_Euler_MLE_sigma_#1_5_timescale.pdf}
\end{minipage}\hfill
\begin{minipage}{0.32\linewidth}\centering
\includegraphics[width=\linewidth]{fig/hist_Euler_MLE_sigma_#1_10_timescale.pdf}
\end{minipage}\hfill
\begin{minipage}{0.32\linewidth}\centering
\includegraphics[width=\linewidth]{fig/hist_Euler_MLE_sigma_#1_20_timescale.pdf}
\end{minipage}

\vspace{2mm}

\textbf{$\beta$}\par\vspace{1mm}
\begin{minipage}{0.32\linewidth}\centering
\includegraphics[width=\linewidth]{fig/hist_Euler_MLE_beta_#1_5_timescale.pdf}
\end{minipage}\hfill
\begin{minipage}{0.32\linewidth}\centering
\includegraphics[width=\linewidth]{fig/hist_Euler_MLE_beta_#1_10_timescale.pdf}
\end{minipage}\hfill
\begin{minipage}{0.32\linewidth}\centering
\includegraphics[width=\linewidth]{fig/hist_Euler_MLE_beta_#1_20_timescale.pdf}
\end{minipage}
\end{figure}
}

\title[Asymptotic inference for skewed stable OU process]
{Asymptotic inference for skewed stable Ornstein-Uhlenbeck process}

\author[E. Kawamo]{Eitaro Kawamo}
\address{
Joint Graduate School of Mathematics for Innovation, Kyushu University, 744 Motooka Nishi-ku Fukuoka 819-0395, Japan}

\author[H. Masuda]{Hiroki Masuda}
\address{Graduate School of Mathematical Sciences, University of Tokyo, 3-8-1 Komaba Meguro-ku Tokyo 153-8914, Japan}
\email{hmasuda@ms.u-tokyo.ac.jp}
\date{\today}

\begin{document}
\mathtoolsset{showonlyrefs=true}

\begin{abstract}
We consider the parametric estimation of the Ornstein-Uhlenbeck process driven by a non-Gaussian $\al$-stable L\'{e}vy process with the stable index $\al>1$ and possibly skewed jumps, based on a discrete-time sample over a fixed period. 
By employing a suitable non-diagonal normalizing matrix, we present the following: the parametric family satisfies the local asymptotic mixed normality with a non-degenerate Fisher information matrix; there exists a local maximum of the log-likelihood function which is asymptotically mixed-normal; the local maximum is asymptotically efficient in the sense that it has maximal concentration around the true value over symmetric convex Borel subsets. In the proof, we prove the asymptotic equivalence between the genuine likelihood and the much simpler Euler-type quasi-likelihood. Furthermore, we propose a simple moment-based method to estimate the parameters of the driving stable L\'{e}vy process, which serves as an initial estimator for numerical search of the (quasi-)likelihood, reducing the computational burden of the optimization to a large extent. 
We also present simulation results, which illustrate the theoretical results and highlight the advantages and disadvantages of the genuine and quasi-likelihood approaches.
\end{abstract}

\maketitle
\section{Introduction}

We consider the parameter estimation of the skewed stable OU ({\ssou}) process
\begin{equation}\label{hm:ssou_sde}
dY_t = (\mu  - \lam Y_t)dt + \sig dJ_t
\end{equation}
based on high-frequency observations $(Y_{t_j})_{j=0}^{n}$, where $t_j = t_{n,j} \coloneqq jh$ with $h = h_n \coloneqq T/n$ and a fixed interval $[0,T]$ and where $J=(J_t)_{t\in[0,T]}$ is a stable {\lp} with stability index $\alpha \in (1,2)$ and skewness parameter $\beta \in (-1,1)$. 
The primary objectives of this paper are two-fold:
\begin{itemize}
    \item Under assumptions $\alpha > 1$ and $\beta = 0$, we show that the log-quasi-likelihood function based on the Euler approximation has a local quadratic expansion expressed by the score function and information matrix. Moreover, there exist local maxima of the likelihood and quasi-likelihood that are asymptotically equivalent and asymptotically mixed-normally distributed if we employ a class of non-diagonal rate matrices.
    \item To reduce the computational burden, we propose a moment estimator of the noise parameter $(\alpha, \sigma, \beta)$ and use them as initial values (plug-in estimators) for optimizing the (quasi-)likelihood of $(\lambda, \mu)$. This hybrid approach balances computational feasibility and estimation precision.
\end{itemize}

The present study was theoretically motivated by \cite{BroMas18}, \cite{CleGlo20}, and \cite{Mas23} together with the references therein. The papers rigorously established the asymptotic properties of the (quas-)likelihood-based inference procedures for models driven by a symmetric stable {\lp} based on high-frequency observations over a fixed period. In \cite{BroMas18}, the local asymptotic normality (LAN) of the true likelihood for a stable {\lp} with drift was proved, along with constructing a practical moment-based one-step estimator that achieved accuracy comparable to the maximum likelihood estimator. Later, in \cite{Mas23}, the local asymptotic mixed normality (LAMN) property of the true likelihood was established for the OU process driven by a symmetric stable {\lp}, and it was shown that efficient estimators can be constructed via $k$-step improvements from quasi-likelihood-based initial estimators. These results provided a theoretical basis for inference in symmetric stable models and motivate the extension to skewed settings considered in this paper.

Originally, this paper was motivated by an application of non-Gaussian stochastic process models to neuroimaging given in \cite{CosBocCauFer2016}. Stochastic modelling of neural signals, particularly under resting-state conditions, has increasingly employed continuous-time frameworks that incorporate both deterministic and stochastic components. The previous work \cite{CosBocCauFer2016} modeled fMRI fluctuations using a generalized Langevin equation with an $\alpha$-stable driving {\lp}. There, the authors claimed that a linear drift can well approximate the deterministic component, consistent with the Ornstein-Uhlenbeck process, and the stochastic component reflects heavy-tailed non-Gaussian fluctuations that are captured by an $\alpha$-stable distribution. However, the exhibited parameter estimation results are not supported by theoretical justification for the asymptotic properties of the estimators. Indeed, to the best of the authors' knowledge, there is no previous theoretical study on statistical inference for the {\ssou} process, and the results in the present paper handling a possibly skewed driving noise are new. 


The rest of this paper is organized as follows.
In Section \ref{hm:sec_preliminaries}, we begin with some necessary background of our model. Section \ref{hm:sec_main} presents the main results, the local asymptotics related to the genuine likelihood and quasi-likelihood. In Section \ref{hm:sec_ME}, we will describe the explicit construction of the easy-to-compute preliminary estimator of the noise parameter. 
Section \ref{hm:sec_sim} presents some numerical experiments. 
In Section \ref{hm:sec_time.scale}, we will discuss estimating the model-time scale in our original model setup where $T>0$ is known.
Finally, Section \ref{sec:main_result_proof} is devoted to the proof of the main result, Theorem \ref{thm:main_result}.

\section{Preliminaries}
\label{hm:sec_preliminaries}

\subsection{Skewed stable distribution}

We denote by $S^{0}_{\alpha} (\beta,\sigma,\mu)$ the non-Gaussian stable distribution corresponding to the characteristic function
\begin{equation}\label{hm:0stable.cf}
\varphi (u;\al,\be,\sig,\mu) \coloneqq 
\exp \left\{ i \mu u - (\sigma |u|)^{\alpha} \left( 1 + i \beta \sgn(u) t_{\al} \left( |\sigma u|^{1 - \alpha} - 1 \right) \right) \right\},
\end{equation}
where 
\begin{equation}
    (\al,\be,\sig,\mu)\in(0,2)\times[-1,1]\times(0,\infty)\times\mbbr.
\end{equation}
The parametrization \eqref{hm:0stable.cf} introduced by \cite{Nol98} is continuous; in particular, it is continuous at $\al=1$ in $\al$ even when $\beta\ne 0$.
The possibly skewed Cauchy case is defined as the limit for $\al\to 1$ 
through $\lim_{\alpha \to 1} \tan(\al\pi/2) (|u|^{1 - \alpha} - 1) = (2/\pi)\log |u|$:
\begin{align}
\varphi (u;1,\beta,\sig,\mu) &= \lim_{\al\to 1}\varphi (u;\al,\be,\sig,\mu)
\nn\\
&=\exp \left\{ i \mu u - \sigma |u| \left( 1 + i \beta \dfrac{2}{\pi} \sgn (u) \log (\sigma |u|) \right) \right\}.
\label{hm:0stable.cf.cauchy}
\end{align}
It is obvious from \eqref{hm:0stable.cf} that for any $a\ne 0$ and $b\in\mbbr$,
\begin{equation}\label{hm:standard.Z}
X\sim S_\al^0(\beta,\sig,\mu) 
\quad \wc \quad 
aX+b \sim S_{\al}^{0}\left(\beta\, \sgn(a), |a|\sig, a\mu+b \right).
\end{equation}
Here and in what follows, $\xi\sim F$ means that a random variable $\xi$ obeys a distribution $F$.

It is known that $S_{\alpha} (\beta,\sigma,\mu)$ for $\beta\in(-1,1)$ admits an everywhere positive smooth Lebesgue density: see \cite{Sat99} or \cite{Sha69}. We denote by the probability density function of $S^{0}_{\al}(\be, 1, 0)$:
\begin{equation}\label{hm:phi_def}
    \phi_{\al,\be}(x) \coloneqq \frac{1}{2\pi} \int e^{-ixu}\vp(u;\al,\beta,1,0)du.
\end{equation}
Let us denote by $\p_a$ the partial differentiation with respect to a variable $a$.
The density $\phi_{\al,\beta}$ is smooth in $(x,\al,\beta)\in\mbbr\times(0,2)\times(-1,1)$, and for each $j,k,l\in \mbbzp$ we have
\begin{equation}
\sup_{x\in\mbbr}
\frac{(1+|x|)^{\al+1+j}}{ \{1+\log(1+|x|)\}^k} \big|\p_x^{j}\p_{\al}^{k}\p_{\be}^l\phi_{\al,\beta}(x)\big| < \infty
\nn
\end{equation}
and also
\begin{equation}
\p_x^{j}\p_{\al}^{k}\p_{\be}^l\phi_{\al,\beta}(x) 
\asymp \frac{\{\log|x|\}^k}{|x|^{\al+1+j}}
\end{equation}
for both $x\to\pm\infty$, where $\asymp$ denotes the asymptotic equivalence up to a multiplicative constant. 
We refer to \cite{DuM73} and \cite[p.1949]{AndCalDav09} for more related details.

\subsection{Skewed stable L\'{e}vy and Ornstein-Uhlenbeck processes}
\label{hm:sec_slp}

It is obvious from \eqref{hm:0stable.cf} that $S_{\alpha} (\beta,\sigma,\mu)$ is infinitely divisible for each admissible parameter value, implying that there exists a {\lp} whose distribution at time $1$ is $S_{\alpha} (\beta,\sigma,\mu)$.
Let
\begin{equation}
    t_{\al} \coloneqq \tan \frac{\al \pi}{2}.
\end{equation}

\begin{lem}
\label{hm:ssou_inc}
Let $J=(J_t)_{t\in[0,T]}$ denote a {\lp} such that $J_1 \sim S_\al^0(\beta,1,0)$. 
Then, for each $h>0$,
\begin{equation}\label{hm:Jh_dist}
J_h \sim 
\left\{
\begin{array}{ll}
\ds{S_\al^0\left(\beta,\, h^{1/\al},\, \beta (h^{1/\al}-h) t_{\al}\right)}
& \al\ne 1, \\[2mm]
\ds{S_1^0\left(\beta,\, h,\, \frac{2}{\pi}h\beta\log h\right)}
& \al= 1.
\end{array}
\right.
\end{equation}
\end{lem}

\begin{proof}
\eqref{hm:Jh_dist} can be derived by direct calculation of the $J_h$'s 
characteristic function $u \mapsto \vp(u;\al,\beta,1,0)^h$ 
through the expressions \eqref{hm:0stable.cf} and \eqref{hm:0stable.cf.cauchy}.
\end{proof}

It follows from \eqref{hm:Jh_dist} that the distribution $\mcl(J_h)$ for $h\ne 1$ has an extra non-trivial drift when $\beta\ne 0$. For each $h>0$, we have $(h^{1/\al}-h) \tan(\al\pi/2) \to (2/\pi)h \log h$ as $\al\to 1$ by l'H\^{o}pital's rule. In this sense, the case $\al\ne 1$ in \eqref{hm:Jh_dist} is the unified expression:
\begin{equation}\label{hm:Jh_dist_unified}
J_h \sim S_\al^0\left(\beta,\, h^{1/\al},\, \beta (h^{1/\al}-h) t_{\al}\right),\qquad \al\in(0,2).
\end{equation}
Moreover, for fixed $\al$ with $\beta\ne 0$, the drift in $\mcl(J_h)$ satisfies that for $h\to 0$,
\begin{equation}
    \beta (h^{1/\al}-h) t_{\al} = \left\{
    \begin{array}{cc}
    O(h^{1/\al}) & \al>1, \\
    O(h \overline{l}) & \al=1,  \\
    O(h) & \al<1.
    \end{array}
    \right.
\end{equation}
This fact will be used later.

\medskip

Next, we consider the {\ssou} process $Y=(Y_t)_{t\in[0,T]}$ driven by a {\lp} $J$ satisfying \eqref{hm:Jh_dist}:
\begin{equation}\label{hm:ou.sde}
dY_t = (\mu - \lam Y_t)dt + \sig dJ_t,
\end{equation}
where, just for simplicity, the initial variable $Y_0=y_0\in\mbbr$ is assumed to be a constant.
As is well-known, the Markovian stochastic differential equation \eqref{hm:ou.sde} admits the explicit solution: for each $j\le n$,
\begin{align}
Y_{t_j}&= e^{-\lambda h}Y_{t_{j-1}} + \frac{\mu}{\lambda}(1- e^{-\lambda h}) 
+ \sigma \int_{t_{j-1}}^{t_j} e^{-\lambda(t_j-u)}dJ_u,
\label{eq:ssOU}
\end{align}
where the last term on the right-hand side is the L\'{e}vy integral, defined through the following limit in probability: for $f_j(s):=\exp(-\lam(t_j - s))$,
\begin{equation}
\int_{t_{j-1}}^{t_j} f_j(u) dJ_u
:= \lim_{m \rightarrow \infty} \sum_{k=1}^{m}f\left(t_{j-1}+
\frac{k-1}{m}h \right) \left(J_{t_{j-1}+\frac{k}{m}h} - J_{t_{j-1}+\frac{k-1}{m}h}\right).
\end{equation}
See \cite{SatYam83} for details.

Let 
\begin{equation}
    \theta \coloneqq (\lam,\mu,\al,\sig, \be) \in \Theta,
\end{equation}
where $\Theta$ is a bounded convex domain whose closure satisfies
\begin{equation}\label{hm:Theta_closure}
\overline{\Theta} \subset \mbbr\times\mbbr\times (1,2)\times(0,\infty)\times(-1,1).
\nonumber
\end{equation}
Note that we are excluding $\al\in(0,1]$ from the admissible region of $\al$
\footnote{We assumed this throughout to make technical arguments concise.}. 
Notably, it does not matter in our study whether or not the process $Y$ is ergodic, that is, $\lam>0$ or not.

Let $\pr$ denote the distribution of $(J,Y)$ corresponding to $\theta$; in particular, $J=(J_t)_{t\in[0,T]}$ a {\lp} such that $J_1 \sim S_\al^0(\beta,1,0)$ under $\pr$.
We set $(\mcf_t)_{t\le T}$ to be the natural filtration generated by $J$: 
\begin{equation}
    \mcf_t:=\sig(J_s:\, s\le t),
\end{equation}
which is known to satisfy the usual hypothesis; see \cite[Theorem I.31]{Pro05}. Thus, we are given the family $\{(\Omega,\mcf,(\mcf_t)_{t\le T},\pr)\}_{\theta\in\overline{\Theta}}$ of a filtered probability spaces. 
We will denote by $\E$ the expectation concerning $\pr$. 
Moreover, we will denote by $\cil$ and $\cip$ the convergences in law and in probability, respectively; all stochastic convergences in probability will be taken under $\pr$ unless otherwise mentioned.

Now, our objective is to estimate the true parameter (assumed to exist) 
\begin{equation}
    \tz \coloneqq (\lam_0,\mz, \al_0,\sig_0, \be_0) \in \Theta
\end{equation}
based on a sample $(Y_{t_j})_{j=0}^{n}$.

To specify the conditional distribution $\mcl(Y_{t_j}|Y_{t_{j-1}})$ under $\pr$, we recall the following lemma; it is more or less well-known, but we give the proof for the sake of reference.

\begin{lem}
\label{lem-p1} 
Let $\vp_1(u):=\vp(u;\al,\beta,1,0)$ and let $f: [0, \infty) \rightarrow \mbbr$ be a real-valued measurable function. Then, for each $t>0$,
\begin{equation}
\log \E \left[ \exp \left( iu \int_{0}^{t} f(s)dJ_s \right) \right] = \int_{0}^{t} \log\vp_{1}\left(u f(s)\right)\,ds
\end{equation}
\end{lem}

\begin{proof}
By the independence and stationarity of $J$'s increments,
\begin{align}
& \log \E \left[ \exp \left( iu \int_{0}^{t} f(s)dJ_s \right) \right]\\
&=\log \E \left[ \exp\left( iu\lim_{m \rightarrow \infty} \sum_{j=1}^{m}f\left(\frac{(j-1)t}{m} \right) (J_{jt/m} - J_{(j-1)t/m})\right)\right]\\
&= \lim_{m \rightarrow \infty} \log \E \left[\exp \left(iu\sum_{j=1}^{m} f\left(\frac{(j-1)t}{m} \right) (J_{jt/m} - J_{(j-1)t/m})\right) \right] \\
&=\lim_{m \rightarrow \infty}\log \prod_{j=1}^{m} \E \left[\exp \left(iuf\left(\frac{(j-1)t}{m} \right) J_{t/m} \right)\right]\\
&=\lim_{m \rightarrow \infty} \sum_{j=1}^{m} \log \vp_{1} \left(uf\left(\frac{(j-1)t}{m} \right)\right)^{t/m}\\
&= \lim_{m \rightarrow \infty} \frac{t}{m}\sum_{j=1}^{m} \log \vp_{1} \left(uf\left(\frac{(j-1)t}{m} \right)\right)\\
&= \int_{0}^{t} \log \vp_{1} \left(uf\left(s \right)\right)ds.
\end{align}
Hence, the result follows.
\end{proof}

Let
\begin{equation}
    \eta(x) \coloneqq \frac1x (1- e^{-x}) = \sum_{k=0}^{\infty} \frac{(-1)^{k}}{(k+1)!}x^{k},
\end{equation}
which is smooth and positive in $\mbbr$, and satisfies that
\begin{equation}
    \forall k\ge 0,\quad \p_x^k\eta(0)=\frac{(-1)^k}{k+1}.
\end{equation}

\begin{lem}
\label{lem:int_dist}
The conditional distribution of $Y_{t_j}$ given $Y_{t_{j-1}}=y$ 
under $\pr$ is
\begin{equation}\label{hm:trans.prob}
    \mcl(Y_{t_j}|Y_{t_{j-1}}=y) \overset{\theta}{=} S_\al^0\left(\beta,\,\sig_h,\, e^{-\lambda h} y+\frac{\mu}{\lambda}(1 - e^{-\lambda h})+ \mu_h\right),
\end{equation}
where
\begin{align}
\sig_h&=\sigma_{h}(\lam, \al,\sig)\coloneqq
h^{1/\alpha}(\eta(\lam \al h))^{1/\alpha}\sigma,\\
\mu_h &=\mu_{h}(\lam,\al, \sig, \be)\coloneqq 
\beta \sigma\left((h\eta(\lam \al h))^{1/\alpha}-h\eta(\lam h)\right)
t_{\al}.
\end{align}
\end{lem}

\begin{proof}
By the expression \eqref{eq:ssOU}, 
it suffices to show that 
\begin{equation}
    \sig \int_{t_{j-1}}^{t_j} e^{-\lam (t_j - u)} dJ_u \sim S_{\al}^{0}(\be,\sig_{h}, \mu_{h}).
\end{equation}
This can be seen as follows:
\begin{align}
&\log \E \left[ \exp\left \{ i u \left( \sig \int_{t_{j-1}}^{t_j} e^{-\lambda(t_j-s) }dJ_s\right) \right \} \right]\\
&= \int_{t_{j-1}}^{t_j} \log (\vp_{1}(u \sig e^{-\lambda (t_j-s)}; \al, \be))ds\\
&= \int_{t_{j-1}}^{t_j} \{ -|u\sig e^{-\lambda (t_j - s)}|^{\al} \left(1+i\be \sgn(u) t_{\al} (|u \sig e^{-\lambda (t_j - s)}|^{1-\al} - 1) \right)\}ds\\
&= -\int_{t_{j-1}}^{t_j}|u \sig e^{- \lam (t_j - s)}|^{\al}ds - i\beta\sgn(u) t_{\al} 
\nn\\
&{}\qquad \times
\left( \int_{t_{j-1}}^{t_j} |u \sig e^{-\lam(t_j - s)}|ds - \int_{t_{j-1}}^{t_j} |u \sig e^{-\lam(t_j - s)}|^{\al}ds \right)\\
&= -|u\sig|^{\al}h\eta(\lam \al h) - i\beta\sgn(u) t_{\al} \left \{ |u \sig| h \eta(\lam h) - |u\sig|^{\al}h\eta(\lam \al h) \right\}\\
&= -(|u\sig|h^{1/\al}\eta(\lam \al h)^{1/\al})^{\al} \left\{ 1+ i\beta \sgn(u) t_{\al} \left( \frac{|u\sig|h\eta(\lam h)}{|u\sig|^{\al}h\eta(\lam \al h)}-1 \right)\right\}\\
&=-(|u\sig|h^{1/\al}\eta(\lam \al h)^{1/\al})^{\al} \left\{ 1+ i\beta \sgn(u) t_{\al} \left((|u\sig|h^{1/\al}\eta(\lam \al h)^{1/\al})^{1 - \al}  -1 \right)\right\}\\
&{}\qquad-(|u\sig|h^{1/\al}\eta(\lam \al h)^{1/\al})^{\al}i\beta \sgn(u) t_{\al} 
\nn\\
&{}\qquad \times
\left( \frac{|u\sig|h\eta(\lam h)}{|u\sig|^{\al}h\eta(\lam \al h)}-(|u\sig|h^{1/\al}\eta(\lam \al h)^{1/\al})^{1-\al} \right)\\
&=-(|u\sig|h^{1/\al}\eta(\lam \al h)^{1/\al})^{\al} \left\{ 1+ i\beta \sgn(u) t_{\al} \left((|u\sig|h^{1/\al}\eta(\lam \al h)^{1/\al})^{1 - \al}  -1 \right)\right\}\\
&\qquad + i\beta \sgn(u) t_{\al} (|u\sig|h^{1/\al}\eta(\lam \al h)^{1/\al} - |u\sig|h\eta(\lam h))\\
&=-(|u\sig|h^{1/\al}\eta(\lam \al h)^{1/\al})^{\al} \left\{ 1+ i\beta \sgn(u) t_{\al} \left((|u\sig|h^{1/\al}\eta(\lam \al h)^{1/\al})^{1 - \al}  -1 \right)\right\}\\
&\qquad+ i\beta \sig t_{\al} (h^{1/\al} \eta(\lam \al h)^{1/\al} - h\eta(\lam h))u\\
&= i\mu_h u -(\sig_h |u|)^\al \left( 1+ i \beta \sgn(u) t_\al (|\sig_h u|^{1-\al} - 1)\right).
\end{align}
Comparing the last expression with \eqref{hm:0stable.cf} ends the proof.
\end{proof}

Since $\mu_h \rightarrow 0$ as $\lam \rightarrow 0$, by \eqref{hm:trans.prob} we have
\begin{equation}
    \mcl(Y_{t_j}|Y_{t_{j-1}}=y) \wc \mcl(y + \mu h + \sig J_h),\qquad \lam\to 0,
\end{equation}
which is consistent with Lemma \ref{hm:ssou_inc}.

\begin{rem}
\label{hm:rem_inv.law}
If $\lam>0$, then the following claims are immediate from Lemma \ref{lem:int_dist}, although we are concerned here with the fixed terminal sampling time $T$.
\begin{itemize}
    \item We have
    \begin{equation}
    \label{hm:rem_inv.law-1}
    Y_t \cil 
    S_{\al}^{0}\left(\be, \sig_{\infty}, \frac{\mu}{\lambda} + \mu_{\infty}\right),
    \qquad t\to\infty,
    \end{equation}
    where
    \begin{align}
        \sig_{\infty} &\coloneqq \lim_{t \rightarrow \infty} \sig_t=\sig\left( \frac{1}{\lam \al}\right)^{1/\al}, 
        \nn\\
        \mu_{\infty} &\coloneqq \lim_{t \rightarrow \infty}\mu_t= \be \sig \left\{ \left(\frac{1}{\lam \al }\right)^{1/\al} - \frac{1}{\lam} \right\} t_\al.
    \end{align}

    \item Moreover, $X$ is exponentially $\beta$-mixing and satisfies that
    \begin{equation}
    \label{hm:rem_inv.law-2}
        \forall q\in(0,\al_0),\quad \sup_{t\ge 0}\E[|X_t|^q]<\infty,
    \end{equation}
    and that for every $F\in L^1(\pi_\theta)$,
    \begin{equation}
    \label{hm:rem_inv.law-3}
        \frac1T \int_0^T F(X_t)dt \cip \int F(x) \pi_\theta(dx),\qquad T\to\infty,
    \end{equation}
    where $\pi_\theta$ denotes the asymptotic stable distribution in \eqref{hm:rem_inv.law-1}.
\end{itemize}
Indeed, from Lemma \ref{lem:int_dist},
    \begin{align}
    Y_t&= e^{-\lambda t}y_0 + \frac{\mu}{\lambda}(1 - e^{-\lambda t}) + \sig\int_{0}^{t} e^{- \lambda(t-u)}dJ_u\\
    &\sim  e^{-\lambda t}y_0 + \frac{\mu}{\lambda}(1 - e^{-\lambda t}) + S_{\al}^{0}(\be, \sig_t,\mu_t) \\
    & \cil \frac{\mu}{\lambda} + S_{\al}^{0}(\be, \sig_{\infty},\mu_{\infty})
    \end{align}
    for $t\to\infty$. 
    The remaining claims \eqref{hm:rem_inv.law-2} and \eqref{hm:rem_inv.law-3} are direct consequences of \cite[Proposition 5.4]{Mas13as}.
\end{rem}

\subsection{Exact Log-likelihood function}

For notational brevity, let
\begin{equation}
\overline{l} \coloneqq\log(1/h),\quad c \coloneqq \eta(\lam \al h)^{-1/\al},
\end{equation}
\begin{equation}
\label{hm:def_ep'}
\ep'_j(\theta) \coloneqq \frac{1}{\sig_h}\left( Y_{t_j} - e^{-\lambda h} Y_{t_{j-1}} - \frac{\mu}{\lambda}(1 - e^{-\lambda h})- \mu_h\right).
\end{equation}
For each $n$, the random variables $\ep'_1(\theta),\dots,\ep'_n(\theta)$ forms an $S_\al^0(\beta,1,0)$-i.i.d. array under $\pr$:
\begin{equation}
    \ep'_1(\theta),\dots,\ep'_n(\theta) \overset{\theta}{\sim} \text{i.i.d.}~S_\al^0(\beta,1,0).
\end{equation}
Recall the notation $\phi_{\al,\be}(x)$ in \eqref{hm:phi_def}. By Lemma \ref{lem:int_dist}, the exact log-likelihood function associated with $\mcl((Y_{t_j})_{j=0}^{n})$ under $\pr$ is given by
\begin{align}
\ell_n(\theta) &\coloneqq \sumj \log\left(\frac{1}{\sig_h}
\phi_{\al,\be}\left(\ep'_j(\theta)\right)\right)
\nn\\
&= \sumj \left\{ - \log\sig_h + \log \phi_{\al,\be}\left(\ep'_j(\theta)\right)\right\}\\
&=\sumj \left(-\log \sig + \frac{1}{\al}\overline{l} + \log c + \log \phi_{\al, \be}(\ep_j'(\theta))\right).
\label{hm:exact.log-lf}
\end{align}
Although $\ell_n(\theta)$ is explicitly given, its numerical optimization requires repeated numerical integration of $\phi_{\al, \be}$. Furthermore, the nonlinearity of $\mu_h$, $\sig_h$, and $c$ with respect to $\theta$ makes the optimization computationally intensive (See Section \ref{hm:sec_sim}). 
To partly circumvent this difficulty, we will introduce a small-time approximation based on the Euler scheme in the following section.

\section{Local asymptotics}
\label{hm:sec_main}

\subsection{Construction of Euler-type quasi-likelihood}

Let 
\begin{equation}
    \D_j\zeta := \zeta_j - \zeta_{j-1}
\end{equation}
for $\zeta=Y$ and $J$.
To centralize the location of $\Delta_j J$, we introduce the notation
\begin{equation}\label{hm:def_xi}
    \xi_h(\al)\coloneqq(1-h^{1-1/\al})t_{\al}.
\end{equation}
By Lemma \ref{hm:ssou_inc} and \eqref{hm:standard.Z}, for each $n\ge 1$, the sequence 
\begin{equation}\label{hm:def_ep0}
    \ep_{0,j}(\theta)\coloneqq h^{-1/\al}(\Delta_j J - \be h^{1/\al}\xi_h(\al))
    \qquad j=1,2,\dots,n
\end{equation}
forms an $S_\al^0(\beta,1,0)$-i.i.d. array under $\pr$. 
In what follows, we will write $\ep_{0,j}$, omitting the explicit dependence on $\theta$.

The Euler approximation of \eqref{hm:ou.sde} under $\pr$ is
\begin{align}\nn
    Y_{t_j} &\overset{\theta}{\approx} Y_{t_{j-1}} + (\mu-\lam Y_{t_{j-1}})h + \sig \D_j J,
\end{align}
which combined with Lemma \ref{hm:ssou_inc} gives the following formal distributional approximation of the conditional probability:
\begin{align}\label{hm:euler.qlf-1}
    \mcl(Y_{t_j}|Y_{t_{j-1}}) &\overset{\theta}{\approx} 
    S_\al^0\left(\beta,\, \sig h^{1/\al},\, 
    Y_{t_{j-1}} + (\mu - \lam Y_{t_{j-1}}) h
    + \beta \sig h^{1/\al} \xi_h(\al)
    \right).
\end{align}
Consequently, 
the variable
\begin{align}
\ep_{j}(\theta)
&\coloneqq \frac{1}{h^{1/\al}\sig}
\left( \D_j Y - (\mu - \lam Y_{t_{j-1}}) h - \beta \sig h^{1/\al} \xi_h(\al) \right)
\label{hm:def_E.ep}\\
&= \frac{1}{h^{1/\al}\sig}
\left( \D_j Y - (\mu - \lam Y_{t_{j-1}}) h\right) - \beta \xi_h(\al),
\nn
\end{align}
is expected to be close to $\ep_{j}'(\theta)$ of \eqref{hm:def_ep'} when $h$ is small. 
Now, the \textit{Euler-type quasi-likelihood function} based on the approximation \eqref{hm:euler.qlf-1} is defined as follows:
    \begin{align}
        \mbbh_{n}(\theta) &\coloneqq \sumj \log \left( \frac{1}{\sig h^{1/\al}}
        \phi_{\al,\beta} \left(\ep_{j}(\theta)\right) \right)\nn\\
        &= \sumj\left( - \log \sig + \al^{-1}\overline{l} + \log \phi_{\al, \be}(\ep_j(\theta))\right).
        \nn
    \end{align}
This gives an approximate version of the exact log-likelihood function $\ell_n(\theta)$.

\subsection{Some notation}

To state our asymptotic results, we need further notation. 
We write 
\begin{equation}
    r_n(\theta) = \sqrt{n} \,h^{1-1/\al}.
\end{equation}
and define the $C^1$-class mapping $\vp_n:\,\overline{\Theta}\to\mbbr^5\otimes\mbbr^5$ taking values in invertible matrices by
\begin{align}
\varphi_n(\theta)&\coloneqq \begin{pmatrix}
r_n(\theta)^{-1} &  & O \\
 & r_n(\theta)^{-1} &\\
& O & \frac{1}{\sqrt{n} \xi_h(\al)} \begin{pmatrix}
\varphi_{33,n}(\theta) & \varphi_{34,n}(\theta) & 0 \\
\varphi_{43,n}(\theta) & \varphi_{44,n}(\theta) & 0 \\
0 & 0 & 1
\end{pmatrix}
\end{pmatrix}\\
&=: \begin{pmatrix}
r_n(\theta)^{-1} I_2 & O \\
O & \frac{1}{\sqrt{n} \xi_h(\al)} \tilde{\varphi}_n(\theta) \\
\end{pmatrix},
\label{hm:def_vp}
\end{align}
where $O$ denotes the zero matrix of appropriate size and the function $\tilde{\varphi}_n$ satisfies the conditions given later; 
the matrix $\varphi_n(\theta)$ will be used to normalize the estimators.
Historically, asymptotic analysis based on non-diagonal matrix norming has been developed to handle the heterogeneity of data more theoretically efficiently; see, for example, \cite{Fah88}.

We will write $\rightarrow_u$ for the usual uniform convergence in $\theta\in\overline{\Theta}$. 
Note that $r_n(\theta)\gtrsim n^{(2-\al)/(2\al)} \to_u \infty$. 
We assume the following conditions regarding $\varphi_n(\theta)$:
there exist functions $\overline{\varphi}_{kl}$ such that
\begin{equation} \label{assump:rate_mat}
\begin{cases}
\varphi_{33,n}(\theta) \rightarrow_{u} \overline{\varphi}_{33}(\theta),\\[1mm]
\varphi_{34,n}(\theta) \rightarrow_{u} \overline{\varphi}_{34}(\theta),\\[1mm]
s_{43,n}(\theta) \coloneqq \al^{-2}\overline{l} \varphi_{33,n}(\theta) + \sig^{-1} \varphi_{43,n}(\theta)\rightarrow_{u} \overline{\varphi}_{43}(\theta),\\[1mm]
s_{44,n}(\theta)\coloneqq \al^{-2}\overline{l} \varphi_{34,n}(\theta) + \sig^{-1} \varphi_{44,n}(\theta) \rightarrow_{u} \overline{\varphi}_{44}(\theta),\\[1mm]
\inf_{\theta \in \overline{\Theta}}\overline{\varphi}_{33}(\theta)>0,\\[1mm]
\inf_{\theta \in \overline{\Theta}} \left|\overline{\varphi}_{33}(\theta) \overline{\varphi}_{44}(\theta) - \overline{\varphi}_{34} (\theta)\overline{\varphi}_{43}(\theta)\right| >0.
\end{cases}
\end{equation}
We note that these uniform-in-$\theta$ conditions are more or less standard in parametric estimation in the context of high-frequency sampling. Previously, they were employed in \cite{BroMas18}, \cite{CleGlo20}, and \cite{Mas23} to study the local asymptotic property of the likelihood associated with homoskedastic stochastic process models driven by a symmetric stable {\lp}. 
See also \cite{BroFuk16} for estimation of the fractional Brownian motion. 

Let us also introduce the related stochastic order symbols. Given continuous random functions $\zeta_{0}(\theta)$ and $\zeta_{n}(\theta)$, $n\ge 1$, and a positive sequence $(a_n)$, we will use the following notation.
\begin{itemize}
    \item We write
    \begin{equation}
        \zeta_{n}(\theta)\cip_u \zeta_{0}(\theta)
    \end{equation}
    if the joint distribution of $\zeta_n$ and $\zeta_0$ are well-defined under $\pr$ and if
    \begin{equation}
        \forall \ep>0,\quad 
        \pr[ |\zeta_{n}(\theta)-\zeta_{0}(\theta)|>\ep ] \to_{u} 0.
    \end{equation}
    \item $\zeta_n(\theta)=o_{u,p}(a_n)$ if $a_n^{-1}\zeta_{n}(\theta)\cip_u 0$;
    \item $\zeta_n(\theta)=O_{u,p}(a_n)$ if for every $\ep>0$ there exists a constant $M>0$ for which 
    \begin{equation}
        \sup_\theta\pr[|a_n^{-1}\zeta_n(\theta)| > M]<\ep.
    \end{equation}
    \item Finally, we write $\zeta_{n}(\theta)\cil_{u}\zeta_{0}(\theta)$ if
\begin{equation}
\left|\int f(x)P^{\zeta_{n}(\theta)}(dx) - \int f(x)P^{\zeta_{0}(\theta)}(dx)\right|\to_{u}0
\nonumber
\end{equation}
as $n\to\infty$ for every bounded uniformly continuous function $f$, where $P^{\mathsf{X}}$ denotes the distribution of a random variable $\mathsf{X}$.

\end{itemize}

\subsection{Main result}
\label{sec:main_result}

Recall that the parameter space $\Theta$ is a bounded convex domain whose closure satisfies \eqref{hm:Theta_closure}.
Let
    \begin{align}
        \D_n^\star(\theta) &:= \vp_n(\theta)^\top \p_\theta \ell_n(\theta), \nn\\
        \mci_n^\star(\theta) &:= -\vp_n(\theta)^\top \p_\theta^2 \ell_n(\theta) \vp_n(\theta).
    \end{align}
For a symmetric positive definite random matrix $A(\theta)\in\mbbr^p \otimes \mbbr^p$ possibly depending on $\theta$, we denote by $A(\theta)^{1/2}$ the positive definite square root of $A(\theta)$ and by $MN_{p,\theta}(0,A(\theta))$ the covariance mixture of $p$-dimensional normal distribution corresponding to the characteristic function
\begin{equation}
    \mbbr^p \ni u\mapsto\E\left[\exp\left(-\frac12 u^\top A(\theta)u\right)\right].
\end{equation}
We are now in a position to state the main result of this paper.

\begin{thm} \label{thm:main_result}
We have the following asymptotic properties.
\begin{enumerate}
\item 
For the positive definite random matrix $\mci(\theta)$ defined in \eqref{hm:def_FIm} below, we have the joint convergence
\begin{equation}\label{hm:joint.conv_Del-I}
    \left( \mci^\star_n(\theta),\, \Delta^\star_n(\theta) \right)
    \cil_u \big( \mci(\theta),\, \mci(\theta)^{1/2}Z \big),
\end{equation}
where $Z$ is a standard normal random vector in $\mbbr^5$ which is defined on an extension of the original probability space and independent of $\mci(\theta)$; hence, in particular, $\Delta^\star_n(\theta) \cil_u MN_{5,\theta}(0, \mci(\theta))$. 
Moreover,
\begin{equation}
\left|\ell_n(\theta + \varphi_{n}(\theta)v) - \ell_n(\theta) - \left( v^{\top} \Delta^\star_n(\theta) - \frac{1}{2} v^{\top} \mci^\star_n(\theta) v \right) \right| \cip_u 0
\end{equation}
for each $v\in\mbbr^5$.

\item With $\pr$-probability tending to $1$, there exist local maxima $\tes^\star$ and $\tes$ of $\ell_n(\theta)$ and $\mbbh_n(\theta)$, respectively, for which
\begin{align}
& \varphi_n(\theta)^{-1} (\tes^\star - \theta) \cil_u MN_{5,\theta}(0, \mci(\theta)^{-1}),
\nn\\
& \left|
\varphi_n(\theta)^{-1} (\tes^\star - \theta) - \varphi_n(\theta)^{-1} (\tes - \theta)\right| \cip_u 0.
\end{align}

\item For any estimator $\theta'_n$ of $\theta$ such that $\vp_n(\theta)^{-1}(\theta'_n - \theta) \cil_u \nu_\theta$ for some distribution $\nu_\theta$, and for any symmetric convex Borel subset $A \subset \mbbr^5$, we have
\begin{equation}
    F_\theta(A) \ge \nu_\theta(A)
\end{equation}
for every $\theta\in\Theta$, where $F_\theta := MN_{5,\theta}(0, \mci(\theta)^{-1})$. 

\end{enumerate}
\end{thm}

\medskip

The proof of Theorem \ref{thm:main_result} is technical and is deferred to Section \ref{sec:main_result_proof}. 
Below, we give a couple of remarks on Theorem \ref{thm:main_result}.
    
\begin{enumerate}[label=\arabic*.]

    \item Theorem \ref{thm:main_result}(1) claims the local asymptotic mixed normality (LAMN). It establishes the asymptotic optimality of all regular estimators. See \cite{Jeg95} for a related detailed account.
    \item 
    As in \cite{BroFuk16}, \cite{BroMas18}, \cite{CleGlo20}, and \cite{Mas23}, the non-diagonality condition of the normalizing matrix $\vp_n(\theta)$ is essential for the non-degeneracy of $\mci(\theta)$. Indeed, we could prove that $\mci(\theta)$ is constantly (for all $\theta\in\Theta$) singular when $\vp_n(\theta)$ is diagonal: $\vp_n(\theta)=\diag(\vp_{1,n}(\theta),\dots,\vp_{5,n}(\theta))$ with $\vp_{k,n}(\theta)\downarrow 0$.
%
    It is the joint estimation of the index $\al$ and the scale $\sig$ that gives rise to the asymptotic singularity and necessitates the non-diagonality of $\vp_n(\theta)$. Under Nolan's $S^0$-parametrization \cite{Nol98}, the off-diagonal components of the rate matrix associated with the skewness parameter $\beta$ and the other parameters can all be zero.

    \item Since $\cip_u$ implies $\cil_u$ as was noted in \cite[Lemma 2]{Swe80}, Theorem \ref{thm:main_result}(2) implies
    \begin{equation}
        \varphi_n(\theta)^{-1} (\tes - \theta) \cil_u MN_{5,\theta}(0, \mci(\theta)^{-1}),
    \end{equation}
    so that the local maxima of $\mbbh_n$ is asymptotically efficient as well in the sense of Theorem \ref{thm:main_result}(3).

    \item Thanks to \eqref{hm:joint.conv_Del-I} and \eqref{hm:outline-6} in Section \ref{sec:main_result_proof}, Studentization is straightforward:
    \begin{equation}
    \mci_n(\tes)^{1/2}\vp_n(\tes)^{-1}(\tes-\tz) \cil N_5(0,I_5).
    \end{equation}

    \item 
    The restriction $\al \in (1,2)$ came from the singular asymptotic behavior of the key quantity $\xi_h(\al)$ at $\al=1$; recall the definition \eqref{hm:def_xi}. 
    This restriction would be a technical one. We will not delve into the details of the case $\alpha\in (0,1]$ in this paper. However, we conjecture that our results associated with the true log-likelihood $\ell_n(\theta)$ would hold even for $\alpha\in (0,1]$ and that the employed continuous parametrization has a stablizing effect on the estimator performance for $\al \approx 1$ with $\beta\ne 0$.
    
    \item 
    \cite{IvaKulMas15} studied the LAN property of the possibly skewed locally stable {\lp}. In our genuine stable driven case, it corresponds to the case where $\lam=0$ and where $(\al,\beta)$ is known. They showed that the rate matrix may be asymmetric if $\beta\ne 0$. It is expected that this difference from \cite{IvaKulMas15} occurred because we are here employing the continuous parametrization.

    \item It should be noted that there is no universal definition of the skewness parameter within the context of L\'{e}vy processes. For example, \cite{Woe04} considered the LAN property of the skewness parameter defined by the exponential tilting of a (symmetric) L\'{e}vy density, namely, $\rho\in\mbbr$ when the L\'{e}vy measure is given by $\nu(dz)=e^{\rho z}g(z)dz$ with $g$ being symmetric in $\mbbr\setminus\{0\}$. This is not the case in our skewed stable case. On the one hand, the LAN property at rate $\sqrt{T}$ in \cite{Woe04}, implying that consistent estimation of this $\rho$ essentially requires that $T=T_n\to\infty$. On the other hand, however, a fixed $T$ is enough in our skewed stable case.

    \item 
    Once again, we emphasize that the asymptotics given in Theorem \ref{thm:main_result} hold for \textit{each} $T>0$; for $\al>1$, a larger $T$ would lead to a better performance of $(\les,\mes)$ since the rate of convergence $r_n(\theta) = T^{1-1/\al} \,n^{1/\al-1/2}$. We could consider the case where $T=T_n\to\infty$ as well. However, the asymptotics of any estimators will essentially depend on the sign of the true value of $\lam$. Even if we assume the ergodicity (namely $\lam>0$: see \cite{Mas04} and \cite{Mas07}), the conventional asymptotic normality fails to hold without an extra device such as the self-weighting; we refer to \cite{Mas10ejs} for related details.
    
\end{enumerate}

\section{Simple estimator of noise parameters}
\label{hm:sec_ME}

In this section, we discuss how to construct a simple estimator of the noise parameters $(\al,\sig,\beta)$ based on moment fitting. We will follow a similar route to what was done in \cite{Kaw23_mt} and \cite{KawMas25}: we will first construct an estimator of $(\al,\sig)$ through symmetrization; then, we will provide an estimator of $\beta$ through centralization. We do not pursue the exact asymptotic distribution, but focus on the convergence rates. Still, we can use the estimators obtained in this section as initial values for the numerical optimization.

We have
\begin{equation} \label{eq: model2}
\Delta_j Y = \int_{t_{j-1}}^{t_j} (\mu - \lambda Y_s)ds + \sigma \Delta_j J,
\end{equation}
where $\sigma \Delta_jJ \sim S_{\alpha}^{0}(\beta, h^{1/\alpha}\sigma,0)$.
Throughout this section, we assume that
\begin{equation}\label{hm:q_range}
    q\in \left(0,\frac{\al}{2}\right).
\end{equation}

\subsection{Index and scale}

To eliminate the skewness of the asymmetric stable noise and to weaken the effect of the drift, we make use of symmetrization through the second-order increments:
\begin{align}
\Delta_j^{2}Y &:= \D_j Y - \D_{j-1}Y \nn\\
&= \sigma \Delta_j^2 J 
+ \lam \left(\int_{t_{j-2}}^{t_{j-1}} (Y_s - Y_{t_{j-2}})ds - \int_{t_{j-1}}^{t_j} (Y_s - Y_{t_{j-2}})ds \right)
\nn\\
&=: \sigma \Delta_j^2 J + \lam A_j,
\end{align}
where $\sigma \Delta_j^2 J \sim S_{\alpha}^{0}(0, (2h)^{1/\alpha}\sigma, 0)$.

Since $q\in(0,1)$, applying the Jensen and the reverse-Jensen inequalities, we get
\begin{align}\label{sym_gap_order}
    & \frac{1}{n-1}\sum_{j=2}^{n}
    \left|\frac1h \int_{t_{j-2}}^{t_{j-1}} h^{-1/\al}(Y_s - Y_{t_{j-2}})ds \right|^q 
    \nn\\
    &\le \frac{1}{n-1}\sum_{j=2}^{n}
    \left( \frac1h \int_{t_{j-2}}^{t_{j-1}} \left|h^{-1/\al}(Y_s - Y_{t_{j-2}})\right| ds \right)^q 
    \nn\\
    &\le 
    \left(\frac{1}{n-1}\sum_{j=2}^{n}
    \frac1h \int_{t_{j-2}}^{t_{j-1}} \left| h^{-1/\al}(Y_s - Y_{t_{j-2}}) \right| ds \right)^q.
\end{align}
The random sequence inside the parentheses $(\dots)$ in the rightmost side is $O_p(1)$; obviously, the same estimate holds if we replace $\int_{t_{j-2}}^{t_{j-1}}$ by $\int_{t_{j-1}}^{t_{j}}$. 
Using these order estimate and the elementary inequality $||a+b|^q-|a|^q| \le |b|^q$ for any $a,b\in\mbbr$, we get
\begin{align}
& \left| \frac{1}{n-1} \sum_{j=2}^{n} \left| \frac{\Delta_j^2 Y}{(2h)^{1/\al}} \right|^q 
-
\frac{1}{n-1} \sum_{j=2}^{n} \left| \frac{\sig \Delta_j^2 J}{(2h)^{1/\al}} \right|^q \right| \nn\\
&\le 
\frac{1}{n-1} \sum_{j=2}^{n} \left| \frac{\lam A_j}{(2h)^{1/\al}} \right|^q 
= O_p(h^q) = o_p(1).
\label{hm:me-1}
\end{align}
Let $m_{1,q}(\alpha)$ denote the $q$-th degree moment of $S_{\alpha}^{0}(0,1,0)$.
From \cite{Kur01}, we know its explicit expression:
\begin{equation}
m_{1,q}(\alpha) := \frac{\Gamma(1-\frac{q}{\alpha})}{\Gamma(1-q)\cos\frac{q \pi}{2}}.
\end{equation}
From \eqref{hm:me-1}, and the fact that $\{\sig (2h)^{-1/\al}\Delta_j^2 J\}_{j=2}^{n}$ forms a $S_{\alpha}^{0}(0,\sig,0)$-i.i.d. array (allowing us to use the central limit theorem under \eqref{hm:q_range}), it follows that
\begin{align}
    & n^{q\wedge (1/2)} \left(
    \frac{1}{n-1} \sum_{j=2}^{n} \left| \frac{\Delta_j^2 Y}{(2h)^{1/\al}} \right|^q 
    -
    \sig^q \, m_{1,q}(\alpha)
    \right) \nn\\
    &= n^{q\wedge (1/2)-q}\times n^{q}\left(
    \frac{1}{n-1} \sum_{j=2}^{n} \left| \frac{\Delta_j^2 Y}{(2h)^{1/\al}} \right|^q 
    -
    \frac{1}{n-1} \sum_{j=2}^{n}\left| \frac{\sig \Delta_j^2 J}{(2h)^{1/\al}} \right|^q
    \right) \nn\\
    &{} \qquad + n^{q\wedge (1/2)-1/2}\times \sqrt{n} \left(
    \frac{1}{n-1} \sum_{j=2}^{n}\left| \frac{\sig \Delta_j^2 J}{(2h)^{1/\al}} \right|^q
    - \sig^q \, m_{1,q}(\alpha) \right)
    \nn\\
    &= O_p(n^{q\wedge (1/2)-q}) + O_p(n^{q\wedge (1/2)-1/2})=O_p(1).
\end{align}

The parameter $\alpha$ can be consistently estimated by solving the following moment ratio equation, which is independent of $\sigma$:
\begin{equation}
\frac{(\frac{1}{n} \sum_{j=2}^{n} |\Delta_j^{2}Y|^{q})^2}{\frac{1}{n} \sum_{j=2}^{n} |\Delta_j^{2}Y|^{2q}} 
\approx \frac{m_{1,q}(\alpha)}{m_{1,2q}(\alpha)} = \frac{\Gamma(1-2q)}{\Gamma(1- \frac{2q}{\alpha})} \left\{ \frac{\Gamma(1- \frac{q}{\alpha})}{\Gamma(1-q)} \right\}^2\frac{\cos q \pi}{(\cos \frac{q \pi}{2})^2}.
\end{equation}
Denote the solution by $\tilde \alpha_n$. 
Then, we can consistently estimate $\sig$ by
\begin{equation}
    \tilde \sigma_n
    := \left(\frac{1}{m_{1,q}(\tilde \alpha_n)} \frac1n \sum_{j=2}^{n} \left| \frac{\Delta_j^2 Y}{(2h)^{1/\tilde{\al}_n}} \right|^q \right)^{1/q}.
\end{equation}
In practice, the moment exponent $q$ should be specified beforehand by the user.

From the above arguments, combined with the continuous mapping theorem and the delta method, it is easy to see that
\begin{equation}\label{hm:al0.sig0_tightness}
    \left(n^{q\wedge (1/2)}(\tilde{\al}_n -\al),\, \frac{n^{q\wedge (1/2)}}{\overline{l}}(\tilde{\sig}_n^q -\sig^q)\right) =O_p(1).
\end{equation}

\begin{rem}
\eqref{hm:al0.sig0_tightness} implies not only the consistency of $(\tilde \al_n, \tilde \sig_n)$ but also that
\begin{equation}
    \log(h^{1/\al-1/\tilde{\al}_n}) = -\log(1/h)(\tilde{\al}_n - \al) /(\tilde{\al}_n \al) = o_p(1).
\end{equation}
Hence, for any estimator $\tilde{\gam}_n$ of $\gam$ such that
\begin{equation}
    \tilde{u}_{\gam,n} := \sqrt{n} \,h^{1-1/\al}(\tilde{\gam}_n -\gam) = O_p(1),
\end{equation}
we have
\begin{equation}
    \sqrt{n}\, h^{1-1/\tilde{\al}_n}(\tilde{\gam}_n -\gam) 
    = \tilde{u}_{\gam,n} \times h^{1/\al-1/\tilde{\al}_n}
    = \tilde{u}_{\gam,n} + o_p(1).
\end{equation}    
\end{rem}

\begin{rem}
The rate of convergence given in \eqref{hm:al0.sig0_tightness} is far from being the best possible. With a slight restriction on the range of $q$, we could directly apply \cite[Theorem 2]{Tod13} to construct an estimator of $(\al,\sig)$ possessing the asymptotic normality.
\end{rem}

\subsection{Skewness parameter}

To estimate $\beta$, we centralize the increments of \eqref{eq: model2} using a second-order difference operator:

\begin{align}
\Delta_j^{C} Y&:=\Delta_{j}Y + \Delta_{j-2}Y -2\Delta_{j-1}Y \\
&=\sig(\Delta_{j}J + \Delta_{j-2}J -2 \Delta_{j-1}J)\\
&{}\qquad + \lam \bigg( \int_{t_{j-2}}^{t_{j-1}} (Y_s - Y_{t_{j-3}})ds - \int_{t_{j-3}}^{t_{j-2}} (Y_s - Y_{t_{j-3}})ds \\
&{}\qquad  - \int_{t_{j-1}}^{t_{j-2}} (Y_s - Y_{t_{j-2}})ds + \int_{t_{j-2}}^{t_{j-1}} (Y_s - Y_{t_{j-2}})ds \bigg)\\
&\eqqcolon \sig \Delta_j^{C} J + \lam B_j.
\end{align}
Recalling that $\Delta_j J \sim S_{\al}^{0}(\beta, h^{1/\al}\sig, 0)$, we get
\begin{align}
&\sig \Delta_{j}^{C}J \sim S_{\alpha}^{0} \left(\left(\frac{2-2^{\alpha}}{2+2^{\alpha}}\right)\beta,(2+2^{\alpha})^{1/\alpha}h^{1/\alpha}\sigma, \left(\frac{2-2^{\alpha}}{2+2^{\alpha}}\right)\beta(2+2^{\alpha})^{1/\alpha}h^{1/\alpha}\sigma t_{\al} \right),
\end{align}
as seen from \eqref{hm:0stable.cf}.

By \eqref{sym_gap_order}, 
\begin{align}\label{hm:me-2}
& \left| \frac{1}{n-1} \sum_{j=3}^{n} \left| \frac{\Delta_j^{C} Y}{(2 + 2^{\al})^{1/\al} h^{1/\al}} \right|^q 
-
\frac{1}{n-1} \sum_{j=3}^{n} \left| \frac{\sig \Delta_j^C J}{(2 + 2^{\al})^{1/\al} h^{1/\al}} \right|^q \right| \nn\\
&\le 
\frac{1}{n-1} \sum_{j=3}^{n} \left| \frac{\lam B_j}{(2 + 2^{\al})^{1/\al} h^{1/\al}} \right|^q 
= O_p(h^q) = o_p(1).
\end{align}

Let $m_{2,q}(\al,\be)$ and $m_{2,\la q\ra}(\al,\be)$ denote, respectively, the $q$-th and $\la q \ra$-th degree moment of $F_{\al,\beta} \coloneqq S_{\alpha}^{0}\left (\left( \frac{2 - 2^{\al}}{2 + 2^{\al}} \right)\be,1 ,\left( \frac{2 - 2^{\al}}{2 + 2^{\al}} \right)\be t_{\al} \right)$.
From \cite{Kur01}, we know its explicit expression:
\begin{align}
m_{2,q}(\al,\be) &:= \frac{\Gamma(1-\frac{q}{\alpha})}{\Gamma(1-q)} \frac{1}{\left(\cos \eta(\al,\be) \right)^{q/\al}} \frac{\cos \frac{q \eta(\al,\be)}{2}}{\cos \frac{q \pi}{2}},\\
m_{2,\la q \ra}(\al,\be) &:= \frac{\Gamma(1-\frac{q}{\alpha})}{\Gamma(1-q)} \frac{1}{\left(\cos \eta(\al,\be) \right)^{q/\al}} \frac{\sin \frac{q \eta(\al,\be)}{2}}{\sin \frac{q \pi}{2}}.
\end{align}
From \eqref{hm:me-2} and the fact that $\sig \Delta_{j}^{C}J/ (2 + 2^{\al})^{1/\al}h^{1/\al}$ forms an 
$F_{\al,\beta}$-i.i.d. array, 
it follows that
\begin{align}
    n^{q\wedge (1/2)} \left(
    \frac{1}{n-2} \sum_{j=3}^{n} \left| \frac{\Delta_j^C Y}{(2 + 2^{\al})^{1/\al} h^{1/\al}} \right|^q 
    -
    \sig^q \, m_{2,q}(\al,\be)
    \right) &= O_p(1),\\
    n^{q\wedge (1/2)} \left(
    \frac{1}{n-2} \sum_{j=3}^{n} \left| \frac{\Delta_j^C Y}{(2 + 2^{\al})^{1/\al} h^{1/\al}} \right|^{\la q \ra} 
    -
    \sig^q \, m_{2,\la q \ra}(\al,\be)
    \right) &= O_p(1).
\end{align}
By the delta method, we get
\begin{equation}
    n^{q\wedge (1/2)} \left(
    \frac{\sum_{j=3}^{n} |\Delta_j^{C}Y|^{\la q \ra}}{\sum_{j=3}^{n} |\Delta_j^{C}Y|^{q}}  - \frac{m_{2,\la q \ra}(\al,\be)}{m_{2,q}(\al,\be)}
    \right) = O_p(1).
\end{equation}
By \eqref{hm:al0.sig0_tightness} and the smoothness of $(\al,\beta) \mapsto (m_{2,q}(\al,\be),m_{2,\la q \ra}(\al,\be))$, we have
\begin{equation}
    n^{q\wedge (1/2)} \left(
    \frac{\sum_{j=3}^{n} |\Delta_j^{C}Y|^{\la q \ra}}{\sum_{j=3}^{n} |\Delta_j^{C}Y|^{q}}  - \frac{m_{2,\la q \ra}(\tilde \al_n,\be)}{m_{2,q}(\tilde \al_n,\be)}
    \right) = O_p(1).
\end{equation}
Now, denote by $\tilde \beta_n$ the solution to the following moment ratio equation:
\begin{equation}
\frac{\sum_{j=3}^{n} |\Delta_j^{C}Y|^{\la q \ra}}{\sum_{j=3}^{n} |\Delta_j^{C}Y|^{q}} 
\approx \frac{m_{2,\la q \ra}(\tilde \al_n,\be)}{m_{2,q}(\tilde \al_n,\be)} = \frac{\tan \frac{q\eta(\tilde \al_n,\be)}{2}}{\tan \frac{q\pi}{2}}.
\end{equation}
By means of the continuous mapping theorem and the delta method, we see that
\begin{equation}
    n^{q\wedge (1/2)} (\tilde{\be}_n -\be) =O_p(1),
\end{equation}
hence in particular, the consistency of $\tilde \be_n$.


\begin{rem}
The estimation methods for $(\al,\sig,\beta)$ discussed so far are directly applicable to the case of a more general drift:
\begin{equation}
    Y_t = Y_0 + \int_0^t \mu_s ds + \sig J_t,
\end{equation}
where $\mu$ is an unobserved (nuisance) adaptive process satisfying mild conditions. This is evident since we did not require any specific structure for the drift part in the above arguments.
\end{rem}

\begin{rem}
In cases where $T=T_n\to\infty$ and $X$ is ergodic ($\lam>0$) and stationary, that is, $Y_t \sim \pi_\theta \coloneqq S_{\al}^{0}\left(\be, \sig_{\infty},\mu_{\infty}' +\frac{\mu}{\lambda}\right)$ for each $t\ge 0$ under $\pr$, then it is possible to consider the method of moments based on the estimating equation (Remark \ref{hm:rem_inv.law})
\begin{equation}
    \frac1n \sumj \psi(Y_{t_{j-1}}) \approx \int \psi(y) \pi_\theta(dy)
\end{equation}
for suitable $\psi$.
\end{rem}

\section{Numerical experiments}
\label{hm:sec_sim}

\subsection{Design}

    %
    

To observe the finite sample performance of the proposed estimators, we conducted numerical experiments for parameter estimation based on $\ell_n(\theta)$ and $\mbbh_n(\theta)$, together with the method of moments presented in Section \ref{hm:sec_ME}.
The simulation design is as follows:

\begin{itemize}
\item The terminal sampling time $T=1,100$.
\item Initial value of stochastic process $Y_0=0$.
\item True values: $\lam =1,\ \mu=2,\ \al \in \left\{1.0, 1.01, 1.5, 1.8\right\},\ \sig=5,\ \be=0.5$, with initial value $Y_0 = 0$; just for reference, we also ran the case $\al=1.0$.
\item Sample size: $n=500,1000,2000$.
\item The degree of the moment: $q=0.2$.
\item The number of the Monte Carlo simulations: $L=1000$.
\end{itemize}
Here, we additionally run simulations for $\al=1$, although we derived theoretical results under $\al > 1$. 

To compute estimators, we first estimated $\al,\sig$, and $\be$ by the method of moments; the estimates can be obtained in a flash. Then, we optimized the objective functions to obtain the MLE and Euler-QMLE by maximizing $\ell_n(\theta)$ or $\mbbh_n(\theta)$ by using the BFGS algorithm implemented in the \texttt{optim} function of R; for the initial values for the numerical optimization, we used the value $(\lambda, \mu) = (2,3)$ for the drift parameters together with the moment estimators for the noise parameters. In addition, to reduce the computation time, we performed the parallel computation using the \texttt{mcapply} function provided by the \texttt{parallel} package. For the calculation, we used a laptop with an Apple M1, 16GB of memory, and 7 CPU cores.

For the generation of random samples and probability density functions, we used the \texttt{libstable4u} package and evaluated the probability density function by the \texttt{stable\_pdf}.
In the special case $\al=1$, where the \texttt{libstable4u} package did not work suitably, the random samples were generated by the CMS method given in \cite{Cha76} and the probability density function was evaluated by the \texttt{dstable} function in the \texttt{stabledist} package.

\subsection{Tables}

Tables \ref{tab:mle-vs-euler-1.0-T1}-\ref{tab:mle-vs-euler-1.8-T1} report the finite-sample performance of the MLE and the Euler-QMLE for various values of $\alpha$ with $T=1$, 
while Tables \ref{tab:mle-vs-euler-1.0}-\ref{tab:mle-vs-euler-1.8} present the corresponding results for $T=100$.

In each Table, \textit{Time} denotes the average computation time for getting one set of estimates over the 1000 Monte Carlo simulations. We observed the following.
\begin{itemize}
    \item Across all cases, the MLE and the Euler-QMLE exhibited similar performance. Both estimators showed negligible bias in most scenarios and converged to the true parameters as $n$ increased, demonstrating their consistency. 
    
    \item As expected from the theoretical findings, the precision of $(\lambda, \mu)$ remains poor unless $T$ is large. Furthermore, we observed upward bias in the estimation of $\sigma$ unless $n$ is sufficiently large.
    
    \item Regarding the standard deviations, those for $\mu$, $\al$, and $\be$ are broadly comparable between the two methods, while the MLE remains preferable for $\lam$ and $\sig$ in terms of its smaller bias. 
    
    \item Computation time for the Euler-QMLE is roughly two to ten times shorter than that of the full MLE for all $n$.
\end{itemize}

In sum, there is a natural trade-off: on the one hand, the MLE yields higher estimation accuracy, while it has higher computational costs; on the other hand, the Euler QMLE is more biased while having shorter computation time.
Still, the Euler-QMLE has both merits for sufficiently large $n$ (here, about $2000$).

\begin{table}[H]
\centering
\scriptsize
\setlength{\tabcolsep}{3pt}
\resizebox{\linewidth}{!}{
\begin{tabular}{c|cc|cc|cc|cc|cc|cc}
\hline
\multirow{2}{*}{$n$}
 & \multicolumn{2}{c|}{$\hat\lambda_n$}
 & \multicolumn{2}{c|}{$\hat\mu_n$}
 & \multicolumn{2}{c|}{$\hat\alpha_n$}
 & \multicolumn{2}{c|}{$\hat\sigma_n$}
 & \multicolumn{2}{c|}{$\hat\beta_n$}
 & \multicolumn{2}{c}{Time (s)} \\
 & MLE & Euler & MLE & Euler & MLE & Euler & MLE & Euler & MLE & Euler & MLE & Euler \\
\hline
500
 & 1.0189 & 1.0180
 & 2.2697 & 2.2365
 & 0.9988 & 0.9989
 & 6.3024 & 6.3090
 & 0.4980 & 0.4971
 & 43.47 & 42.93      \\
 & (0.2656) & (0.2651)
 & (2.7004) & (2.6014)
 & (0.0564) & (0.0568)
 & (19.9627) & (19.9868)
 & (0.0755) & (0.0776)
 &          &         \\
\hline
1000
 & 1.0174 & 1.0171
 & 2.3355 & 2.3456
 & 0.9959 & 0.9959
 & 5.8660 & 5.8736
 & 0.5007 & 0.5004
 & 89.25 & 108.91   \\
 & (0.2069) & (0.2065)
 & (2.1318) & (2.1972)
 & (0.0409) & (0.0409)
 & (11.3751) & (11.3808)
 & (0.0540) & (0.0555)
 &          &         \\
\hline
2000
 & 1.0072 & 1.0066
 & 2.4892 & 2.4915
 & 0.9938 & 0.9948
 & 5.6004 & 5.3321
 & 0.5018 & 0.5029
 & 112.16 & 145.97    \\
 & (0.1405) & (0.1399)
 & (1.7420) & (1.6535)
 & (0.0303) & (0.0248)
 & (5.2421) & (1.2992)
 & (0.0399) & (0.0355)
 &          &         \\
\hline
\end{tabular}}
\caption{True values: $\alpha=1.0$, $(\lambda,\mu,\sigma,\beta)=(1,2,5,0.5)$, $T=1$.}
\label{tab:mle-vs-euler-1.0-T1}
\end{table}

\begin{table}[H]
\centering
\scriptsize
\setlength{\tabcolsep}{3pt}
\resizebox{\linewidth}{!}{
\begin{tabular}{c|cc|cc|cc|cc|cc|cc}
\hline
\multirow{2}{*}{$n$}
 & \multicolumn{2}{c|}{$\hat\lambda_n$}
 & \multicolumn{2}{c|}{$\hat\mu_n$}
 & \multicolumn{2}{c|}{$\hat\alpha_n$}
 & \multicolumn{2}{c|}{$\hat\sigma_n$}
 & \multicolumn{2}{c|}{$\hat\beta_n$}
 & \multicolumn{2}{c}{Time (s)} \\
 & MLE & Euler & MLE & Euler & MLE & Euler & MLE & Euler & MLE & Euler & MLE & Euler \\
\hline
500
 & 1.0535 & 1.0498
 & 2.5022 & 2.5102
 & 1.0040 & 1.0048
 & 5.9835 & 5.7668
 & 0.5007 & 0.5020
 &  19.24    &  19.29       \\
 & (0.3219) & (0.3190)
 & (2.8059)  & (2.8386)
 & (0.0548) & (0.0532)
 & (10.1616) & (8.3449)
 & (0.0730) & (0.0700)
 &           &          \\
\hline
1000
 & 1.0193 & 1.0195
 & 2.5111 & 2.4366
 & 1.0048 & 1.0051
 & 5.5863 & 5.5952
 & 0.5024 & 0.5014
 & 34.97     & 31.59       \\
 & (0.1976) & (0.1974)
 & (2.4464)  & (2.0847)
 & (0.0374) & (0.0380)
 & (5.5092) & (5.6466)
 & (0.0524) & (0.0549)
 &           &          \\
\hline
2000
 & 1.0036 & 1.0026
 & 2.6578 & 2.6428
 & 1.0022 & 1.0024
 & 5.4416 & 5.3871
 & 0.5038 & 0.5038
 & 49.45   & 50.88      \\
 & (0.1571) & (0.1594)
 & (1.4847)  & (1.4762)
 & (0.0245) & (0.0233)
 & (1.9589) & (0.9241)
 & (0.0355) & (0.0326)
 &           &          \\
\hline
\end{tabular}}
\caption{True values: $\alpha=1.01$, $(\lambda,\mu,\sigma,\beta)=(1,2,5,0.5)$, $T=1$.}
\label{tab:mle-vs-euler-1.01-T1}
\end{table}

\begin{table}[H]
\centering
\scriptsize
\setlength{\tabcolsep}{3pt}
\resizebox{\linewidth}{!}{
\begin{tabular}{c|cc|cc|cc|cc|cc|cc}
\hline
\multirow{2}{*}{$n$}
 & \multicolumn{2}{c|}{$\hat\lambda_n$}
 & \multicolumn{2}{c|}{$\hat\mu_n$}
 & \multicolumn{2}{c|}{$\hat\alpha_n$}
 & \multicolumn{2}{c|}{$\hat\sigma_n$}
 & \multicolumn{2}{c|}{$\hat\beta_n$}
 & \multicolumn{2}{c}{Time (s)} \\
 & MLE & Euler & MLE & Euler & MLE & Euler & MLE & Euler & MLE & Euler & MLE & Euler \\
\hline
500
 & 2.0345 & 1.9437
 & 2.7498 & 2.5361
 & 1.5006 & 1.4962
 & 5.1058 & 5.1480
 & 0.5055 & 0.4941
 & 26.17   & 20.00      \\
 & (1.8317) & (1.7158)
 & (7.0510) & (7.1591)
 & (0.0681) & (0.0671)
 & (0.9411) & (0.9336)
 & (0.1191) & (0.1171)
 &          &          \\
\hline
1000
 & 1.8365 & 1.8067
 & 2.5613 & 2.5450
 & 1.5012 & 1.5013
 & 5.0520 & 5.0447
 & 0.5023 & 0.5025
 & 50.15 & 37.05       \\
 & (1.6518) & (1.5986)
 & (6.2549) & (6.1811)
 & (0.0481) & (0.0478)
 & (0.7101) & (0.7094)
 & (0.0875) & (0.0871)
 &          &          \\
\hline
2000
 & 1.5820 & 1.6037
 & 2.5000 & 2.5293
 & 1.4990 & 1.4994
 & 5.0554 & 5.0475
 & 0.5021 & 0.5020
 &  67.56  & 68.18       \\
 & (1.3842) & (1.3997)
 & (5.2342) & (5.2386)
 & (0.0344) & (0.0344)
 & (0.5635) & (0.5614)
 & (0.0616) & (0.0608)
 &          &          \\
\hline
\end{tabular}}
\caption{True values: $\alpha=1.5$, $(\lambda,\mu,\sigma,\beta)=(1,2,5,0.5)$, $T=1$.}
\label{tab:mle-vs-euler-1.5-T1}
\end{table}

\begin{table}[H]
\centering
\scriptsize
\setlength{\tabcolsep}{3pt}
\resizebox{\linewidth}{!}{
\begin{tabular}{c|cc|cc|cc|cc|cc|cc}
\hline
\multirow{2}{*}{$n$}
 & \multicolumn{2}{c|}{$\hat\lambda_n$}
 & \multicolumn{2}{c|}{$\hat\mu_n$}
 & \multicolumn{2}{c|}{$\hat\alpha_n$}
 & \multicolumn{2}{c|}{$\hat\sigma_n$}
 & \multicolumn{2}{c|}{$\hat\beta_n$}
 & \multicolumn{2}{c}{Time (s)} \\
 & MLE & Euler & MLE & Euler & MLE & Euler & MLE & Euler & MLE & Euler & MLE & Euler \\
\hline
500
 & 4.0254 & 3.9328
 & 4.3578 & 3.8034
 & 1.7866 & 1.7912
 & 5.1647 & 5.1027
 & 0.4699 & 0.4767
 &  23.49 & 23.05       \\
 & (3.4876) & (3.3354)
 & (15.0040) & (14.5220)
 & (0.0613) & (0.0586)
 & (0.5906) & (0.5539)
 & (0.2338) & (0.2239)
 &          &          \\
\hline
1000
 & 3.6723 & 3.6314
 & 4.0432 & 3.4872
 & 1.7970 & 1.7970
 & 5.0584 & 5.0476
 & 0.5022 & 0.5008
 &  48.71   & 43.27       \\
 & (3.4117) & (3.3793)
 & (13.3460) & (13.2150)
 & (0.0425) & (0.0430)
 & (0.4352) & (0.4420)
 & (0.1683) & (0.1706)
 &          &          \\
\hline
2000
 & 3.2226 & 3.1905
 & 3.1643 & 3.1563
 & 1.7990 & 1.8000
 & 5.0269 & 5.0101
 & 0.5025 & 0.5062
 &  81.33  & 81.20     \\
 & (3.0107) & (2.9563)
 & (11.5850) & (11.3060)
 & (0.0315) & (0.0306)
 & (0.3570) & (0.3410)
 & (0.1246) & (0.1209)
 &          &          \\
\hline
\end{tabular}}
\caption{True values: $\alpha=1.8$, $(\lambda,\mu,\sigma,\beta)=(1,2,5,0.5)$, $T=1$.}
\label{tab:mle-vs-euler-1.8-T1}
\end{table}

\begin{table}[H]
\centering
\scriptsize
\setlength{\tabcolsep}{3pt}
\resizebox{\linewidth}{!}{
\begin{tabular}{c|cc|cc|cc|cc|cc|cc}
\hline
\multirow{2}{*}{$n$}
 & \multicolumn{2}{c|}{$\hat\lambda_n$}
 & \multicolumn{2}{c|}{$\hat\mu_n$}
 & \multicolumn{2}{c|}{$\hat\alpha_n$}
 & \multicolumn{2}{c|}{$\hat\sigma_n$}
 & \multicolumn{2}{c|}{$\hat\beta_n$}
 & \multicolumn{2}{c}{Time (s)} \\
 & MLE & Euler & MLE & Euler & MLE & Euler & MLE & Euler & MLE & Euler & MLE & Euler \\
\hline
500
 & 1.0000 & 0.9064
 & 1.7983 & 1.7241
 & 0.9979 & 0.9870
 & 5.0788 & 4.7769
 & 0.5014 & 0.4963
 & 54.08  & 44.39 \\
 & (0.0136) & (0.0108)
 & (0.5960) & (0.6479)
 & (0.0535) & (0.0725)
 & (0.5829) & (0.9639)
 & (0.0717) & (0.0938)
 &         &        \\
\hline
1000
 & 1.0000 & 0.9516
 & 1.8731 & 1.8534
 & 0.9959 & 0.9886
 & 5.0832 & 4.9825
 & 0.5022 & 0.4992
 & 67.97  & 96.41 \\
 & (0.0084) & (0.0075)
 & (0.5558) & (0.5970)
 & (0.0385) & (0.0533)
 & (0.5189) & (0.9934)
 & (0.0497) & (0.0666)
 &         &        \\
\hline
2000
 & 1.0001 & 0.9755
 & 1.9584 & 1.9596
 & 0.9976 & 0.9926
 & 5.0736 & 5.0787
 & 0.5020 & 0.4993
 & 193.79 & 200.61 \\
 & (0.0057) & (0.0053)
 & (0.5278) & (0.5620)
 & (0.0274) & (0.0390)
 & (0.4644) & (1.0340)
 & (0.0365) & (0.0467)
 &         &        \\
\hline
\end{tabular}}
\caption{True values: $\alpha=1.0$, $(\lambda,\mu,\sigma,\beta)=(1,2,5,0.5)$, $T=100$.}
\label{tab:mle-vs-euler-1.0}
\end{table}

\begin{table}[H]
\centering
\scriptsize
\setlength{\tabcolsep}{3pt}
\resizebox{\linewidth}{!}{%
\begin{tabular}{c|cc|cc|cc|cc|cc|cc}
\hline
\multirow{2}{*}{$n$}
 & \multicolumn{2}{c|}{$\hat\lambda_n$}
 & \multicolumn{2}{c|}{$\hat\mu_n$}
 & \multicolumn{2}{c|}{$\hat\alpha_n$}
 & \multicolumn{2}{c|}{$\hat\sigma_n$}
 & \multicolumn{2}{c|}{$\hat\beta_n$}
 & \multicolumn{2}{c}{Time (s)} \\
 & MLE & Euler & MLE & Euler & MLE & Euler & MLE & Euler & MLE & Euler & MLE & Euler \\
\hline
500
 & 1.0005 & 0.9068
 & 2.0867 & 1.8721
 & 0.9991 & 1.0025
 & 5.2592 & 4.6529
 & 0.4982 & 0.5015
 & 55.69  & 22.19 \\
 & (0.0129) & (0.0108)
 & (0.7410) & (0.5533)
 & (0.0693) & (0.0588)
 & (1.3303) & (0.6268)
 & (0.0764) & (0.0714)
 &          &        \\
\hline
1000
 & 0.9998 & 0.9068
 & 2.1233 & 1.8721
 & 1.0010 & 1.0025
 & 5.2789 & 4.6529
 & 0.5002 & 0.5015
 & 115.37 & 37.07 \\
 & (0.0085) & (0.0108)
 & (0.8341) & (0.5533)
 & (0.0559) & (0.0588)
 & (1.5185) & (0.6268)
 & (0.0589) & (0.0714)
 &          &        \\
\hline
2000
 & 0.9998 & 0.9753
 & 2.1523 & 2.0802
 & 1.0021 & 1.0026
 & 5.2780 & 5.0380
 & 0.5007 & 0.5010
 & 237.12 & 83.92 \\
 & (0.0058) & (0.0056)
 & (0.8394) & (0.5356)
 & (0.0442) & (0.0309)
 & (1.6871) & (0.5819)
 & (0.0426) & (0.0386)
 &          &        \\
\hline
\end{tabular}}
\caption{True values:$\alpha=1.01$, $(\lambda,\mu,\sigma,\beta)=(1,2,5,0.5)$, $T=100$.}
\label{tab:mle-vs-euler-1.01}
\end{table}

\begin{table}[H]
\centering
\scriptsize
\setlength{\tabcolsep}{3pt}
\resizebox{\linewidth}{!}{%
\begin{tabular}{c|cc|cc|cc|cc|cc|cc}
\hline
\multirow{2}{*}{$n$}
 & \multicolumn{2}{c|}{$\hat\lambda_n$}
 & \multicolumn{2}{c|}{$\hat\mu_n$}
 & \multicolumn{2}{c|}{$\hat\alpha_n$}
 & \multicolumn{2}{c|}{$\hat\sigma_n$}
 & \multicolumn{2}{c|}{$\hat\beta_n$}
 & \multicolumn{2}{c}{Time (s)} \\
 & MLE & Euler & MLE & Euler & MLE & Euler & MLE & Euler & MLE & Euler & MLE & Euler \\
\hline
500
 & 1.0151 & 0.9123
 & 2.0548 & 1.8203
 & 1.5018 & 1.5006
 & 5.0021 & 4.5351
 & 0.5040 & 0.5043
 & 28.14  & 9.94 \\
 & (0.1023) & (0.0562)
 & (0.7200) & (0.5939)
 & (0.0683) & (0.0699)
 & (0.2866) & (0.2756)
 & (0.1246) & (0.1210)
 &        &       \\
\hline
1000
 & 1.0086 & 0.9572
 & 2.0359 & 1.9350
 & 1.5006 & 1.5004
 & 5.0108 & 4.7499
 & 0.5039 & 0.5025
 & 58.16  & 20.68 \\
 & (0.0729) & (0.0539)
 & (0.6759) & (0.6389)
 & (0.0479) & (0.0480)
 & (0.2481) & (0.3350)
 & (0.0874) & (0.0877)
 &        &       \\
\hline
2000
 & 1.0095 & 0.9792
 & 2.0644 & 1.9882
 & 1.5000 & 1.4993
 & 5.0090 & 4.8822
 & 0.5025 & 0.5030
 & 651.66 & 45.34 \\
 & (0.0849) & (0.0512)
 & (0.6936) & (0.6350)
 & (0.0361) & (0.0339)
 & (0.2228) & (0.2475)
 & (0.0636) & (0.0633)
 &        &       \\
\hline
\end{tabular}}
\caption{True values: $\alpha=1.5$, $(\lambda,\mu,\sigma,\beta)=(1,2,5,0.5)$, $T=100$.}
\label{tab:mle-vs-euler-1.5}
\end{table}

\begin{table}[H]
\centering
\scriptsize
\setlength{\tabcolsep}{3pt}
\resizebox{\linewidth}{!}{%
\begin{tabular}{c|cc|cc|cc|cc|cc|cc}
\hline
\multirow{2}{*}{$n$}
 & \multicolumn{2}{c|}{$\hat\lambda_n$}
 & \multicolumn{2}{c|}{$\hat\mu_n$}
 & \multicolumn{2}{c|}{$\hat\alpha_n$}
 & \multicolumn{2}{c|}{$\hat\sigma_n$}
 & \multicolumn{2}{c|}{$\hat\beta_n$}
 & \multicolumn{2}{c}{Time (s)} \\
 & MLE & Euler & MLE & Euler & MLE & Euler & MLE & Euler & MLE & Euler & MLE & Euler \\
\hline
500
 & 1.0262 & 0.9259
 & 2.0996 & 1.8348
 & 1.7891 & 1.7902
 & 5.0248 & 4.3619
 & 0.4694 & 0.4807
 & 114.34 & 9.89 \\
 & (0.1234) & (0.1030)
 & (0.7840) & (0.6553)
 & (0.0604) & (0.0604)
 & (0.2105) & (0.6052)
 & (0.2318) & (0.2250)
 &        &       \\
\hline
1000
 & 1.0259 & 0.9710
 & 2.0872 & 1.9770
 & 1.7969 & 1.7968
 & 5.0151 & 4.6728
 & 0.4978 & 0.4970
 & 196.85  & 24.75 \\
 & (0.1262) & (0.1070)
 & (0.7658) & (0.7172)
 & (0.0427) & (0.0430)
 & (0.1754) & (0.5485)
 & (0.1715) & (0.1819)
 &        &       \\
\hline
2000
 & 1.0389 & 0.9954
 & 2.1198 & 2.0546
 & 1.7981 & 1.7985
 & 5.0167 & 4.8709
 & 0.5020 & 0.5046
 & 363.93 & 61.85 \\
 & (0.1670) & (0.1077)
 & (0.7707) & (0.7423)
 & (0.0339) & (0.0319)
 & (0.1665) & (0.2707)
 & (0.1318) & (0.1298)
 &        &       \\
\hline
\end{tabular}}
\caption{True values: $\alpha=1.8$, $(\lambda,\mu,\sigma,\beta)=(1,2,5,0.5)$, $T=100$.}
\label{tab:mle-vs-euler-1.8}
\end{table}

\subsection{Histograms}

Figures \ref{fig:param-by-n_mle-vs-euler-1.0}-\ref{fig:param-by-n_mle-vs-euler-1.8} display histograms of the normalized estimators. 
Each estimator is normalized by its theoretical rate of convergence, as prescribed in \cite[Remark 2 following Theorem 1]{BroMas18}:
\begin{align}
&\les: \sqrt{n} h^{1-1/\al} (\les - \lam), \quad \mes: \sqrt{n} h^{1-1/\al} (\mes - \mu)\\
&\aes: \sqrt{n} (\aes - \al), \quad \ses: \frac{\sqrt{n}}{\overline{l}}(\ses - \sig), \quad \bes: \sqrt{n} (\bes - \be).
\end{align}
The histograms of the normalized estimators become closer to a Gaussian-like shape; as we have shown in Theorem \ref{thm:main_result}, this holds for the estimators of the noise parameters $(\al,\beta,\sig)$.
The MLE histograms are mostly peaked and unimodal around zero for all parameters, whereas those of the Euler-QMLE exhibit leftward biases for $\lam$ and $\sig$. 
For $\al$ and $\be$, we cannot see a clear difference between the two methods.
The normalized $\mu$-histograms are the widest when $\al$ is near 1, and they tighten as $\al$ increases; note that the scales of $\mu$-histograms in the four figures are rather different across the different $\al$-values. It should be noted that the normalized estimator $\hat{\lambda}$ follows a mixed-normal asymptotic distribution. Consequently, the distribution of $\hat{\lambda}$ is leptokurtic, that is, it exhibits heavier tails and a sharper peak at the origin than the standard normal distribution.



  \newcommand{\CellPair}[3]{%
    \begin{subfigure}[t]{\colw}
      \centering
      \begin{minipage}{0.49\linewidth}
        \centering
        \includegraphics[width=\linewidth]{#1}\\[-2pt]
        \scriptsize MLE
      \end{minipage}\hfill
      \begin{minipage}{0.49\linewidth}
        \centering
        \includegraphics[width=\linewidth]{#2}\\[-2pt]
        \scriptsize Euler
      \end{minipage}
      \caption{#3}
    \end{subfigure}%
  }
  
\begin{figure}[t]
  \centering
  \captionsetup[subfigure]{justification=centering}
  \newcommand{\colw}{0.32\linewidth} 

  \newcommand{\rowtitle}[1]{\par\medskip\textbf{ #1}\par\smallskip}

  \rowtitle{$\lambda$}
  \CellPair{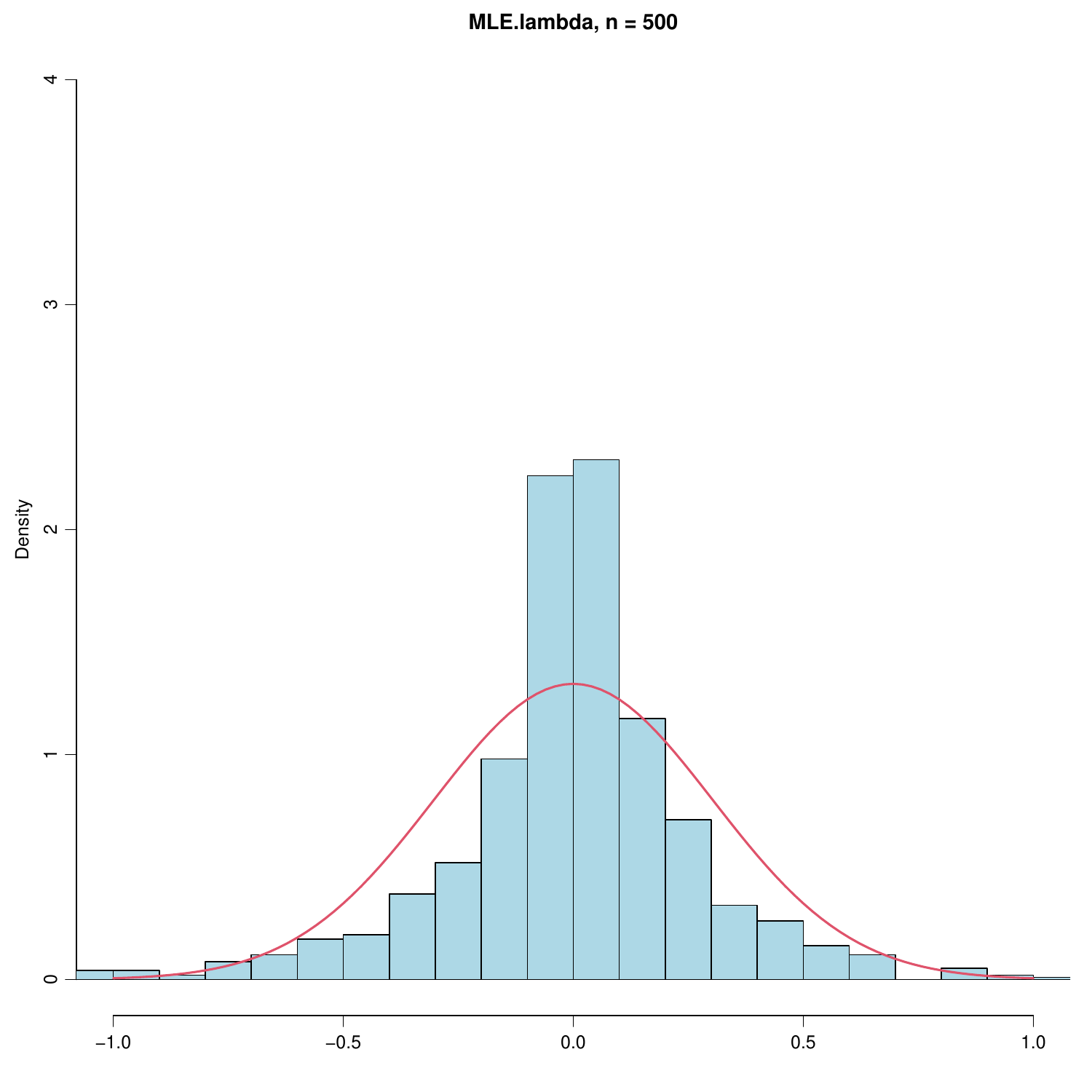}{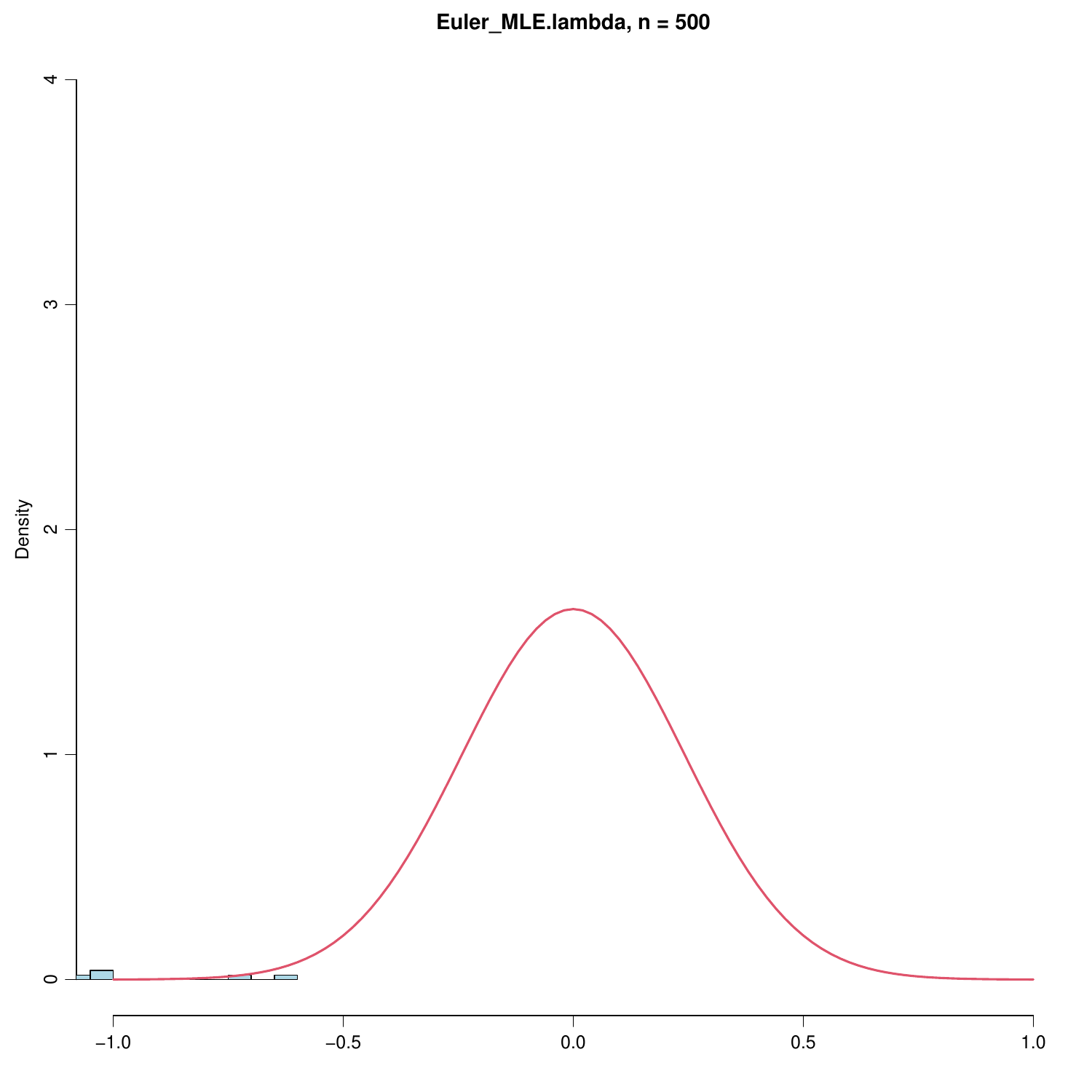}{$n=500$}\hfill
  \CellPair{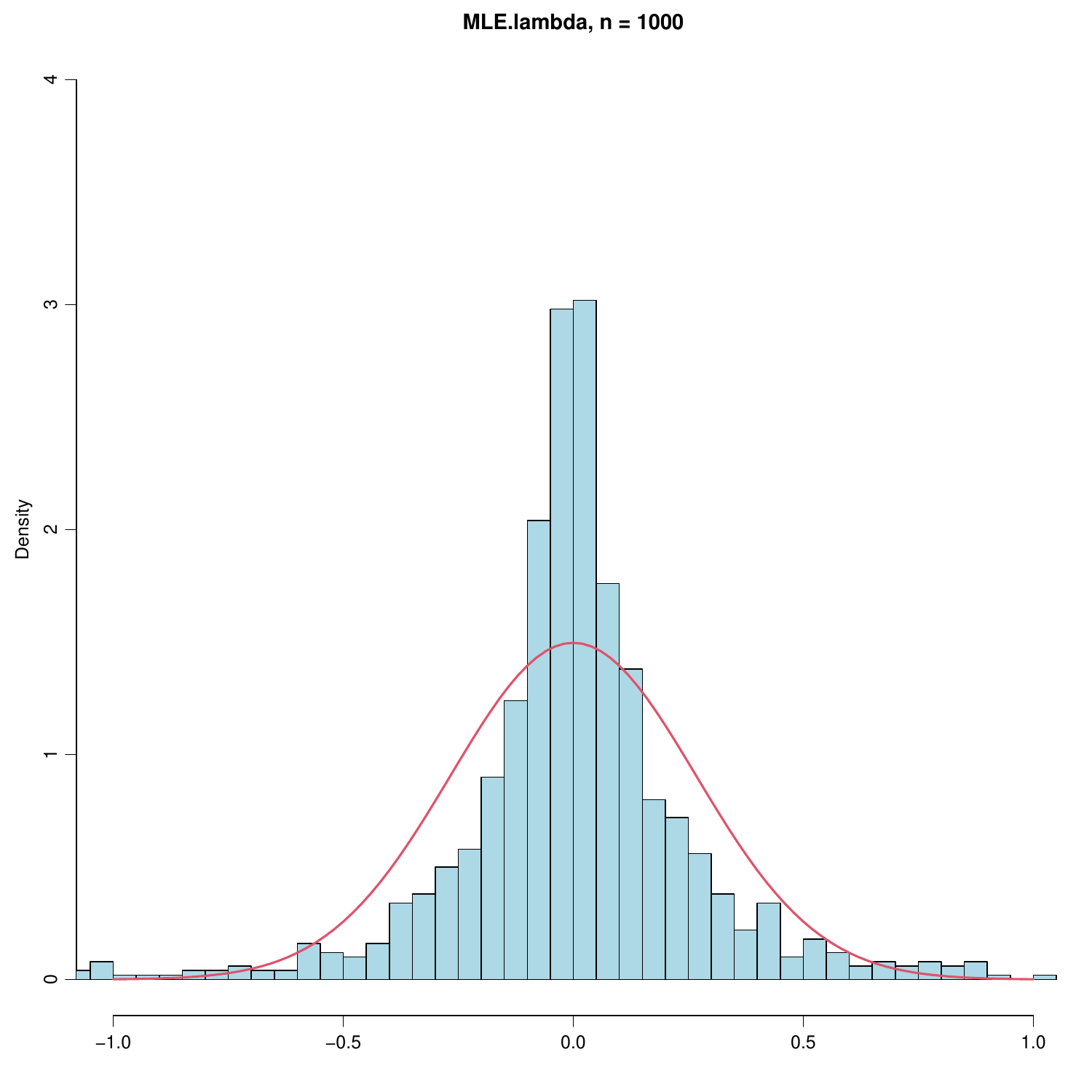}{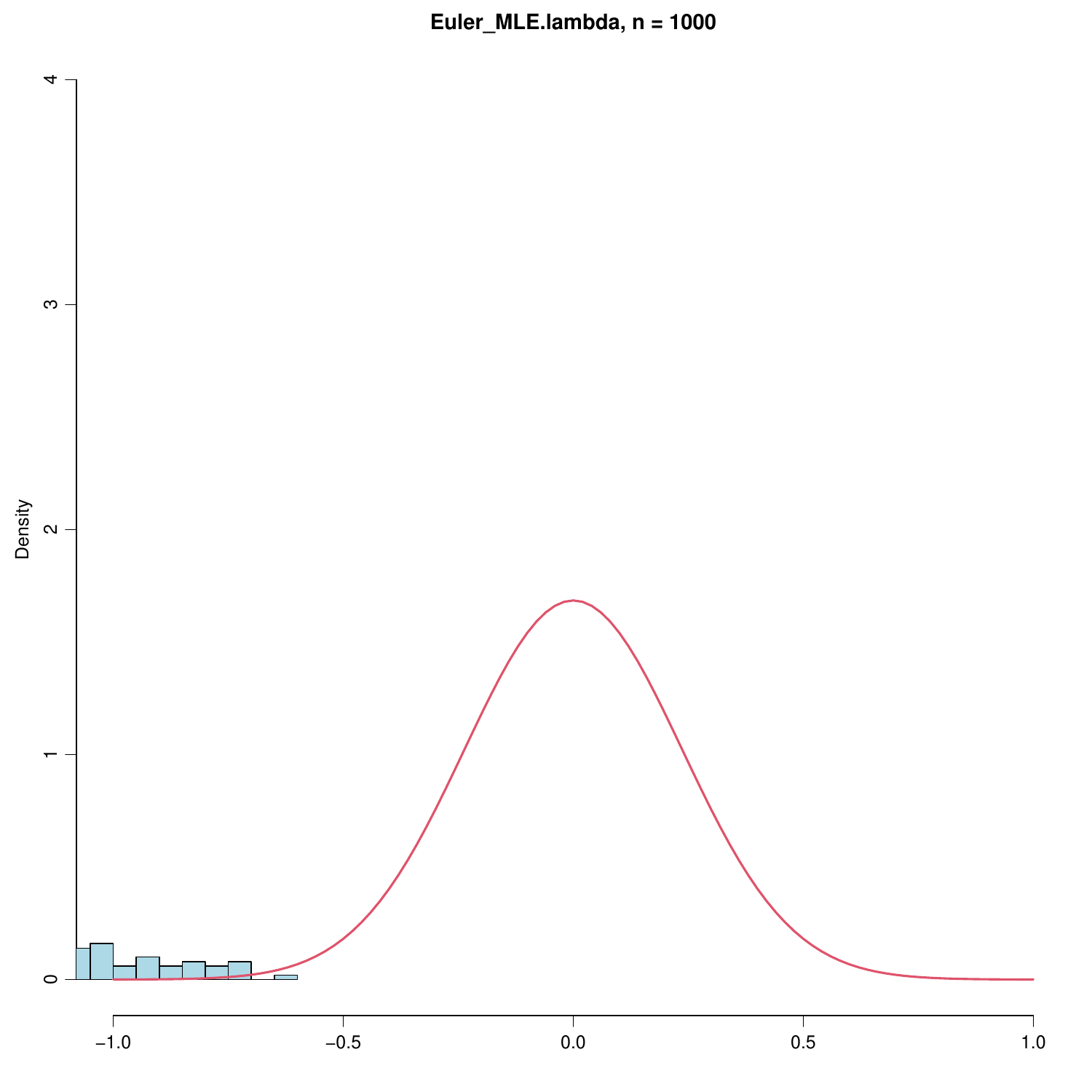}{$n=1000$}\hfill
  \CellPair{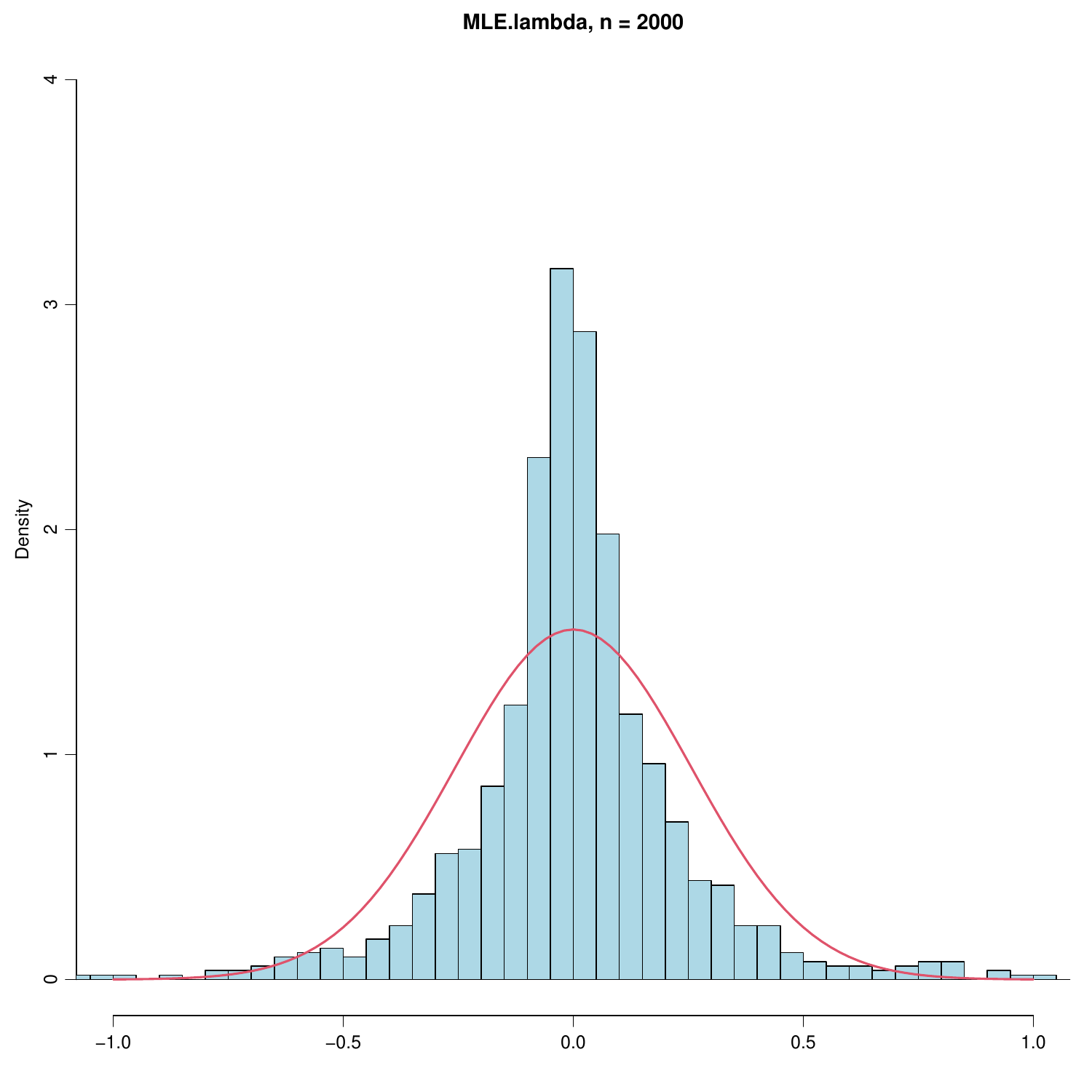}{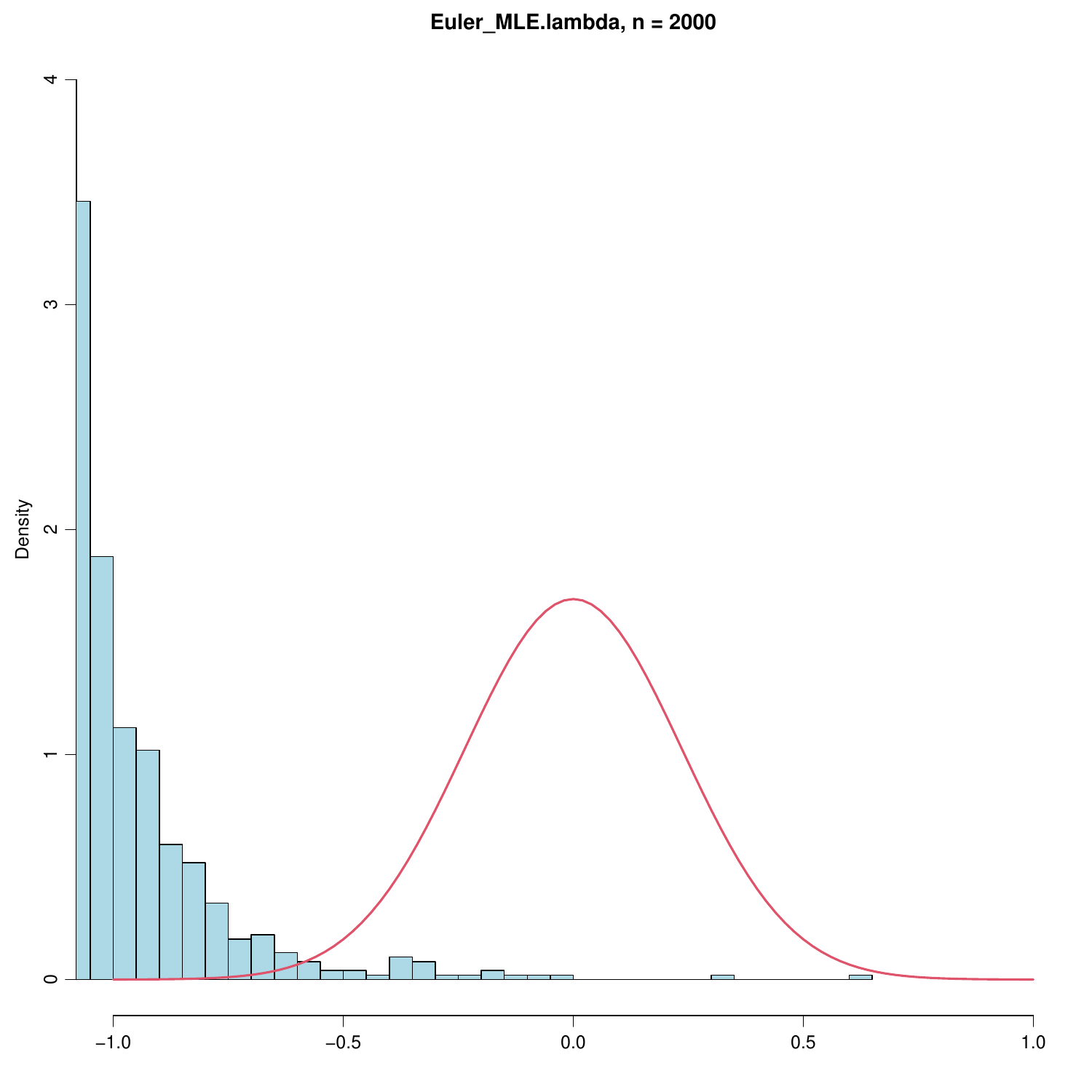}{$n=2000$}
  
  \vspace{5mm}
  
  \rowtitle{$\mu$}
  \CellPair{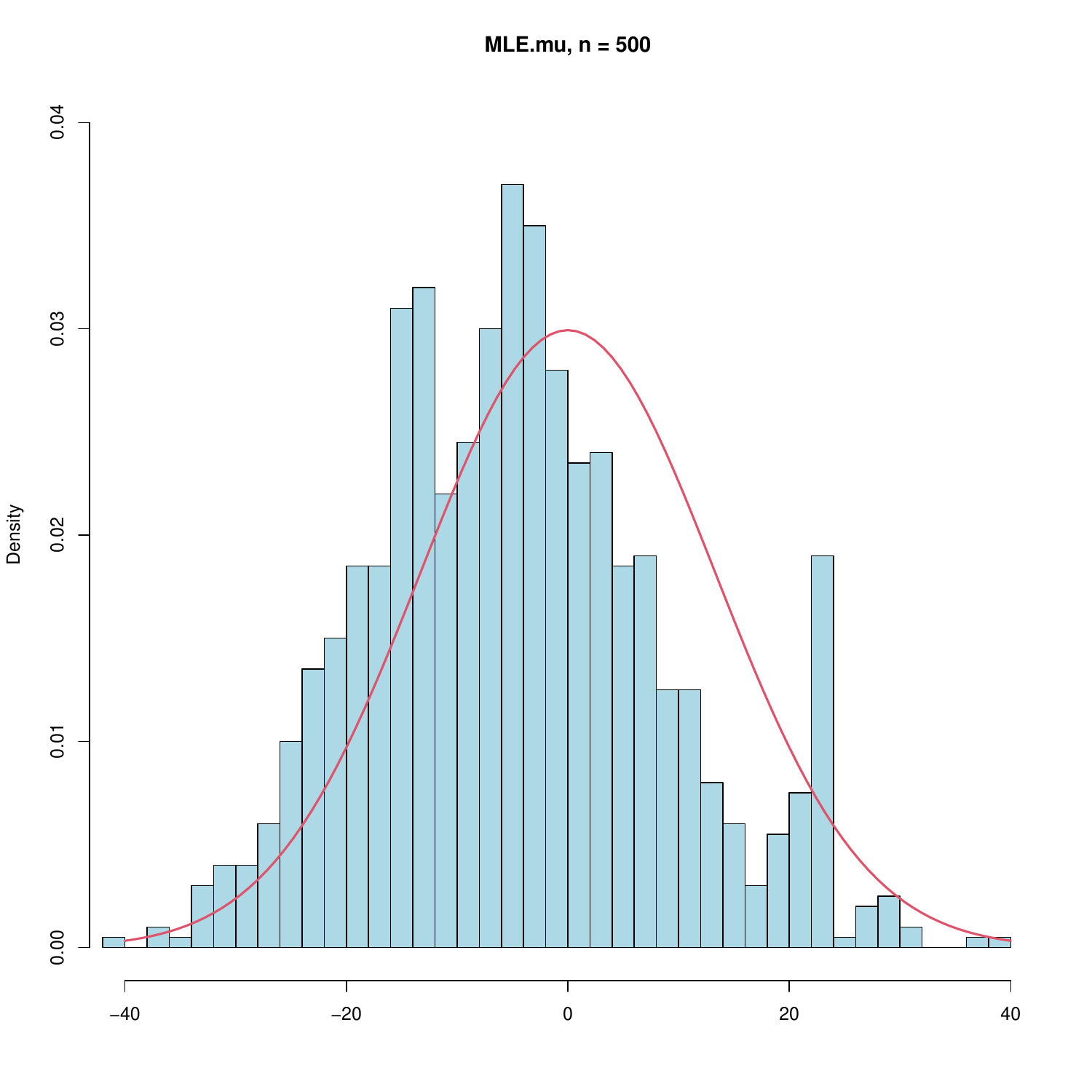}{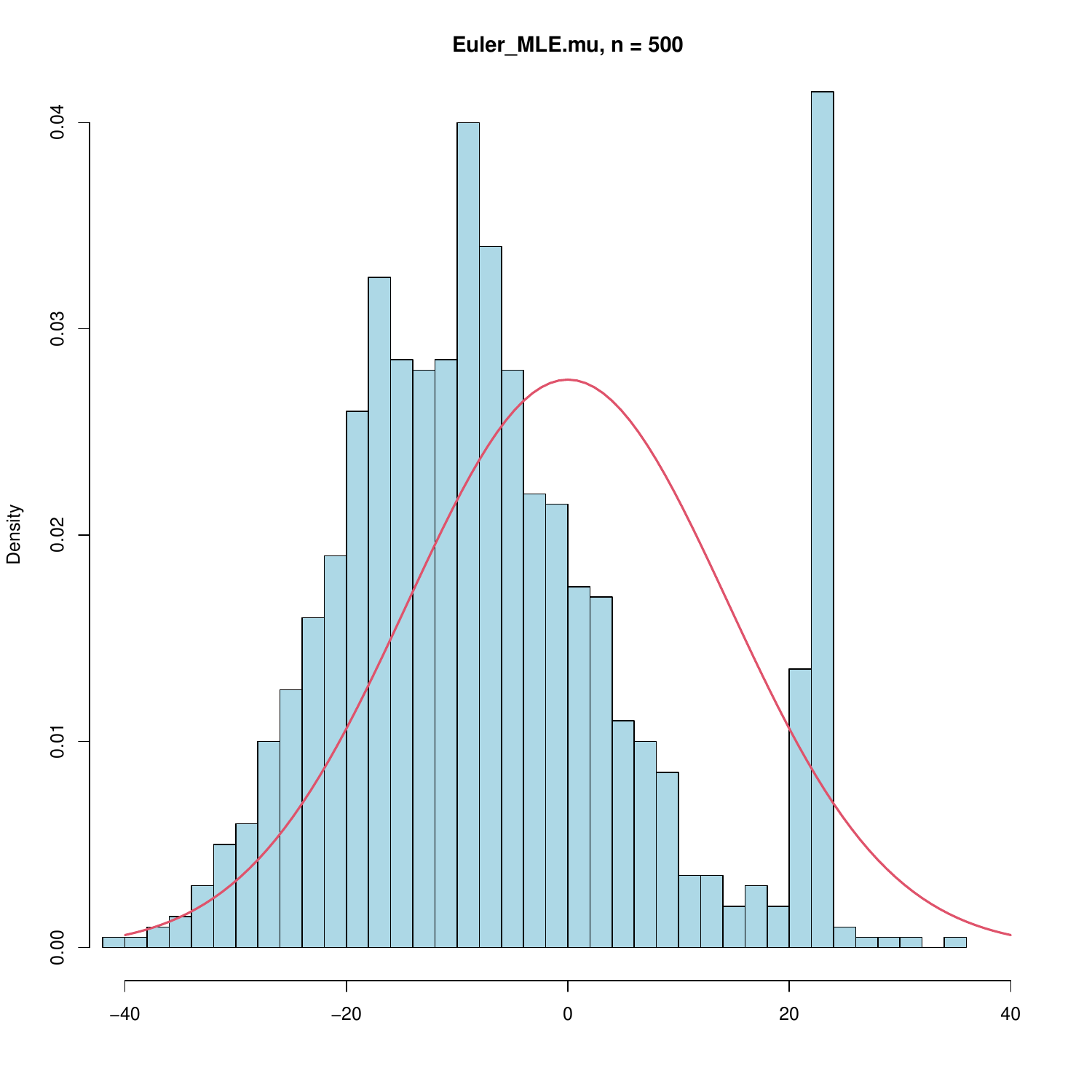}{$n=500$}\hfill
  \CellPair{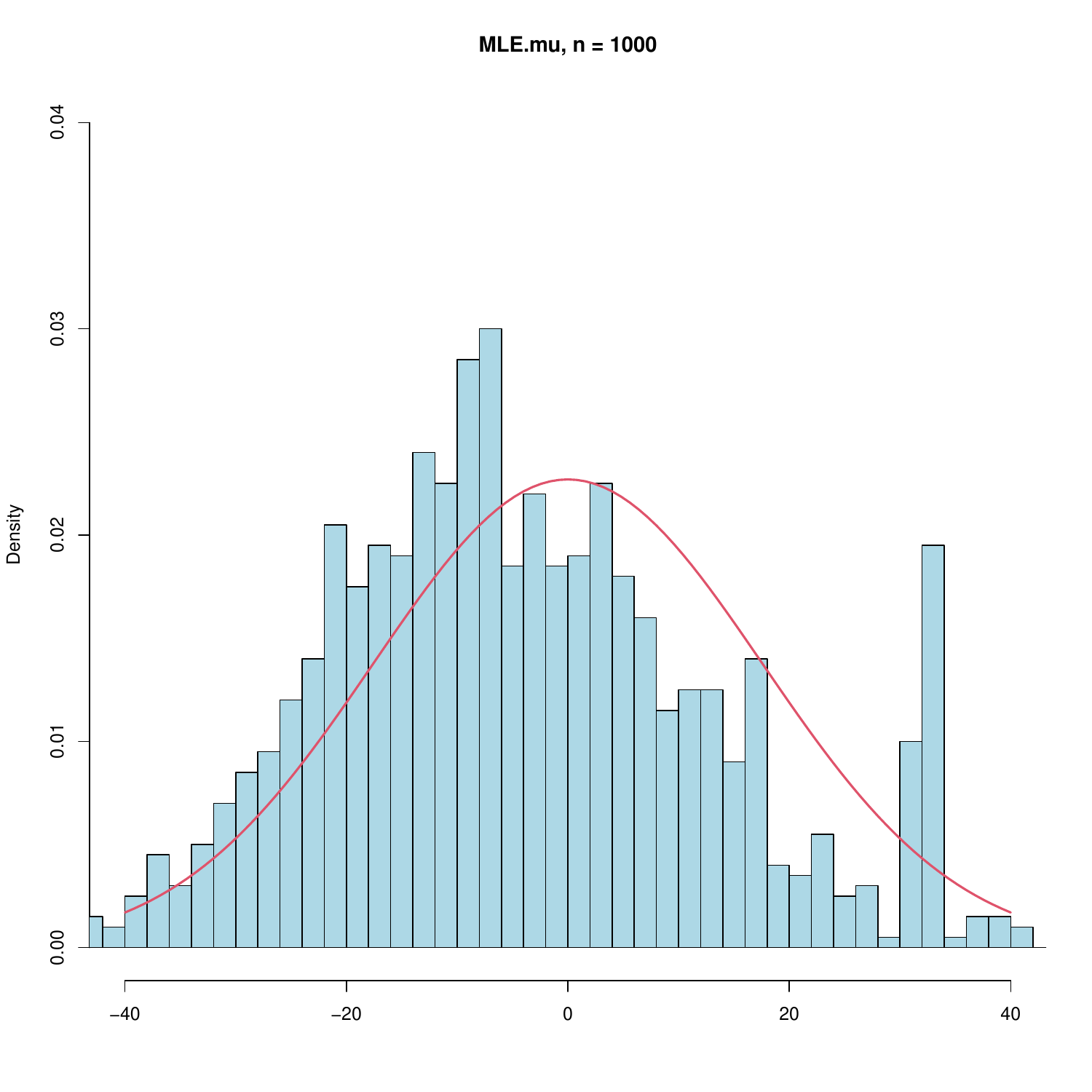}{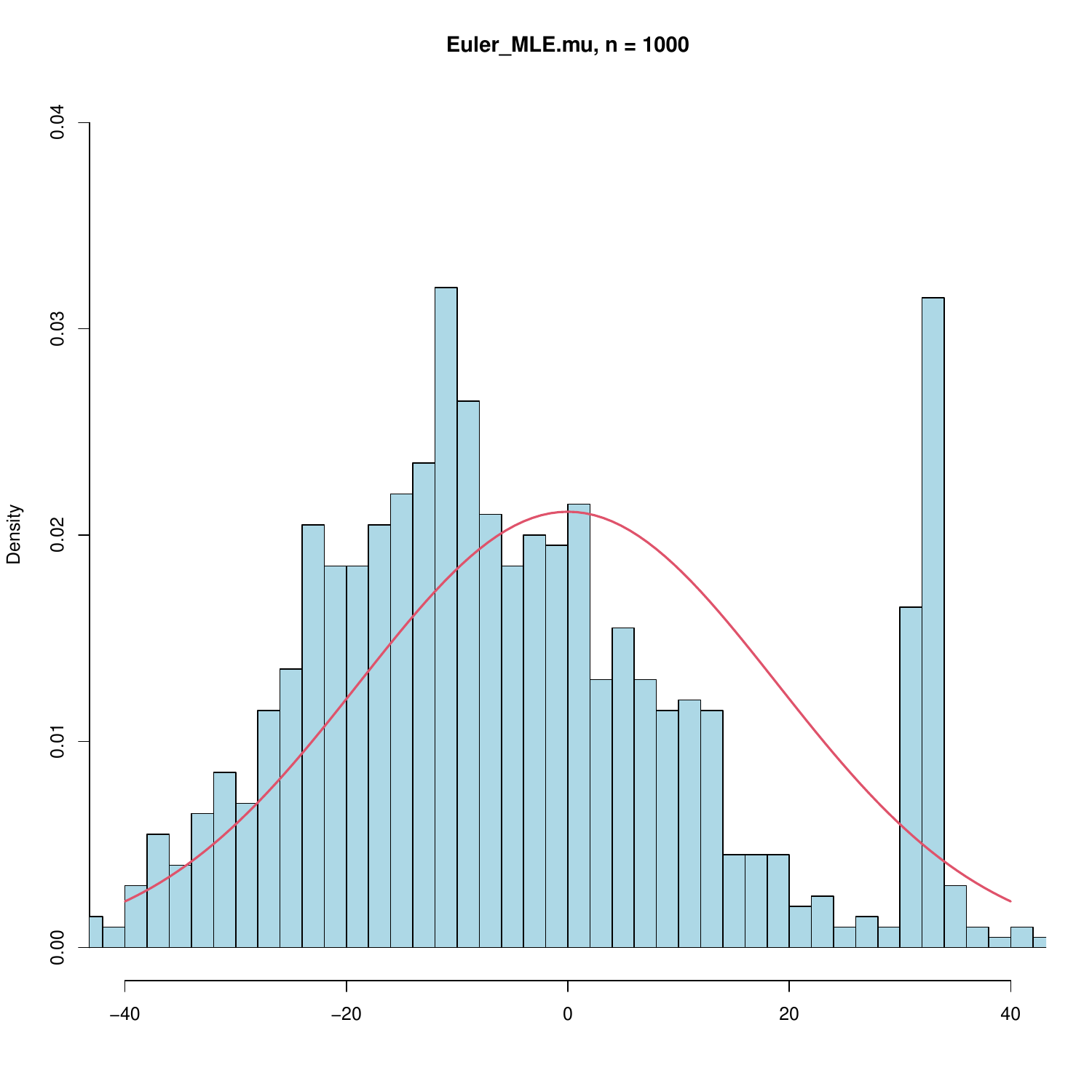}{$n=1000$}\hfill
  \CellPair{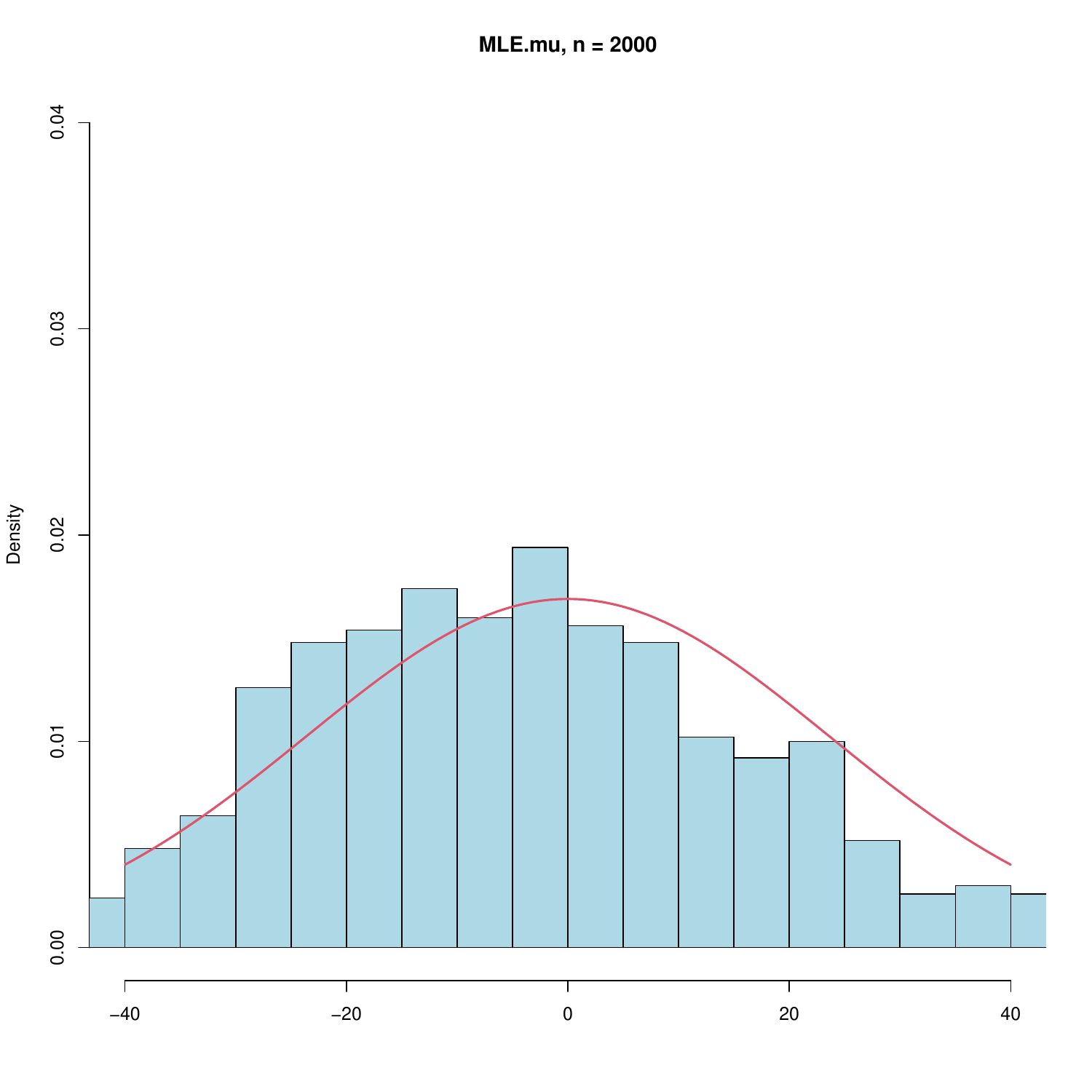}{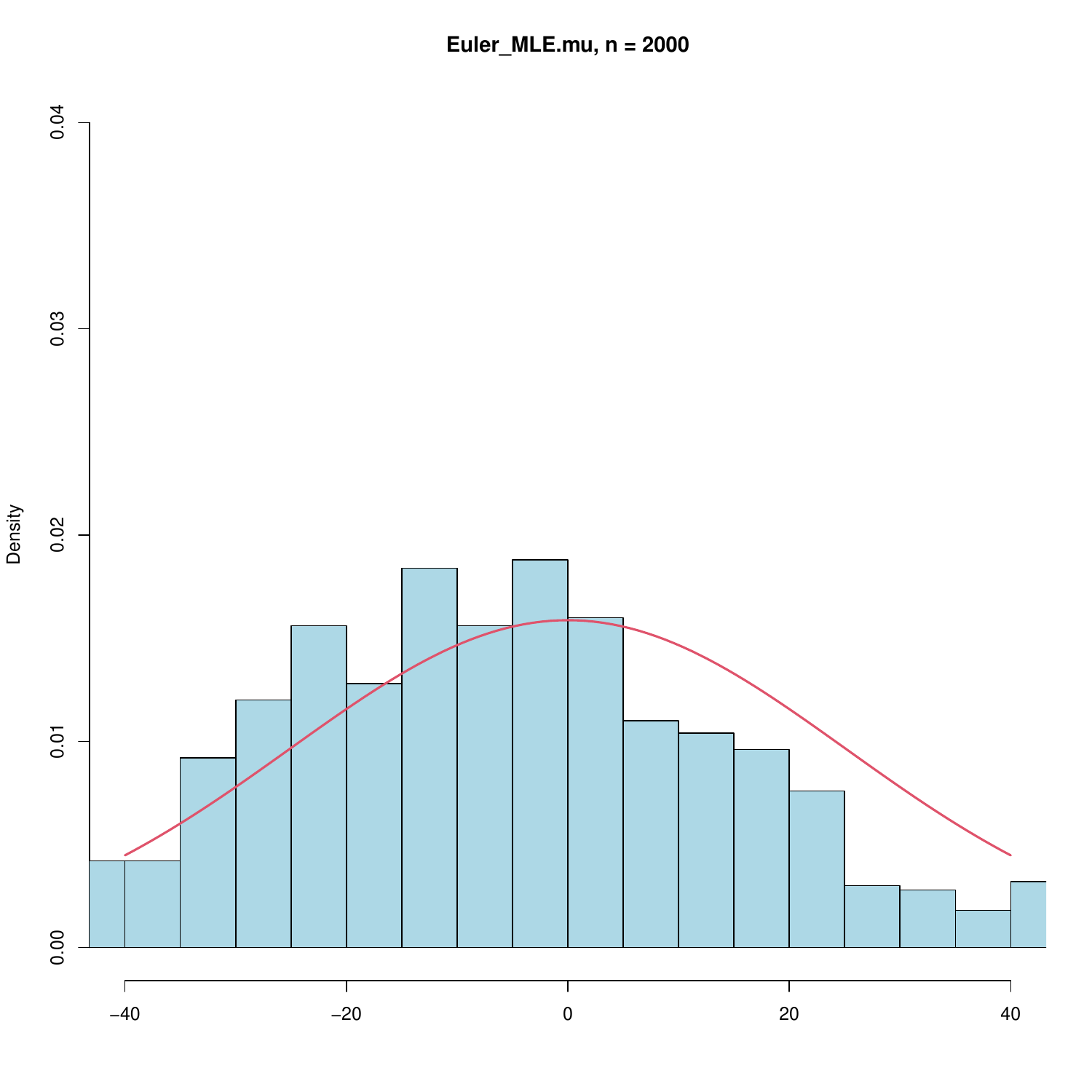}{$n=2000$}

  \vspace{5mm}
  
  \rowtitle{$\al$}
  \CellPair{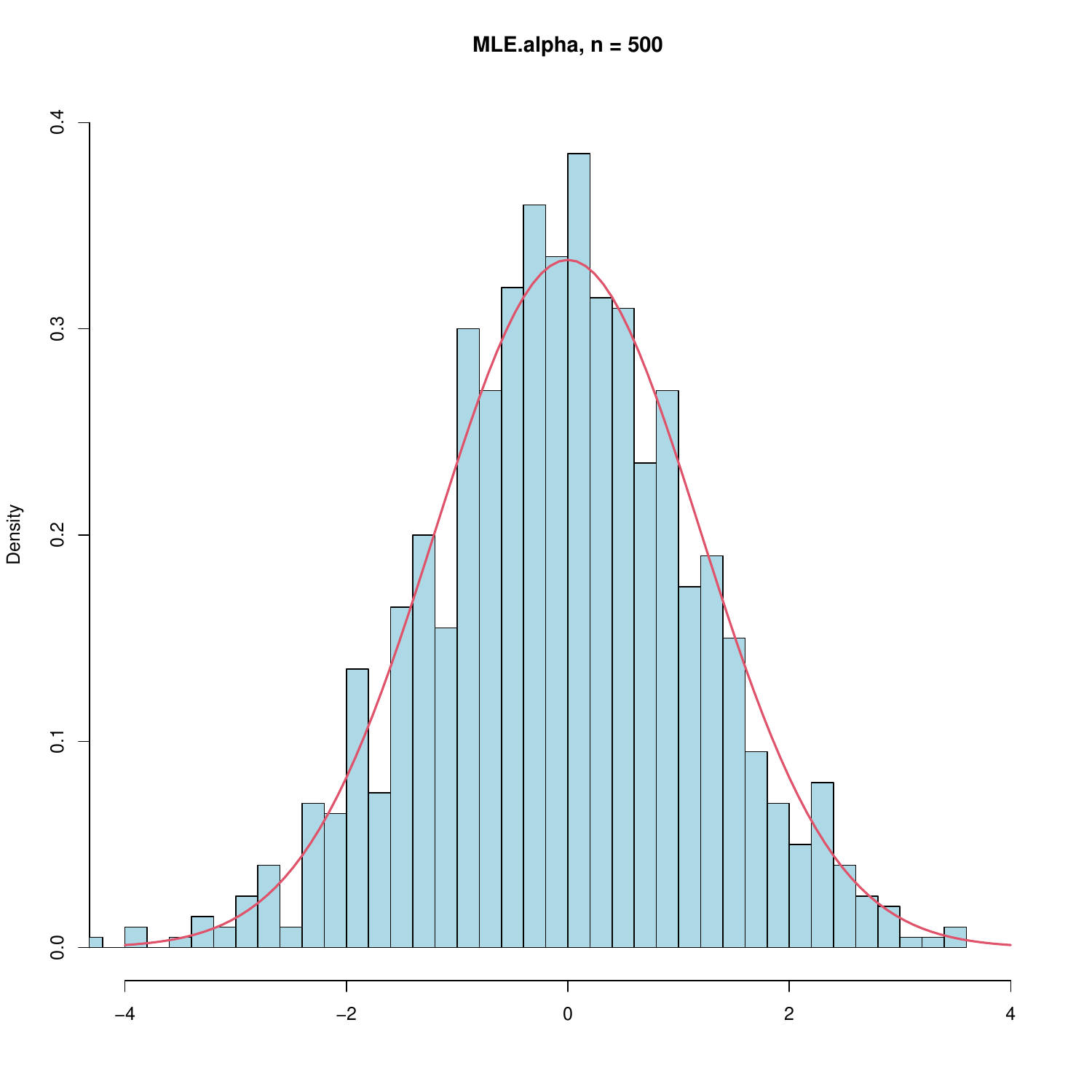}{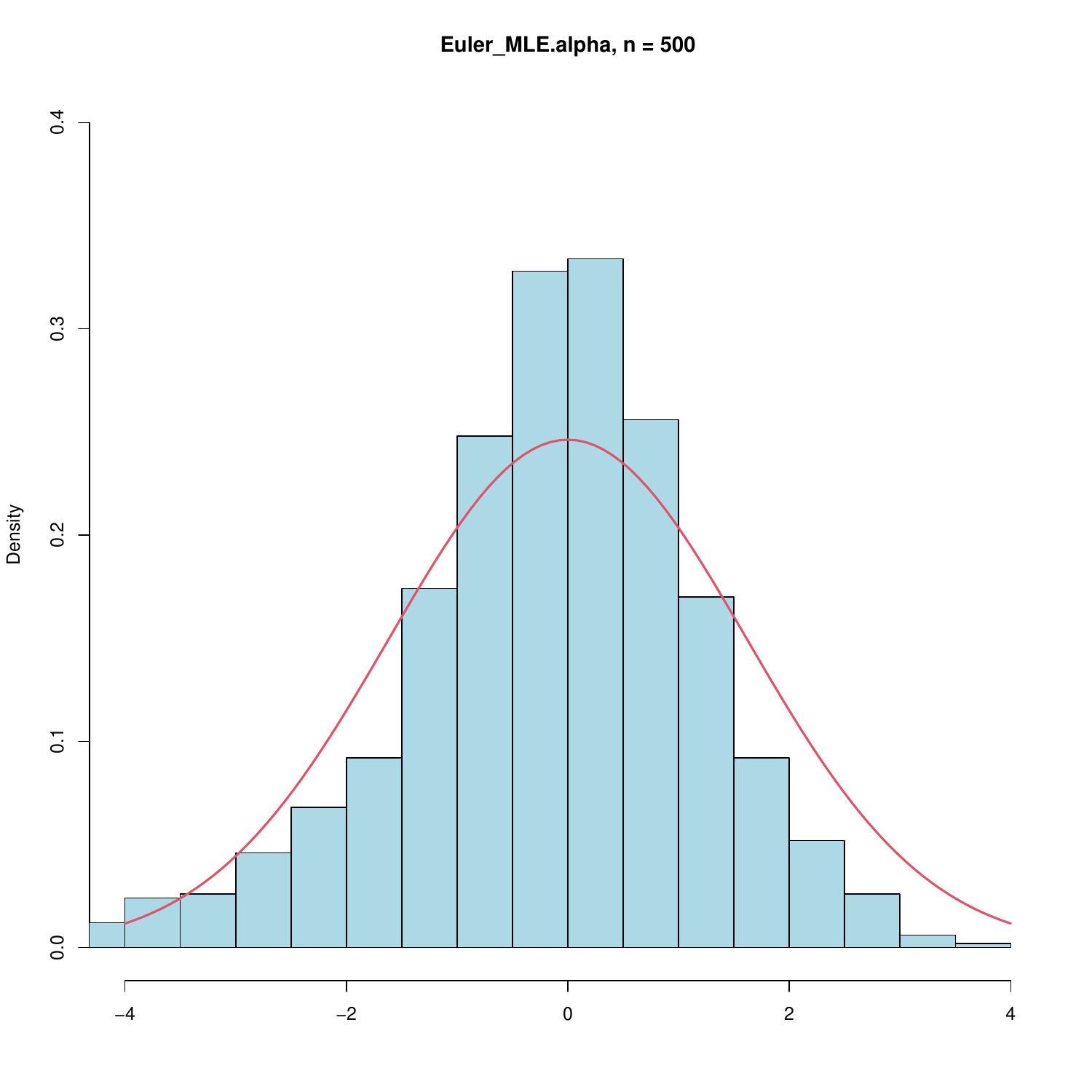}{$n=500$}\hfill
  \CellPair{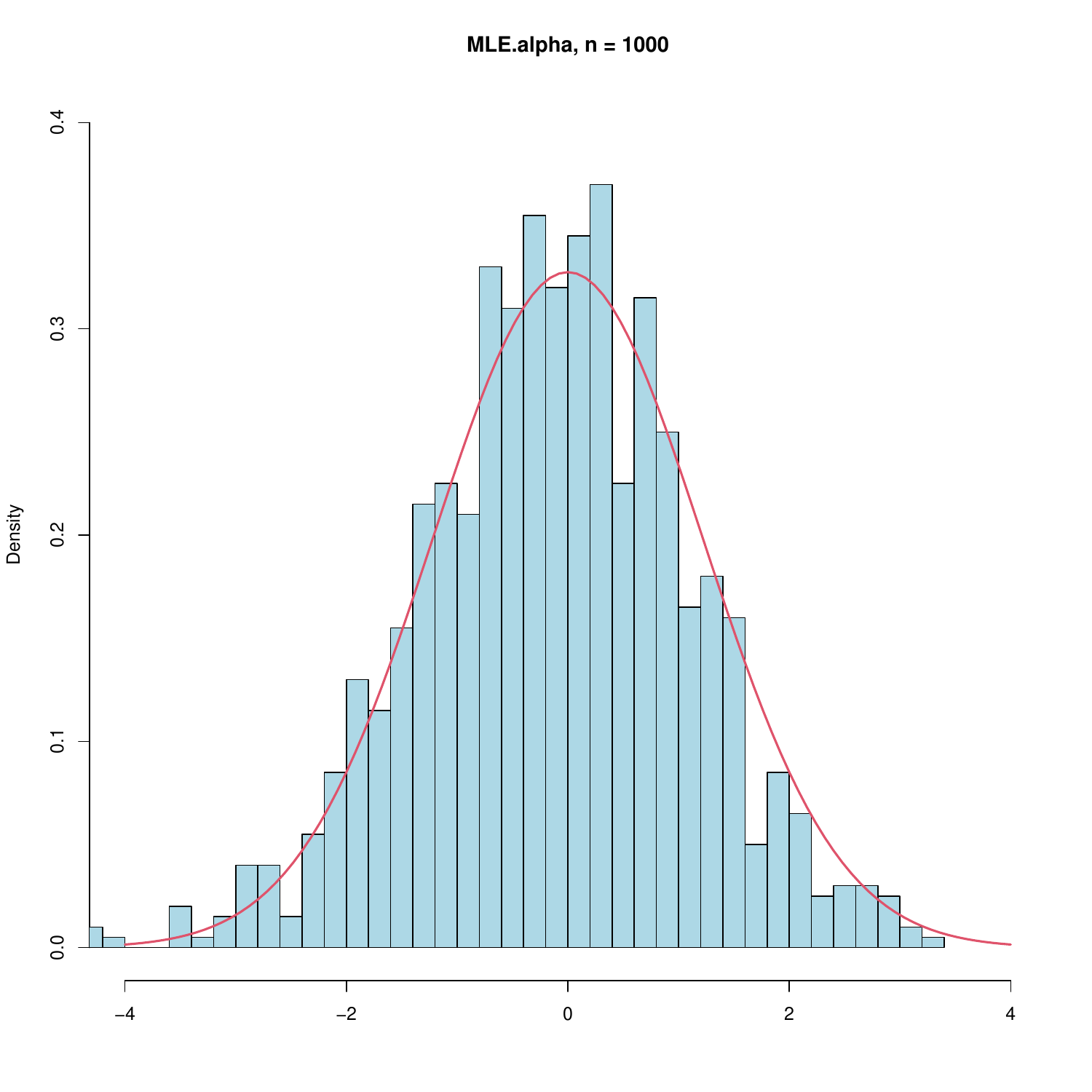}{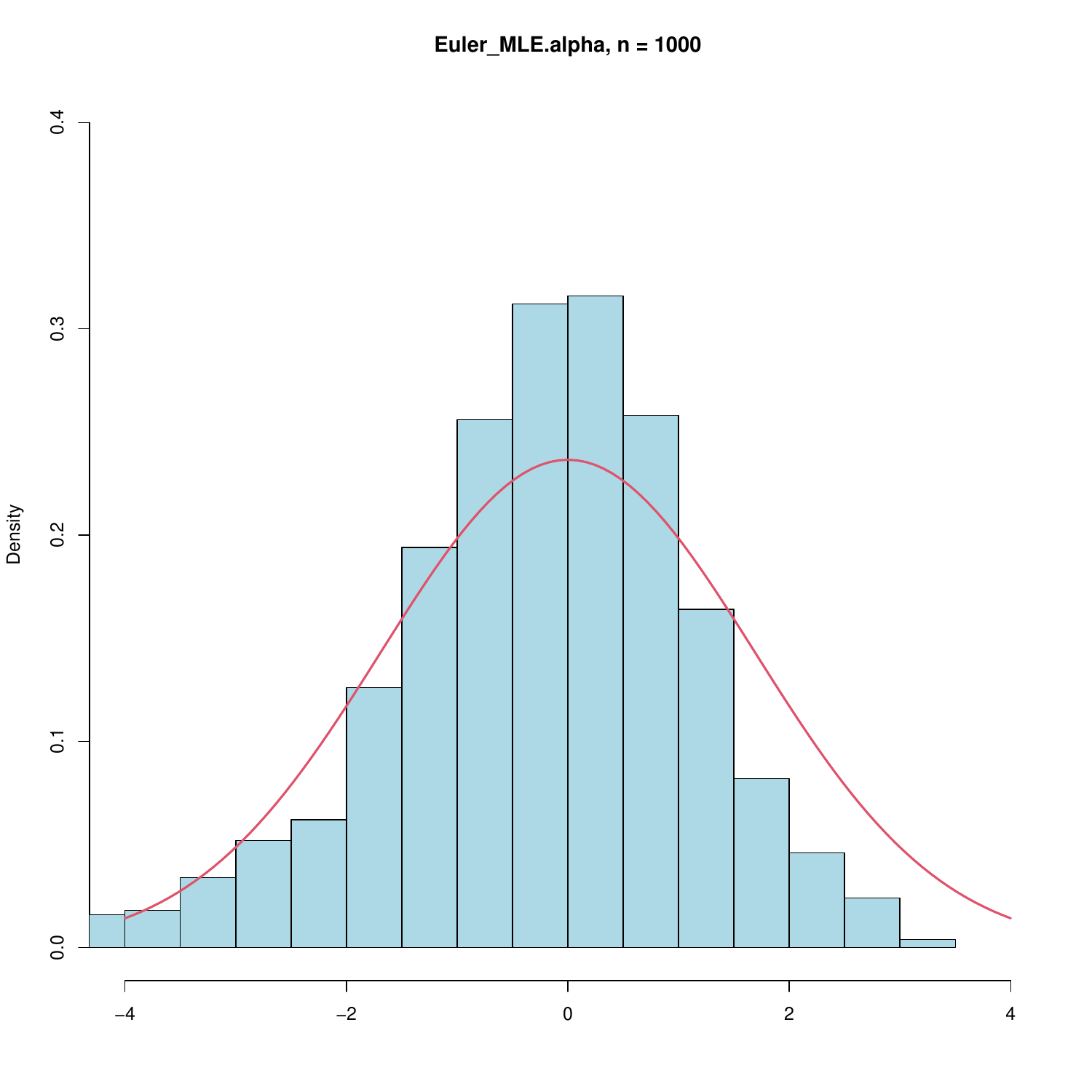}{$n=1000$}\hfill
  \CellPair{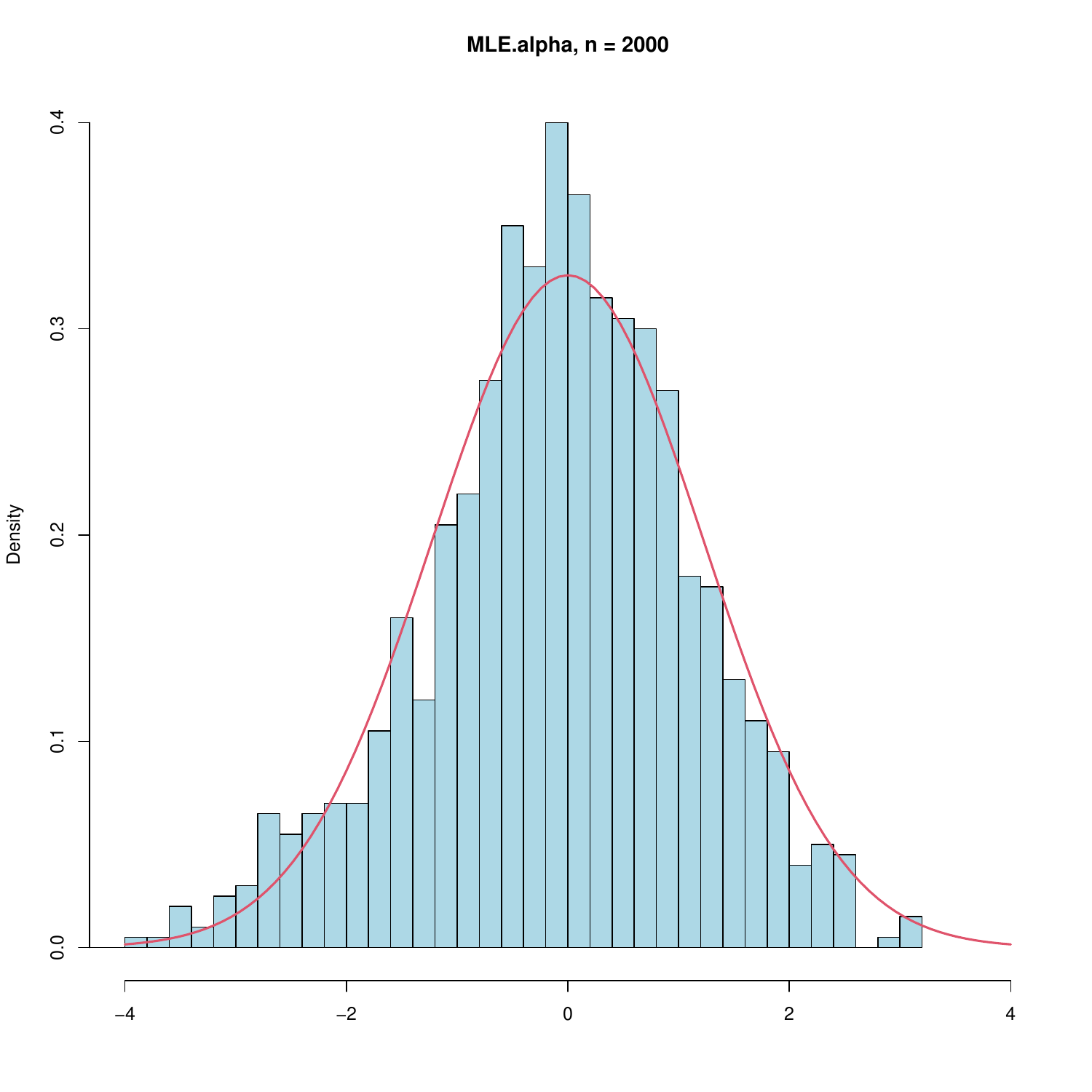}{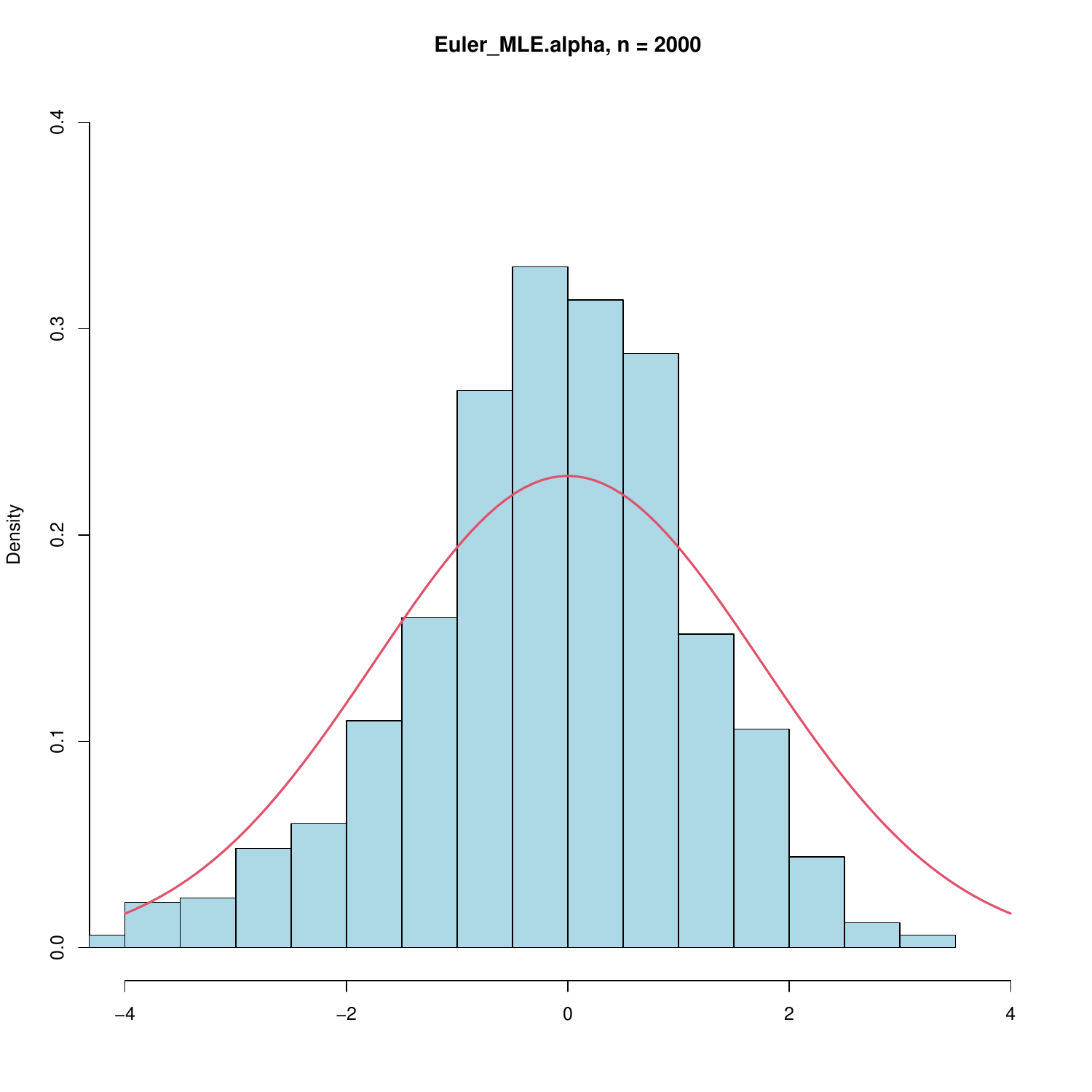}{$n=2000$}

  \vspace{5mm}
  
  \rowtitle{$\sig$}
  \CellPair{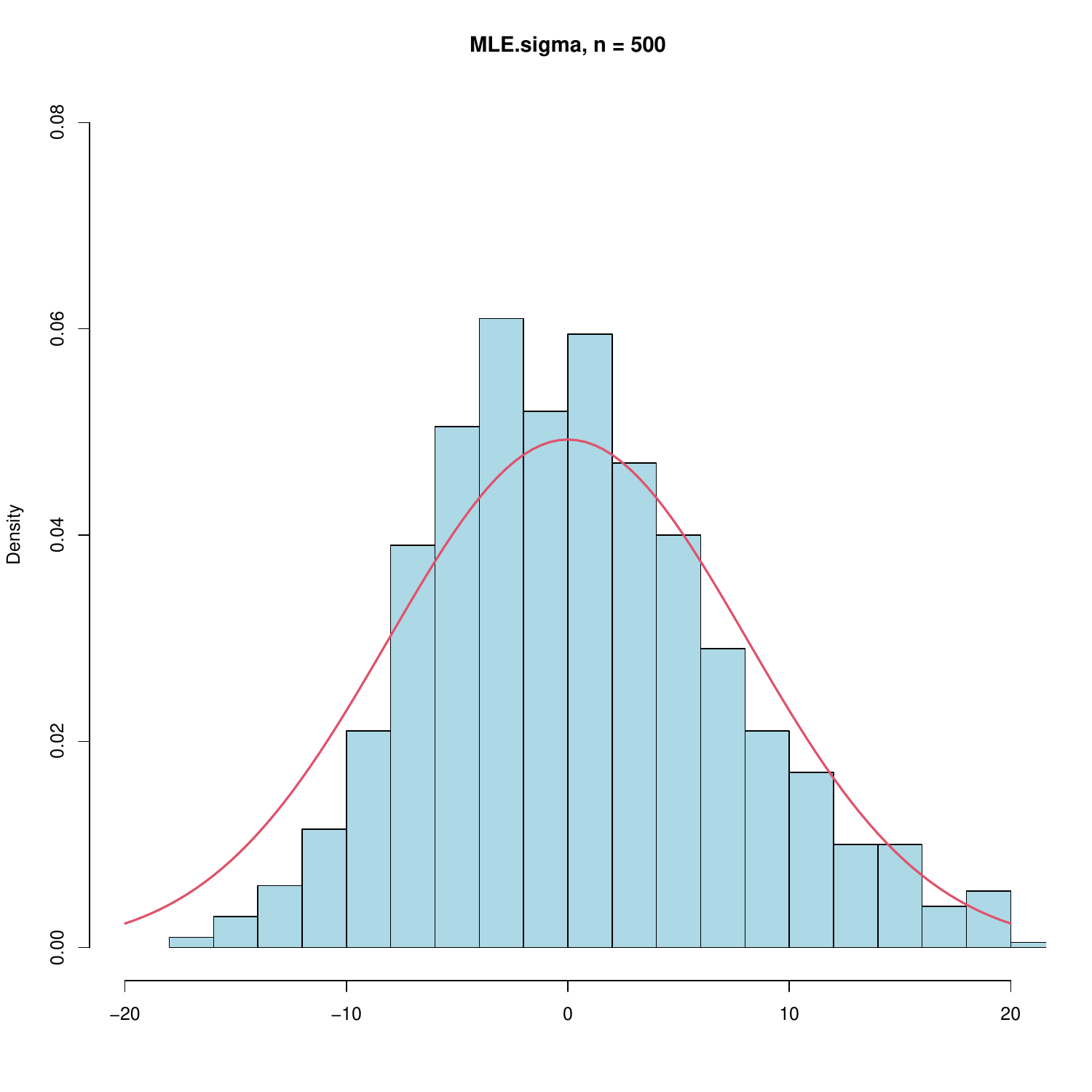}{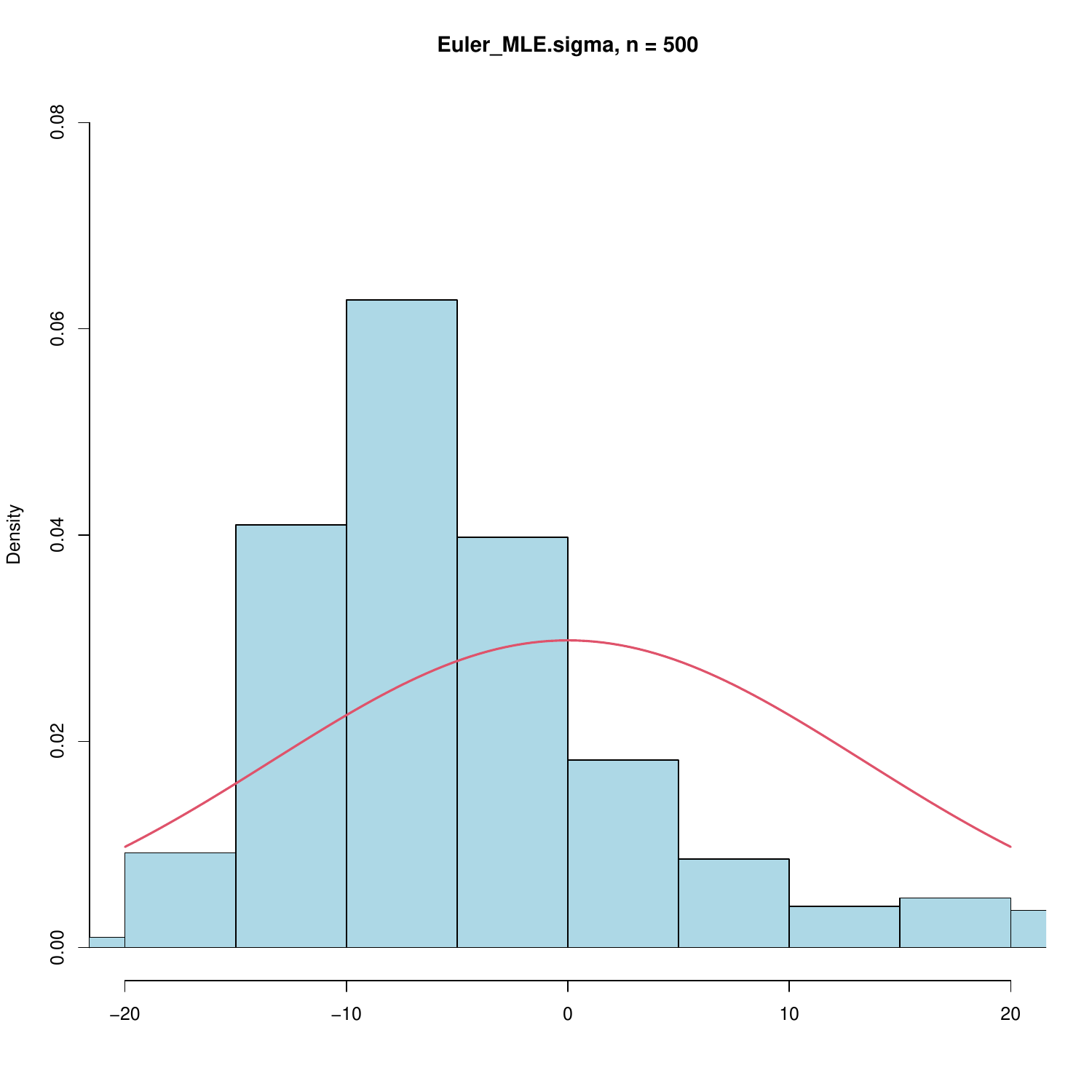}{$n=500$}\hfill
  \CellPair{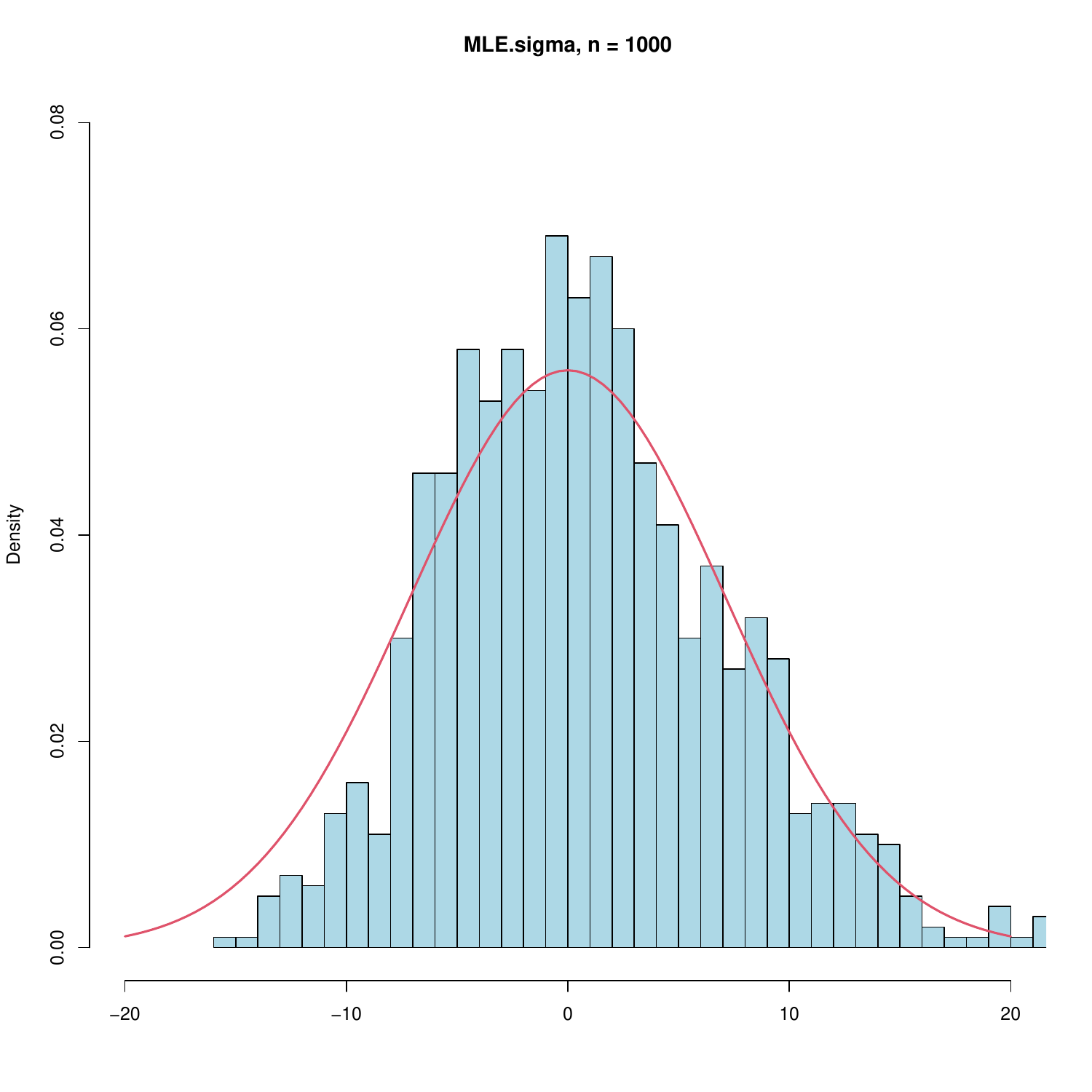}{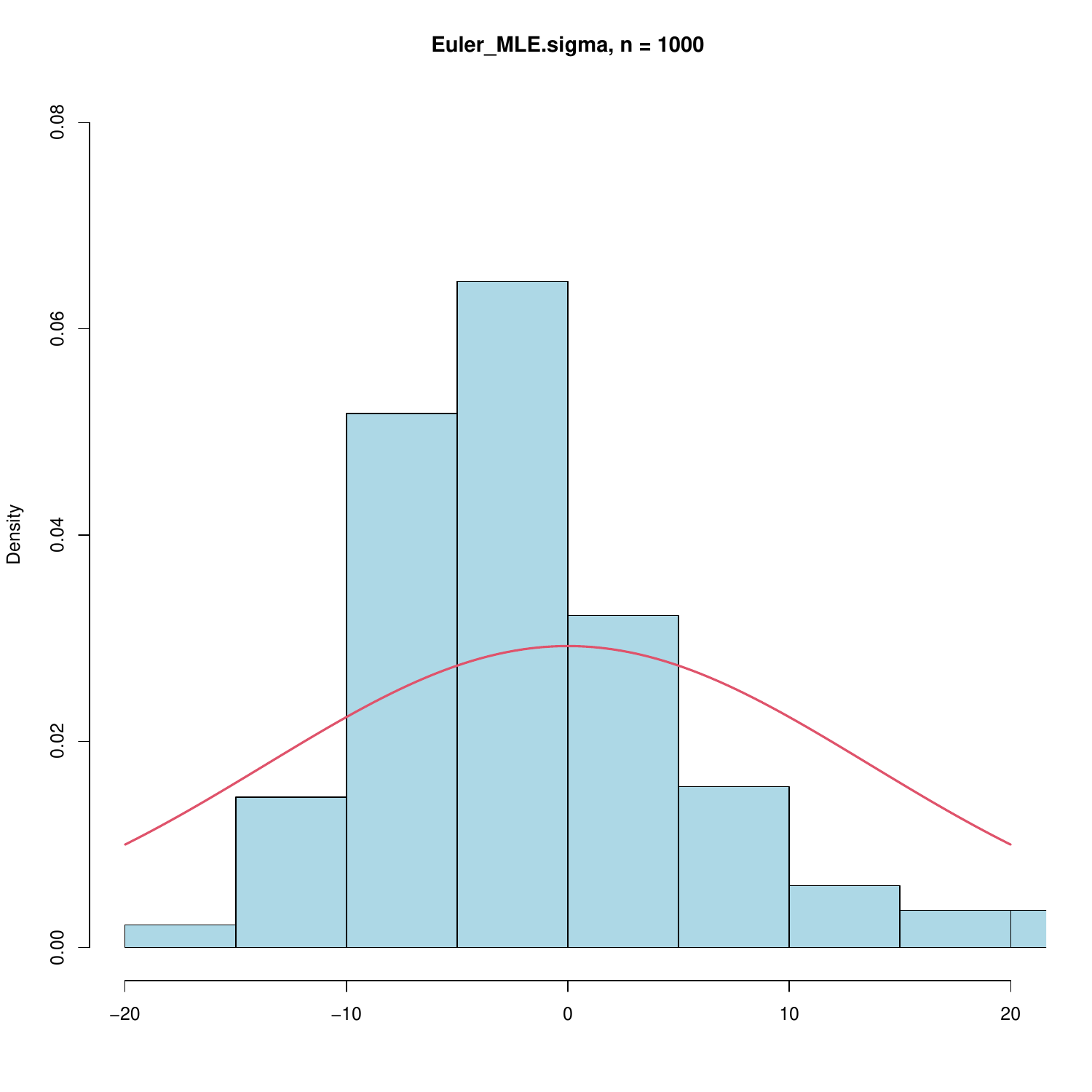}{$n=1000$}\hfill
  \CellPair{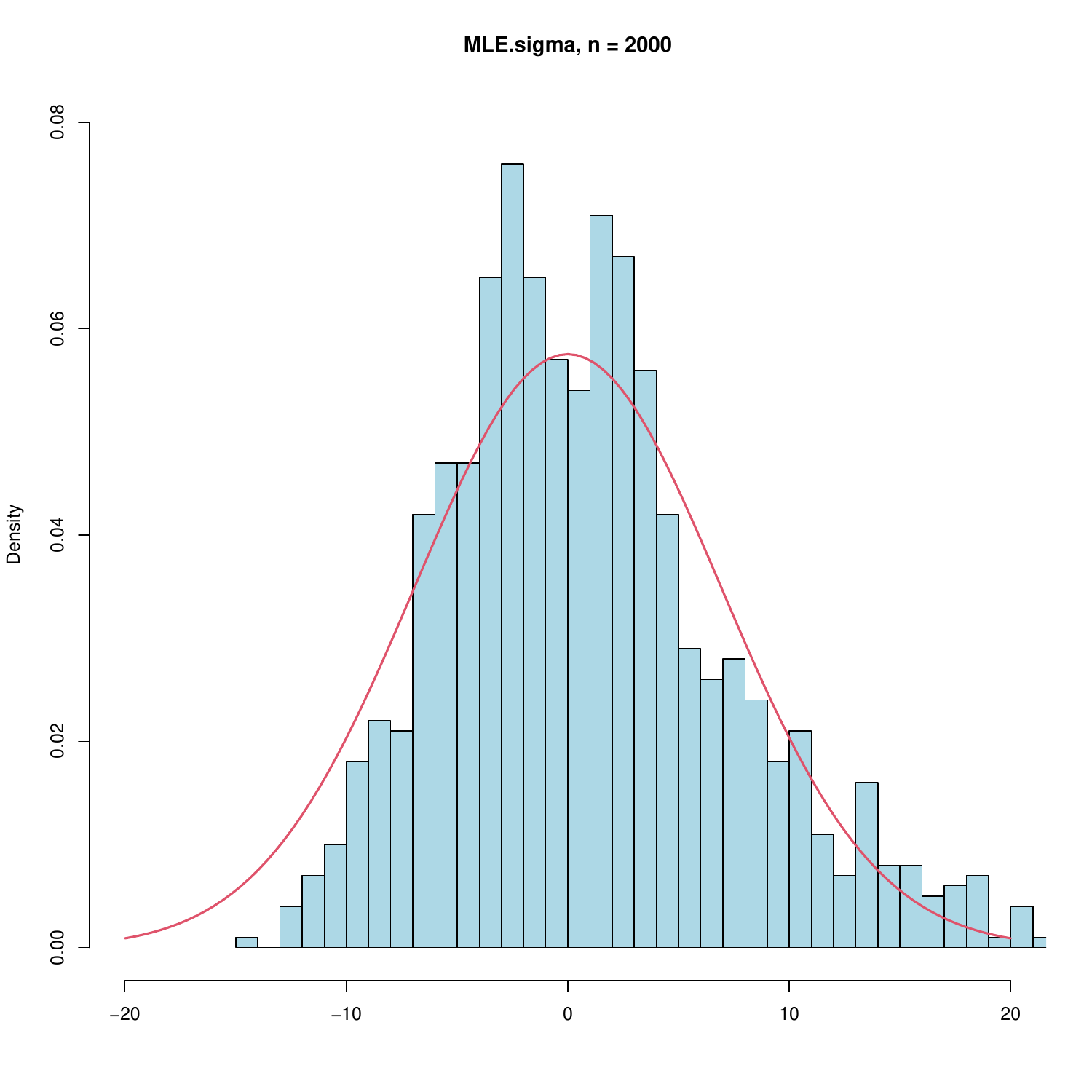}{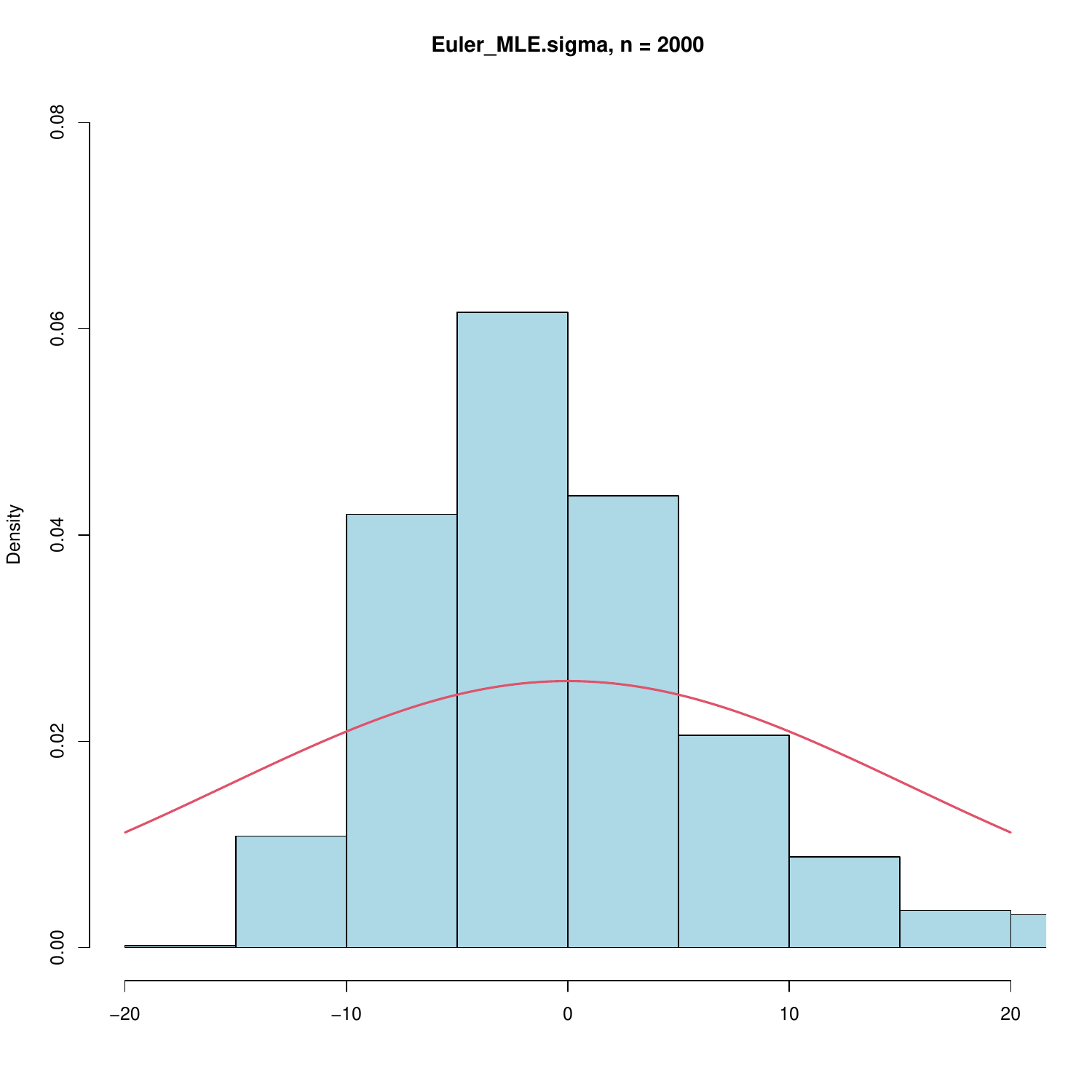}{$n=2000$}

  \vspace{5mm}

  \rowtitle{$\beta$}
  \CellPair{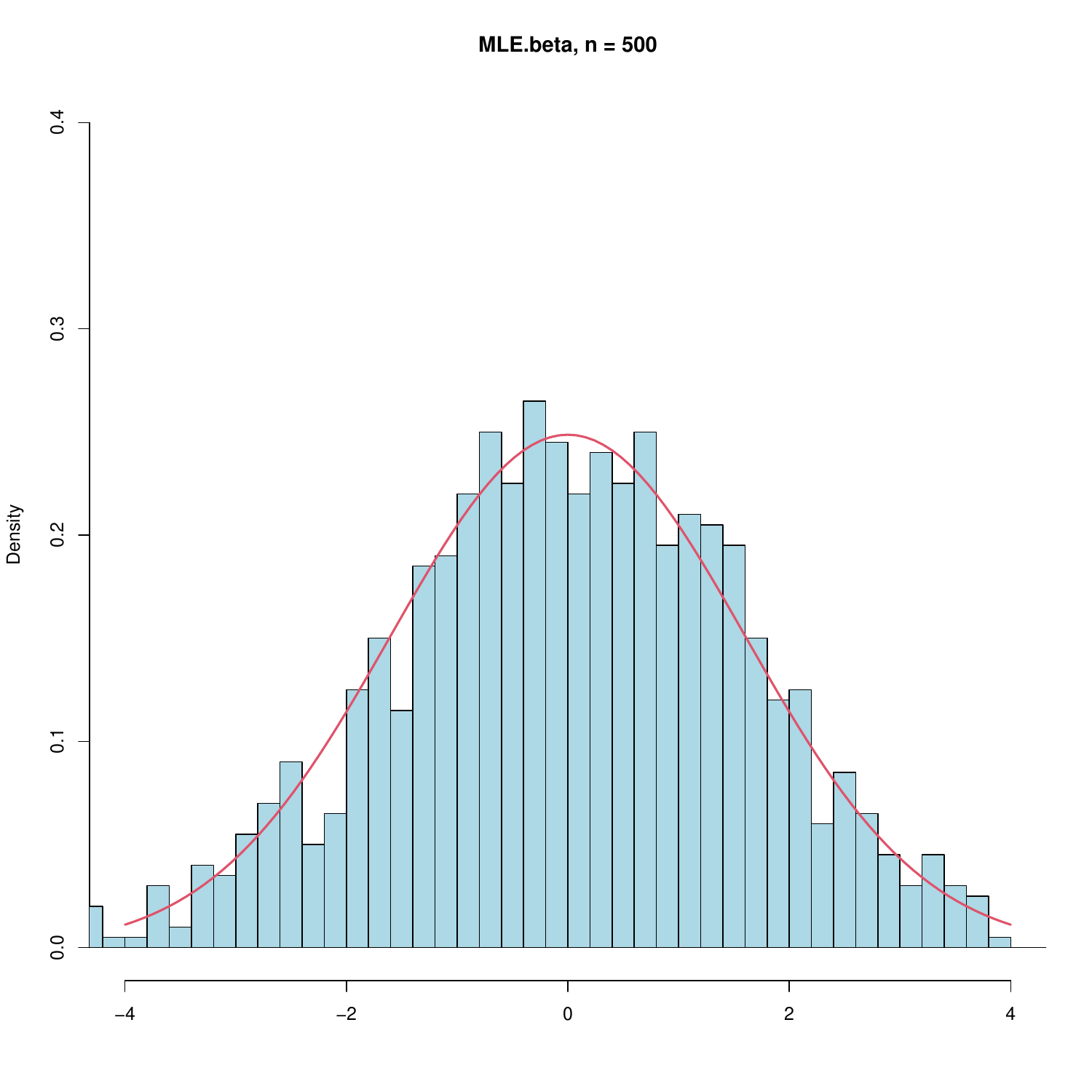}{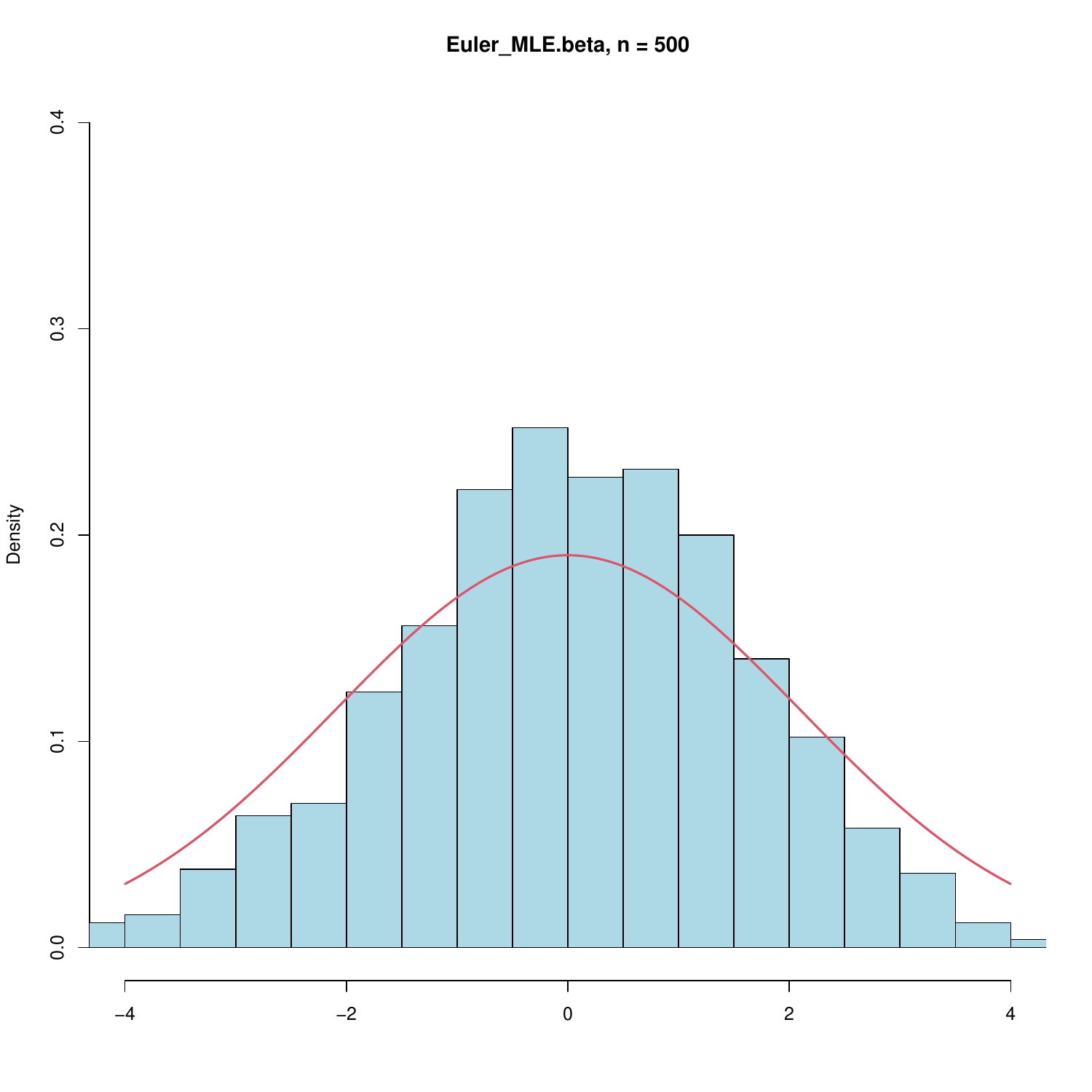}{$n=500$}\hfill
  \CellPair{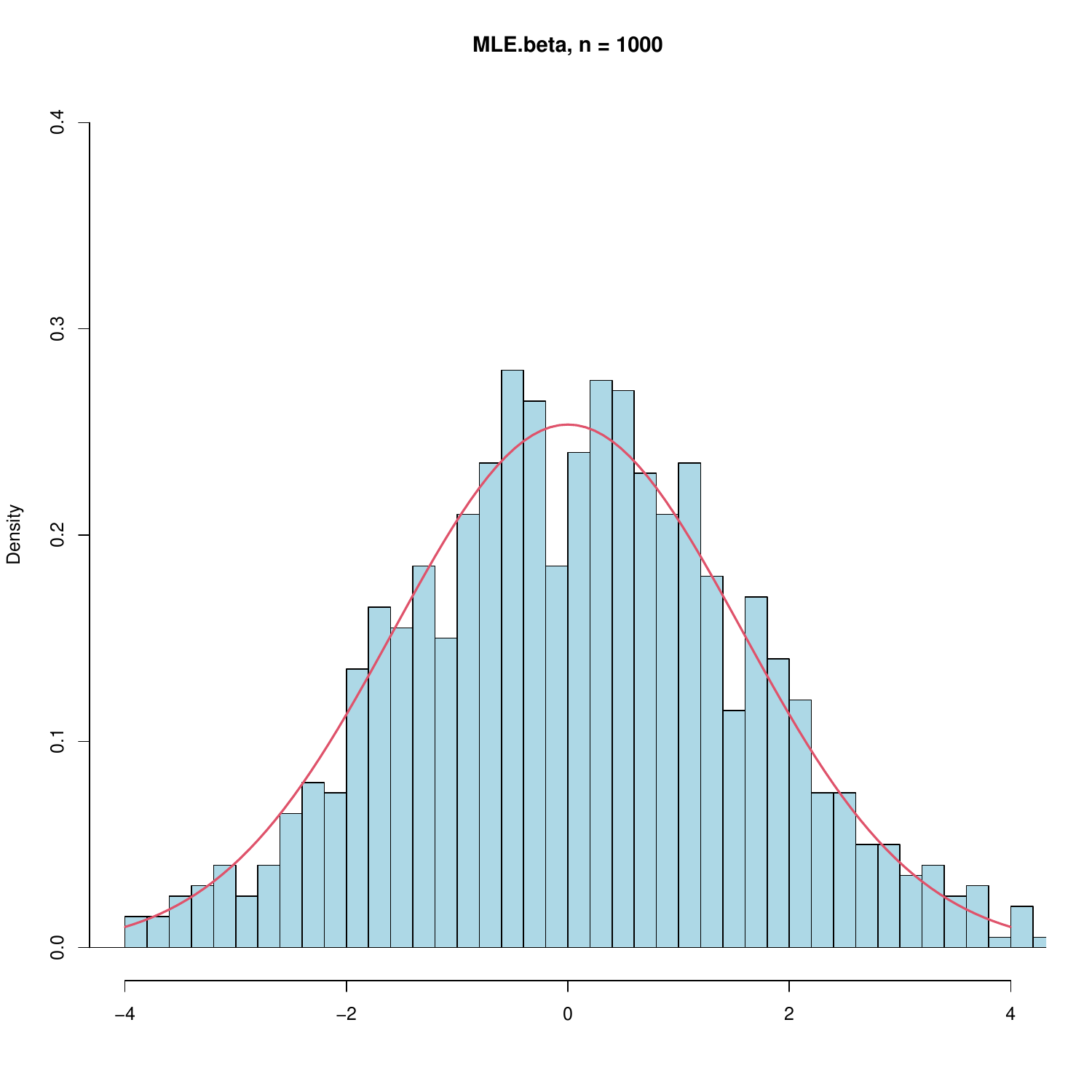}{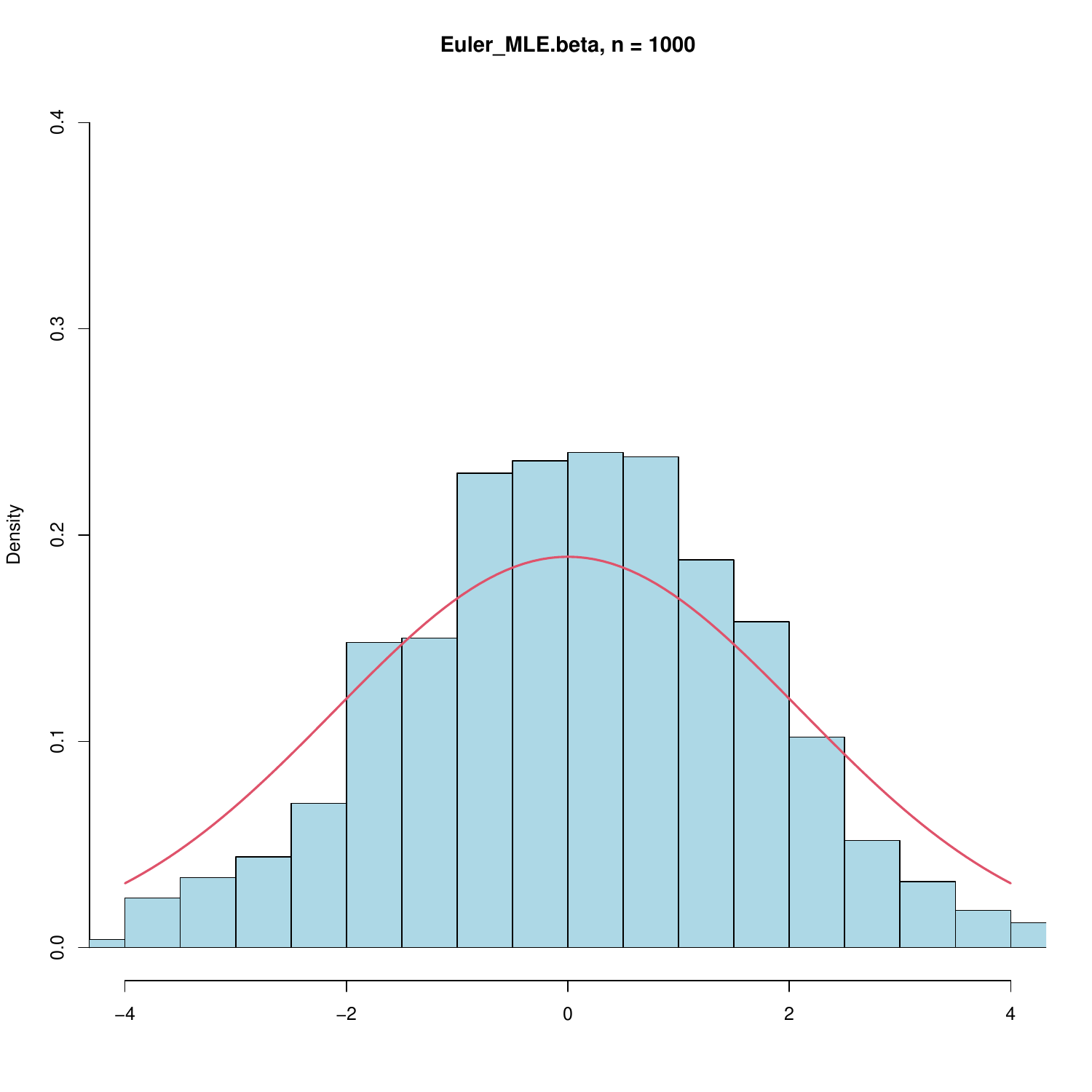}{$n=1000$}\hfill
  \CellPair{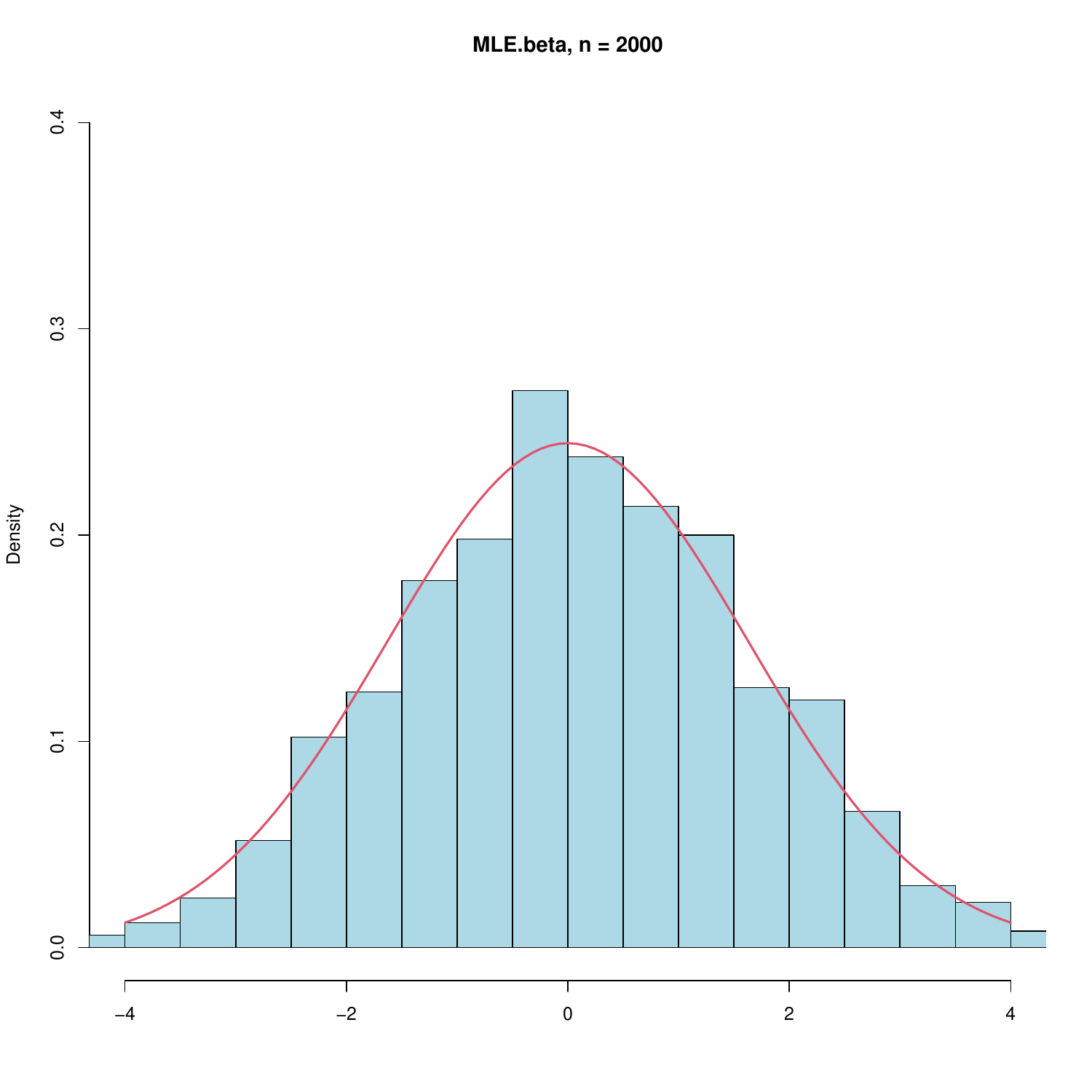}{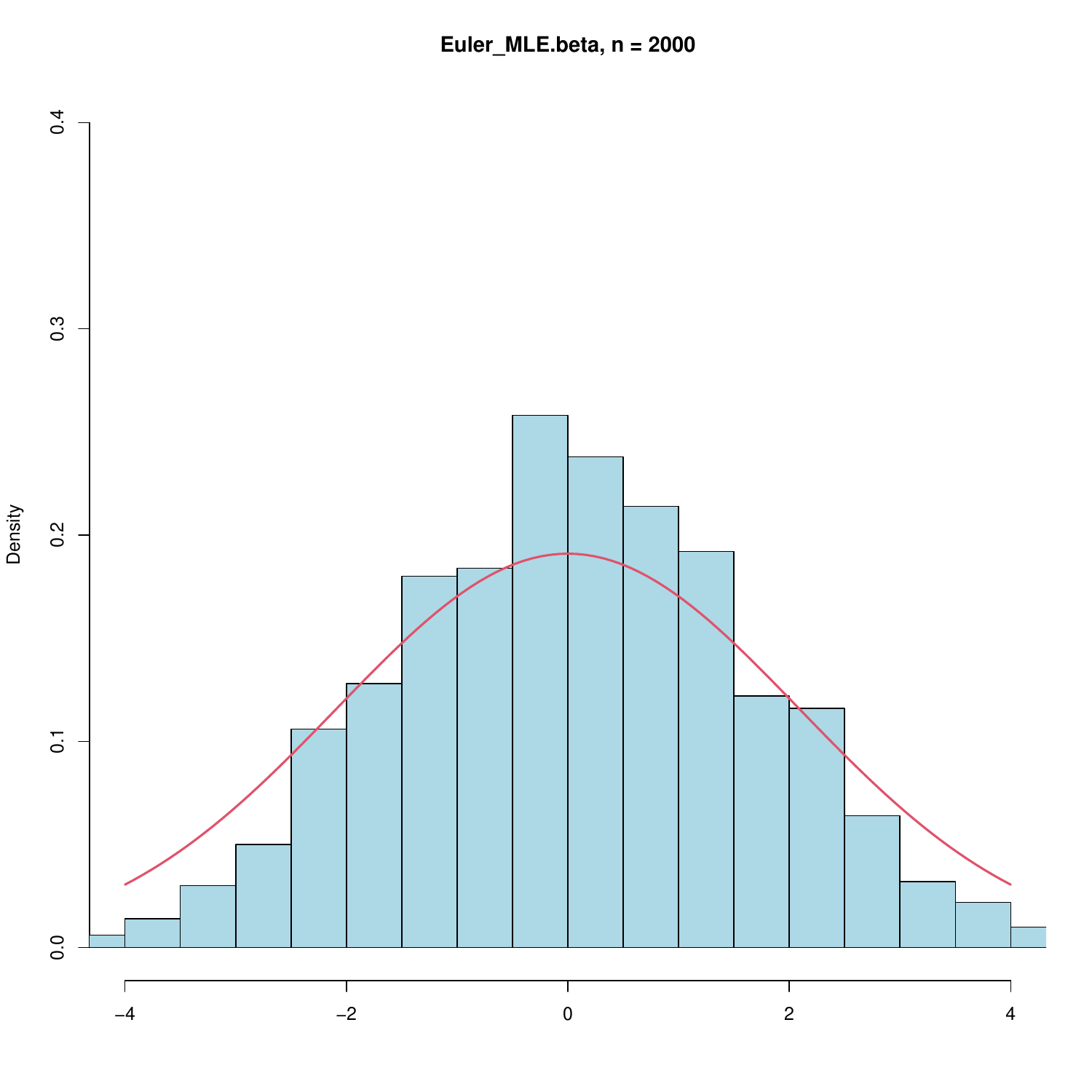}{$n=2000$}
  
  \caption{Comparison between histograms of the normalized MLE and Euler-QMLE with $\alpha=1.0$, $(\lambda,\mu,\sigma,\beta)=(1,2,5,0.5)$.
  }
  \label{fig:param-by-n_mle-vs-euler-1.0}
\end{figure}


\begin{figure}[t]
  \centering
  \captionsetup[subfigure]{justification=centering}
  \newcommand{\colw}{0.32\linewidth} 


  \newcommand{\rowtitle}[1]{\par\medskip\textbf{ #1}\par\smallskip}

  \rowtitle{$\lam$}
  \CellPair{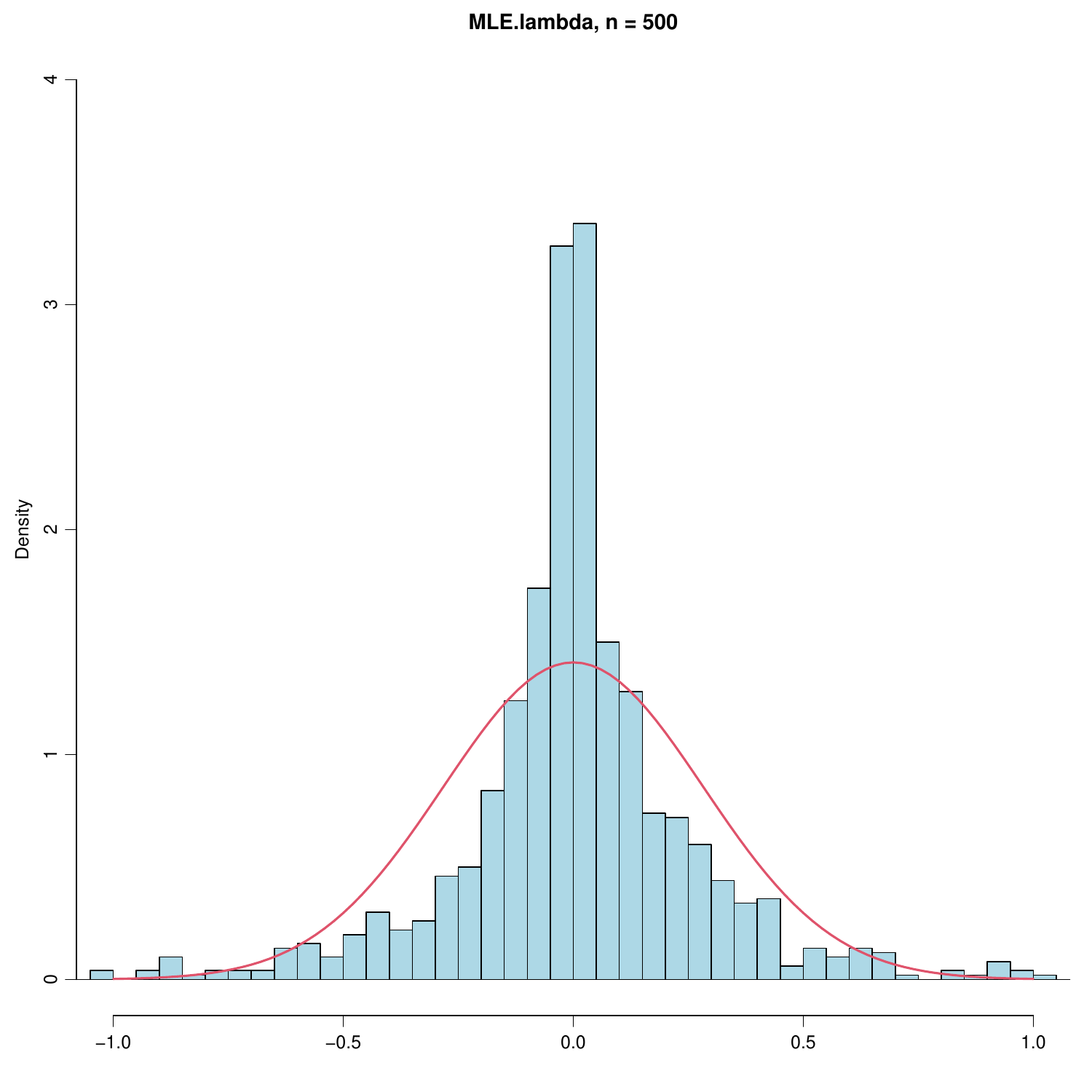}{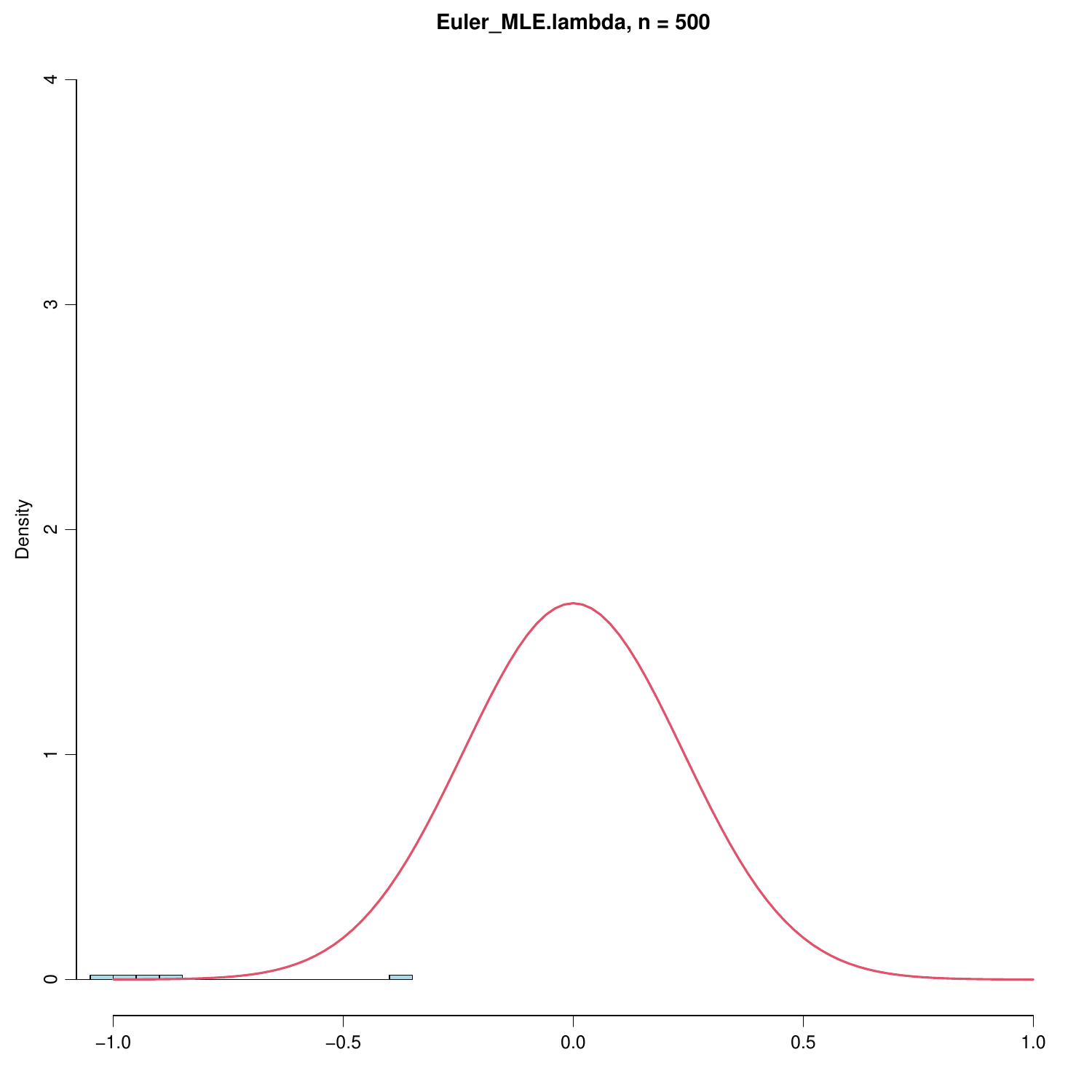}{$n=500$}\hfill
  \CellPair{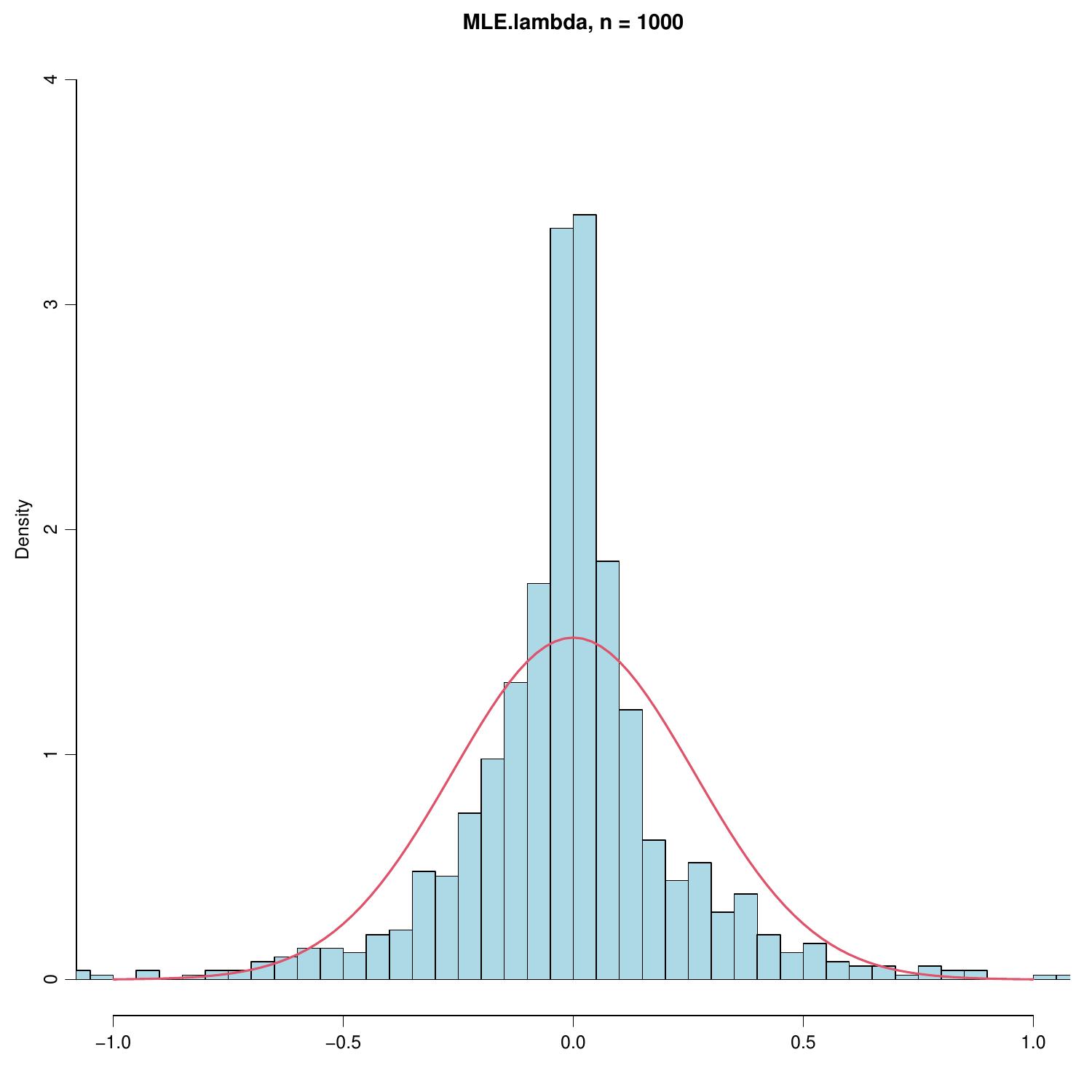}{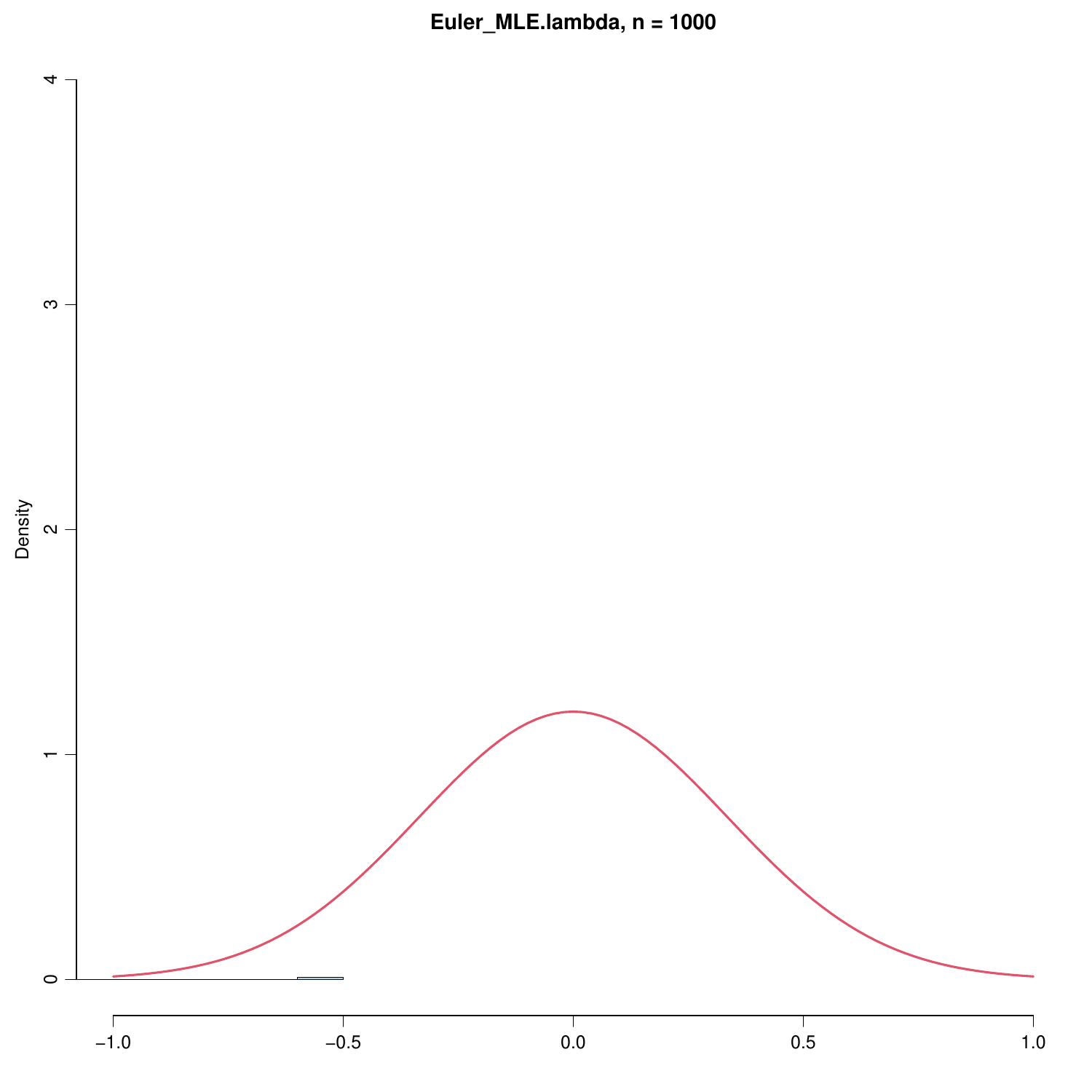}{$n=1000$}\hfill
  \CellPair{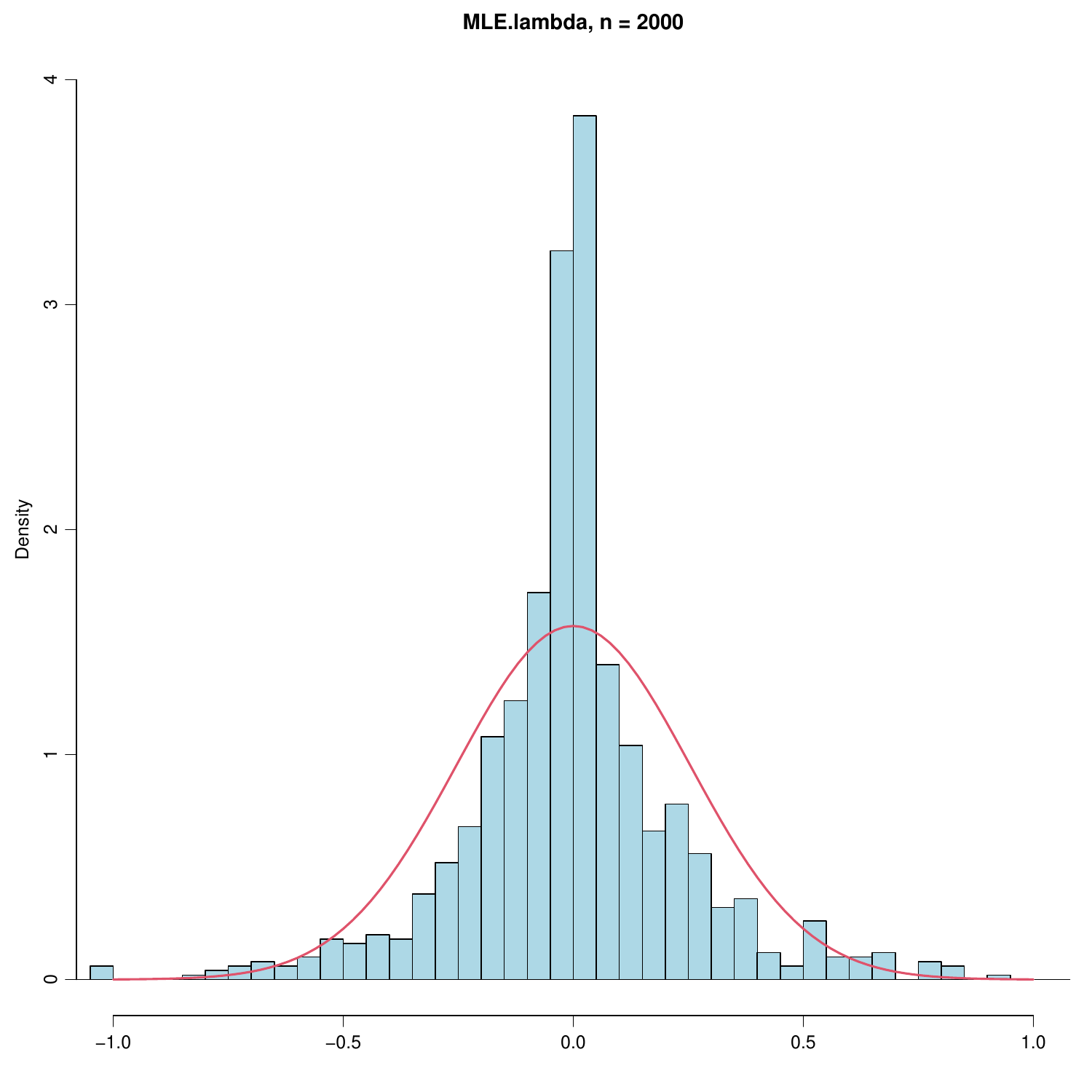}{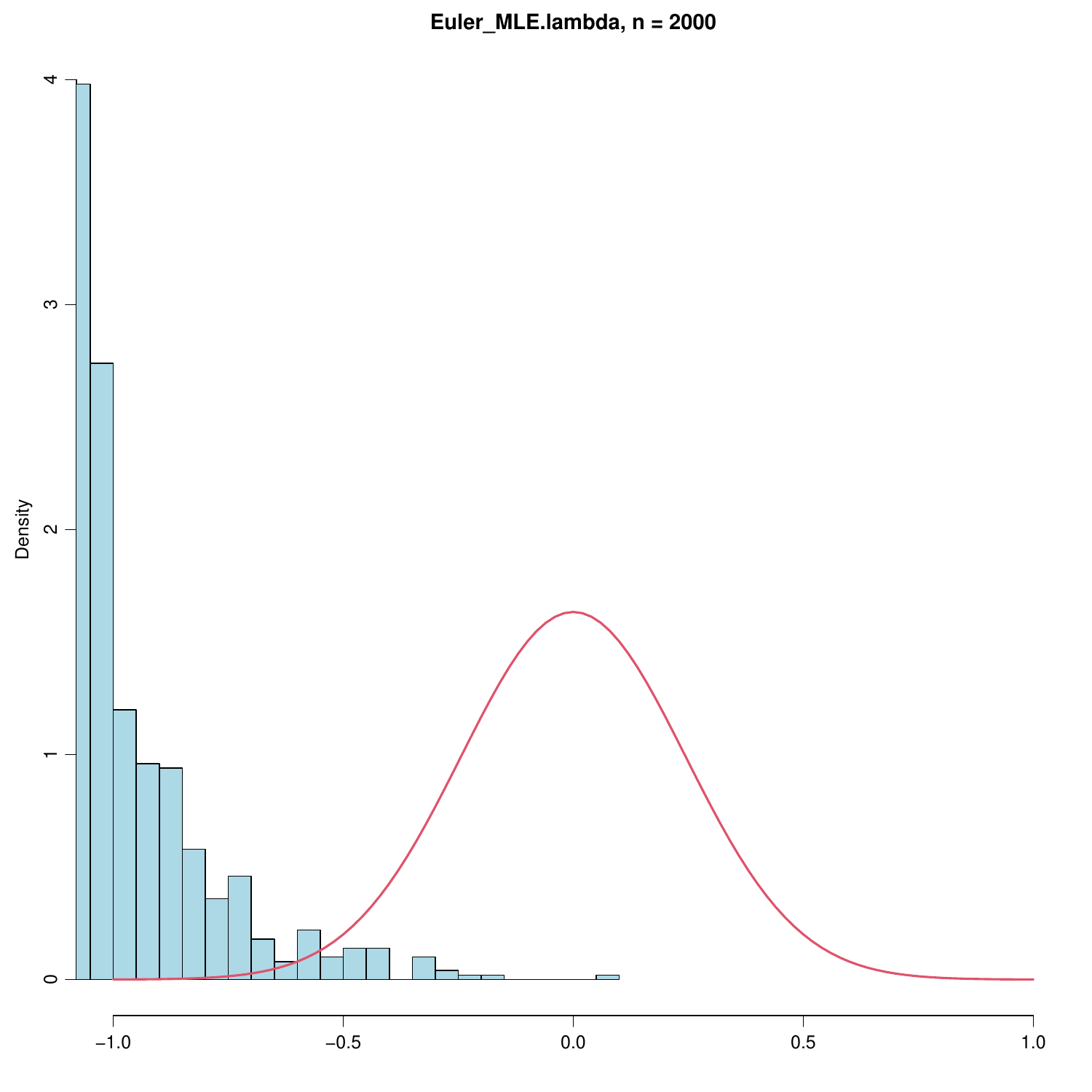}{$n=2000$}
  
  \vspace{5mm}
  
  \rowtitle{$\mu$}
  \CellPair{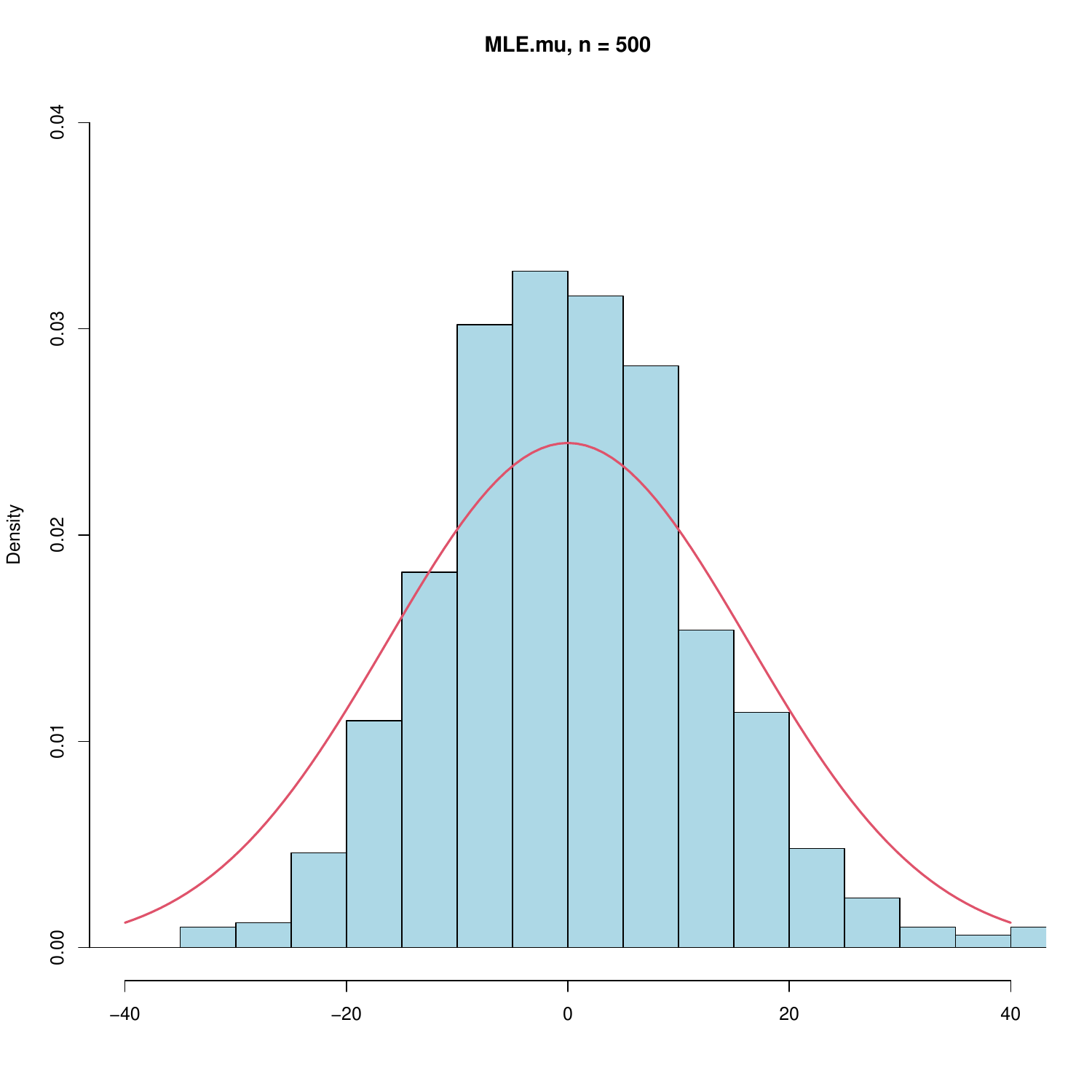}{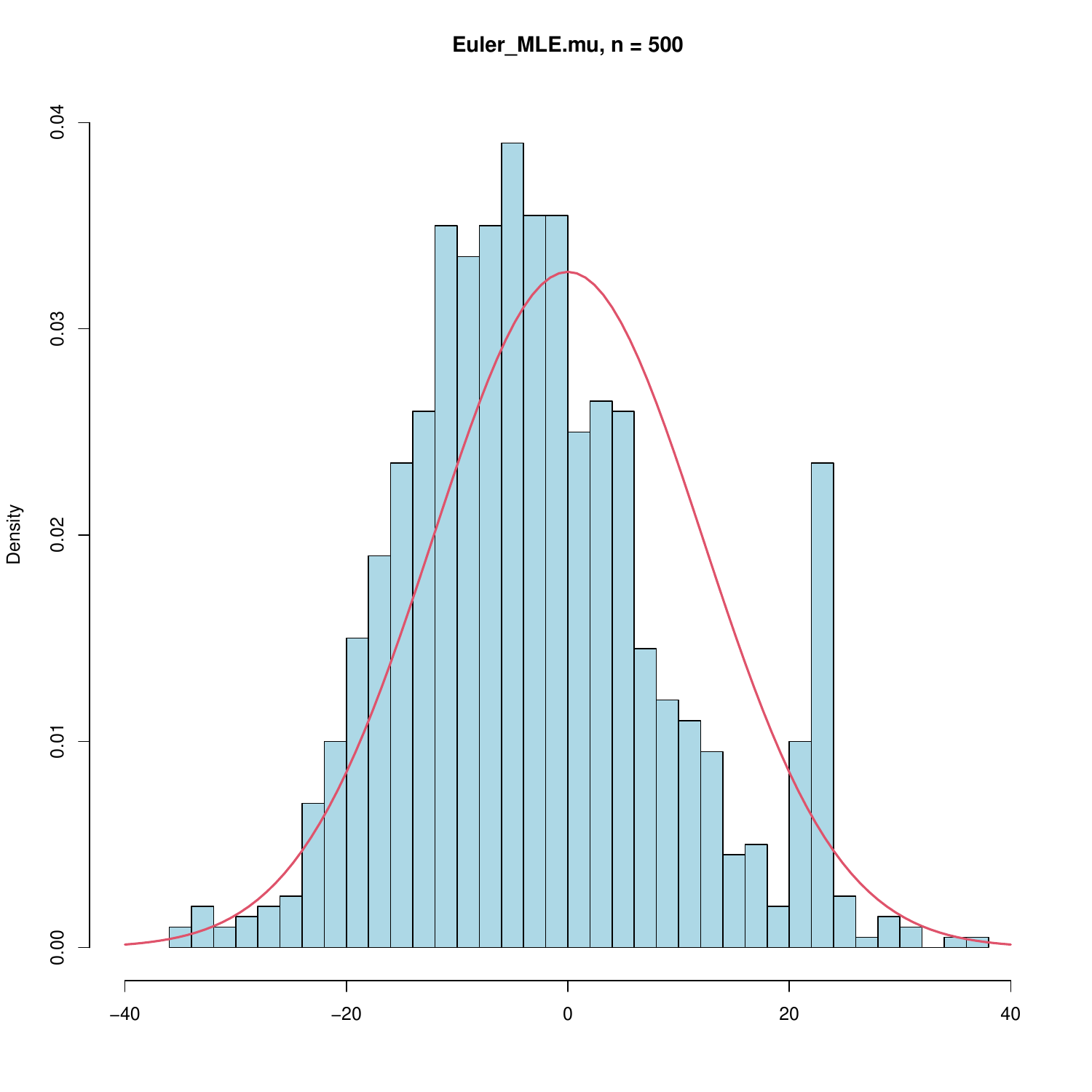}{$n=500$}\hfill
  \CellPair{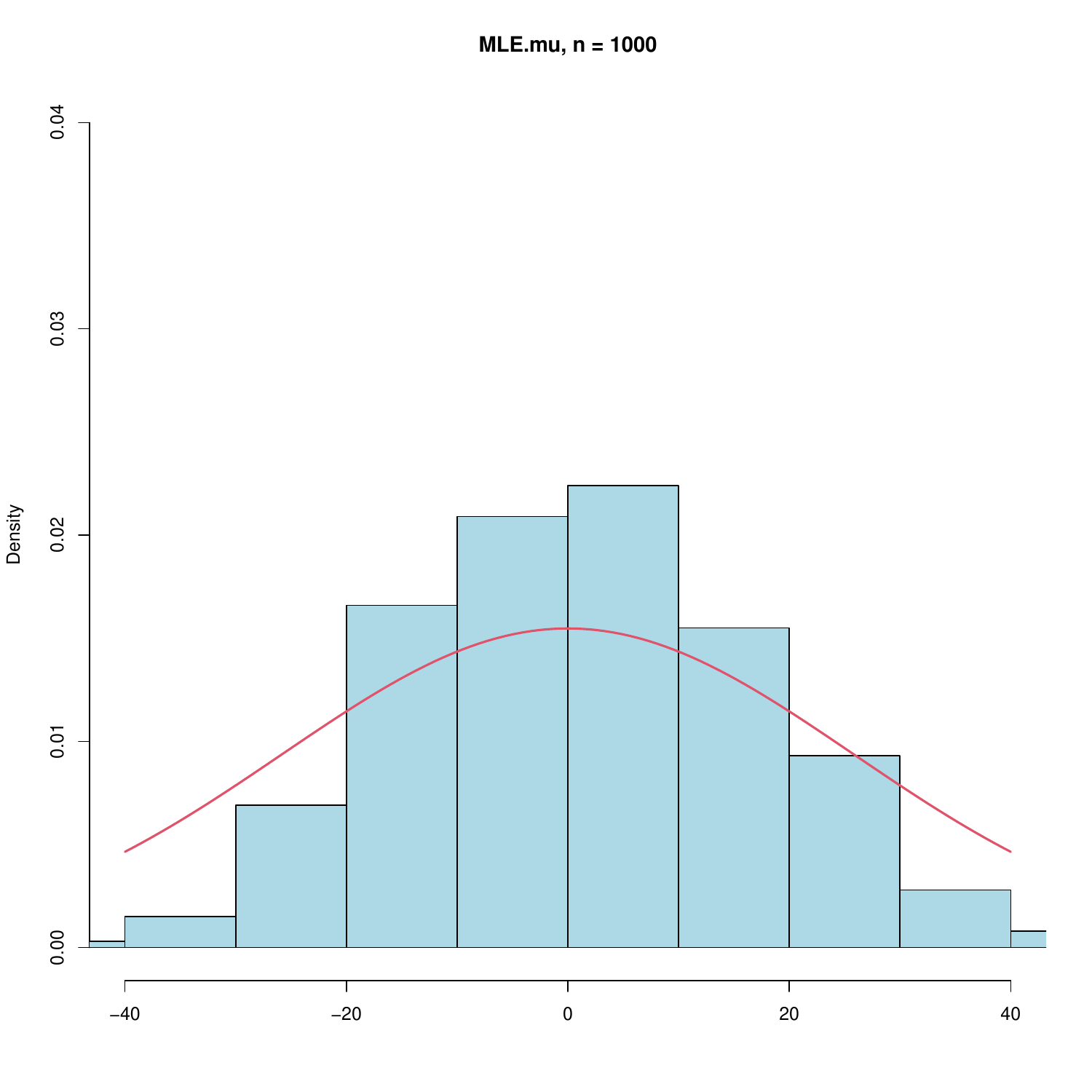}{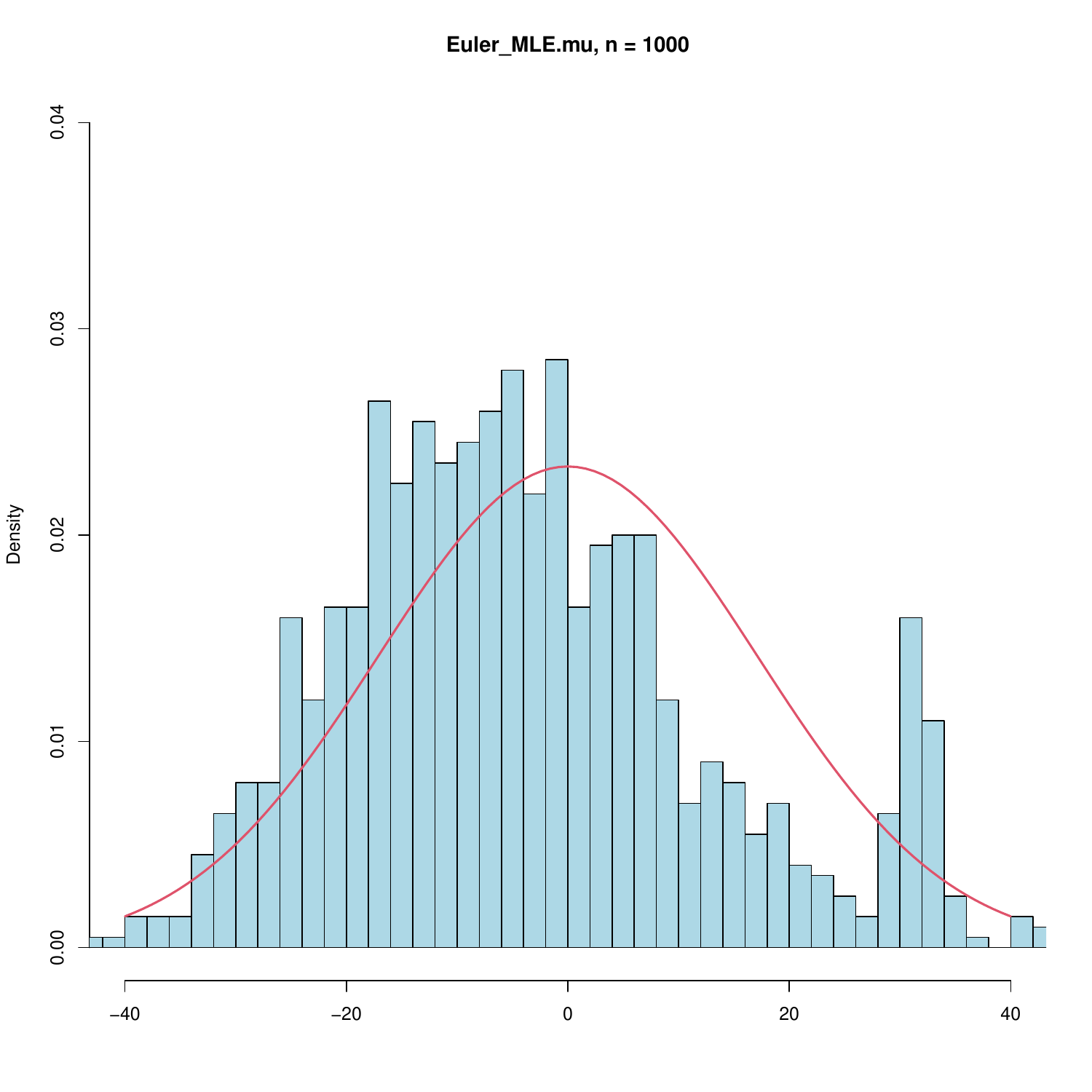}{$n=1000$}\hfill
  \CellPair{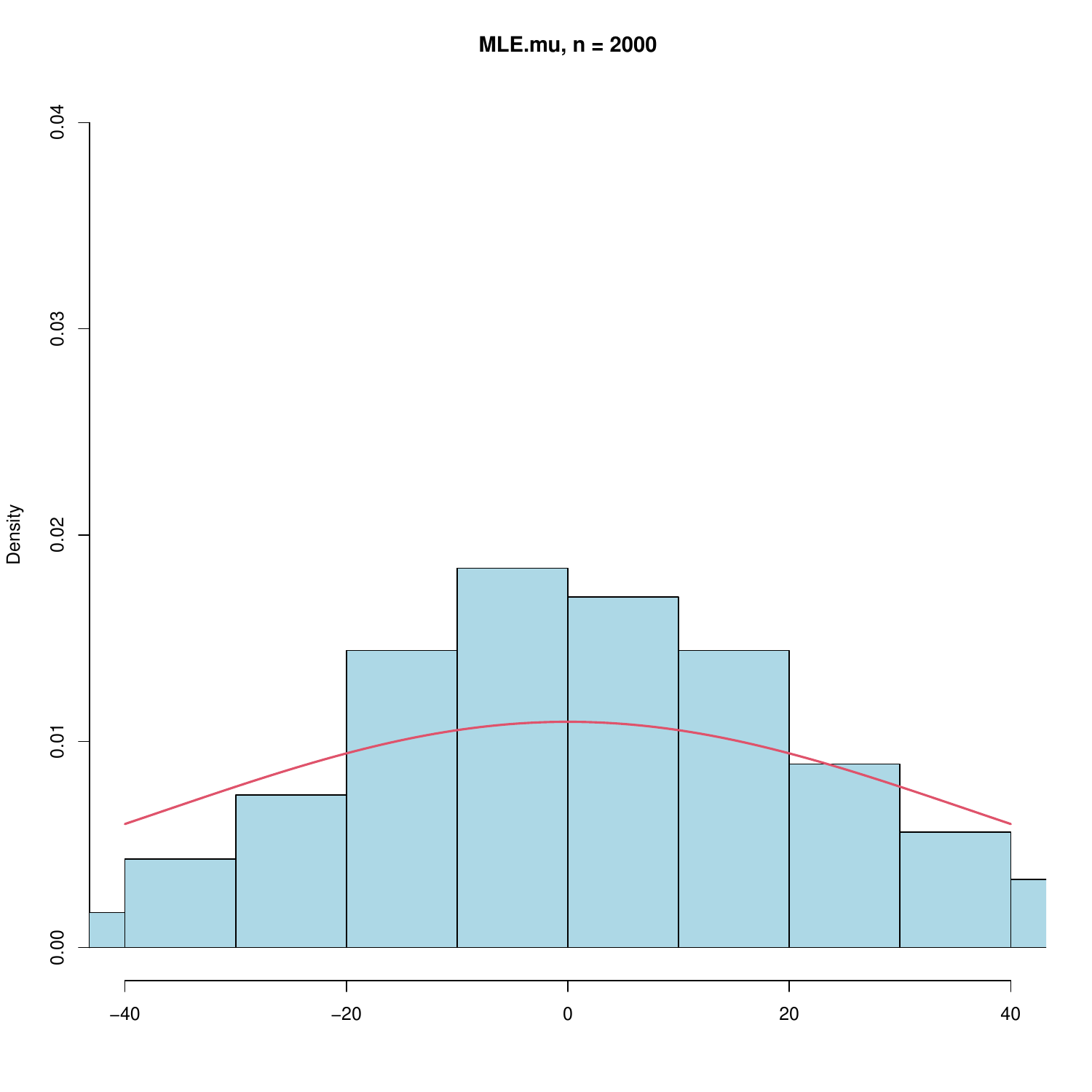}{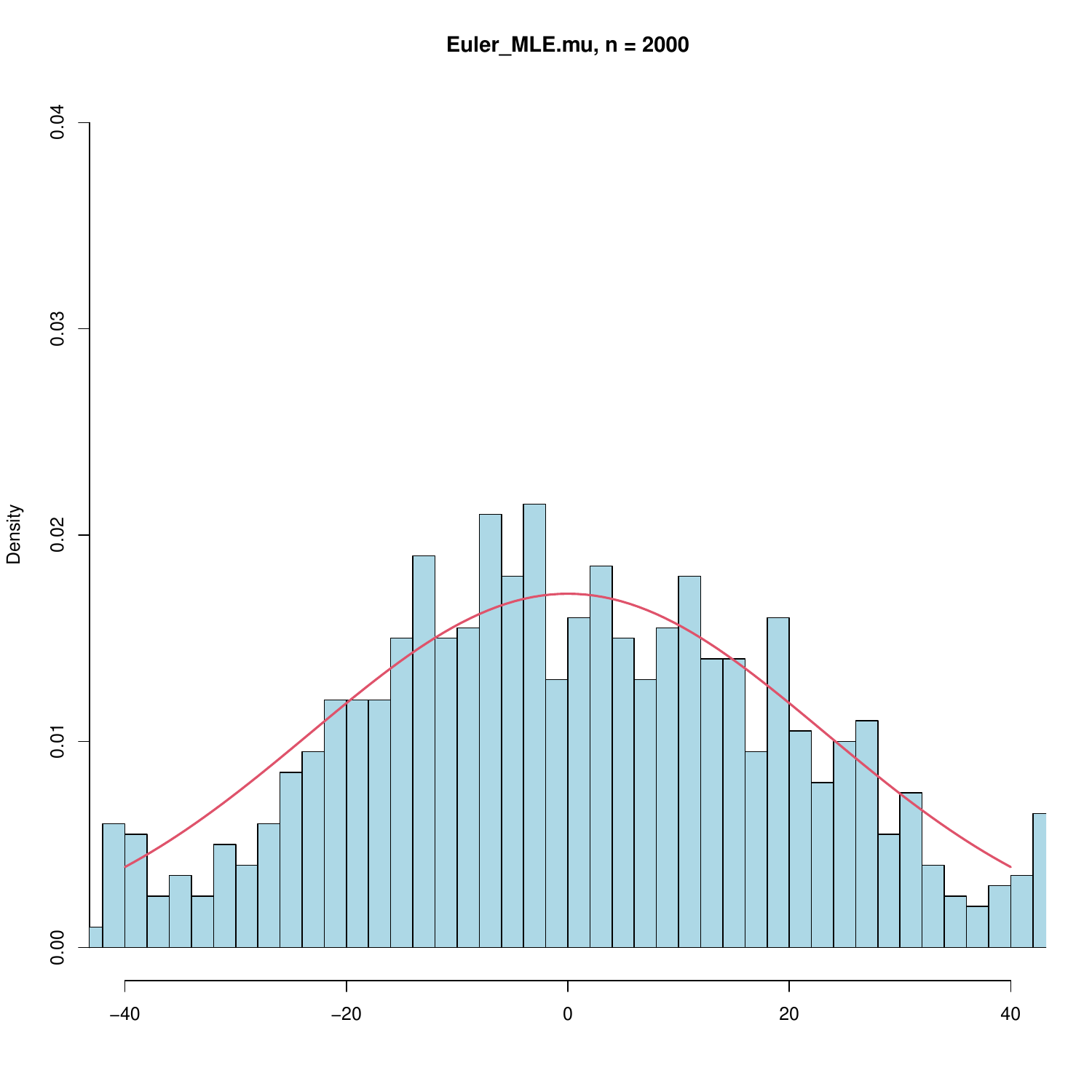}{$n=2000$}

  \vspace{5mm}
  
  \rowtitle{$\alpha$}
  \CellPair{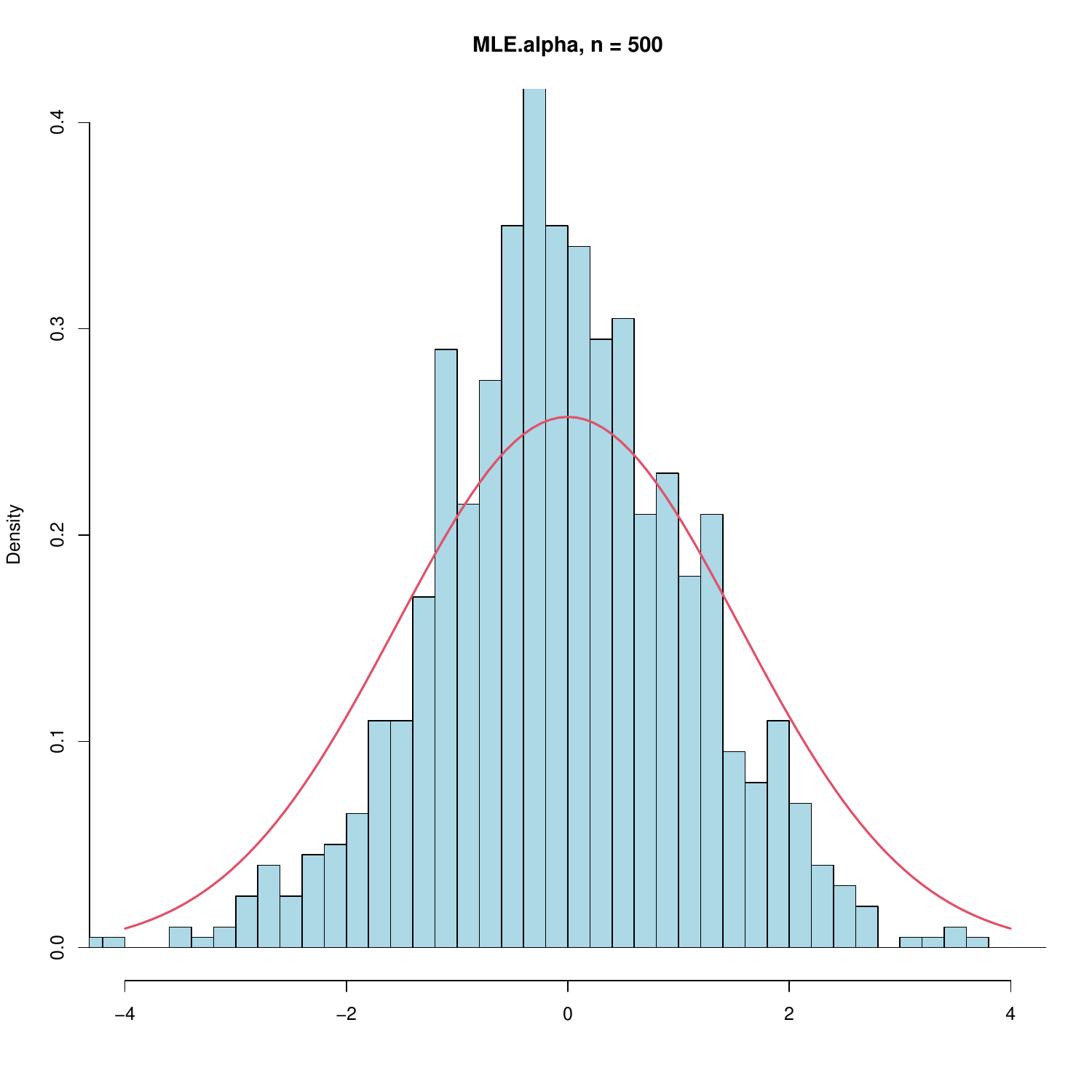}{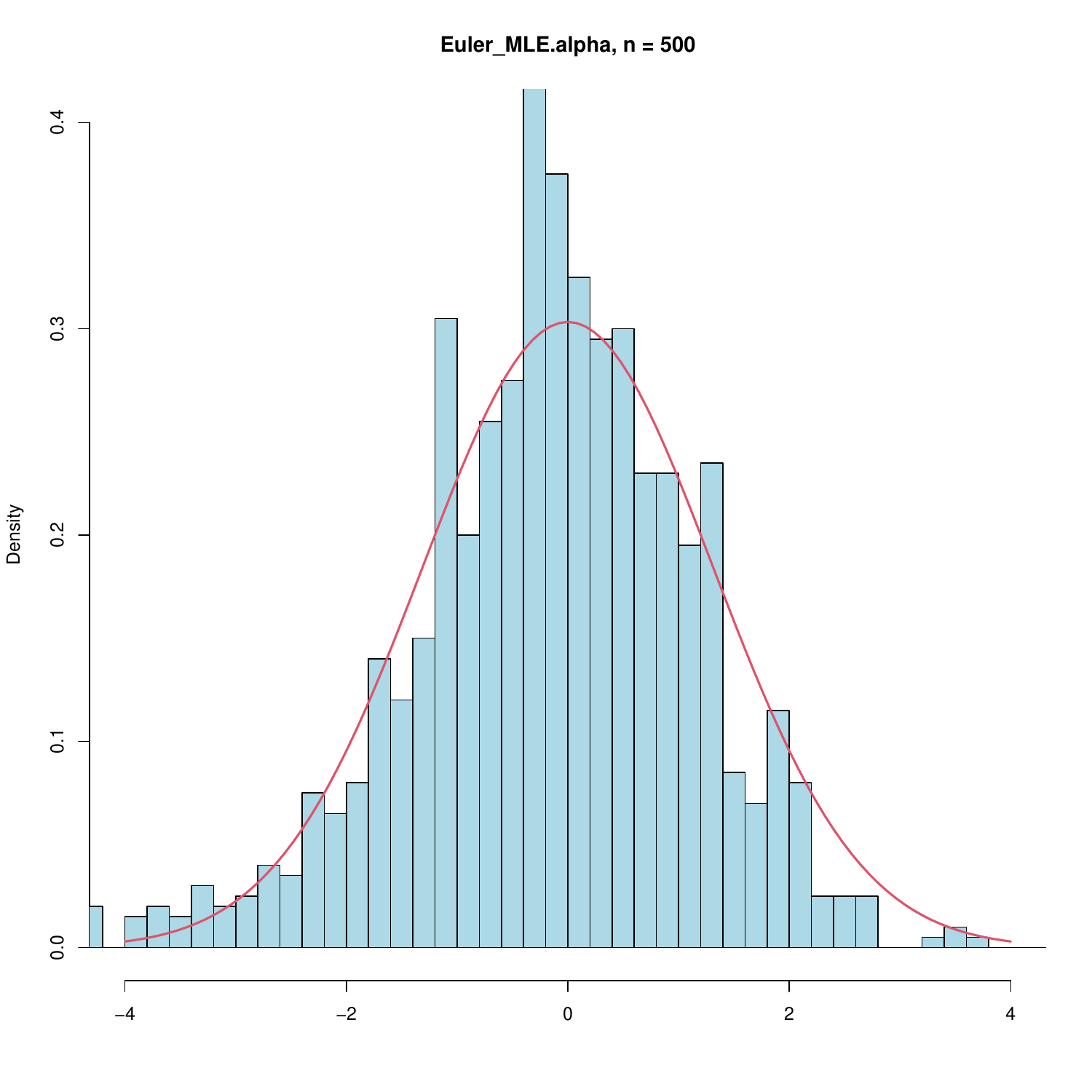}{$n=500$}\hfill
  \CellPair{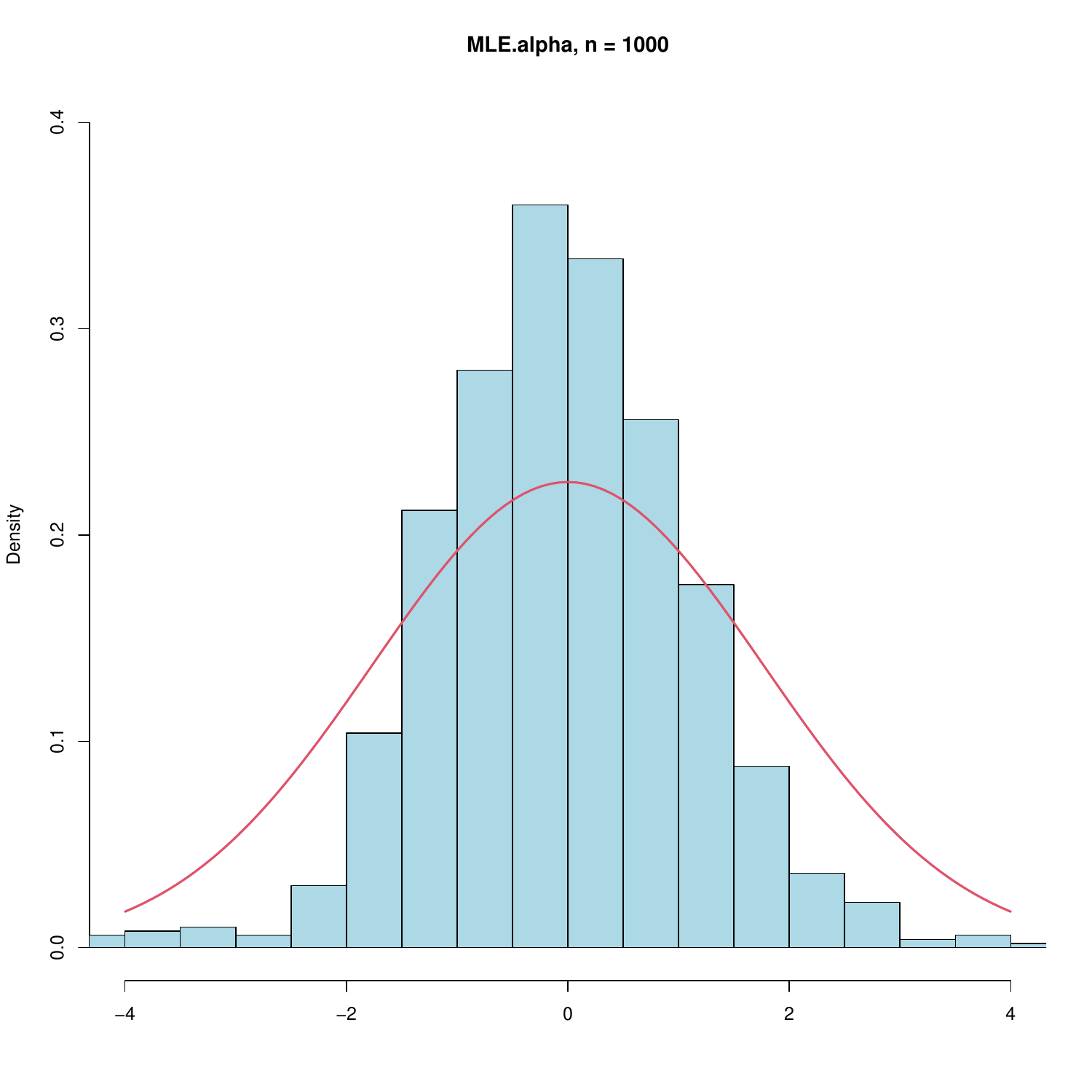}{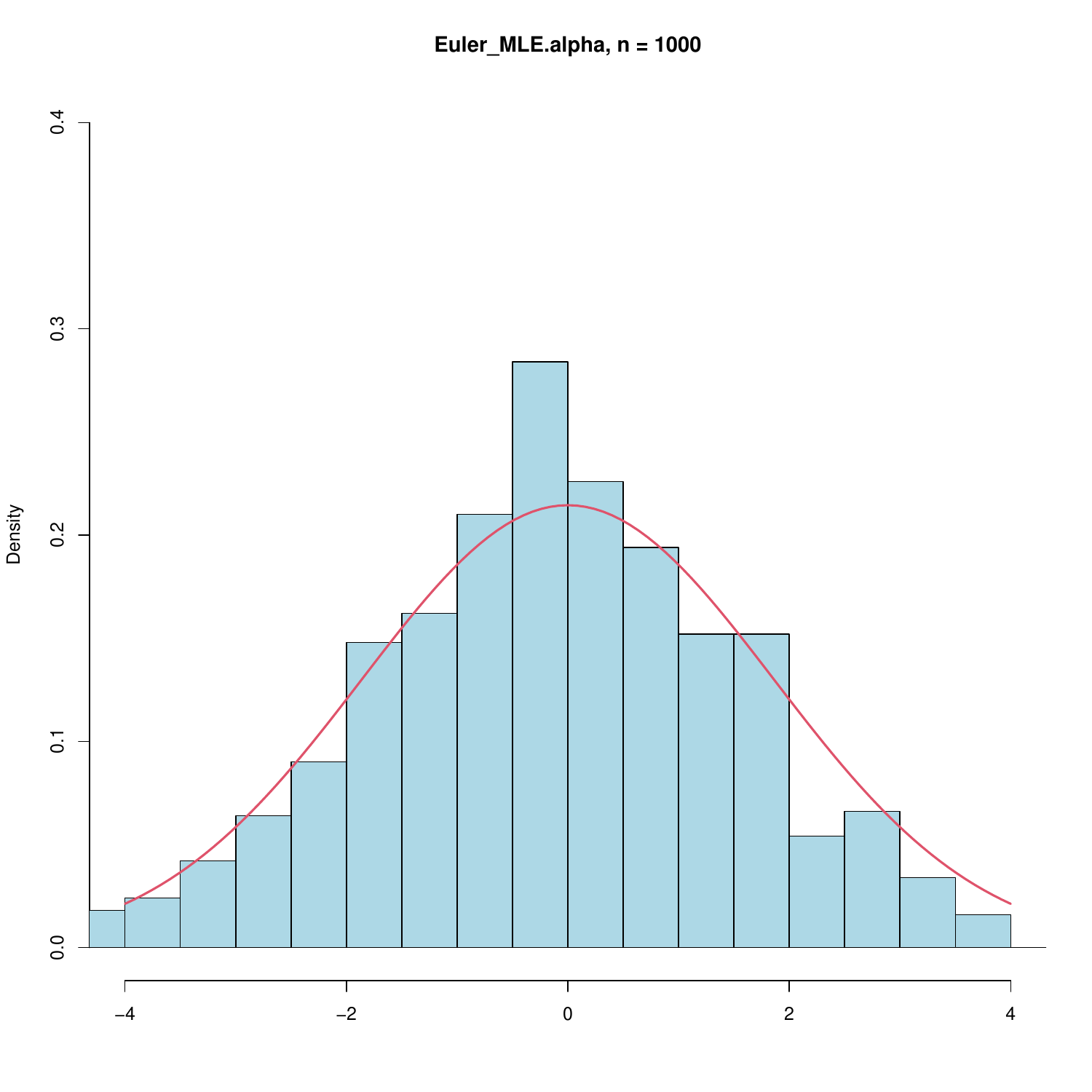}{$n=1000$}\hfill
  \CellPair{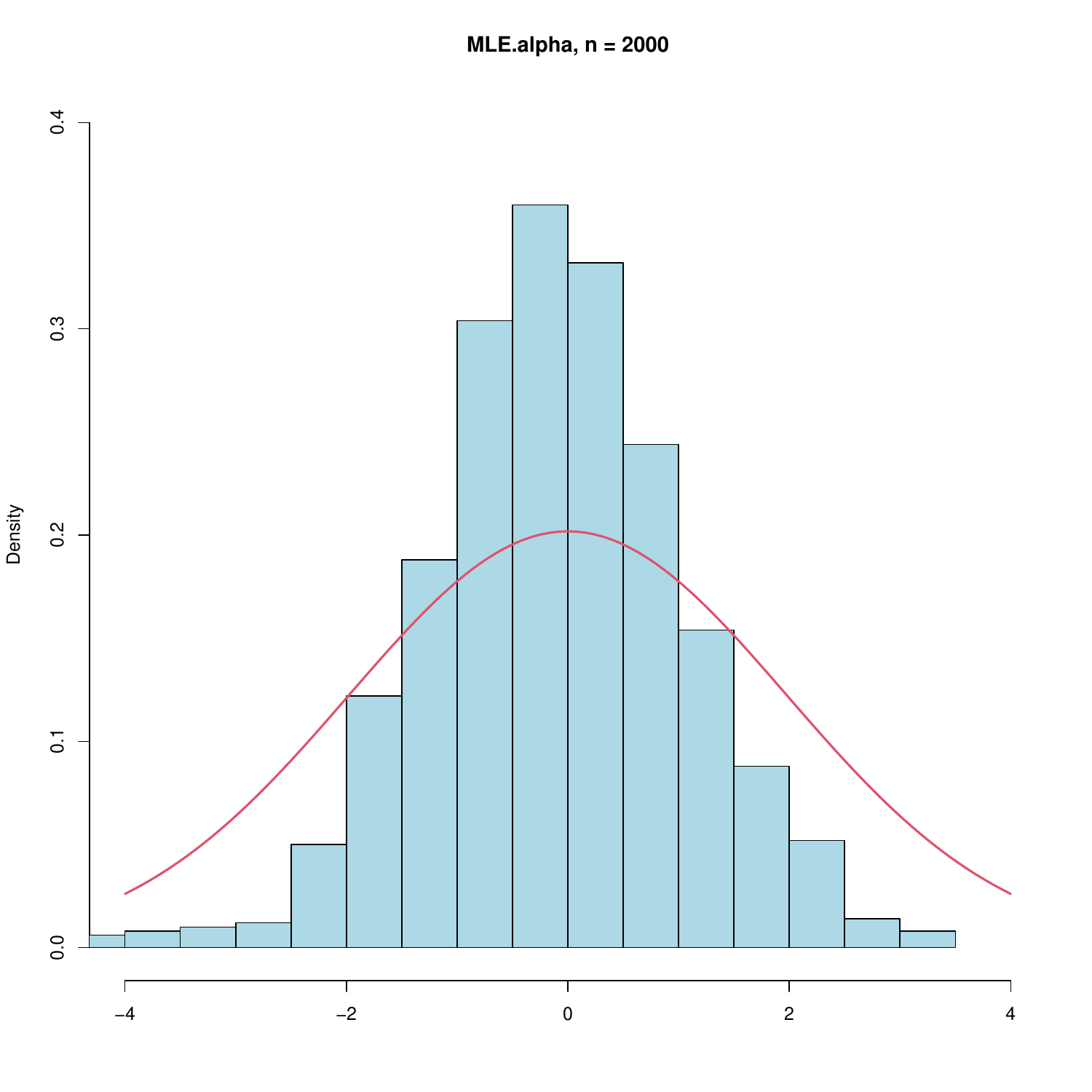}{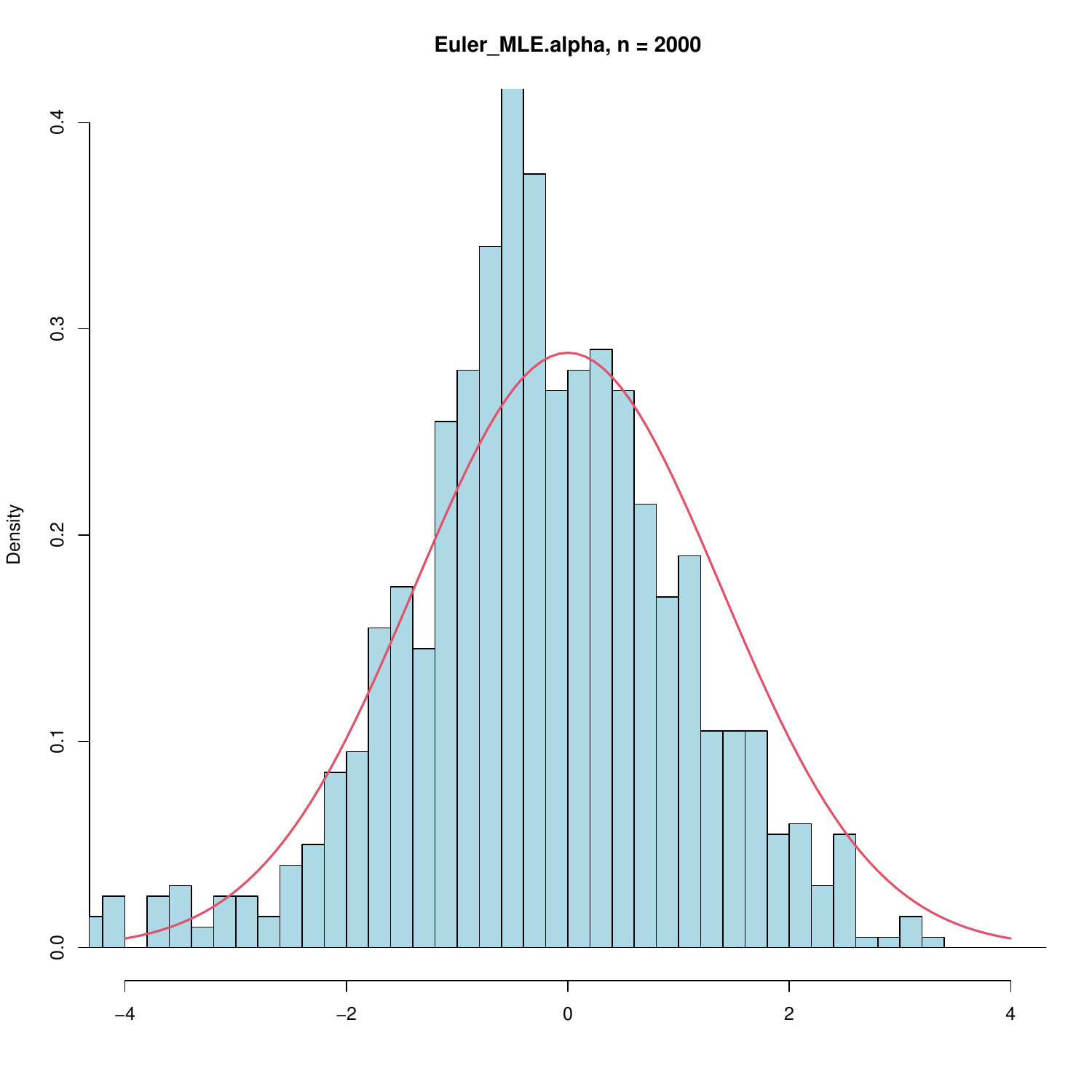}{$n=2000$}

  \vspace{5mm}
  
  \rowtitle{$\sigma$}
  \CellPair{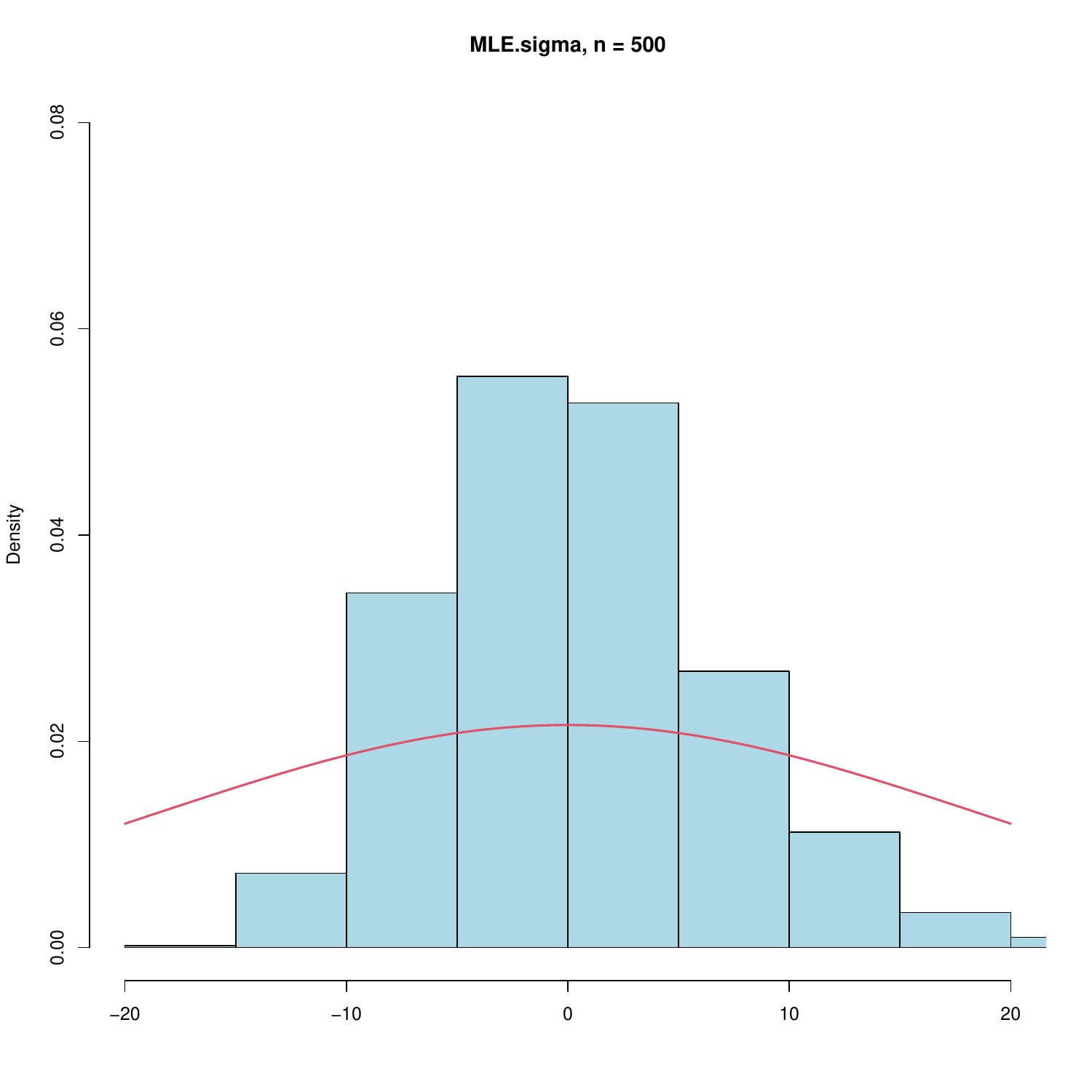}{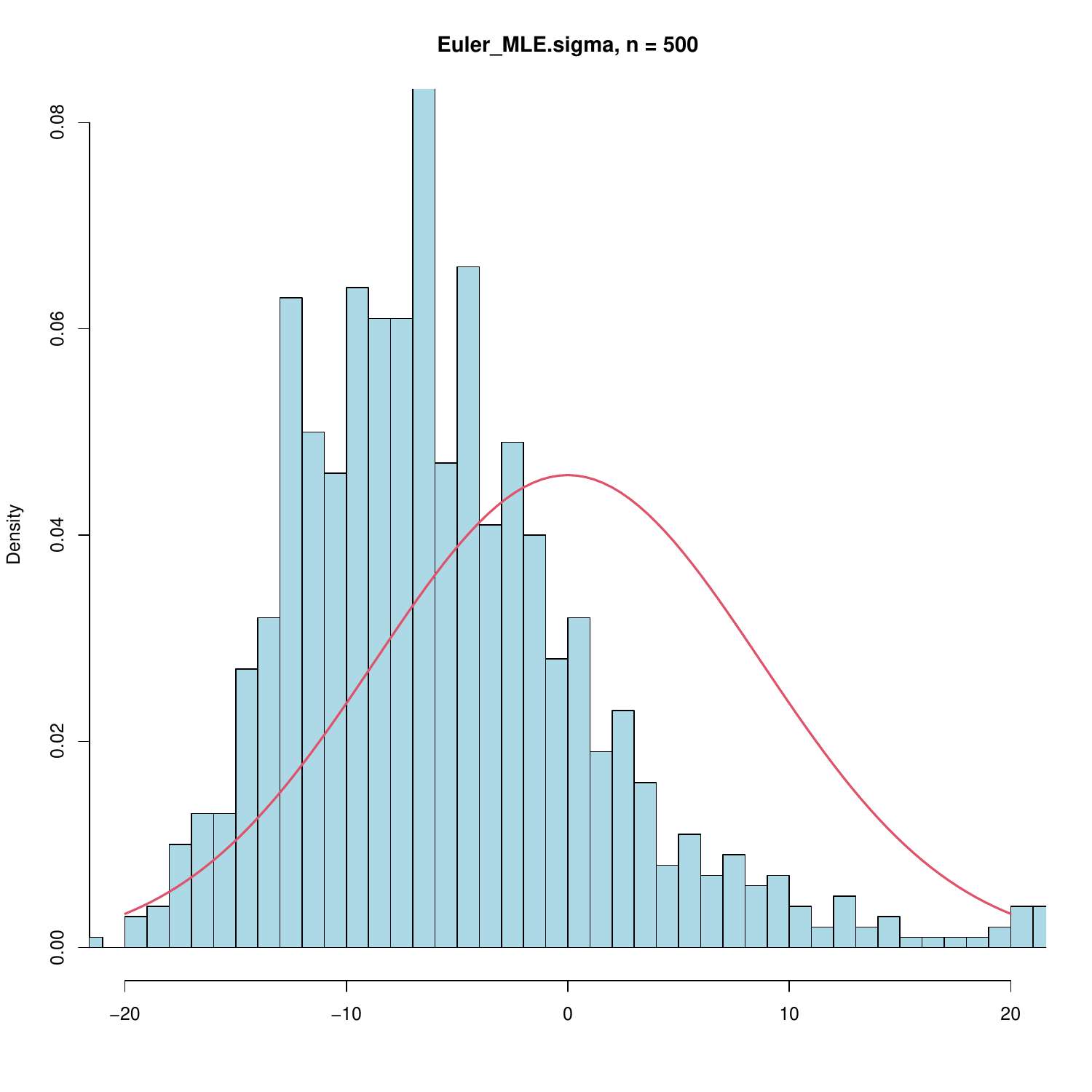}{$n=500$}\hfill
  \CellPair{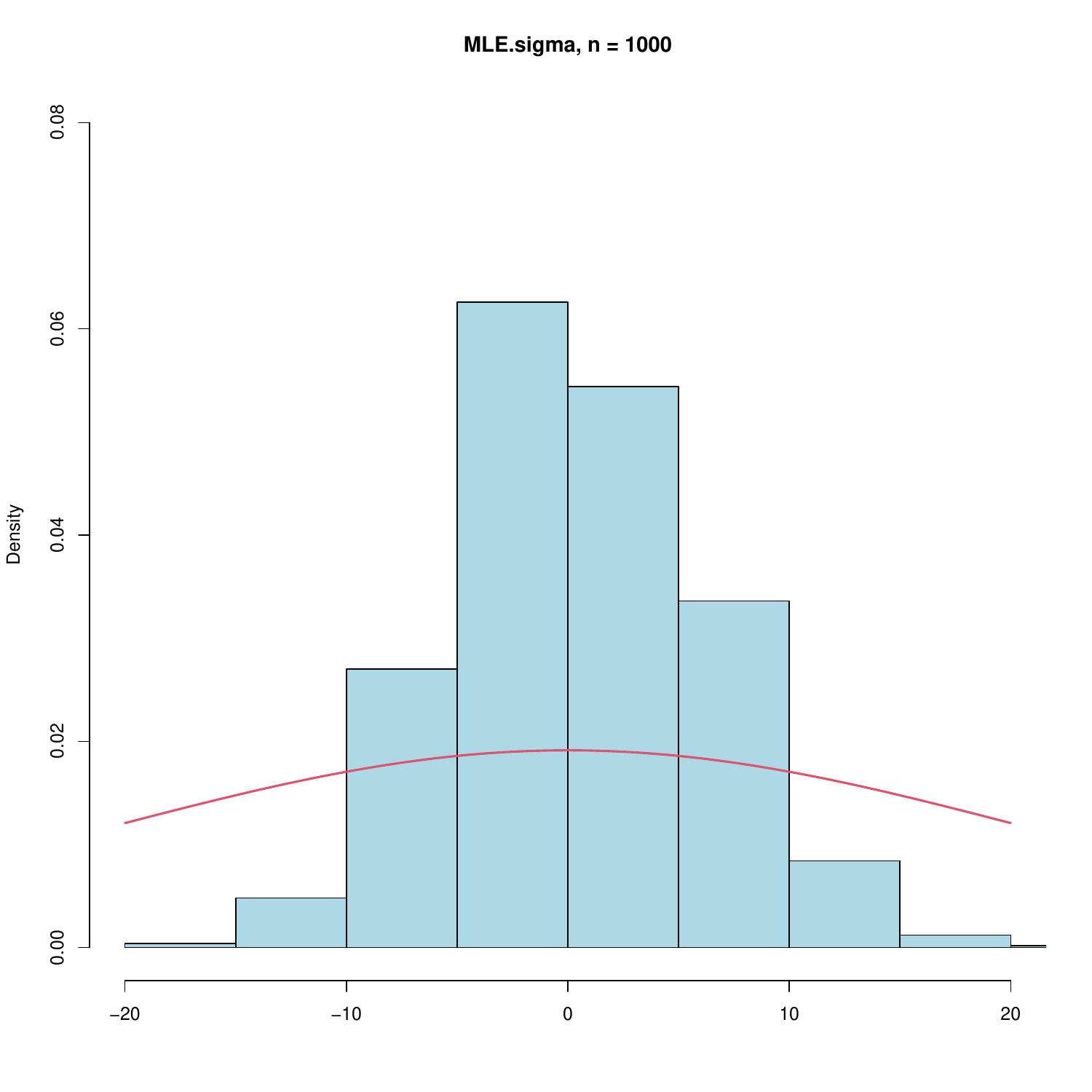}{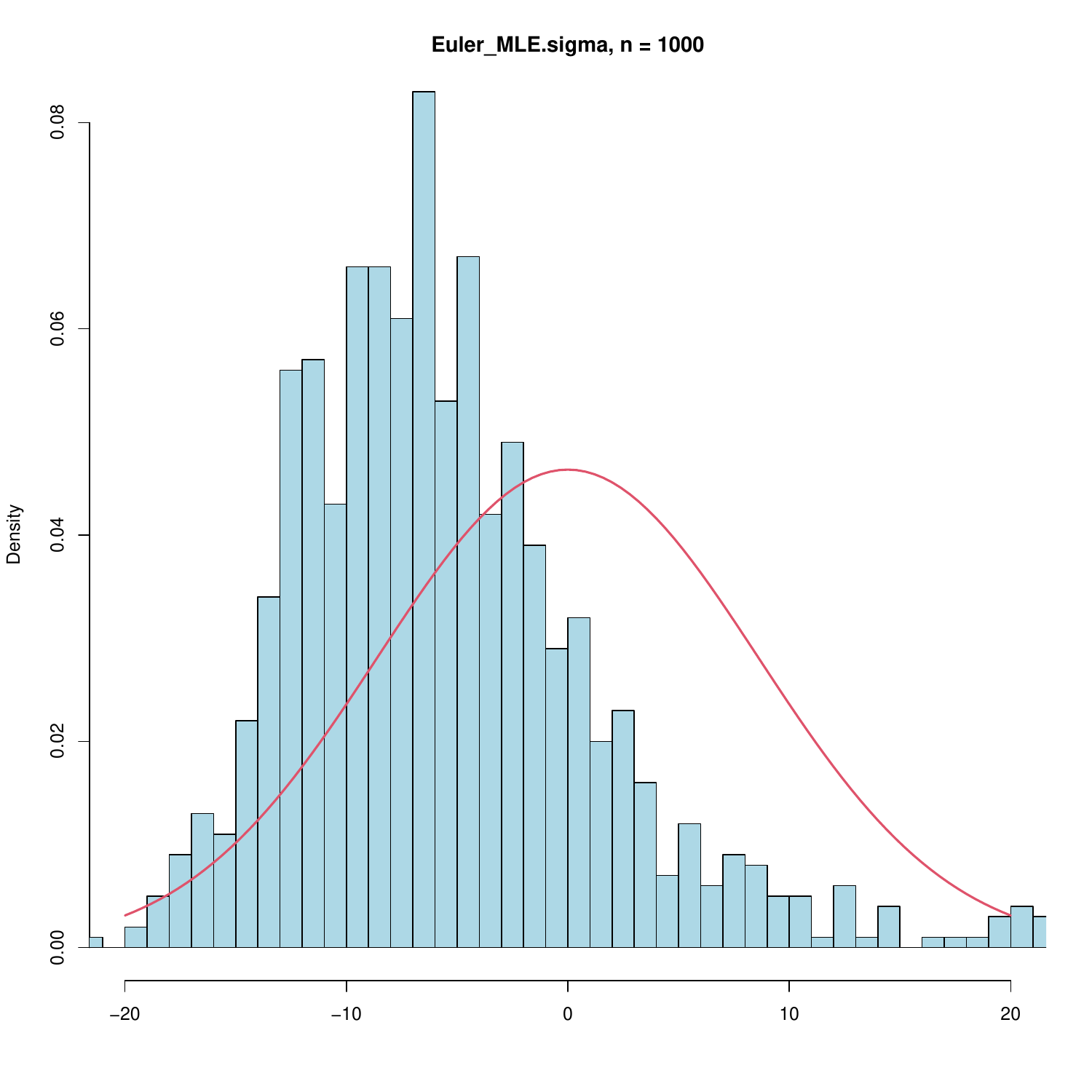}{$n=1000$}\hfill
  \CellPair{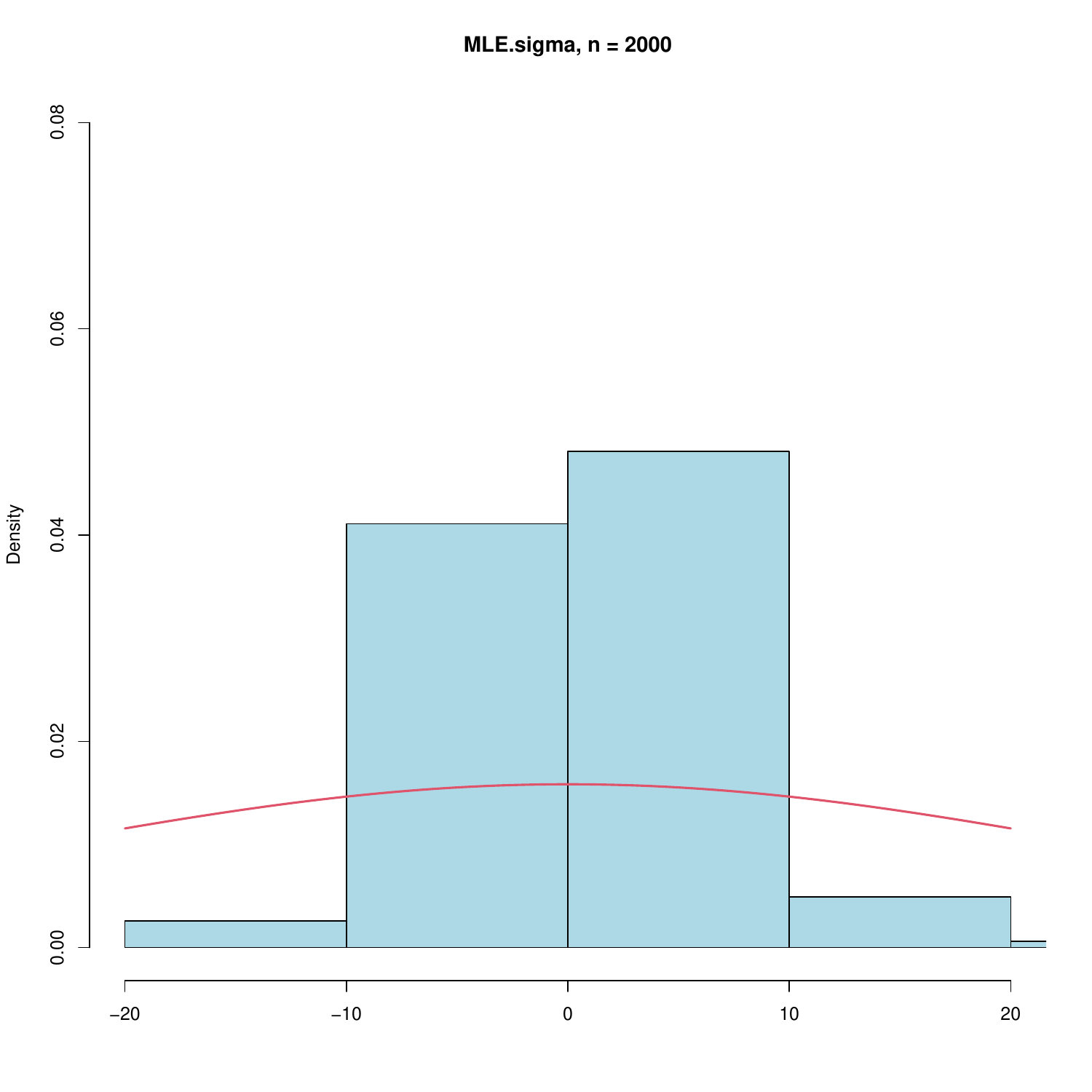}{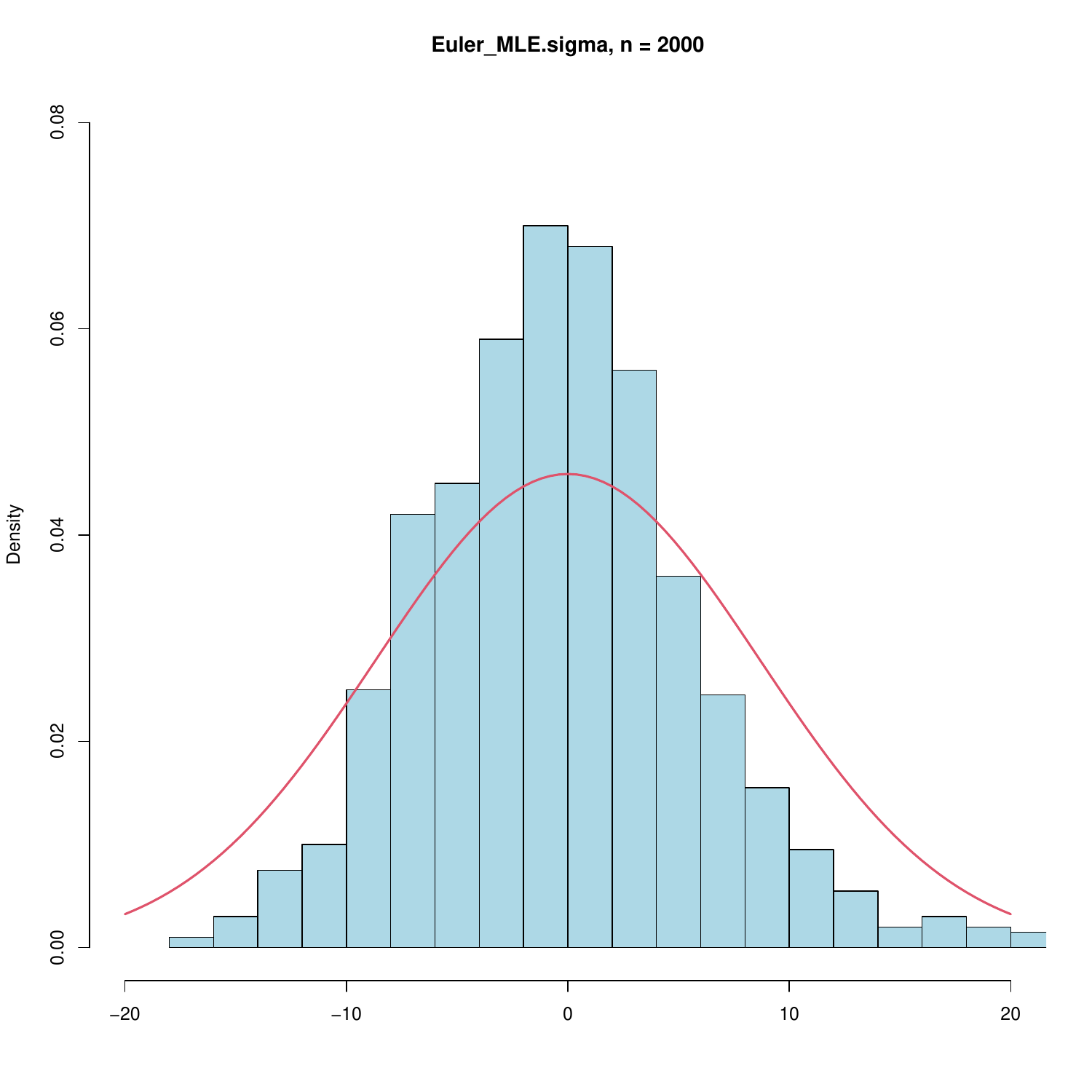}{$n=2000$}

  \vspace{5mm}

  \rowtitle{$\beta$}
  \CellPair{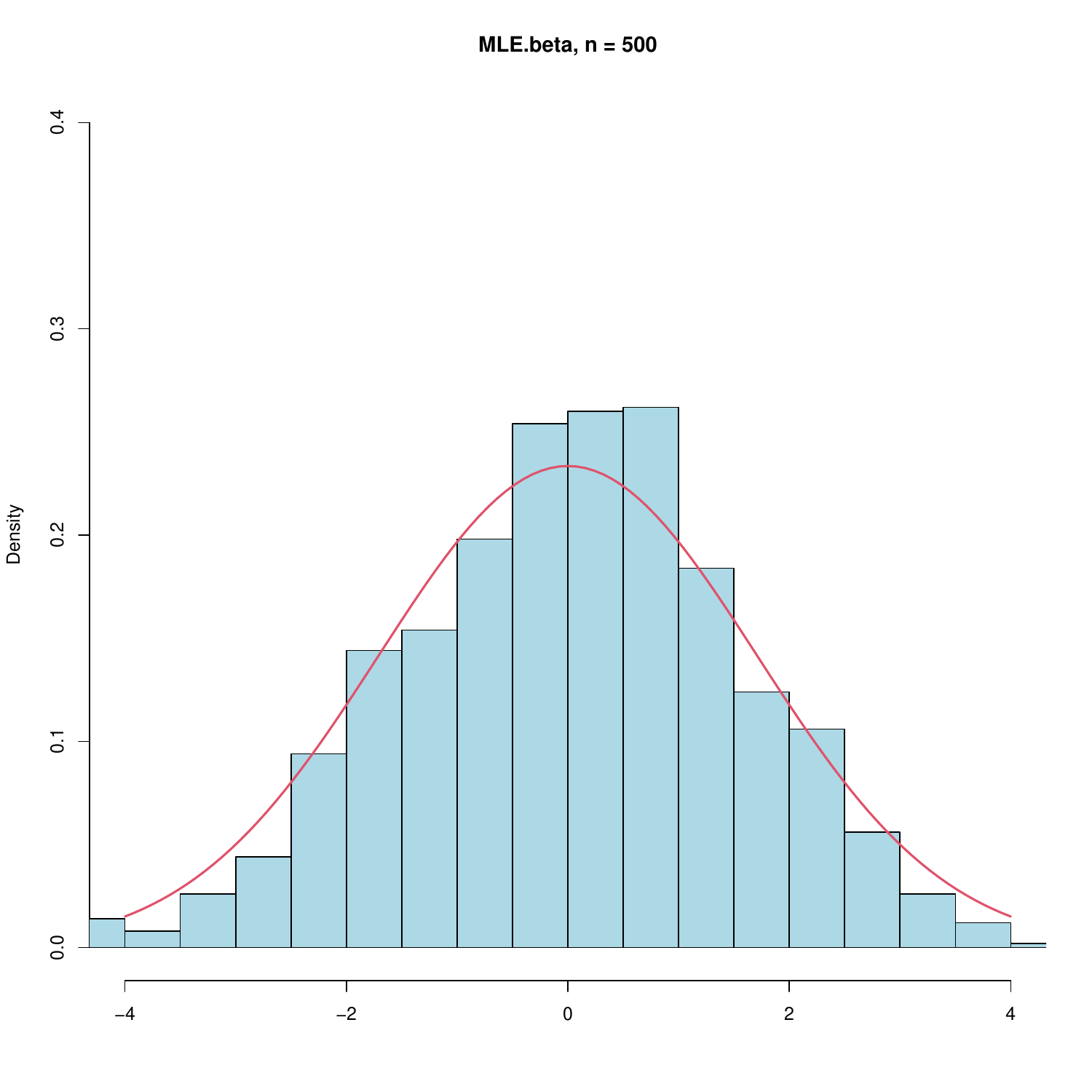}{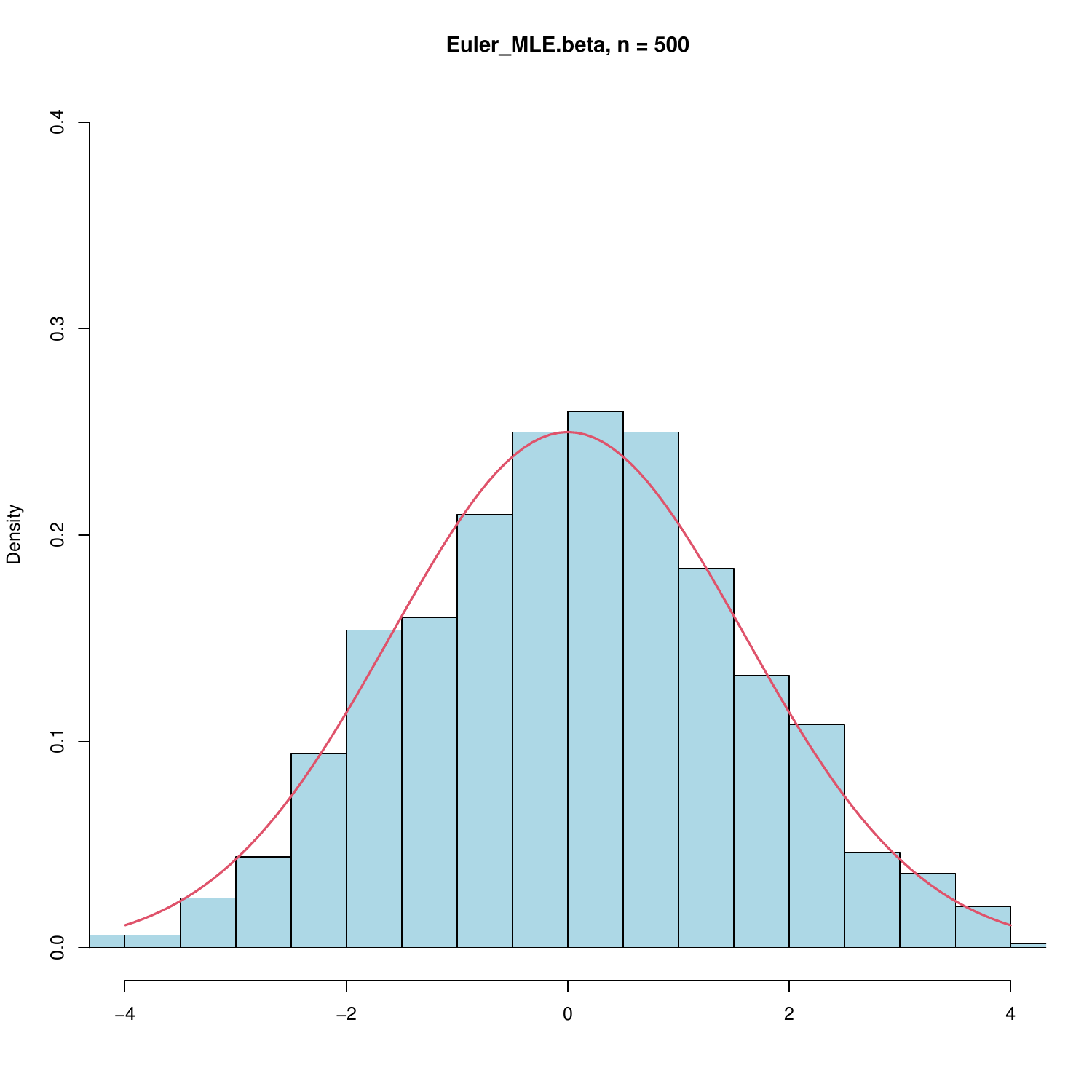}{$n=500$}\hfill
  \CellPair{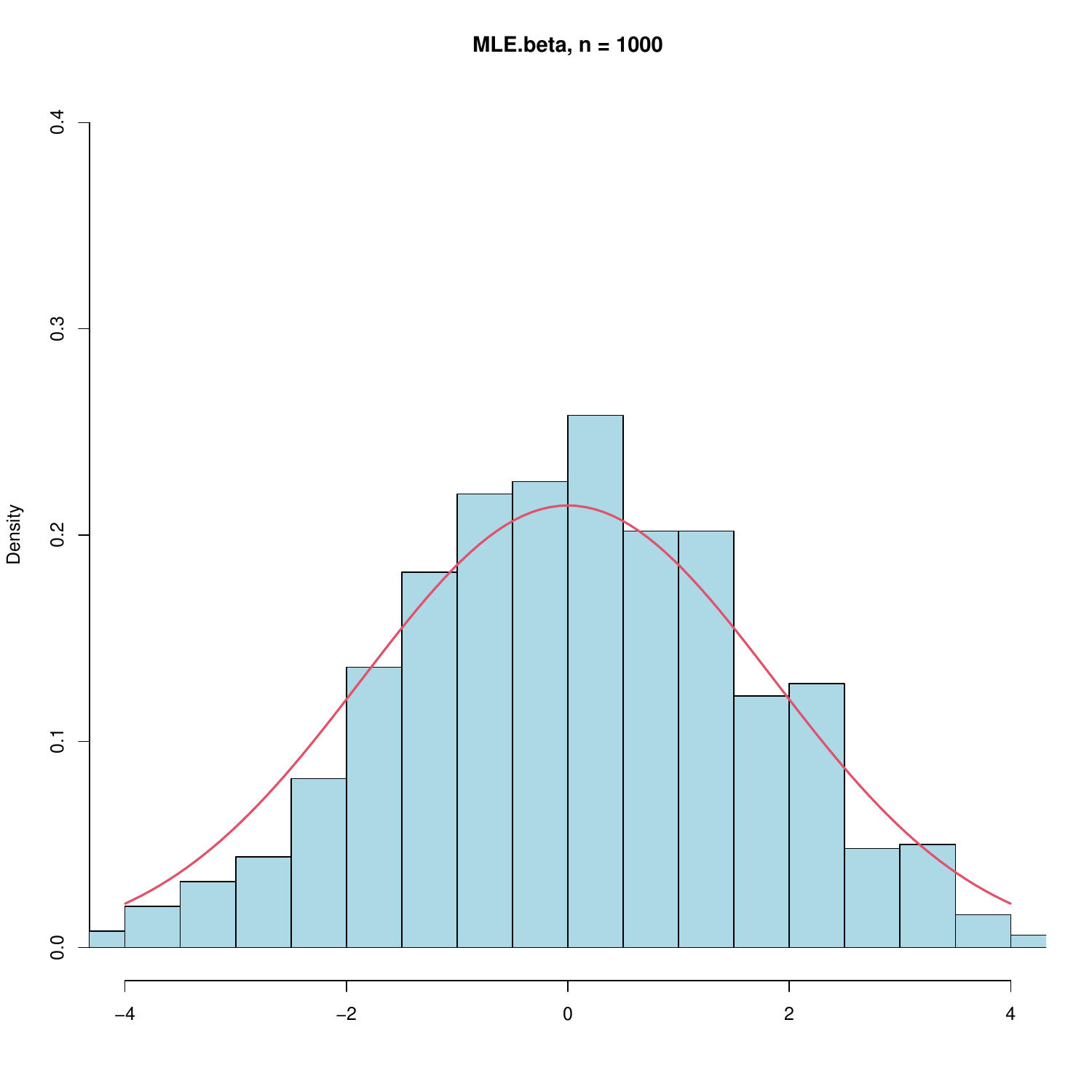}{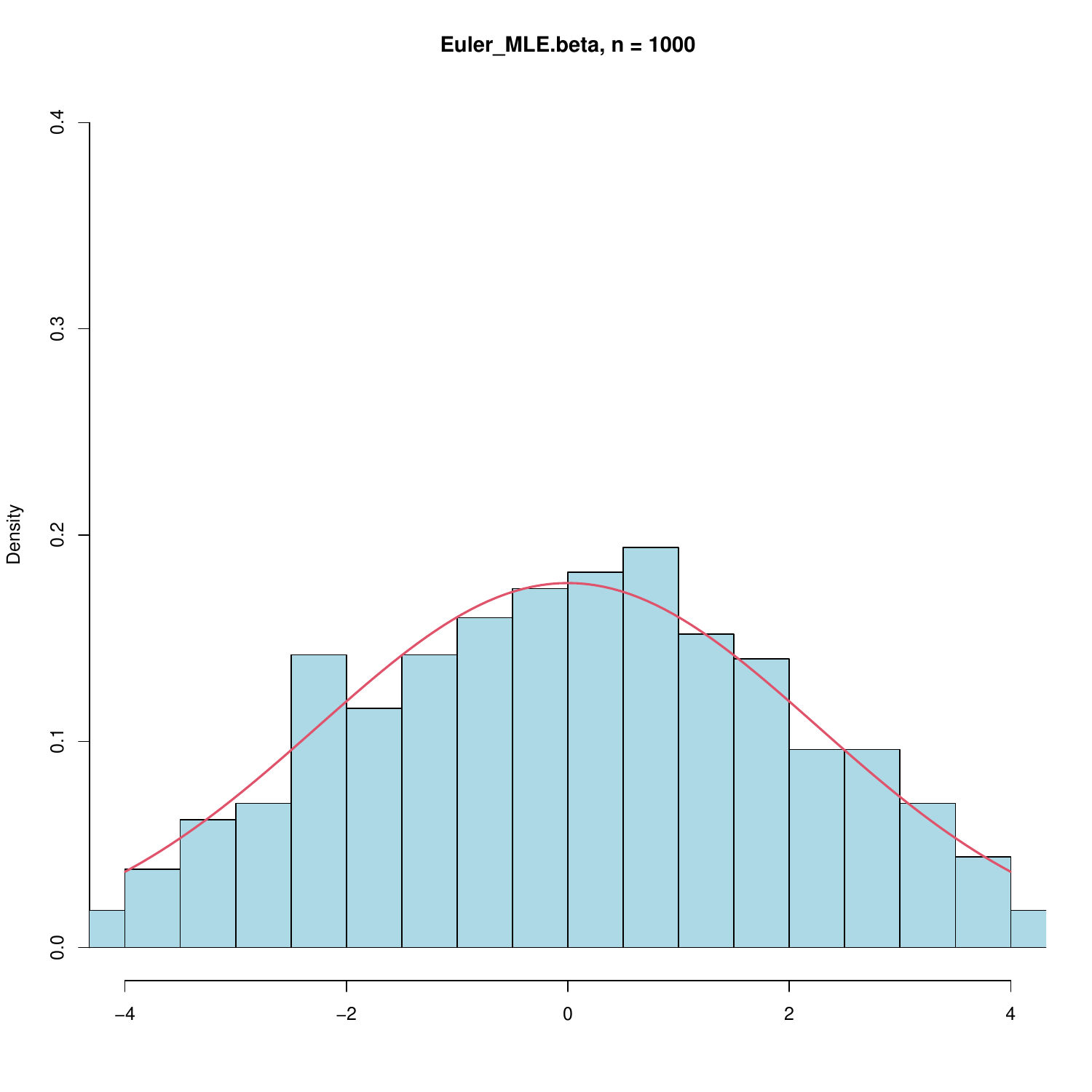}{$n=1000$}\hfill
  \CellPair{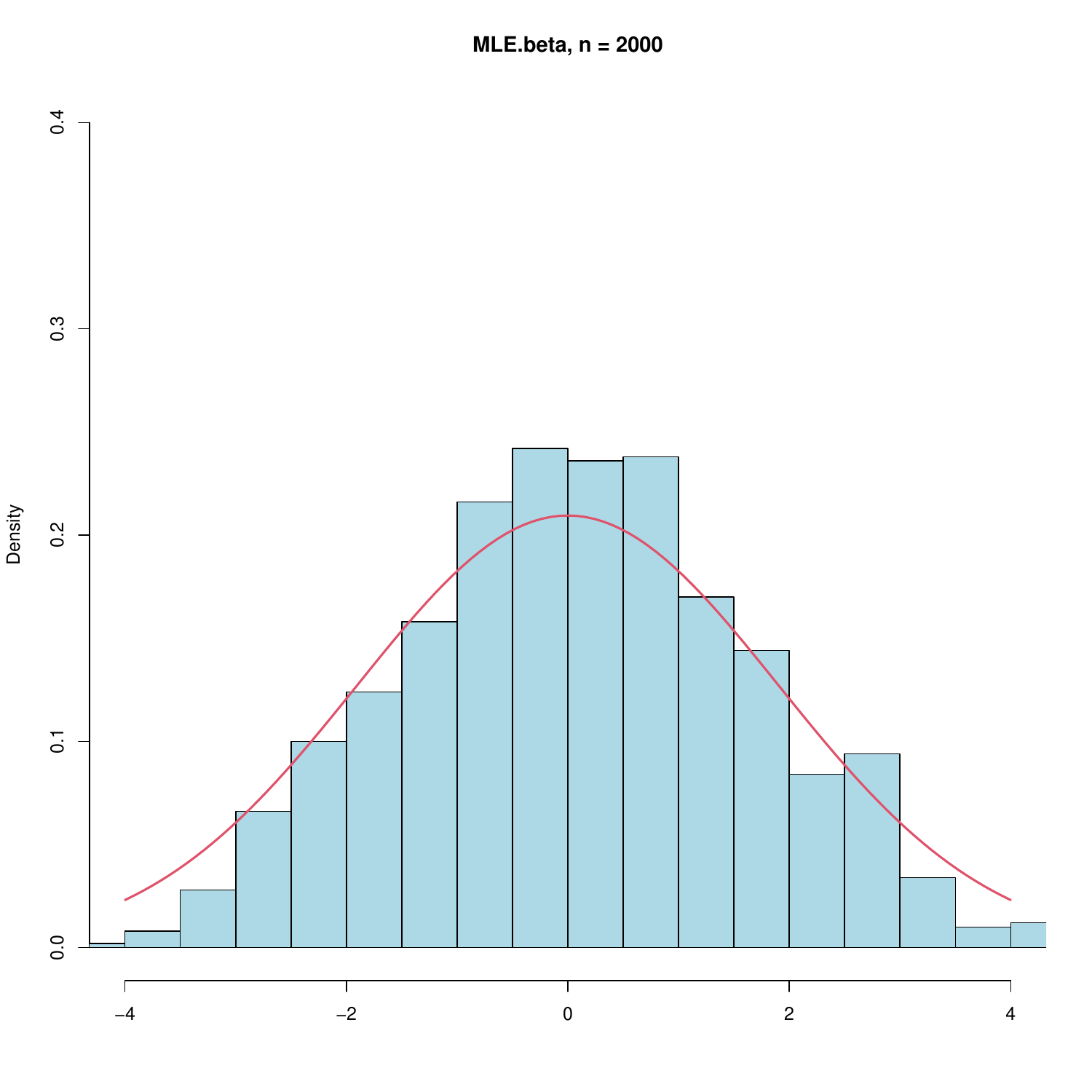}{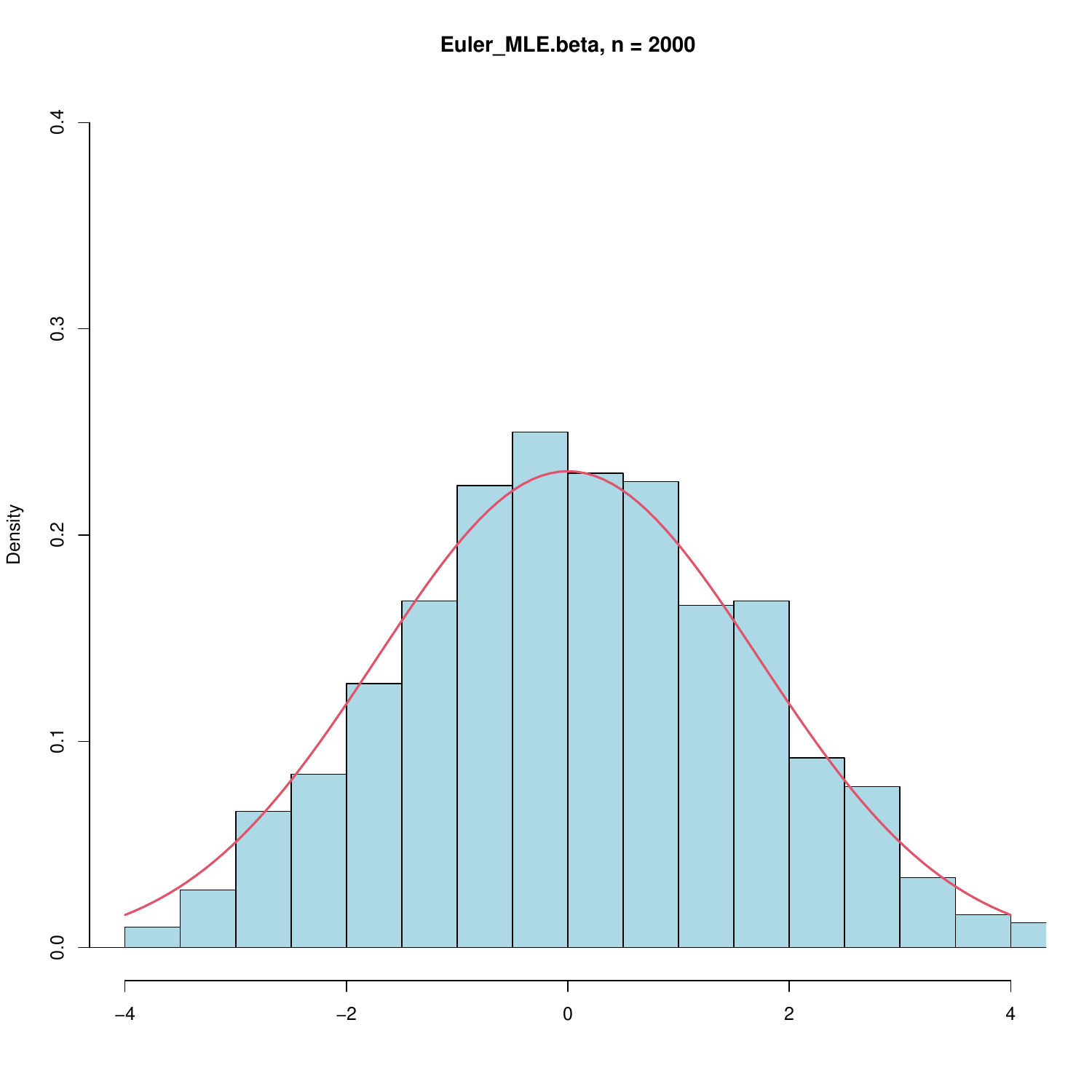}{$n=2000$}
  
  \caption{Comparison between histograms of the normalized MLE and Euler-QMLE with $\alpha=1.01$, $(\lambda,\mu,\sigma,\beta)=(1,2,5,0.5)$.}
  \label{fig:param-by-n_mle-vs-euler-1.01}
\end{figure}


\begin{figure}[t]
  \centering
  \captionsetup[subfigure]{justification=centering}
  \newcommand{\colw}{0.32\linewidth} 

  \newcommand{\rowtitle}[1]{\par\medskip\textbf{#1}\par\smallskip}

  \rowtitle{$\lam$}
  \CellPair{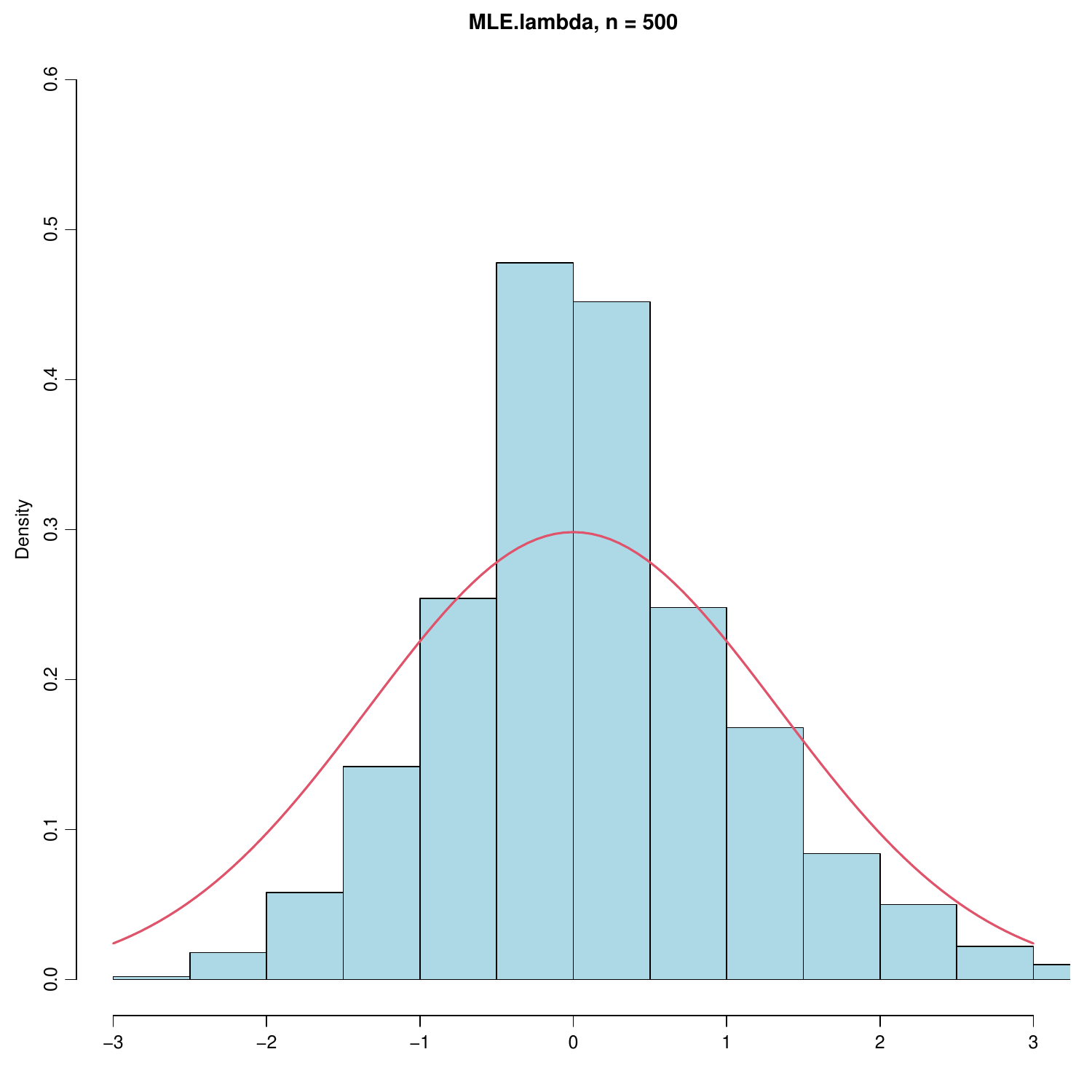}{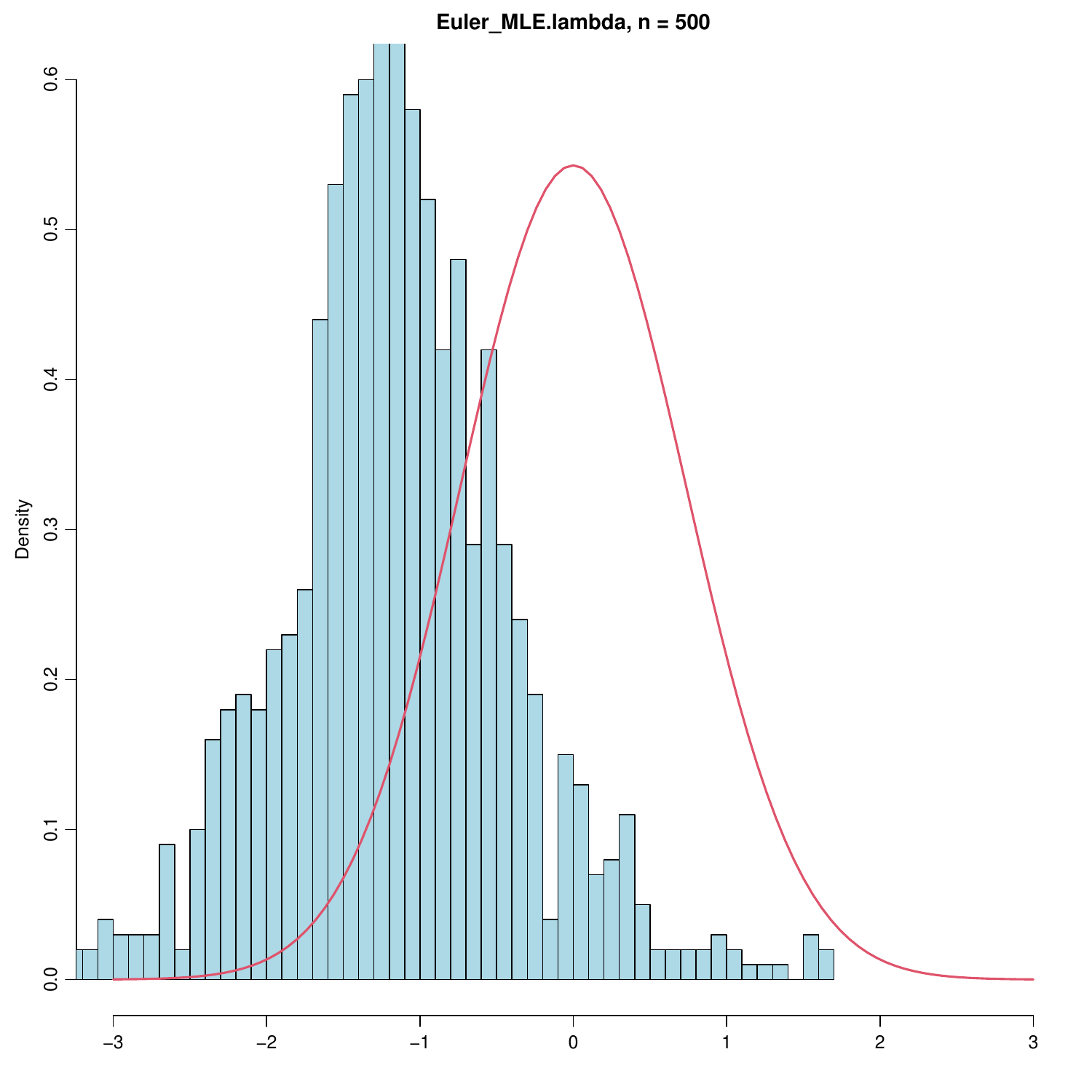}{$n=500$}\hfill
  \CellPair{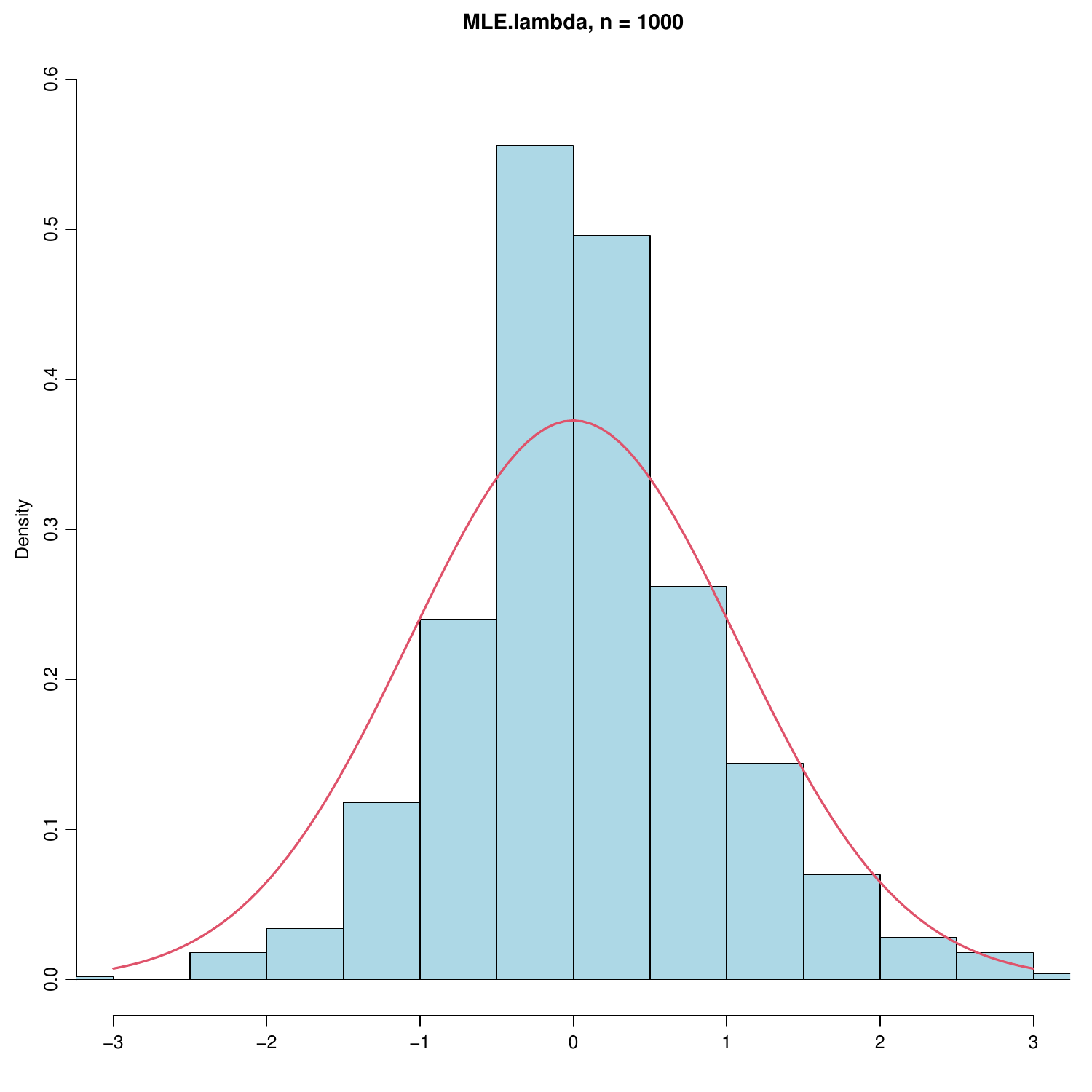}{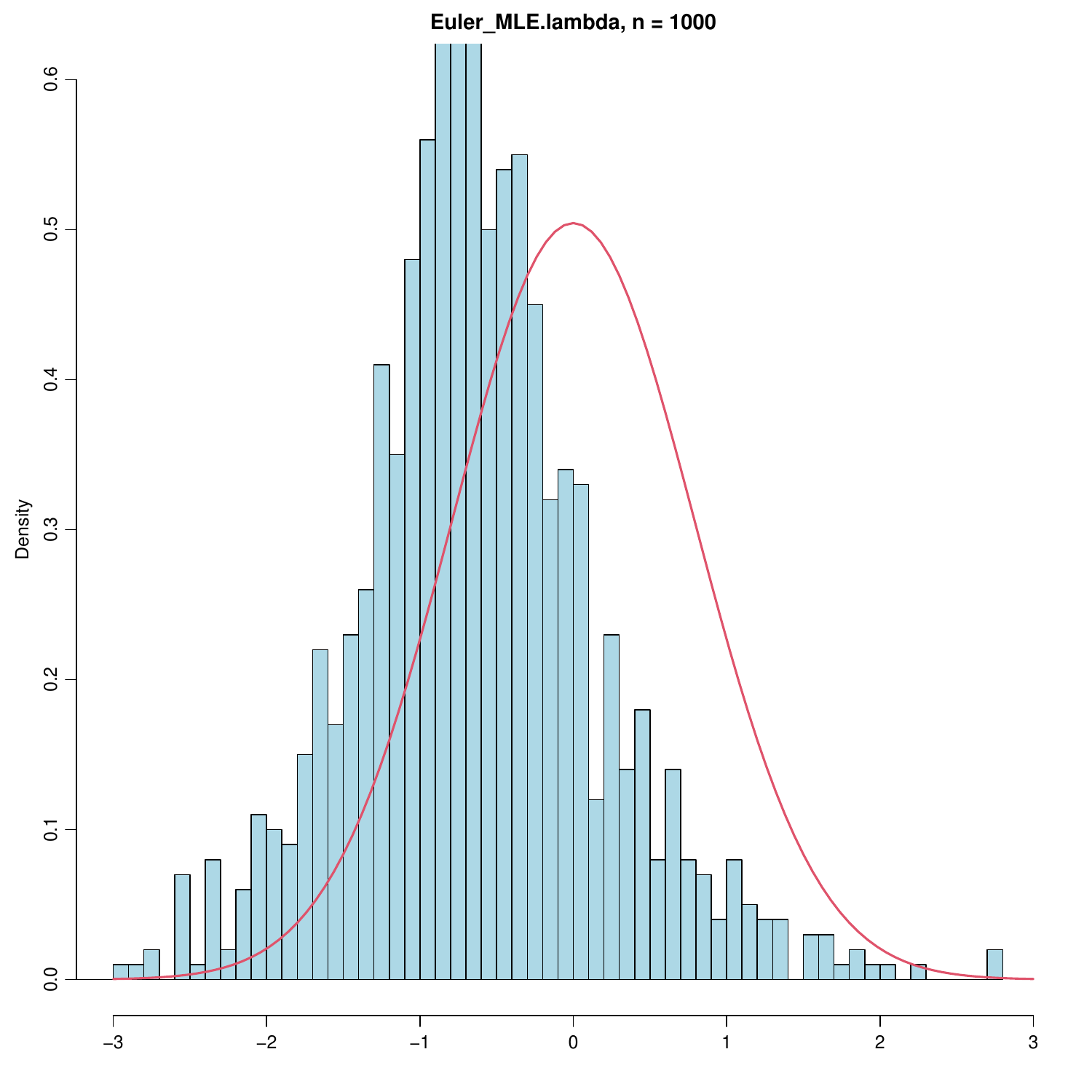}{$n=1000$}\hfill
  \CellPair{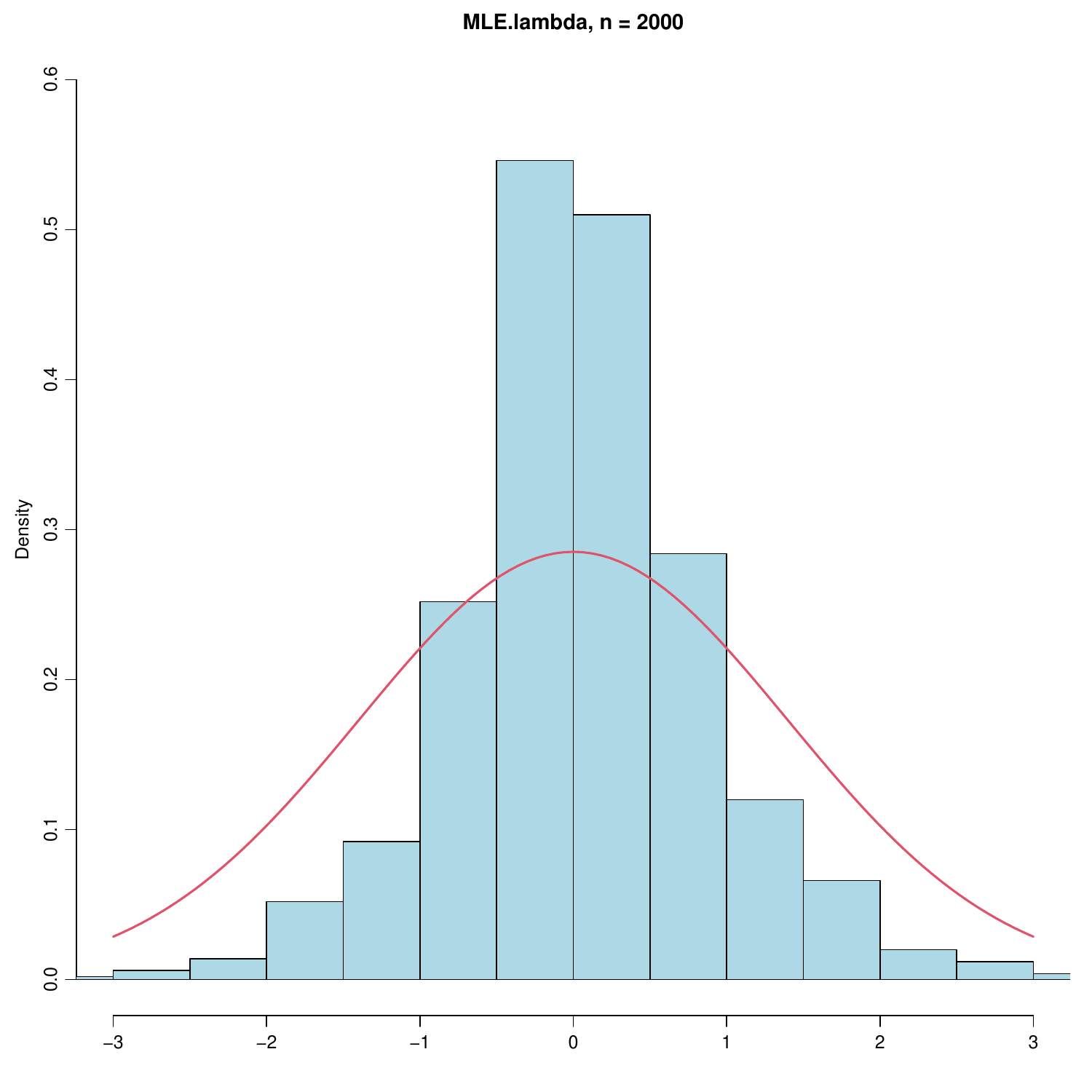}{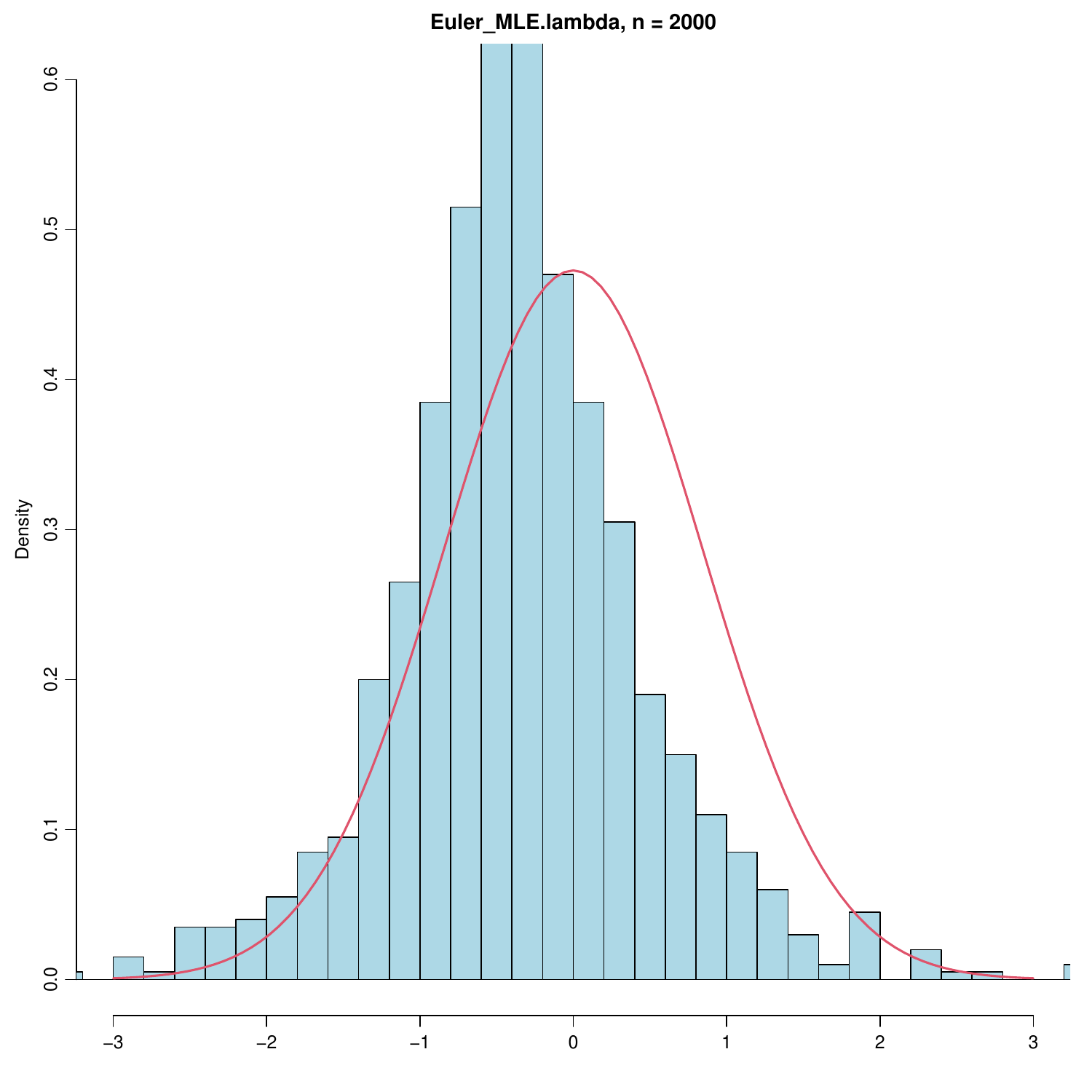}{$n=2000$}
  
  \vspace{5mm}
  
  \rowtitle{$\mu$}
  \CellPair{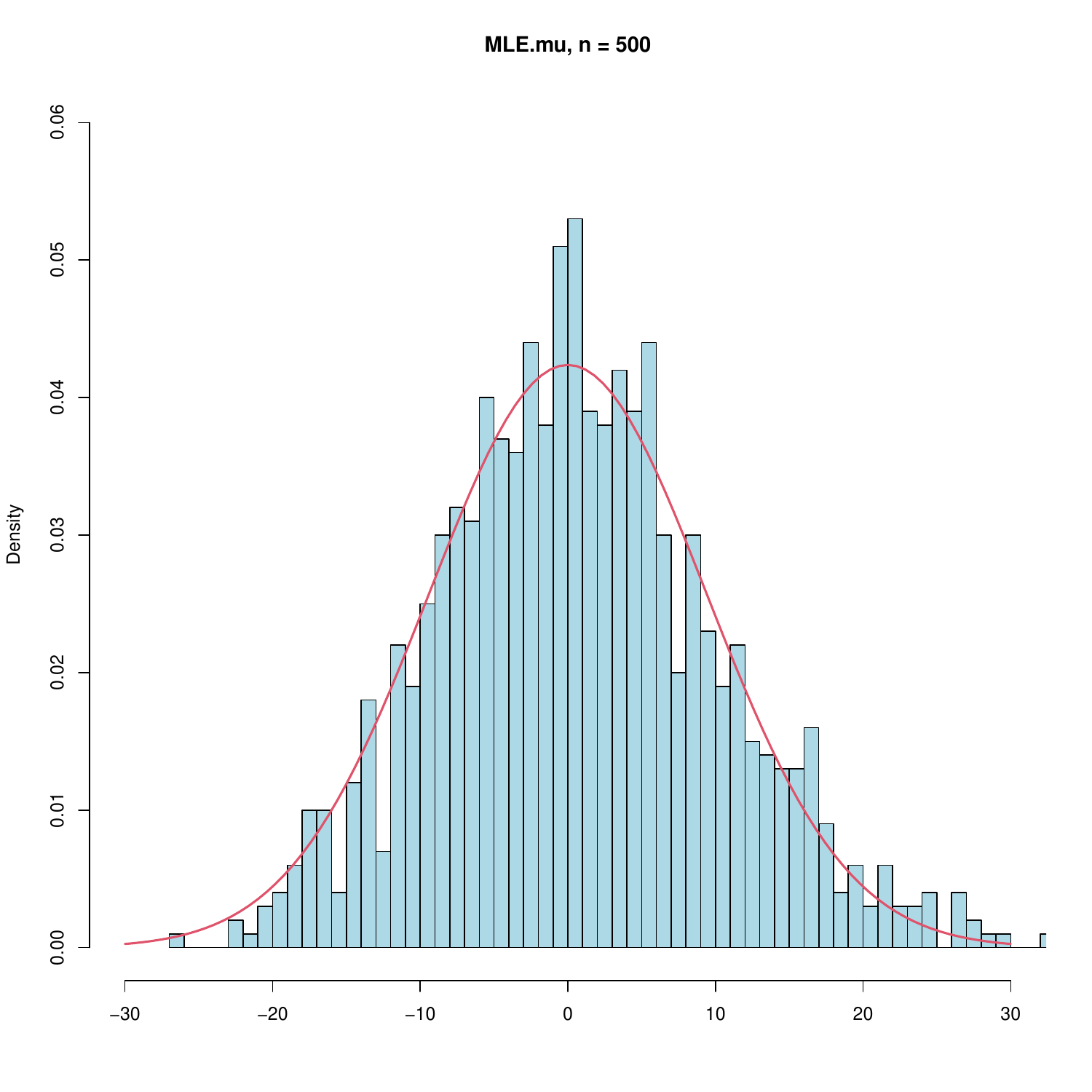}{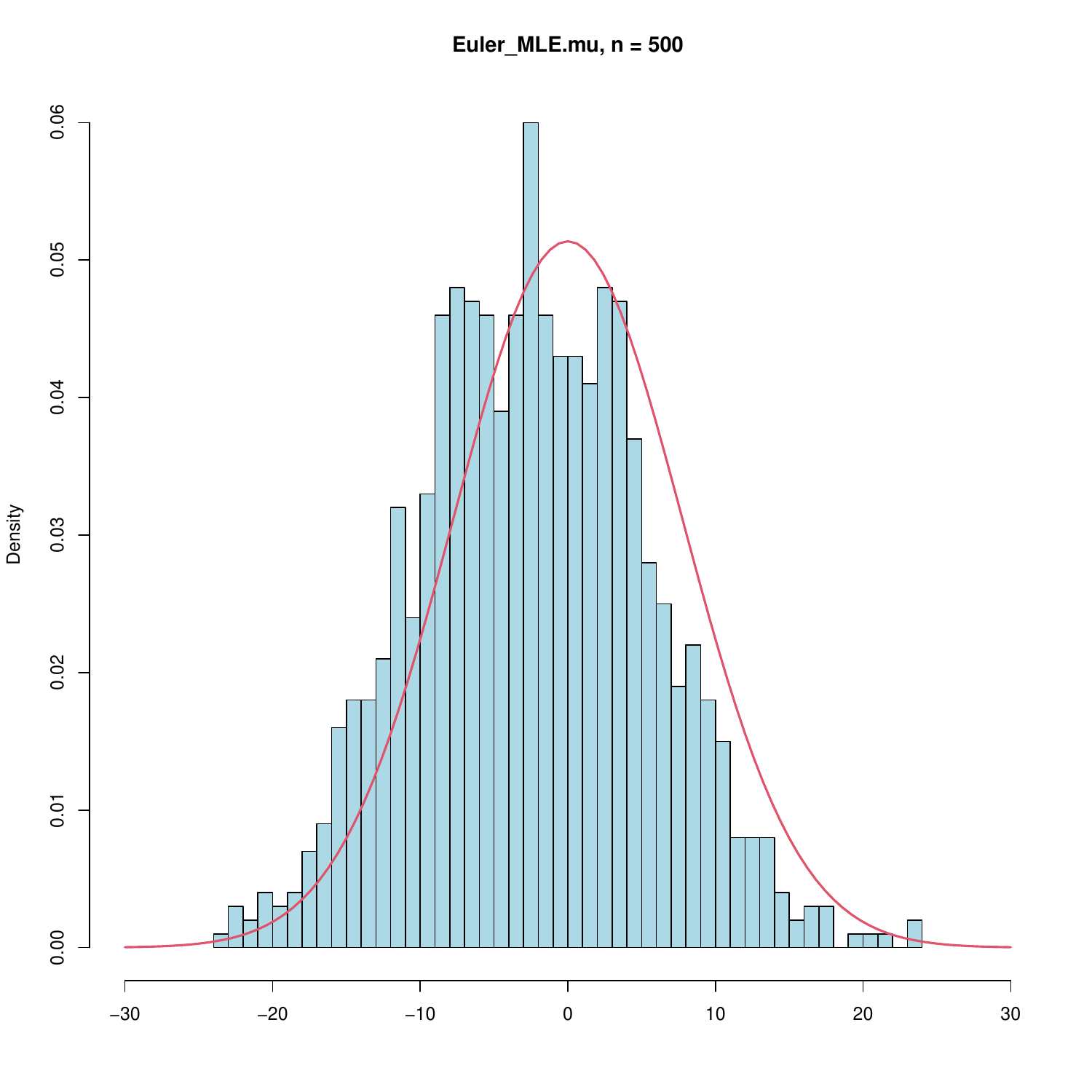}{$n=500$}\hfill
  \CellPair{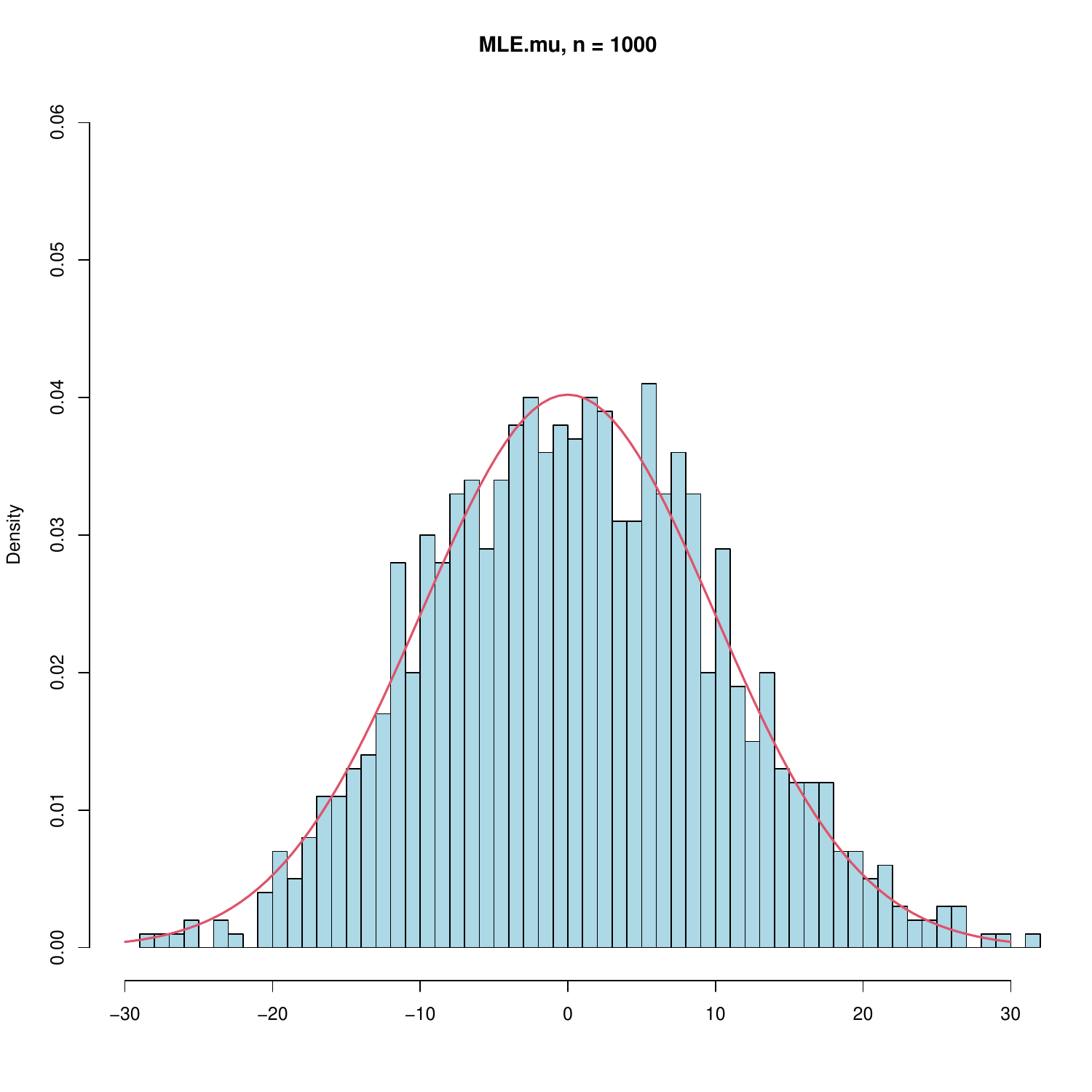}{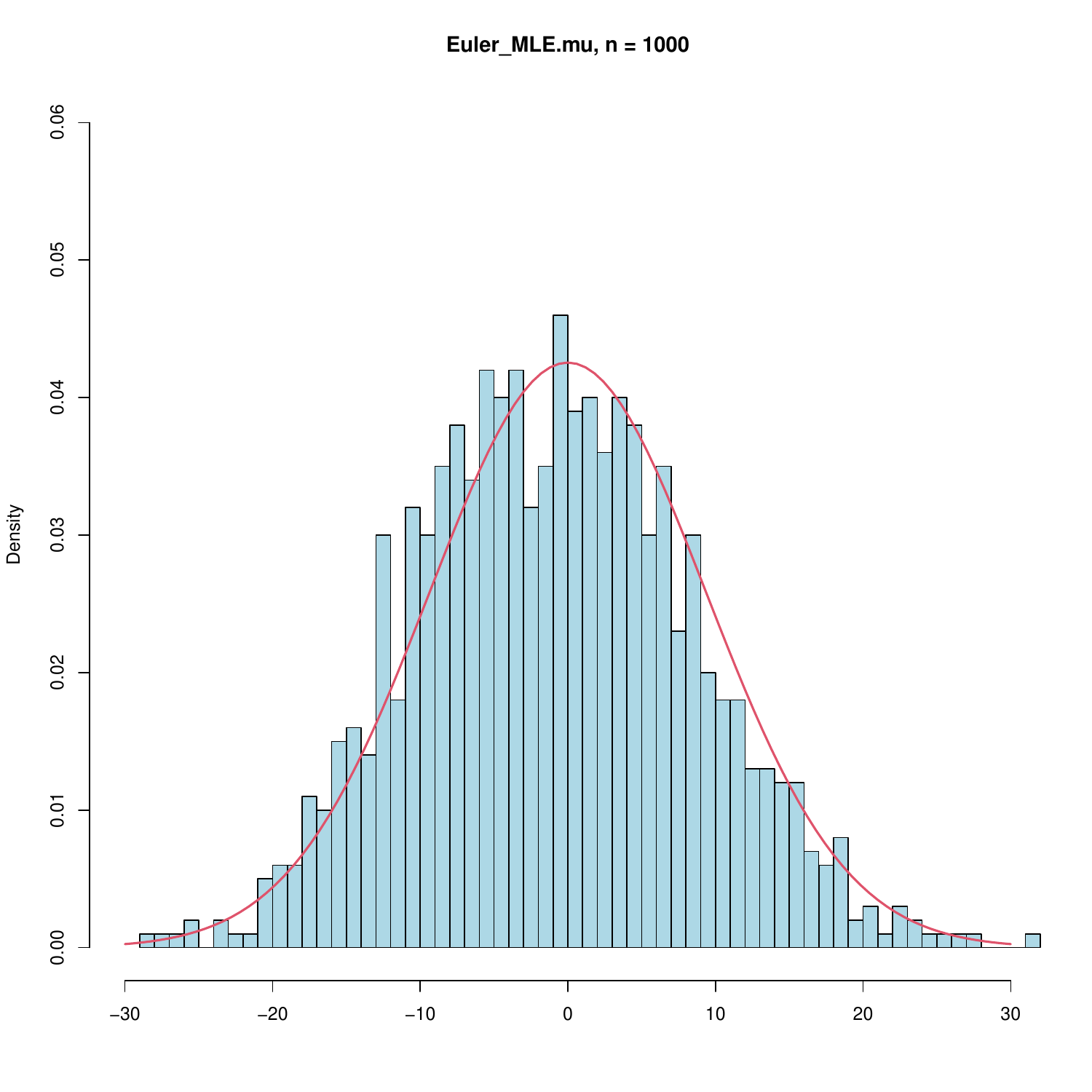}{$n=1000$}\hfill
  \CellPair{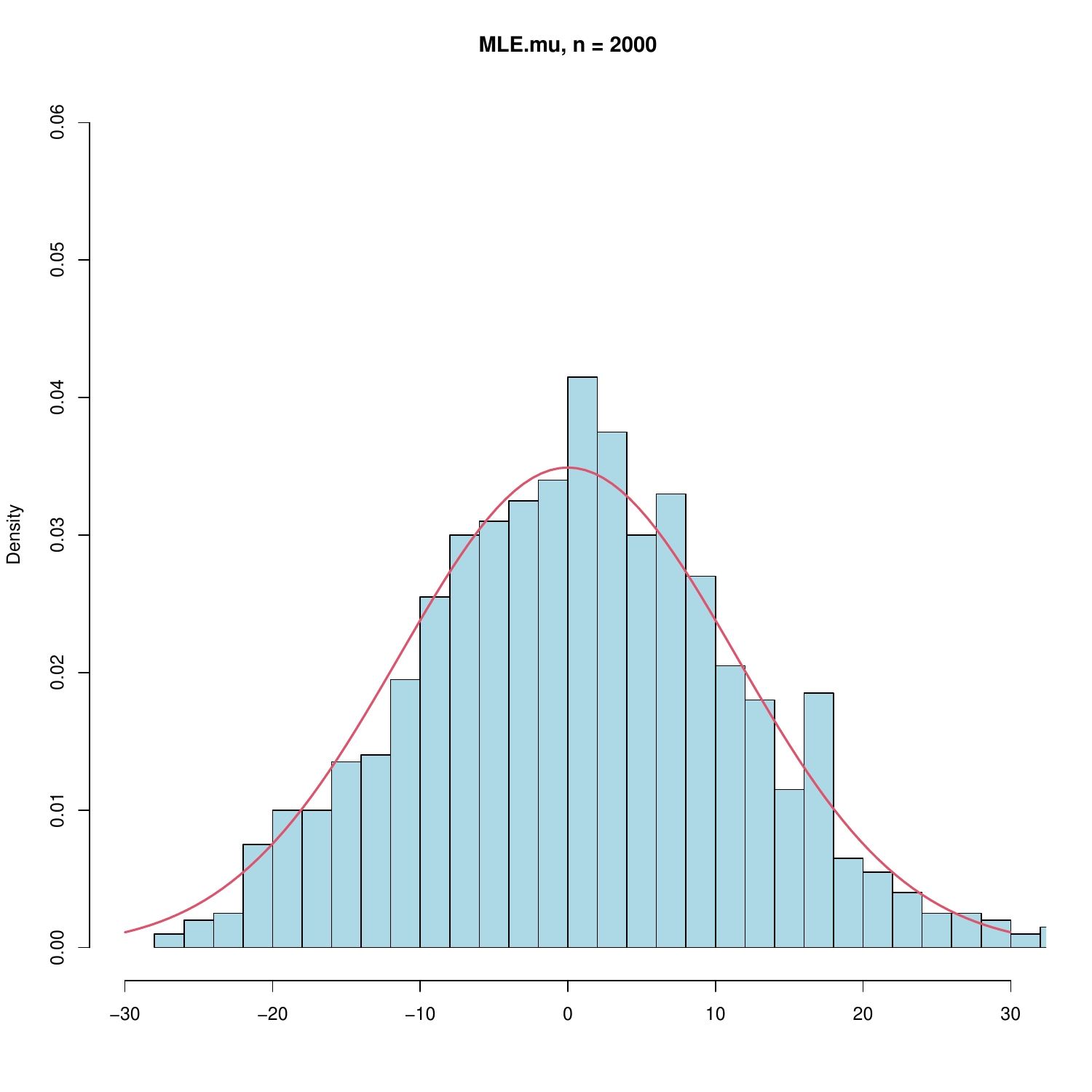}{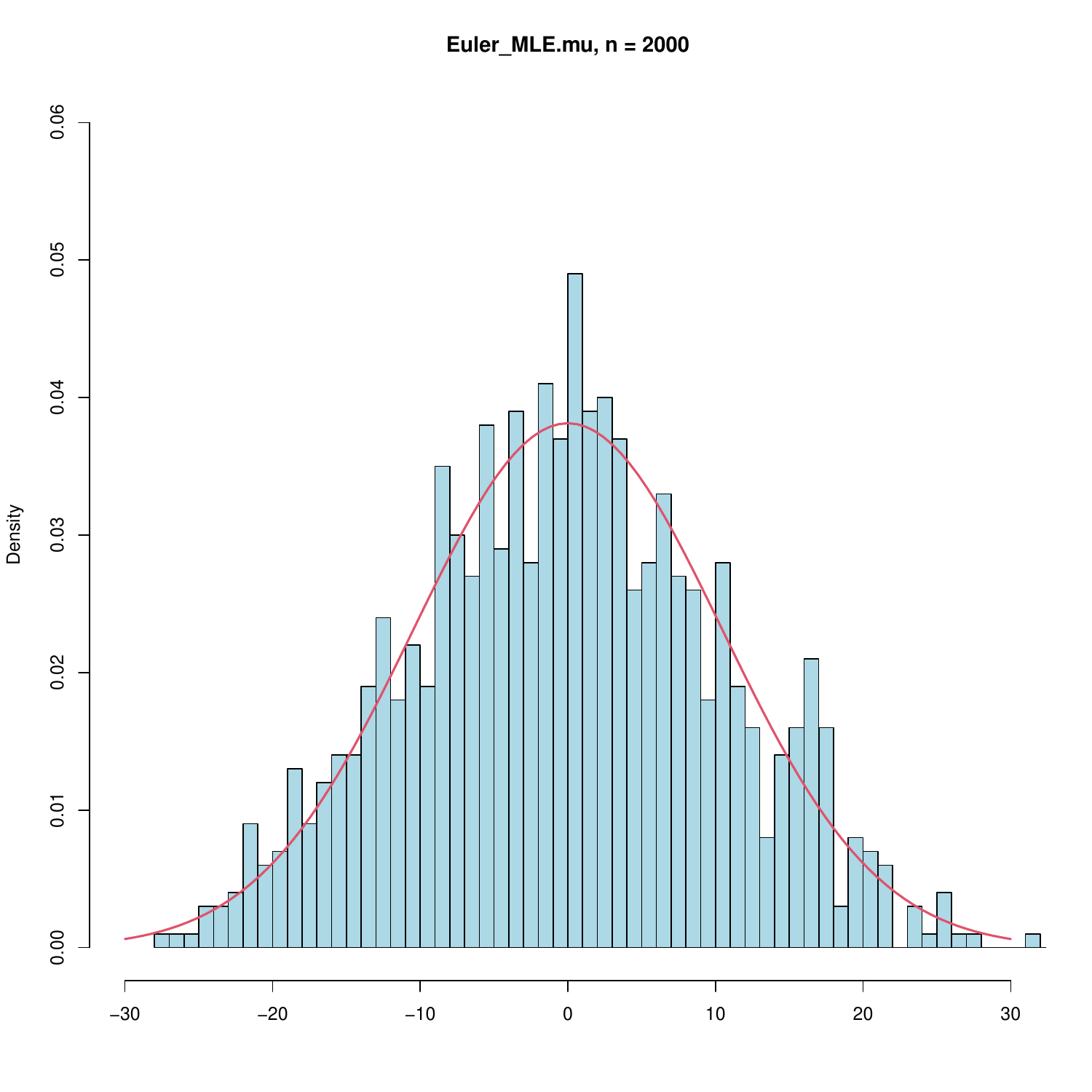}{$n=2000$}

  \vspace{5mm}
  
  \rowtitle{$\al$}
  \CellPair{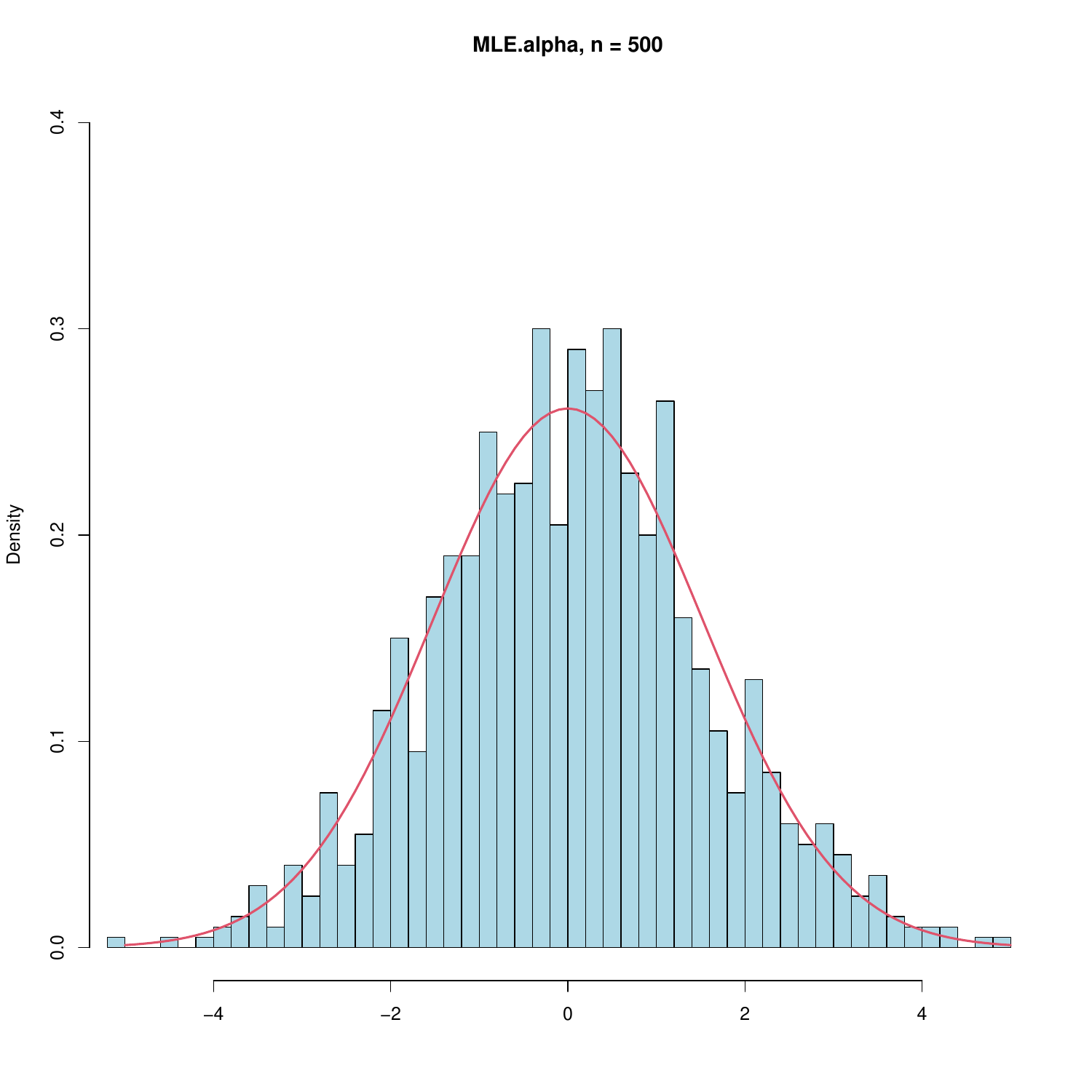}{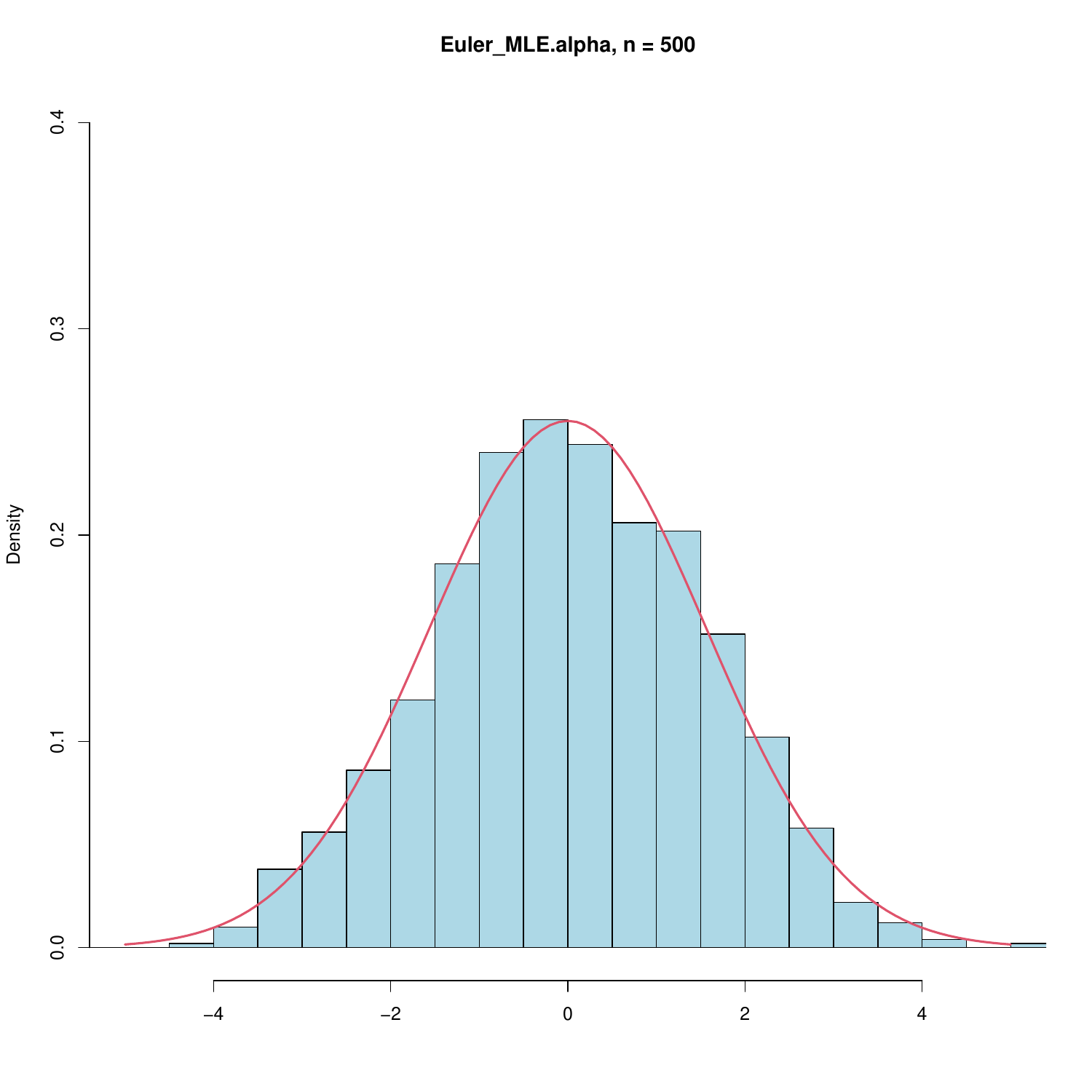}{$n=500$}\hfill
  \CellPair{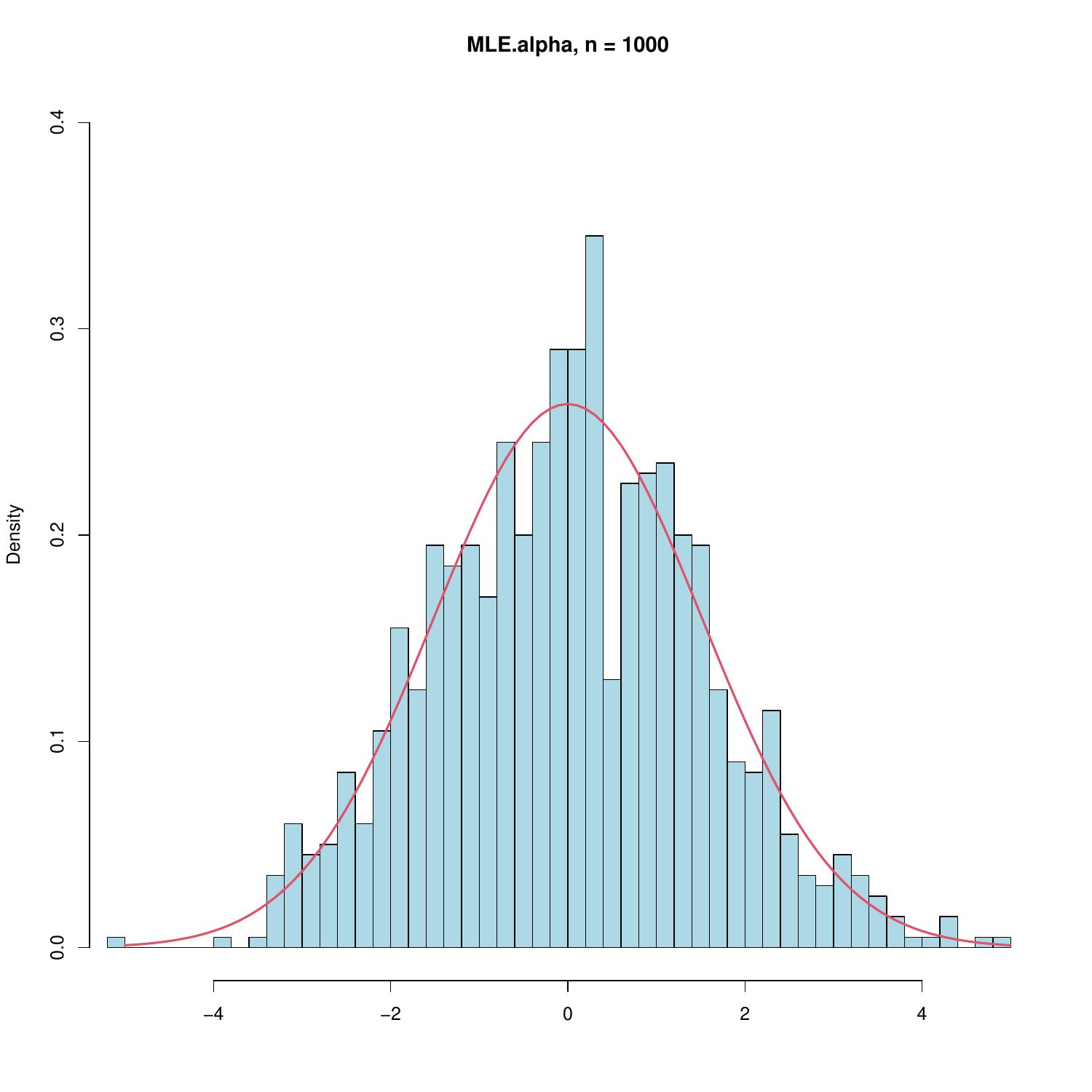}{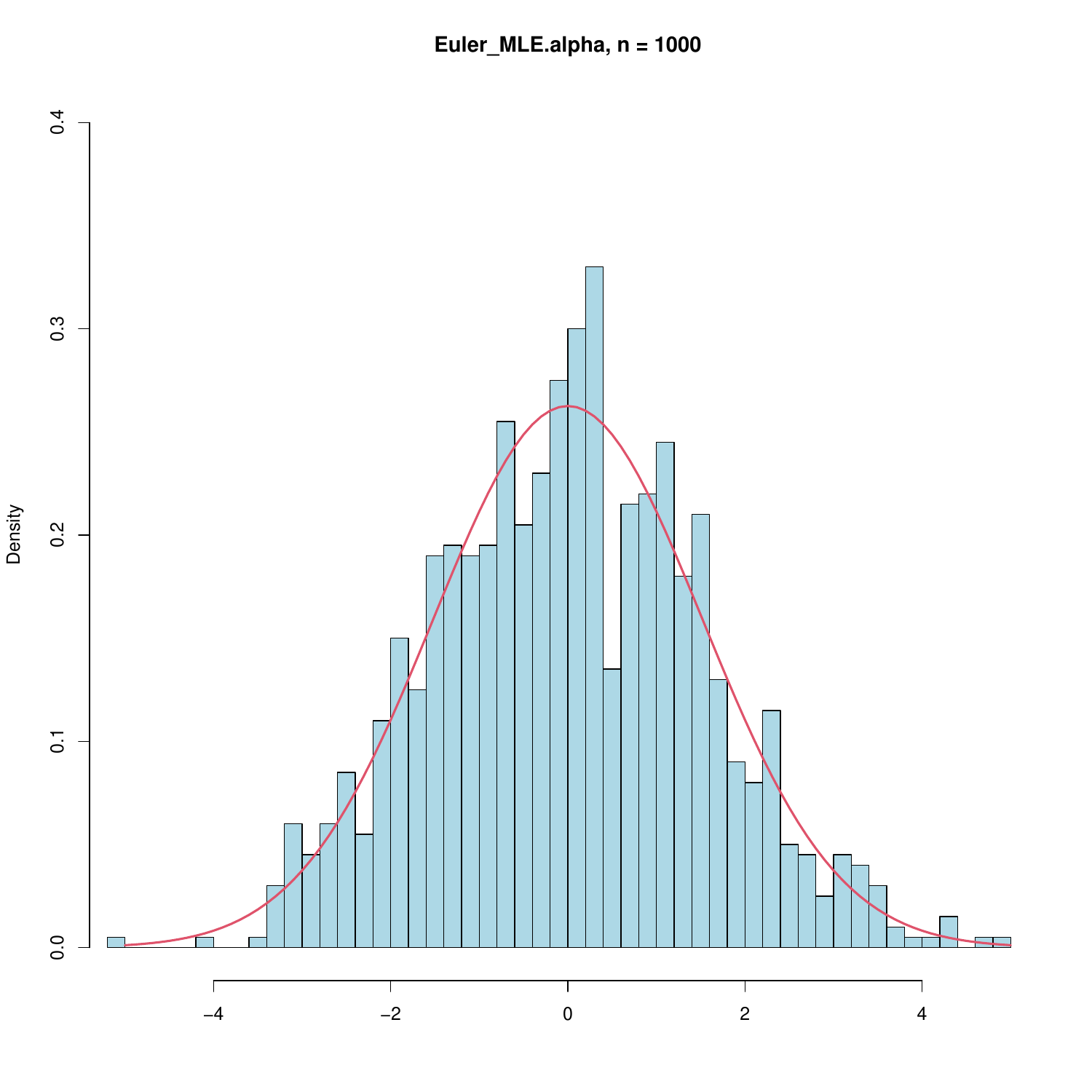}{$n=1000$}\hfill
  \CellPair{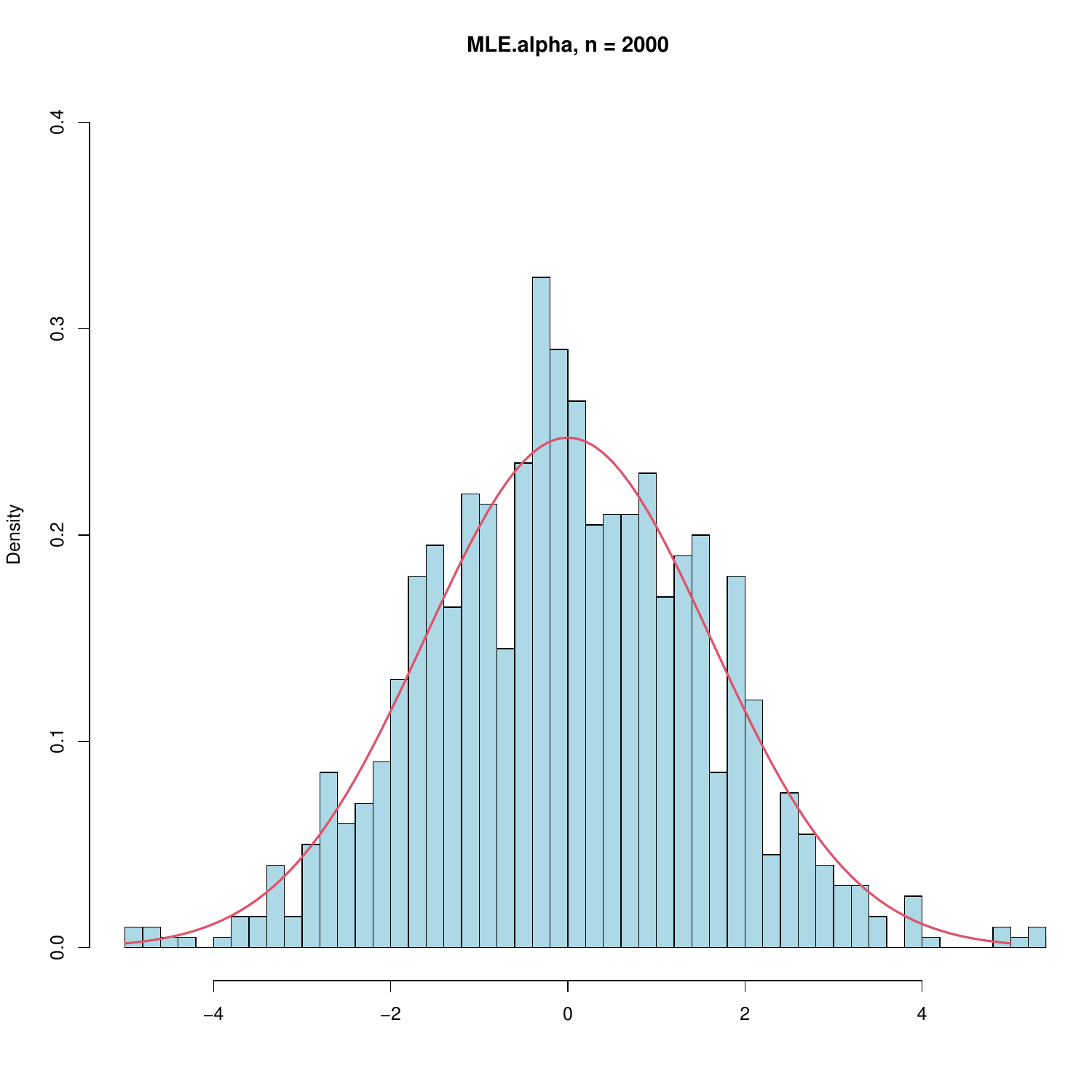}{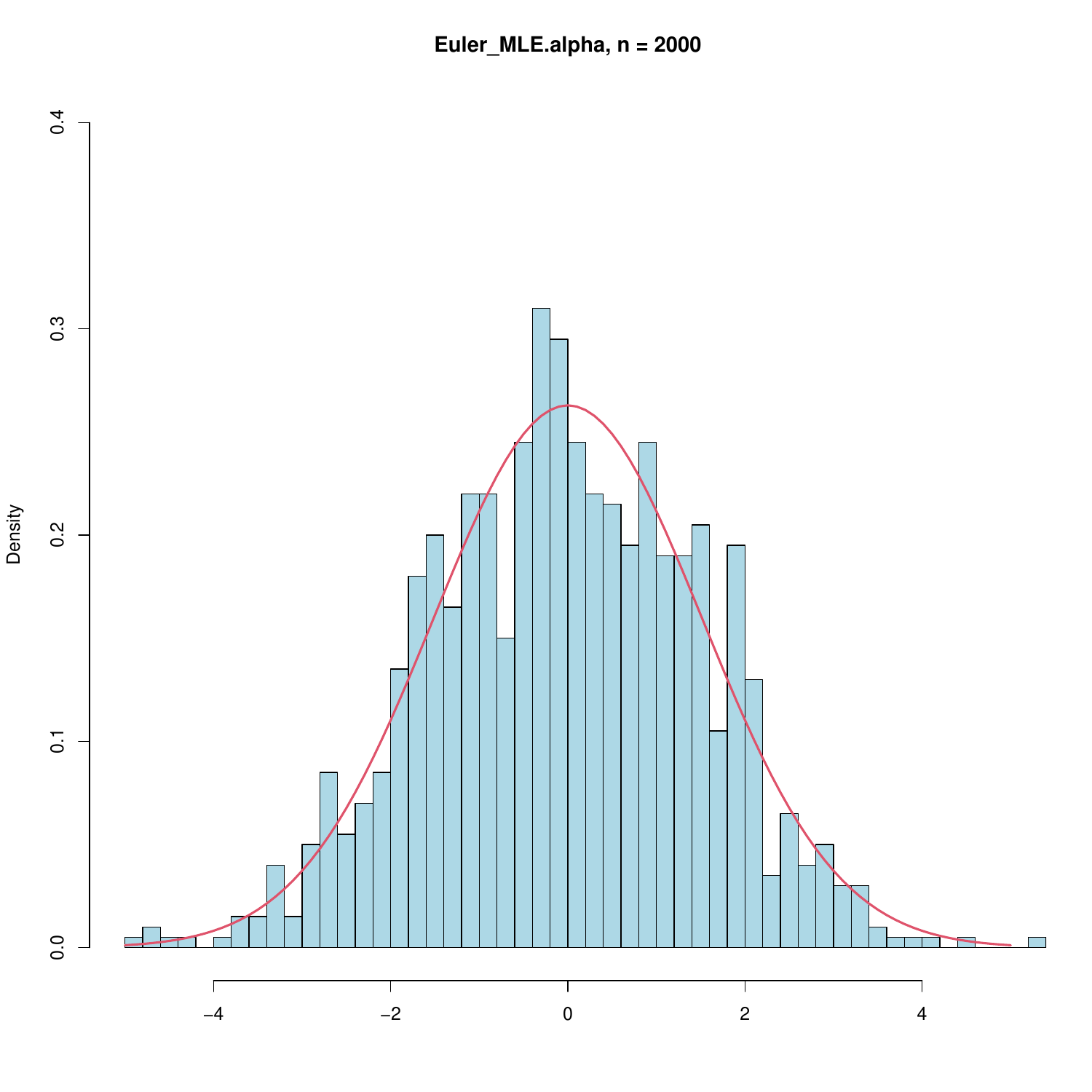}{$n=2000$}

  \vspace{5mm}
  
  \rowtitle{$\sig$}
  \CellPair{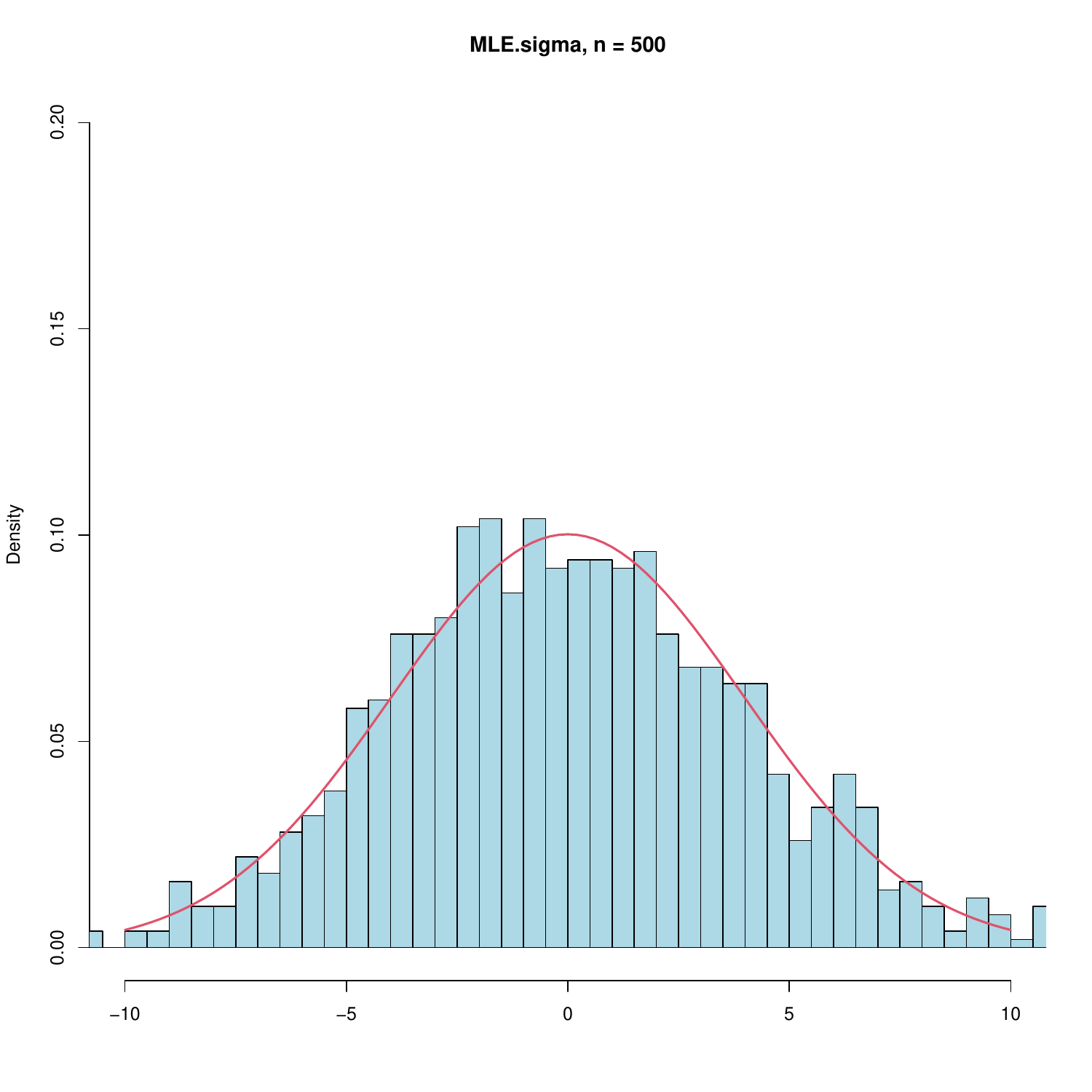}{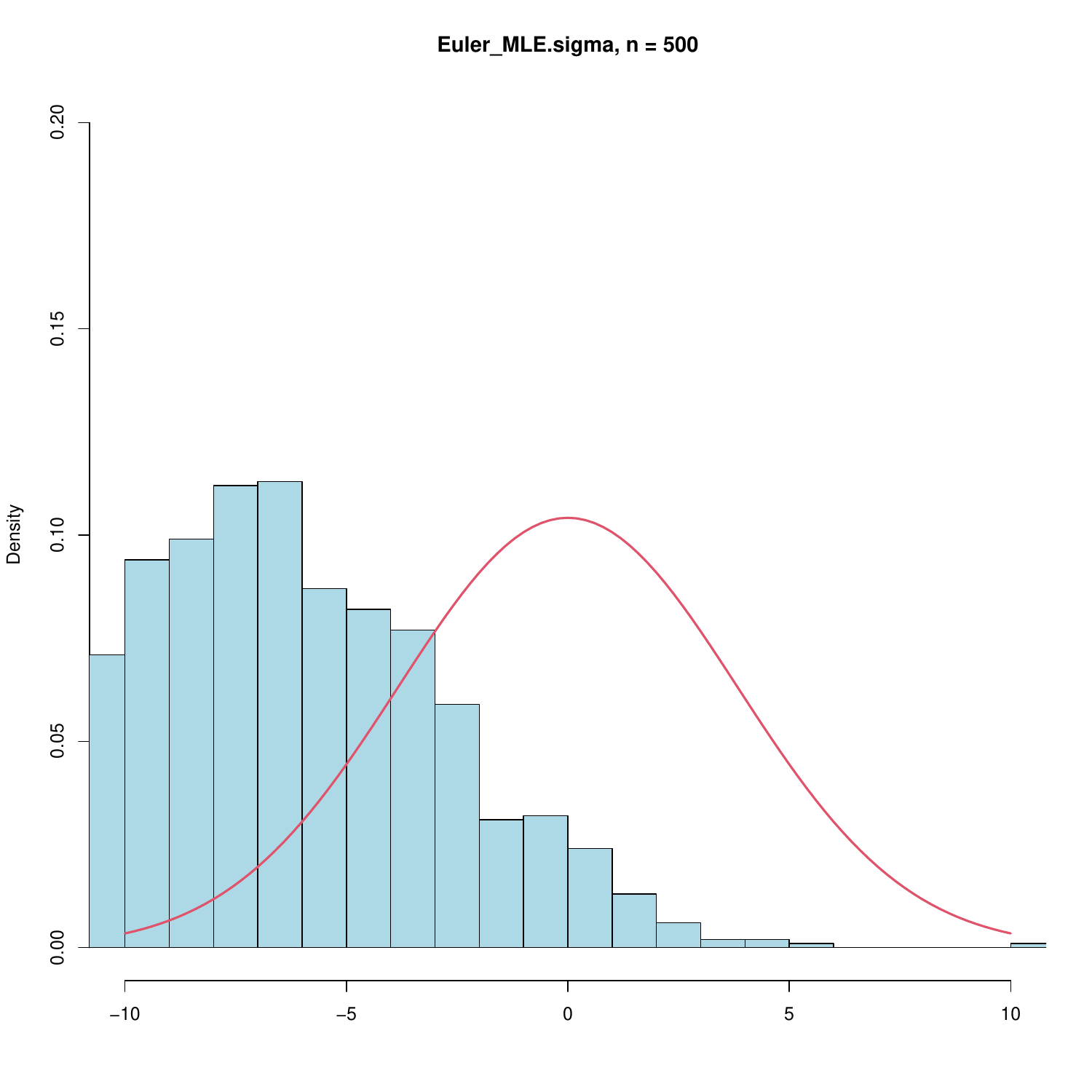}{$n=500$}\hfill
  \CellPair{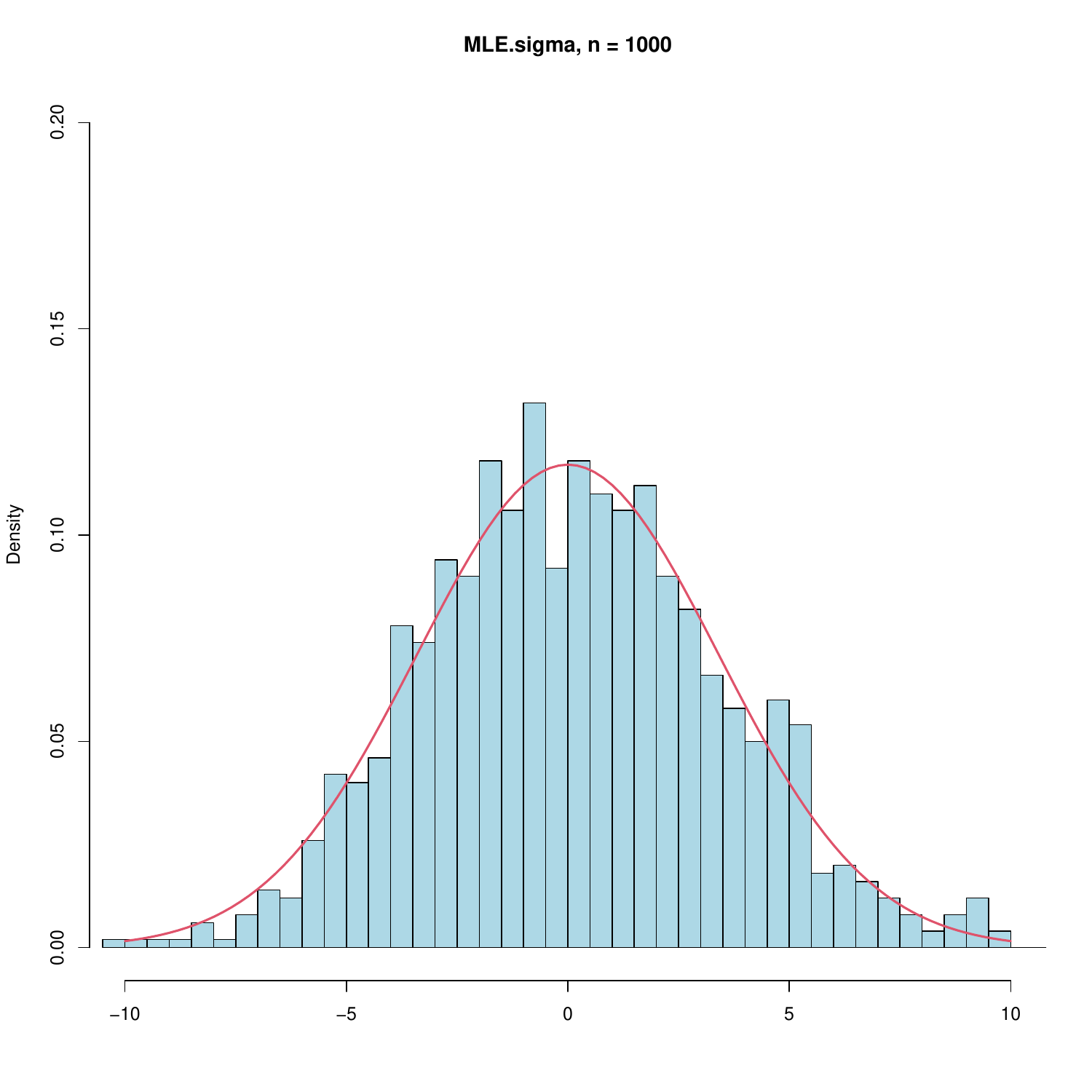}{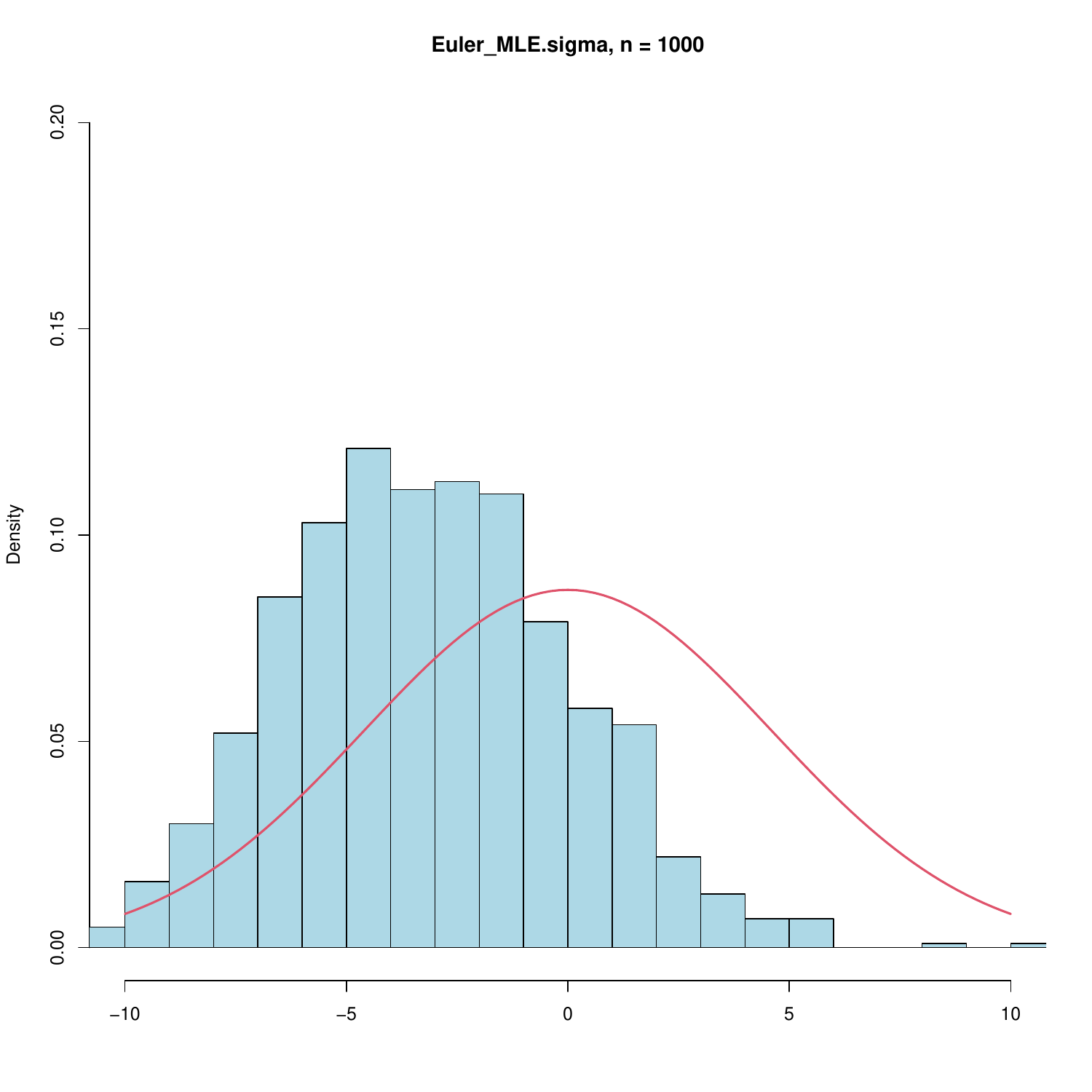}{$n=1000$}\hfill
  \CellPair{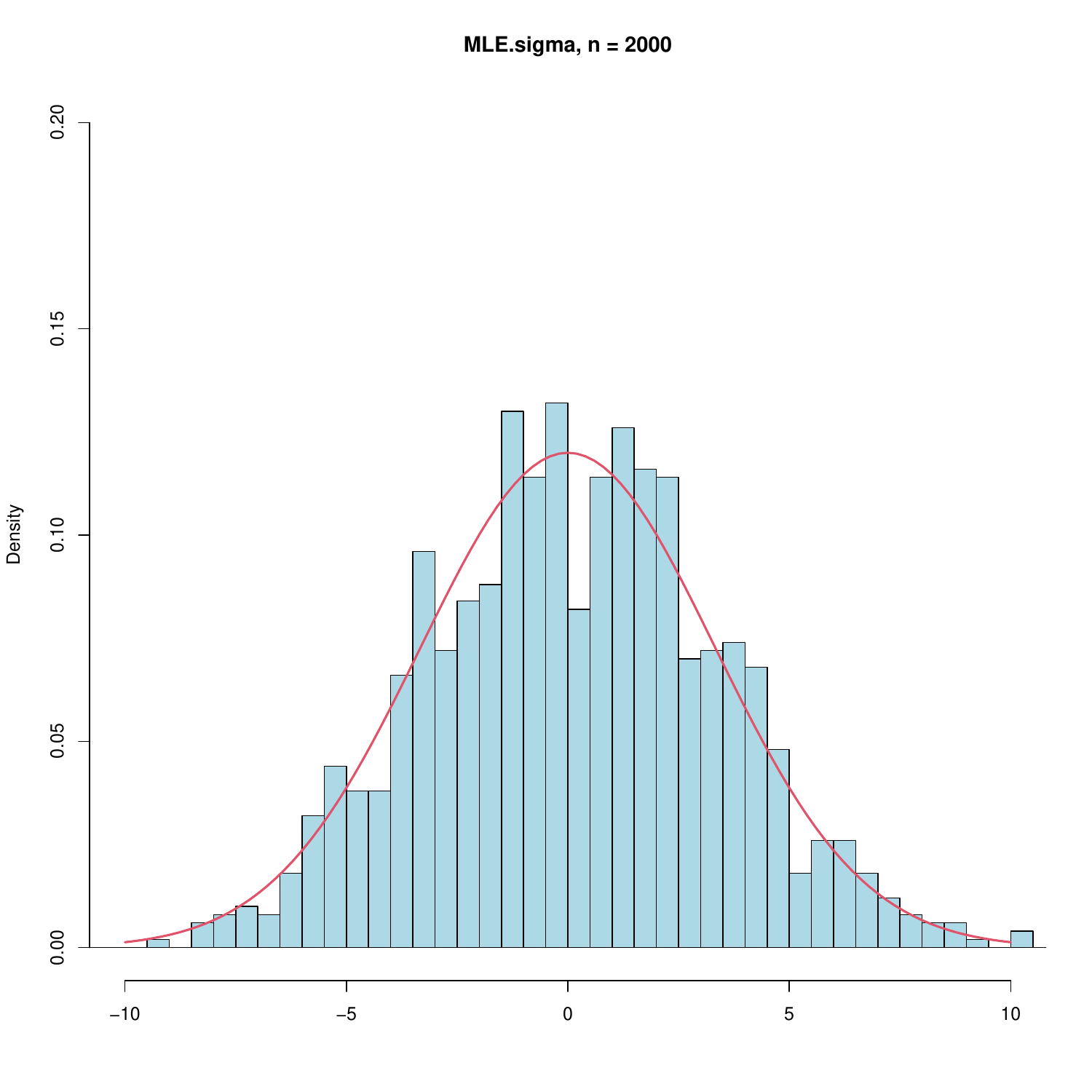}{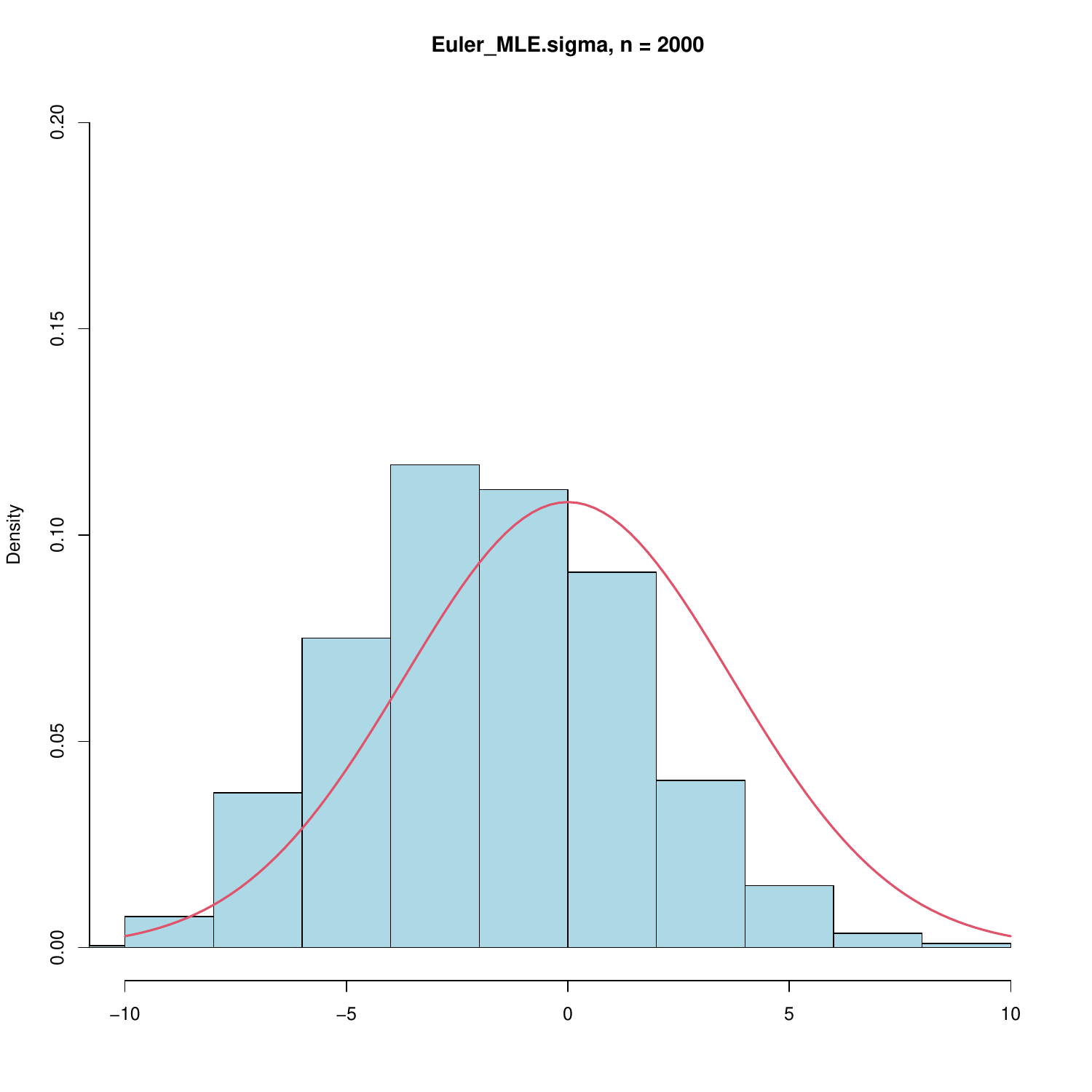}{$n=2000$}

  \vspace{5mm}

  \rowtitle{$\be$}
  \CellPair{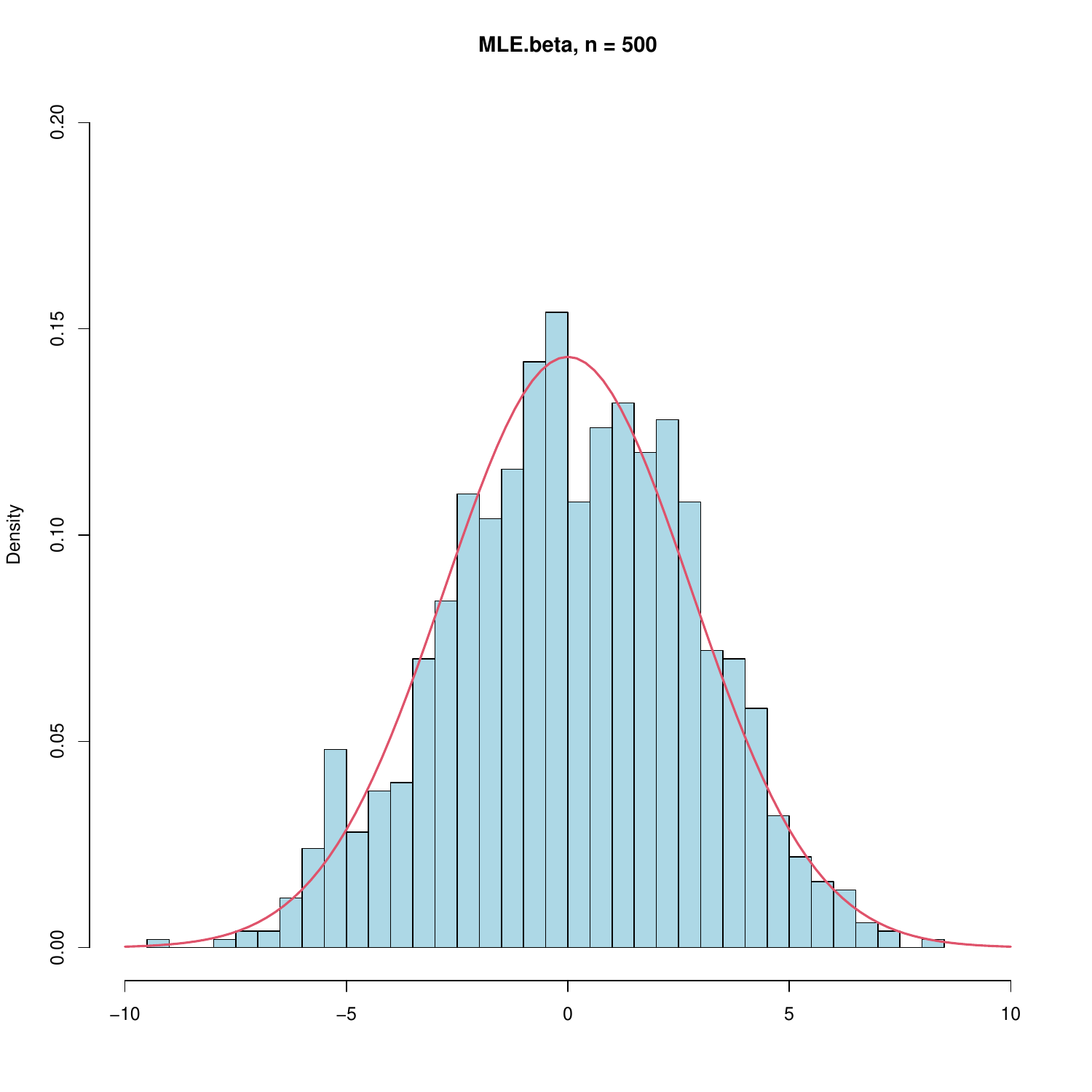}{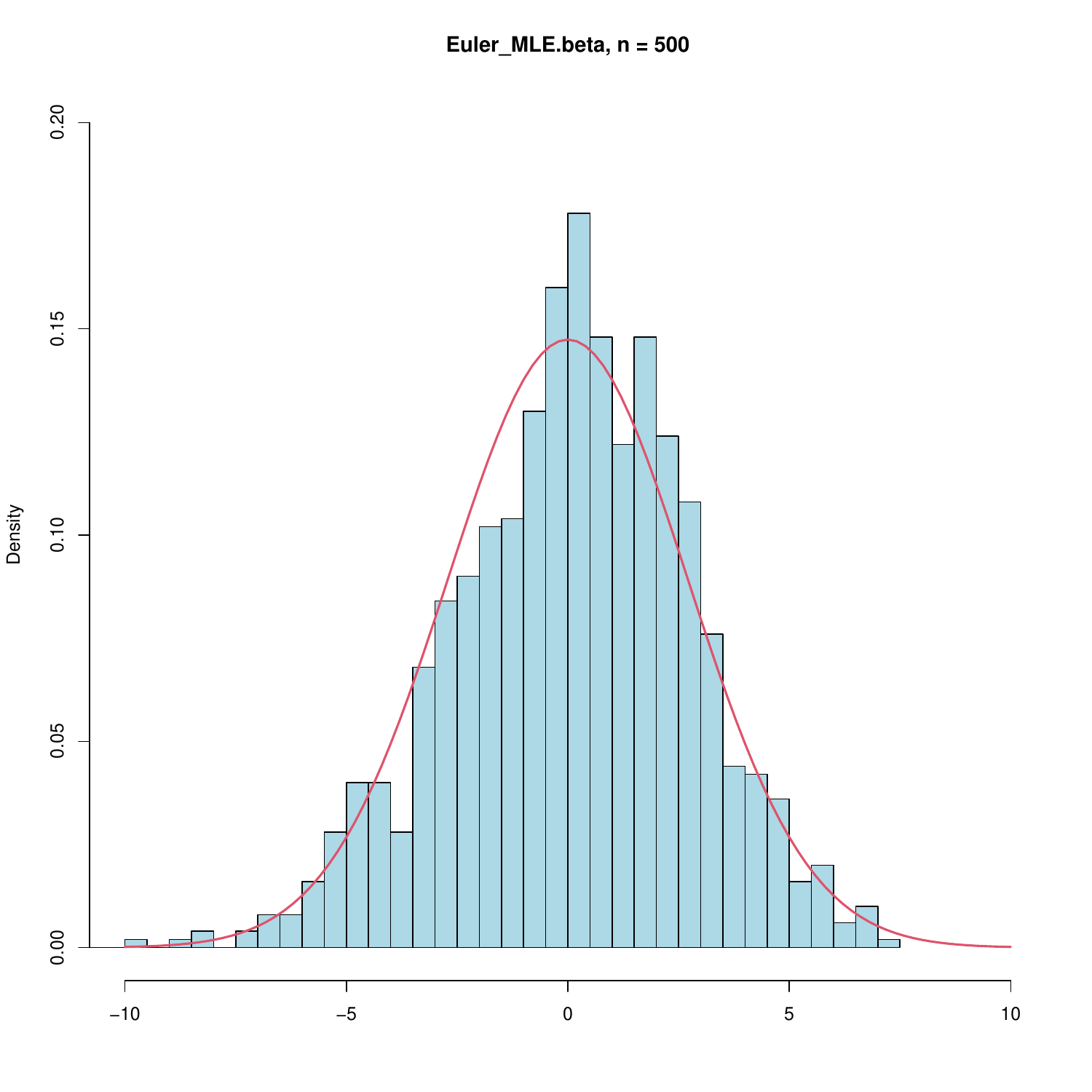}{$n=500$}\hfill
  \CellPair{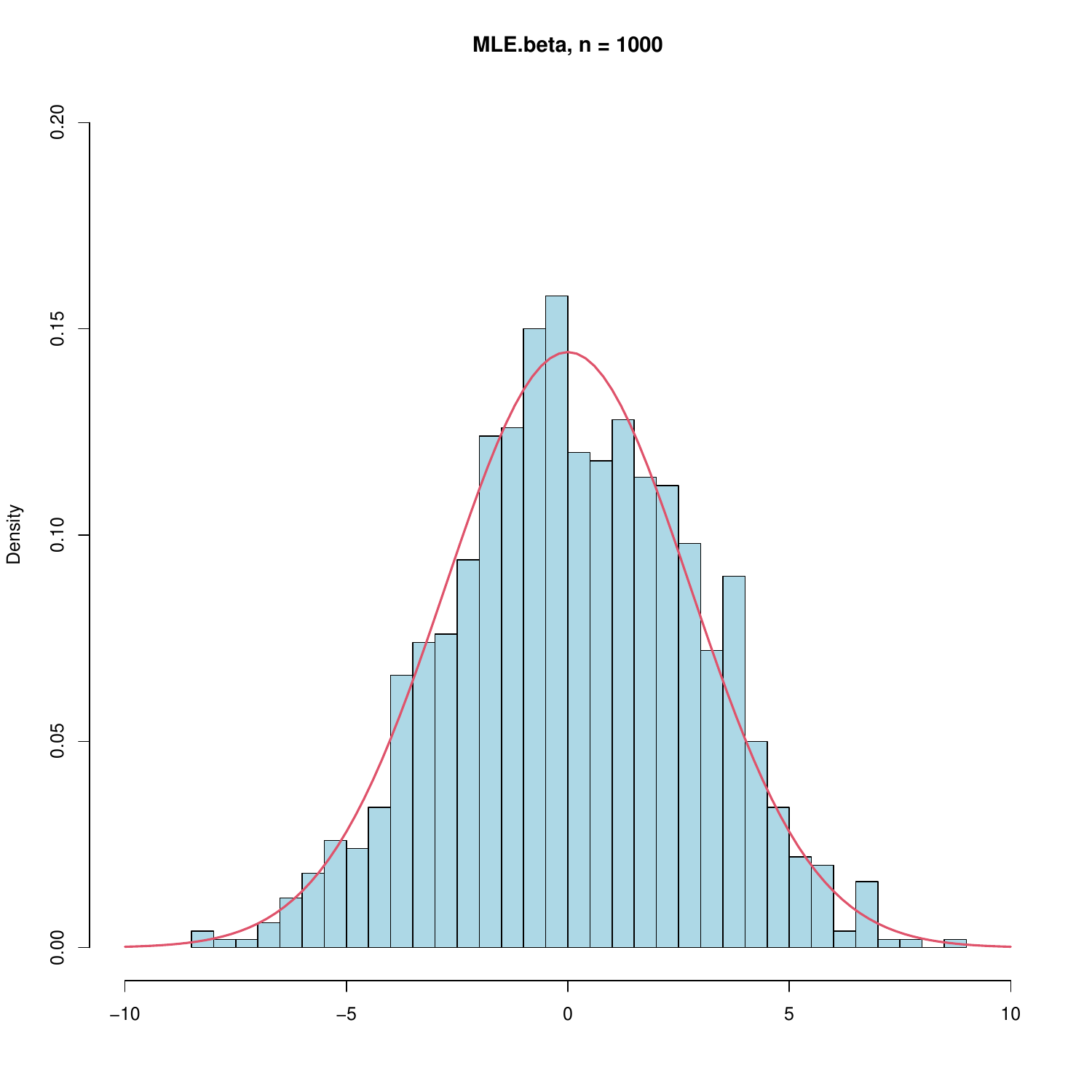}{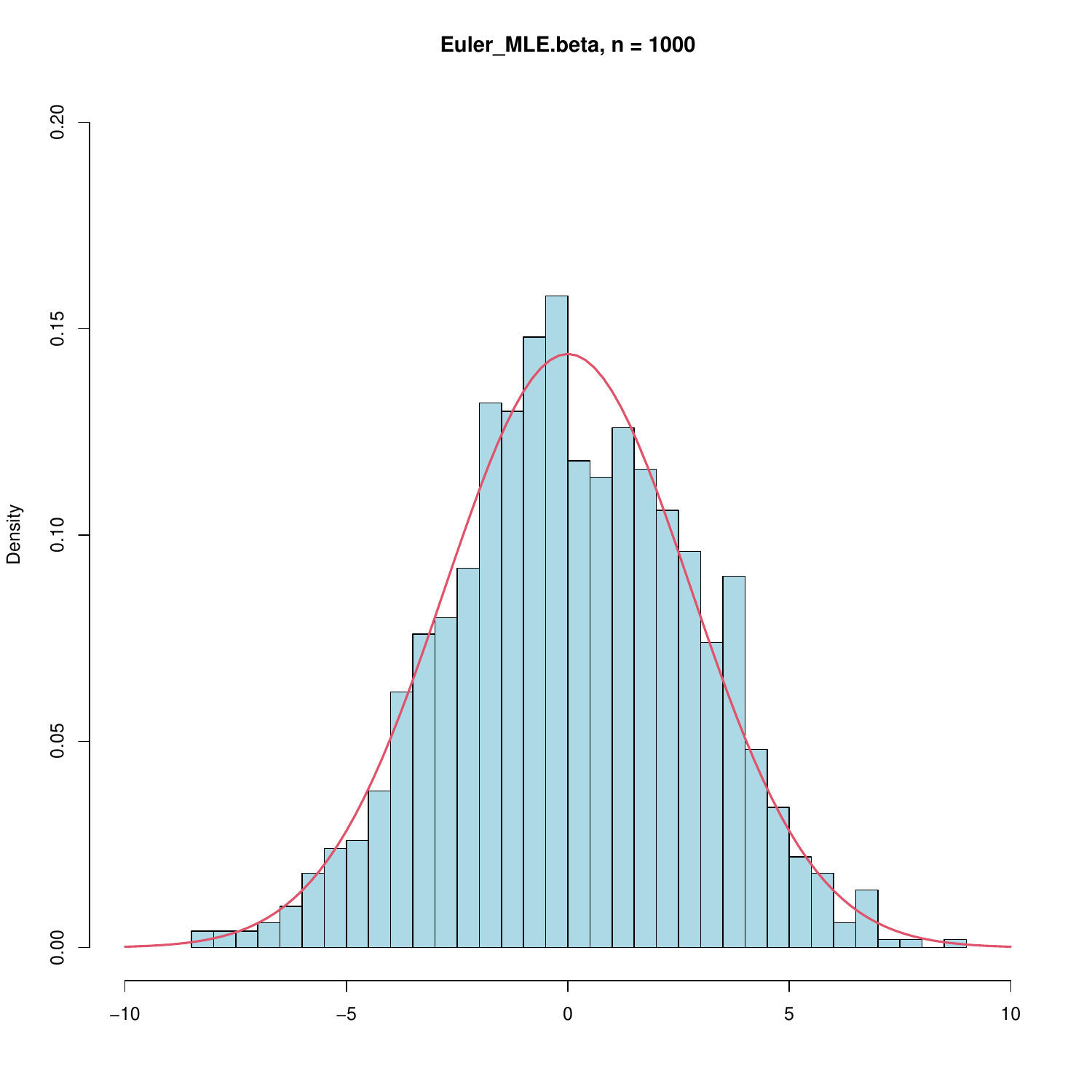}{$n=1000$}\hfill
  \CellPair{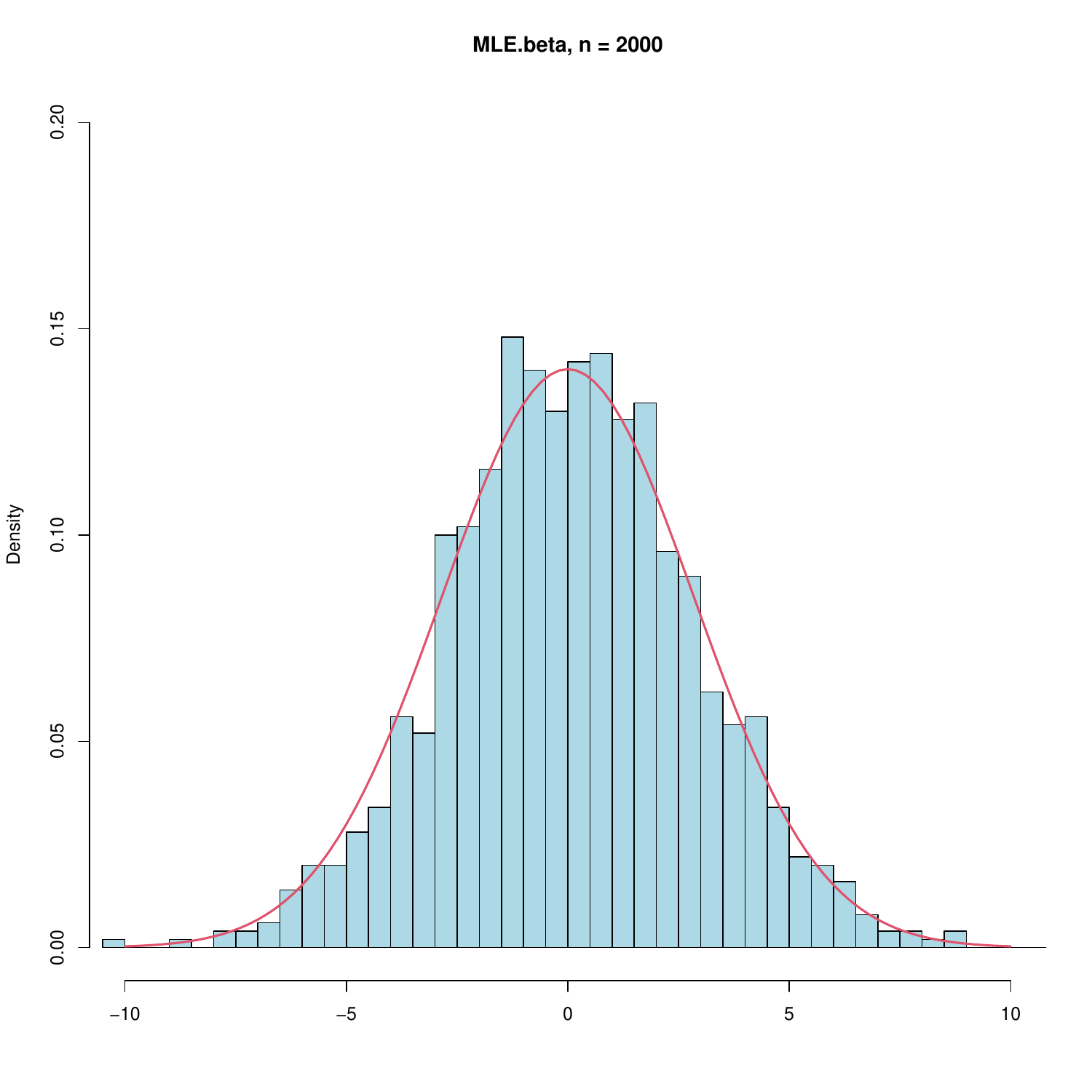}{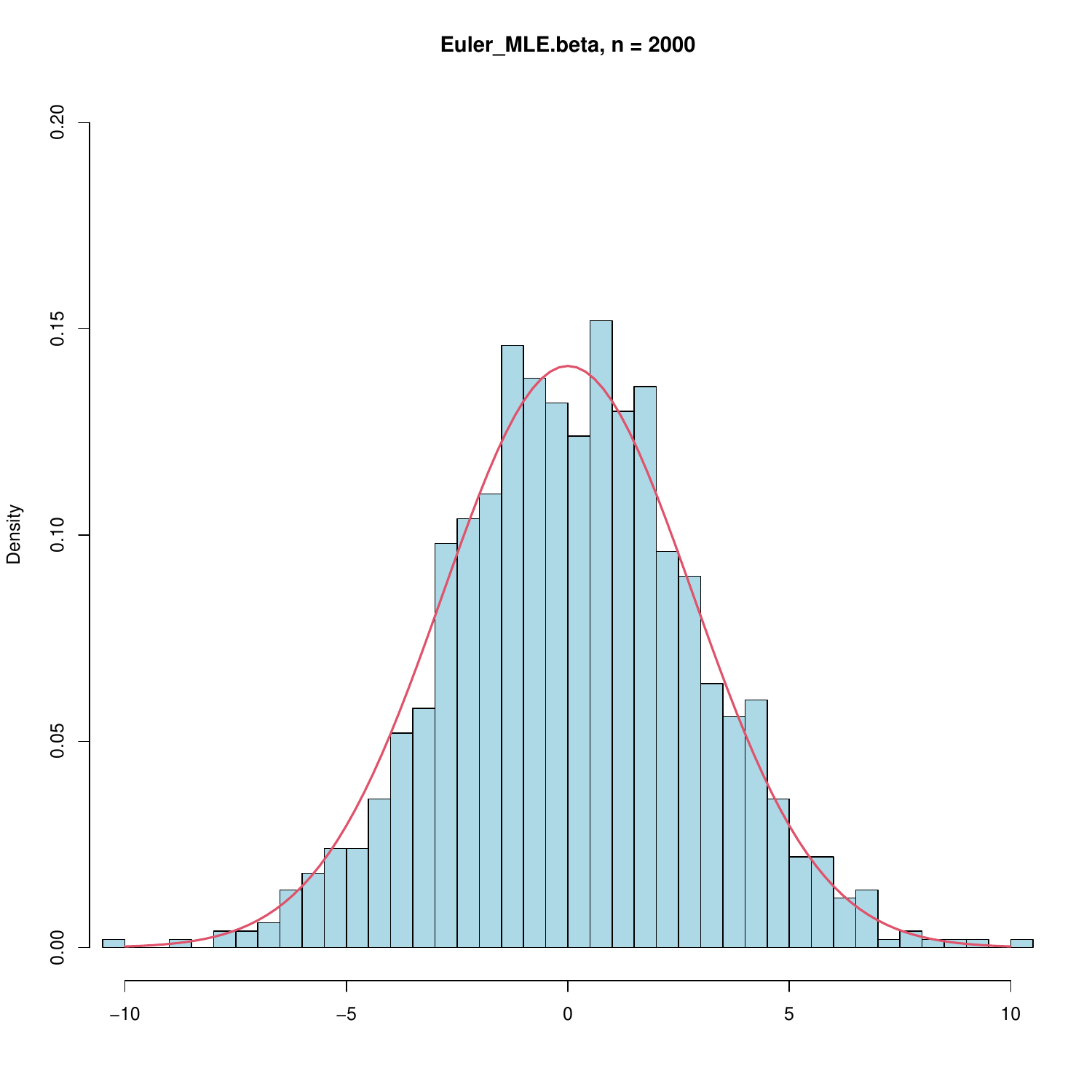}{$n=2000$}
  
  \caption{Comparison between histograms of the normalized MLE and Euler-QMLE with $\alpha=1.5$, $(\lambda,\mu,\sigma,\beta)=(1,2,5,0.5)$.}
  \label{fig:param-by-n_mle-vs-euler-1.5}
\end{figure}


\begin{figure}[t]
  \centering
  \captionsetup[subfigure]{justification=centering}
  \newcommand{\colw}{0.32\linewidth} 

  \newcommand{\rowtitle}[1]{\par\medskip\textbf{ #1}\par\smallskip}

  \rowtitle{$\lam$}
  \CellPair{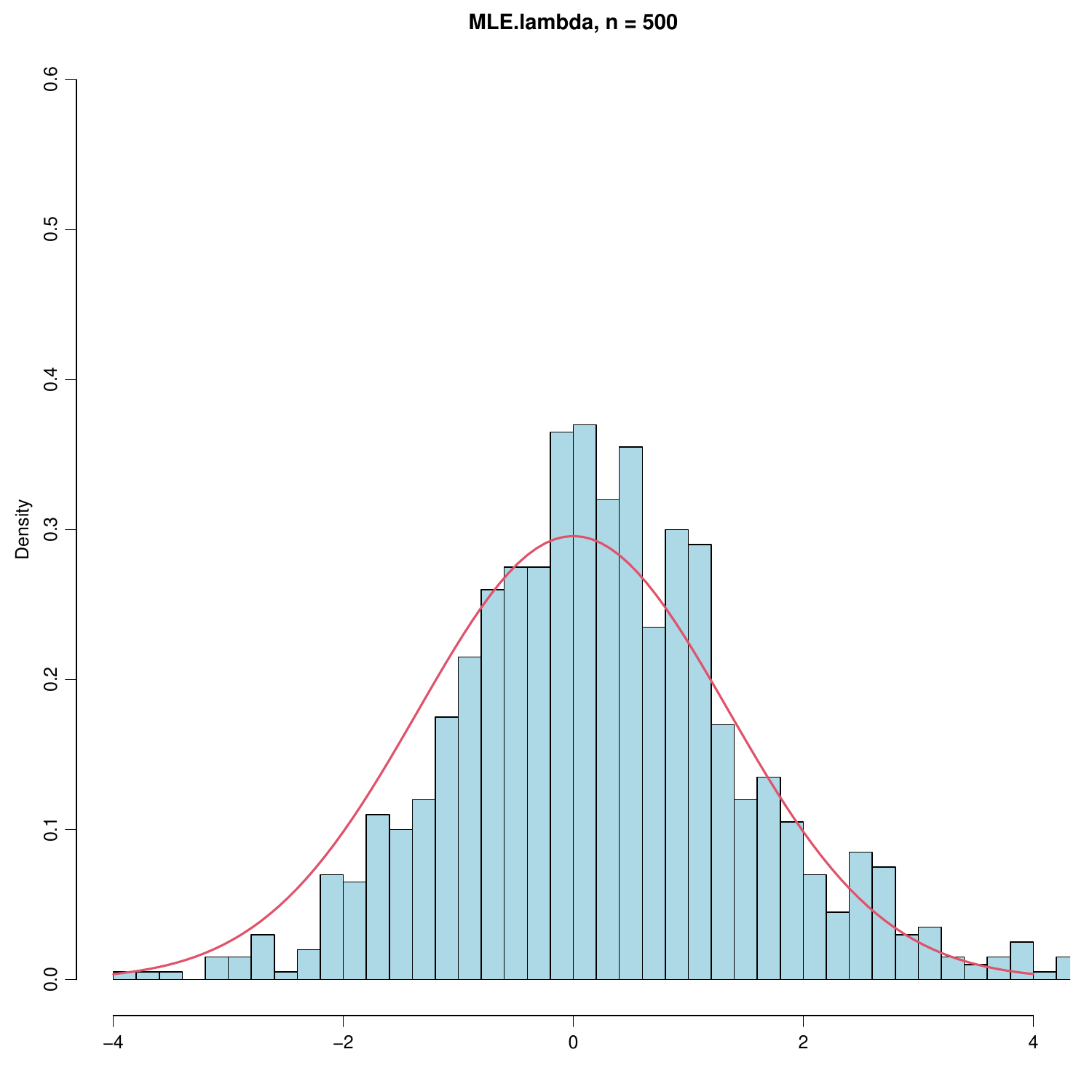}{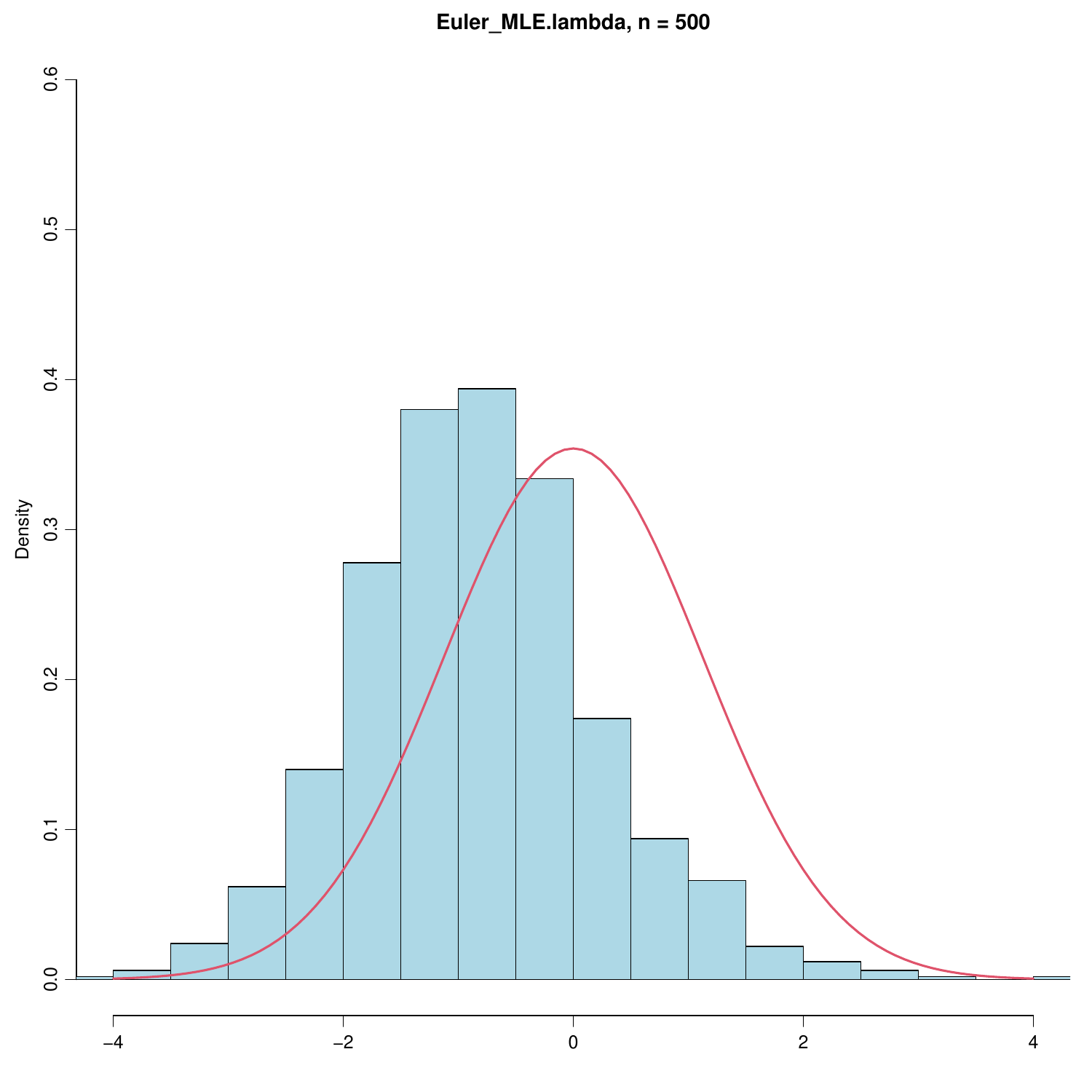}{$n=500$}\hfill
  \CellPair{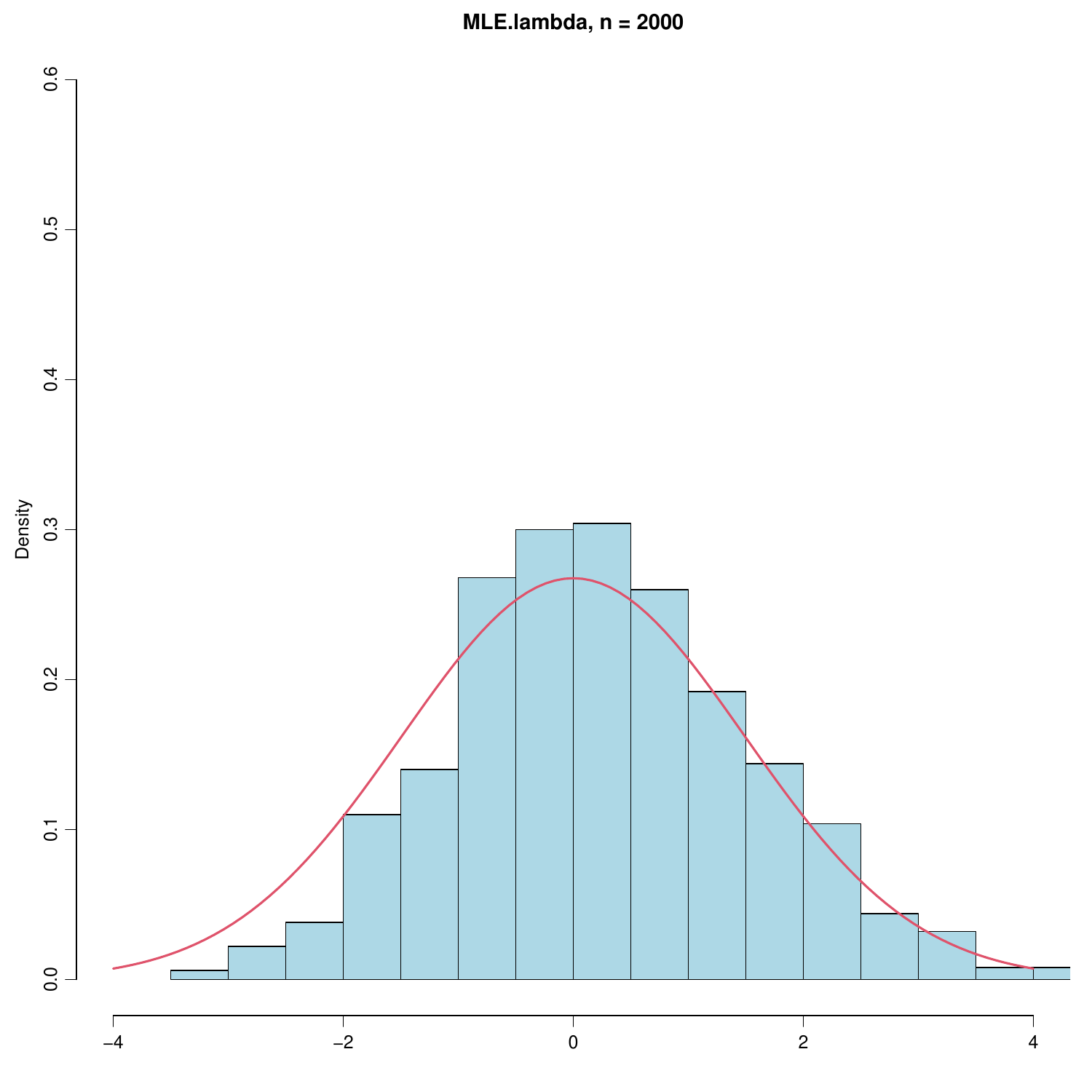}{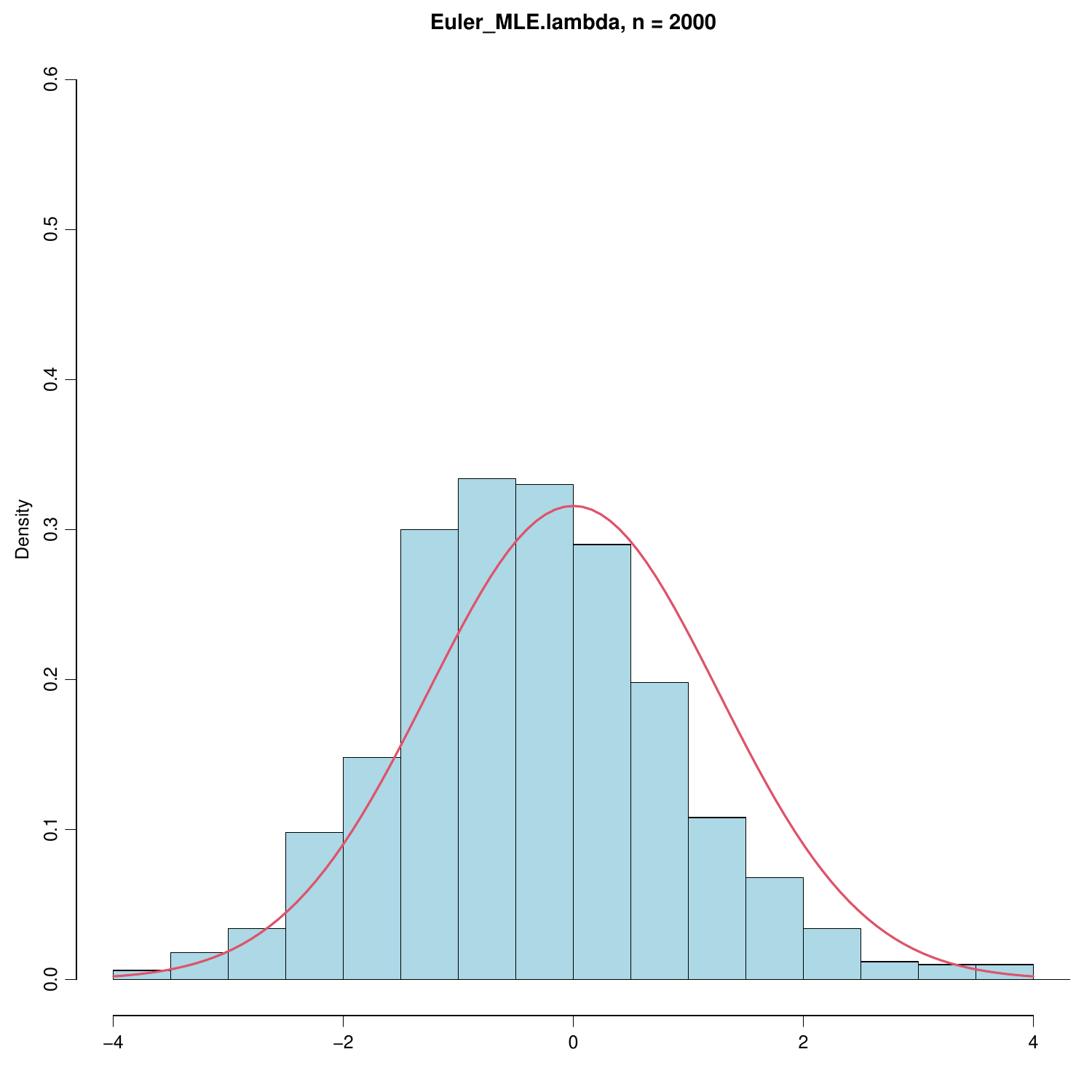}{$n=1000$}\hfill
  \CellPair{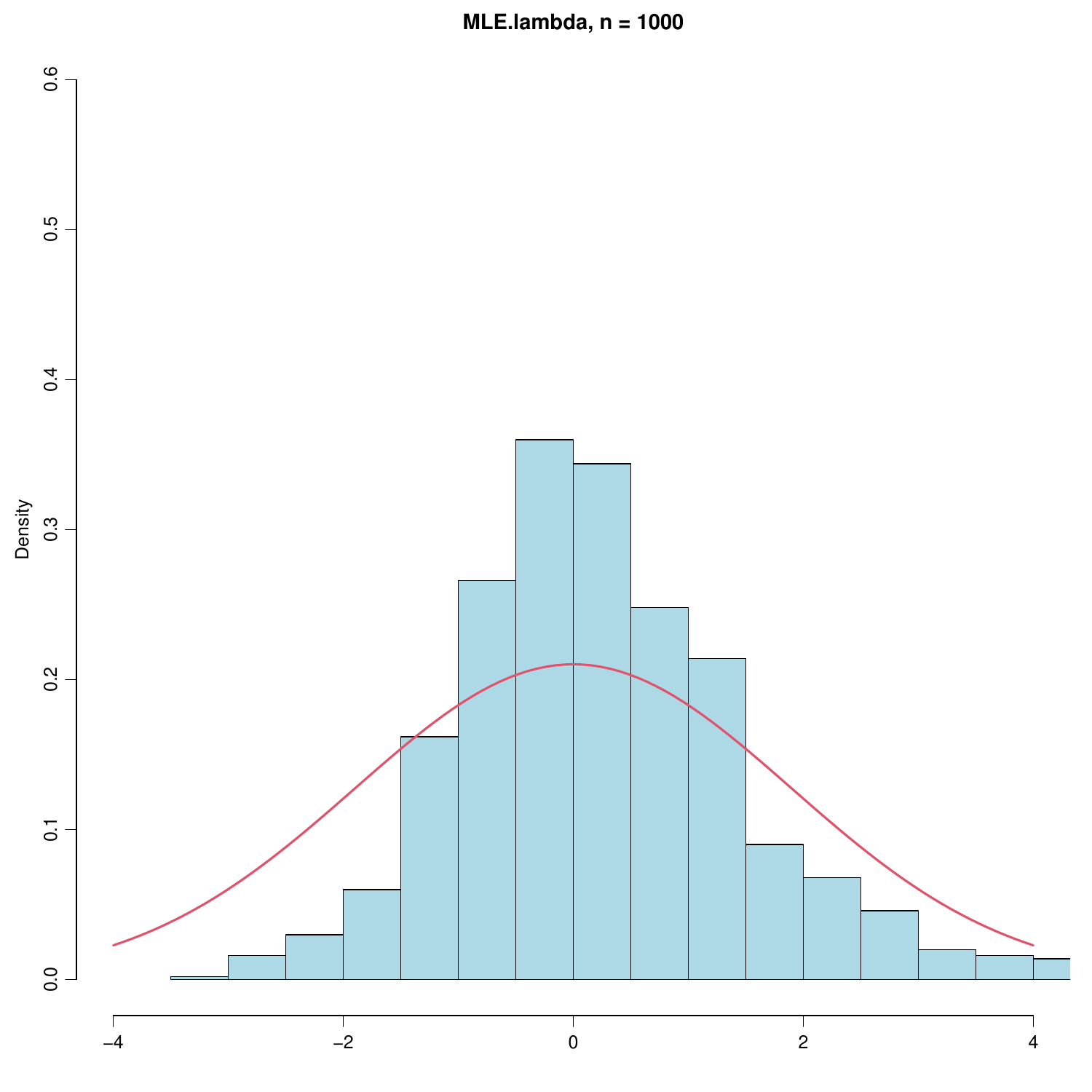}{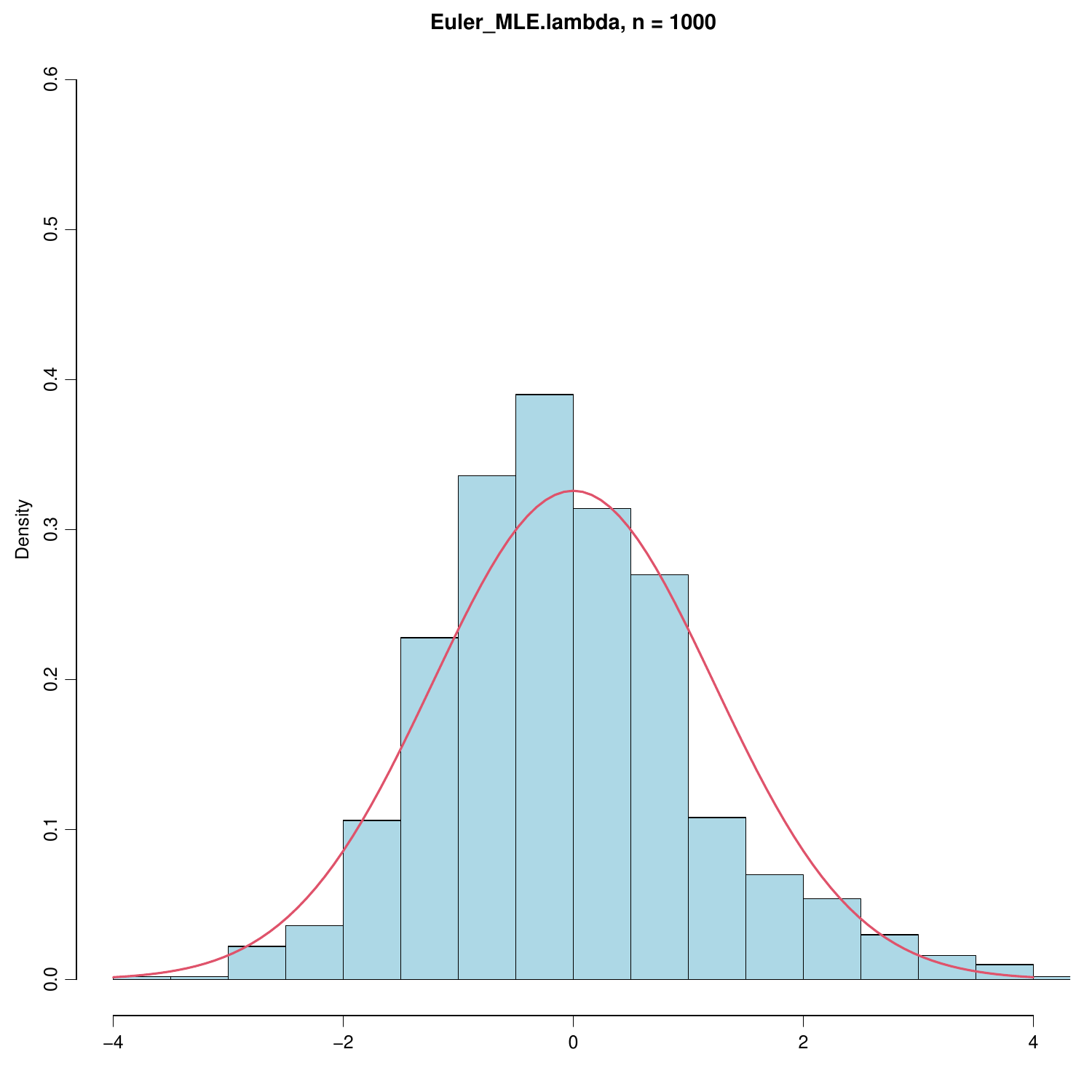}{$n=2000$}
  
  \vspace{5mm}
  
  \rowtitle{$\mu$}
  \CellPair{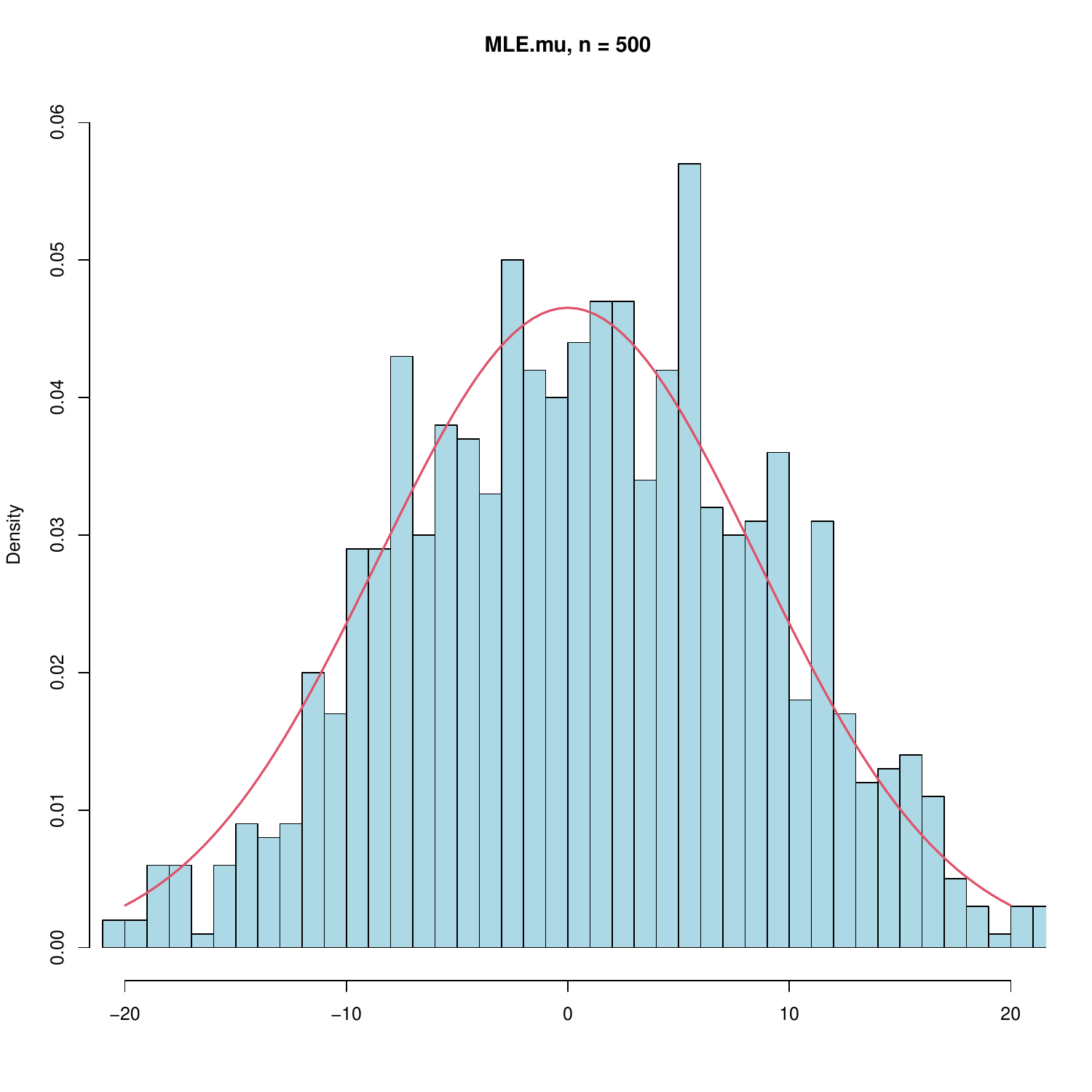}{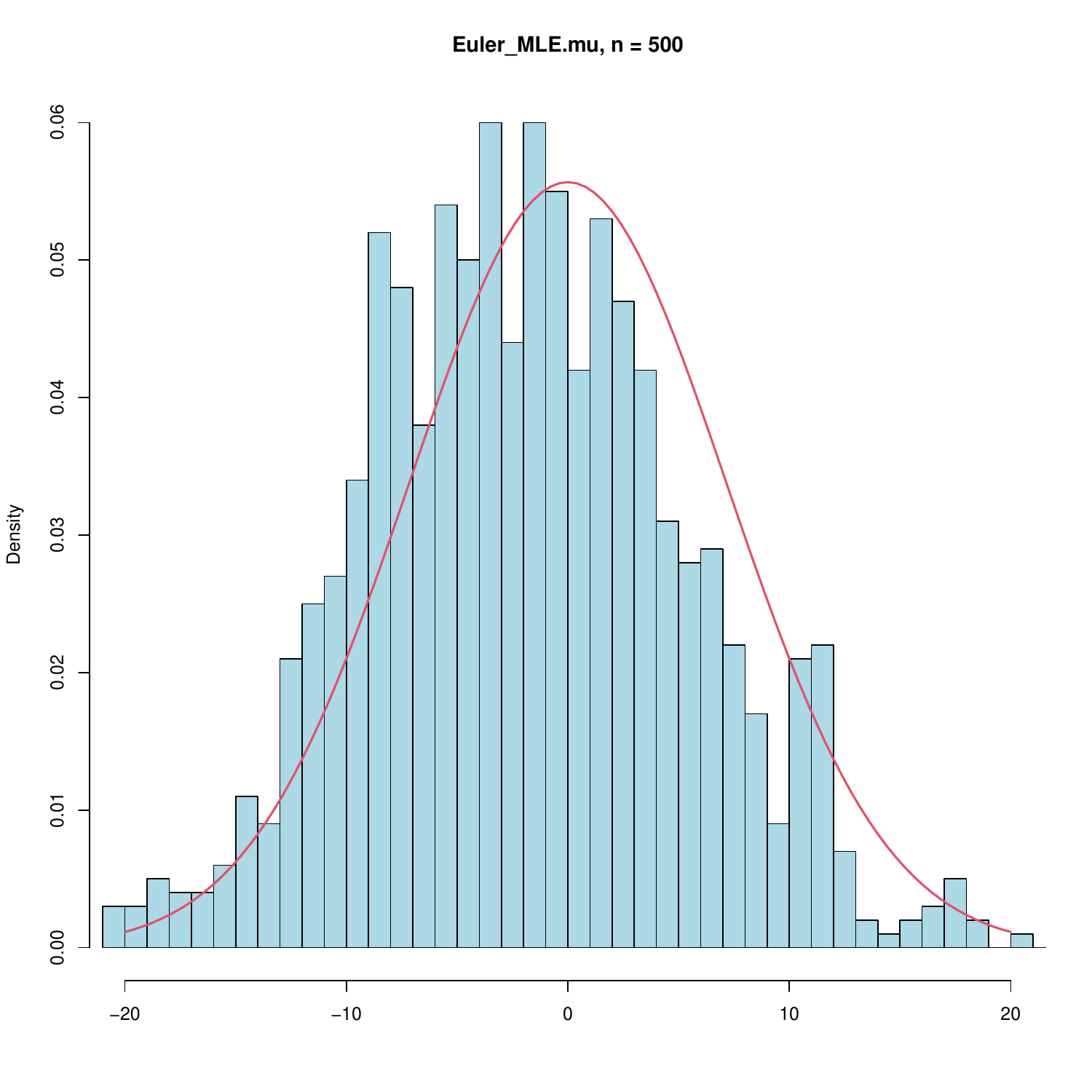}{$n=500$}\hfill
  \CellPair{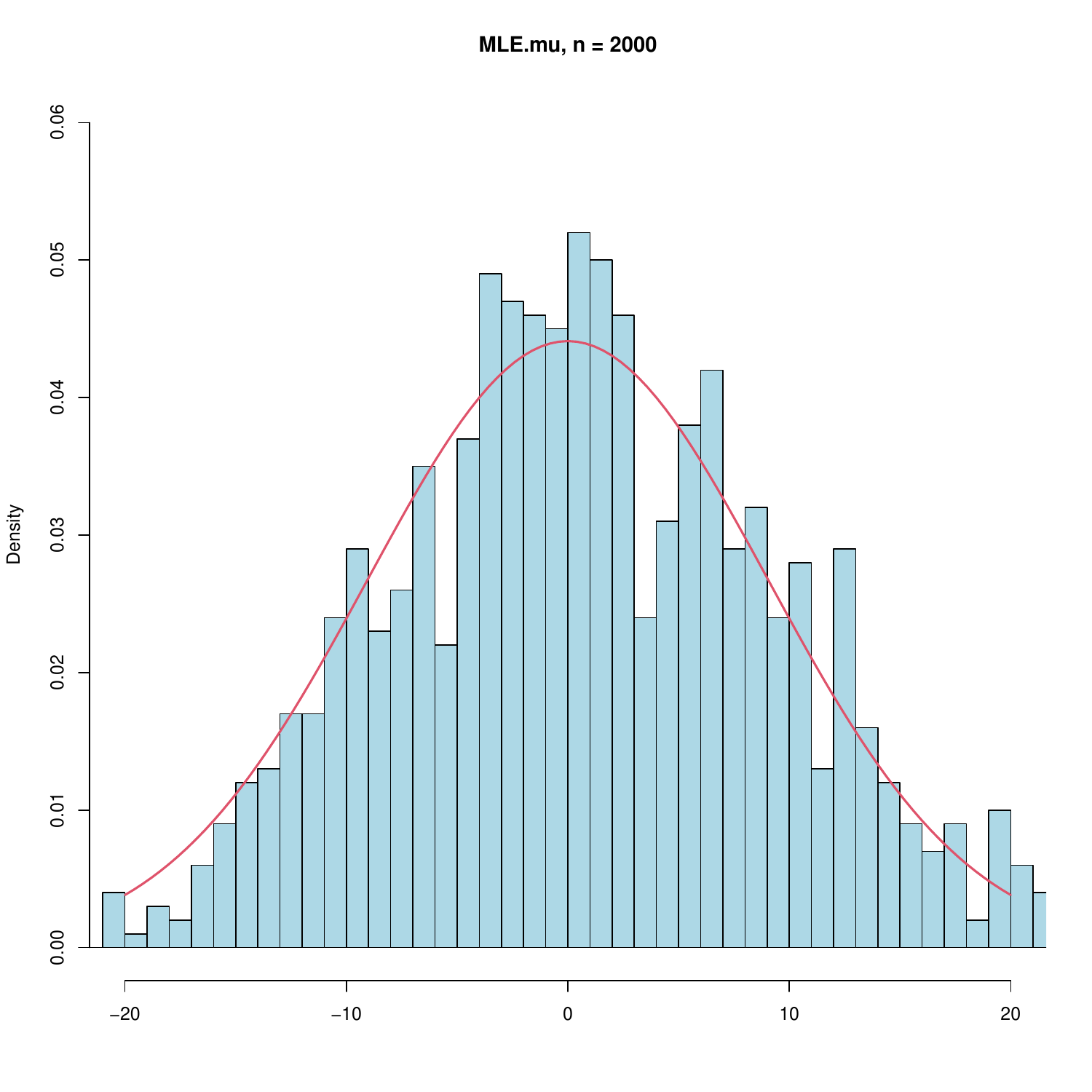}{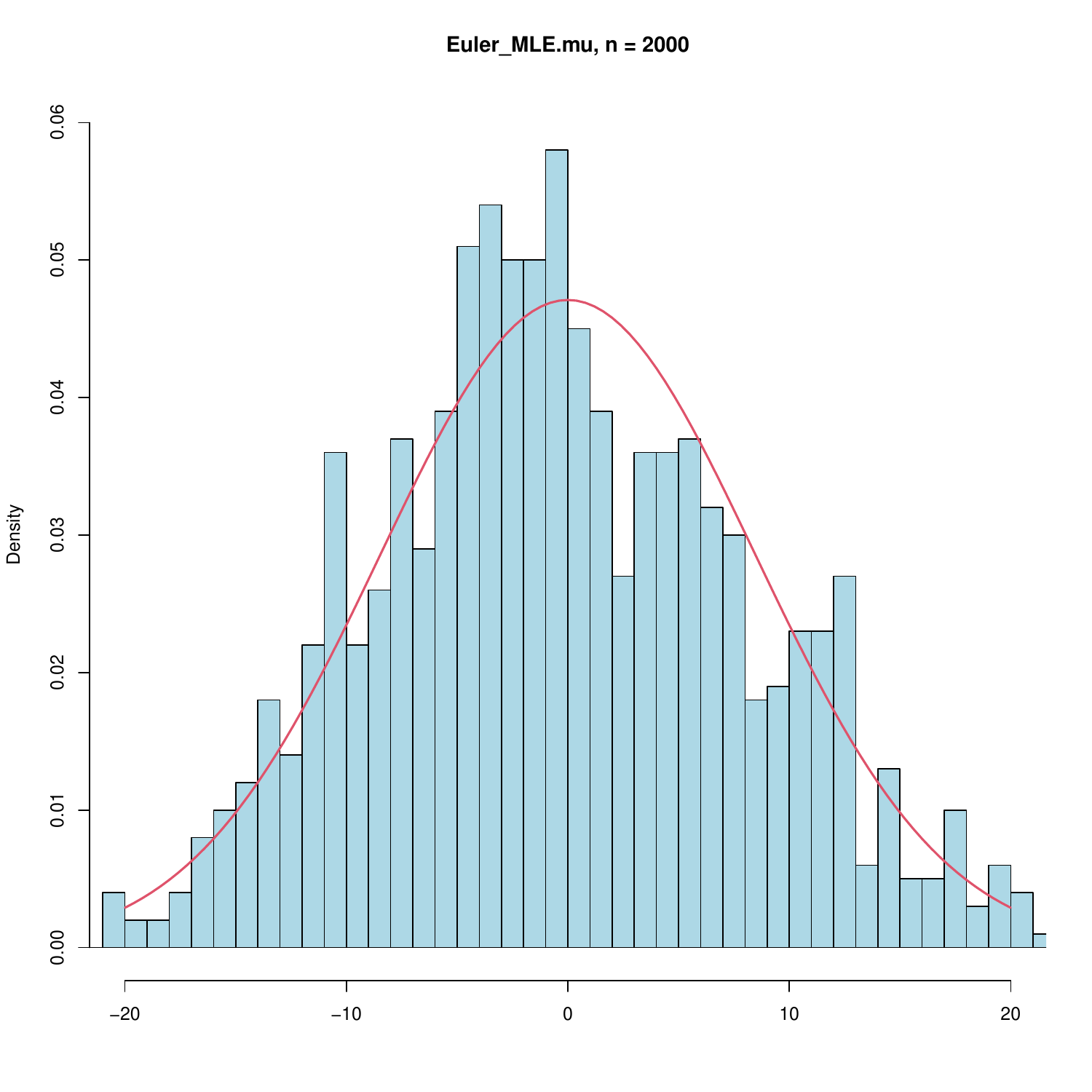}{$n=1000$}\hfill
  \CellPair{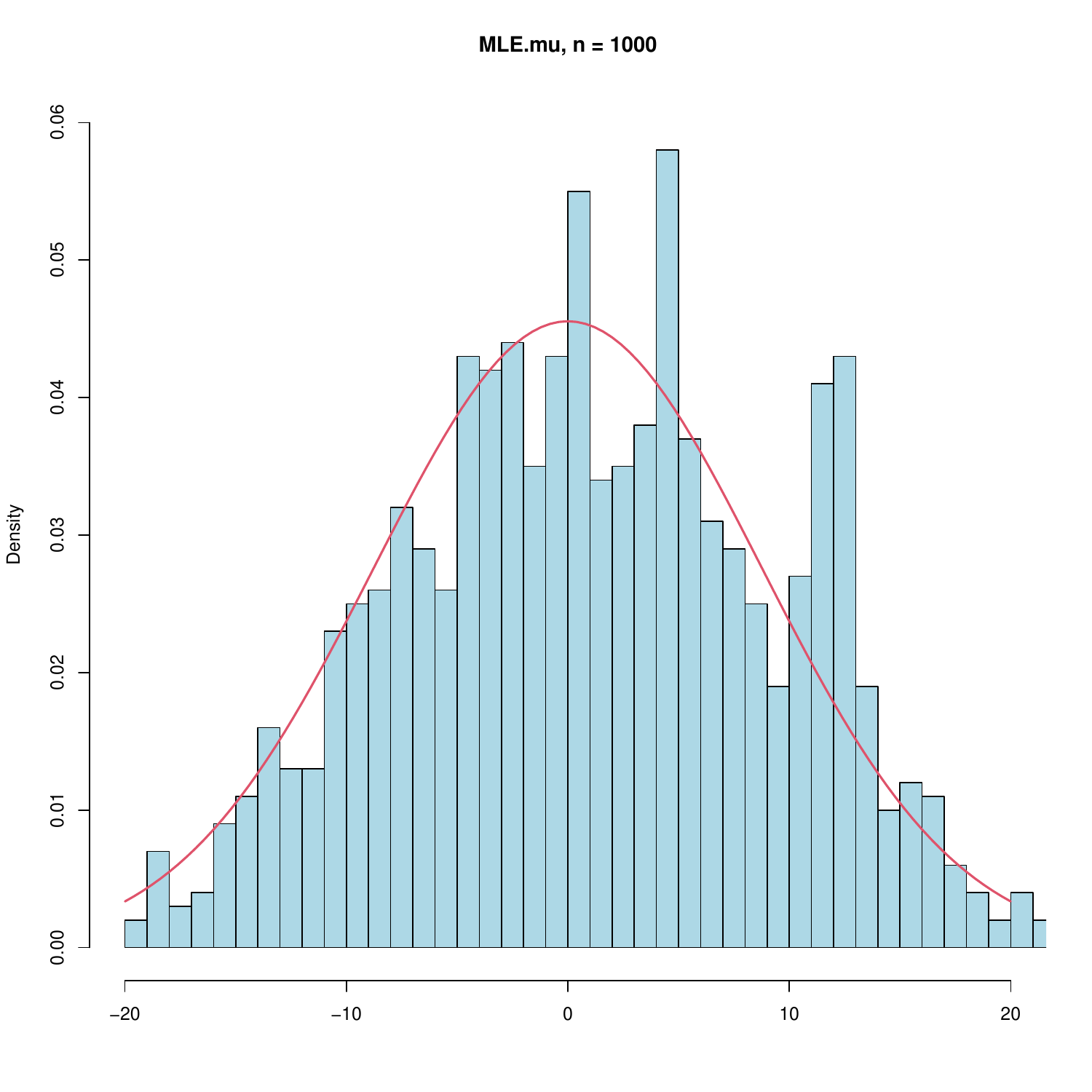}{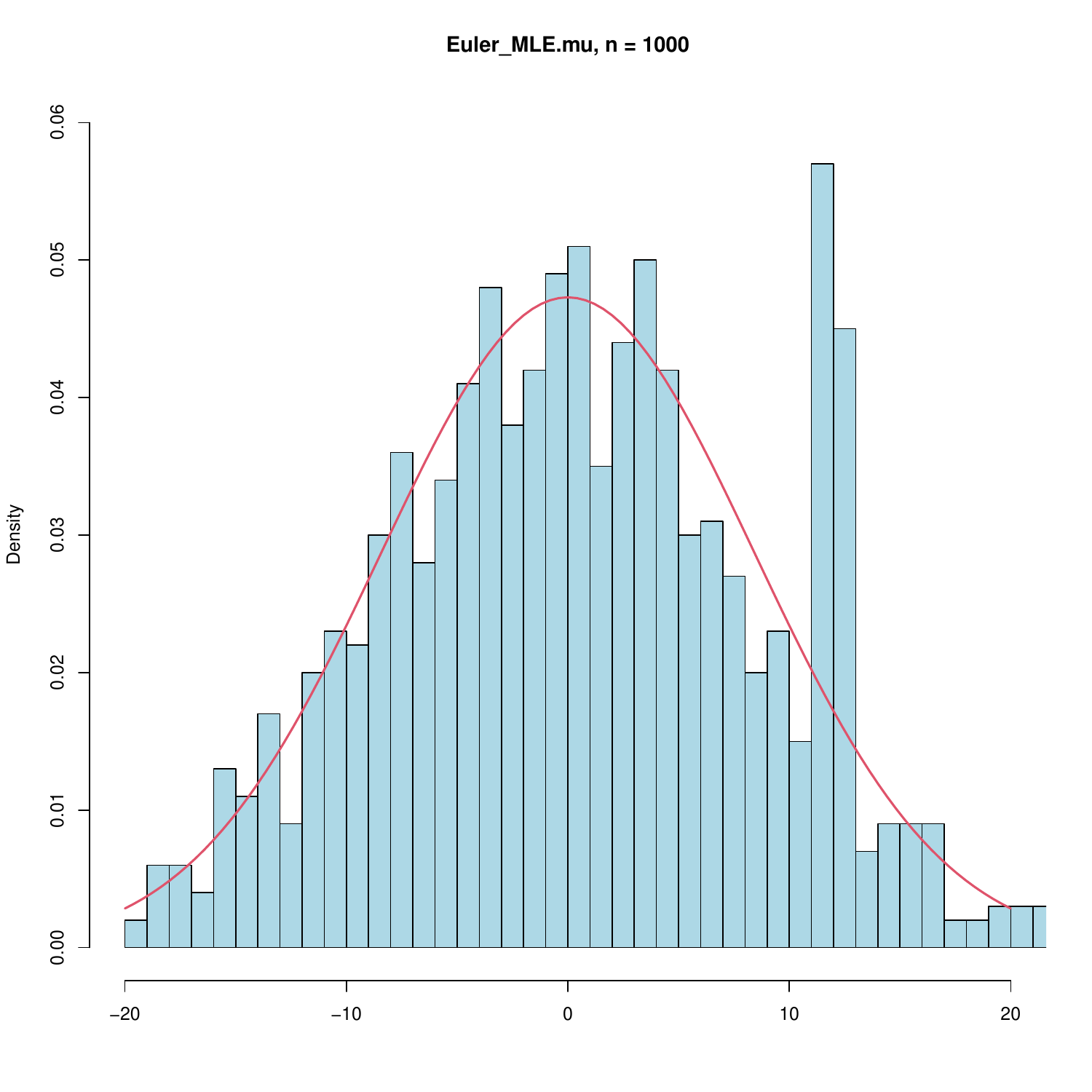}{$n=2000$}

  \vspace{5mm}
  
  \rowtitle{$\al$}
  \CellPair{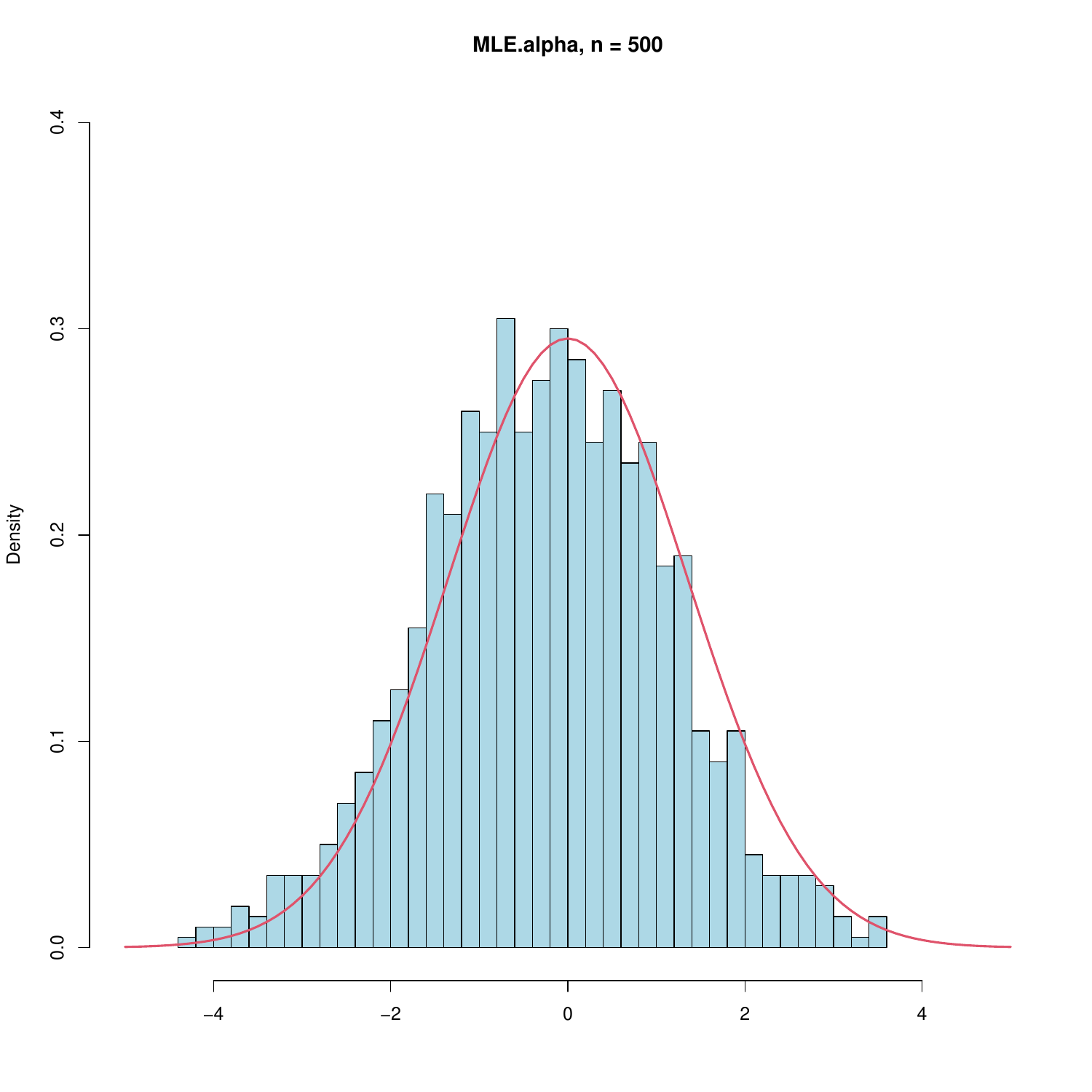}{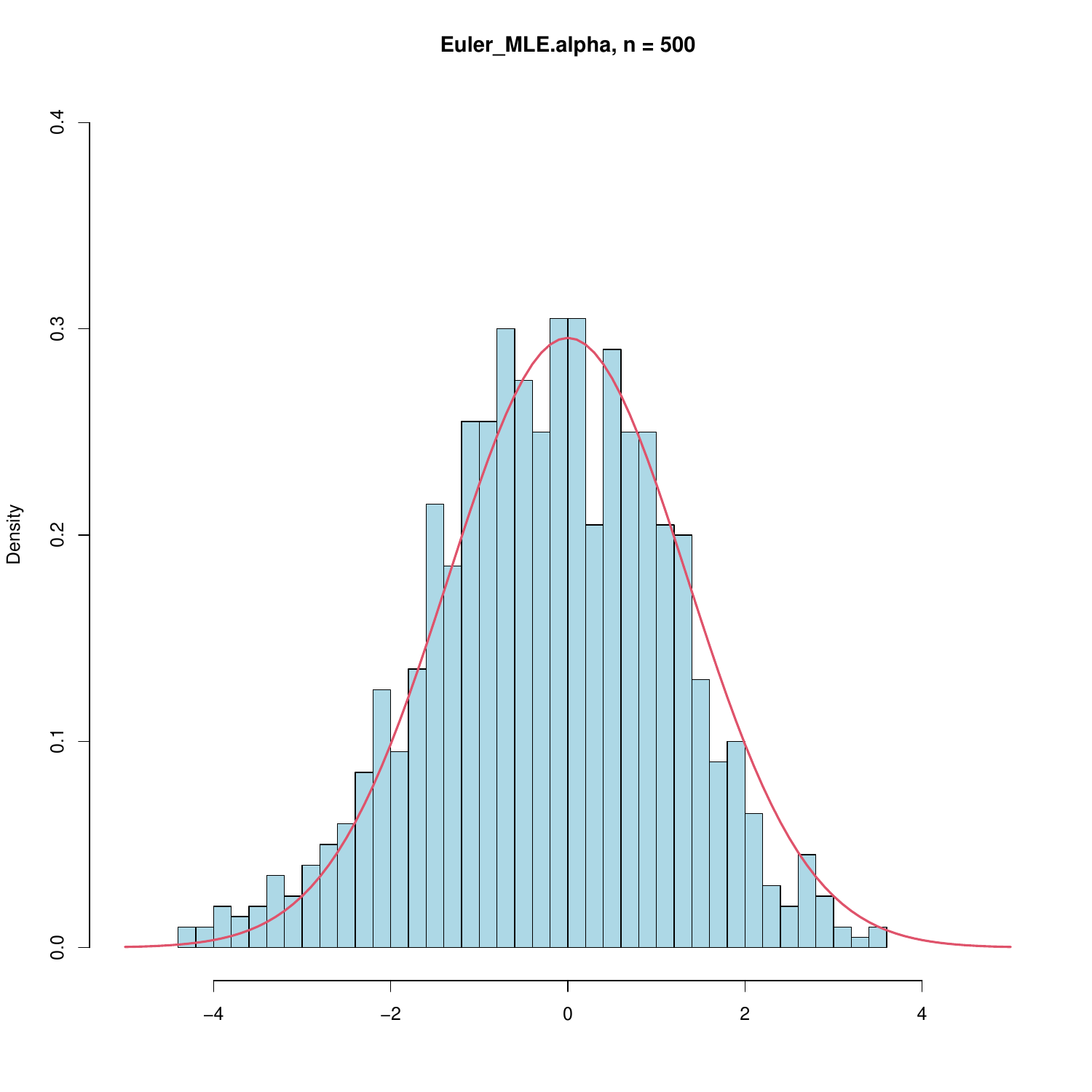}{$n=500$}\hfill
  \CellPair{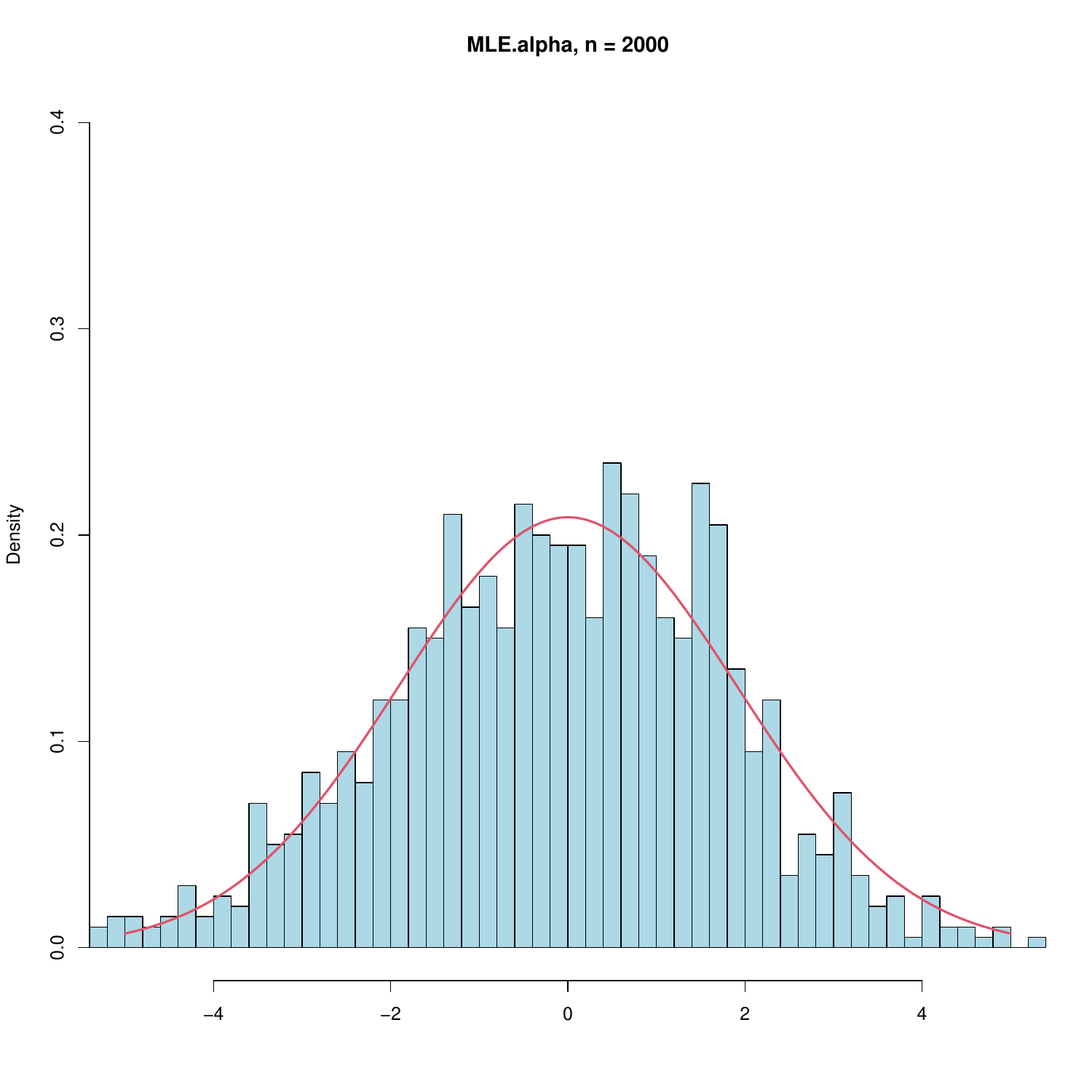}{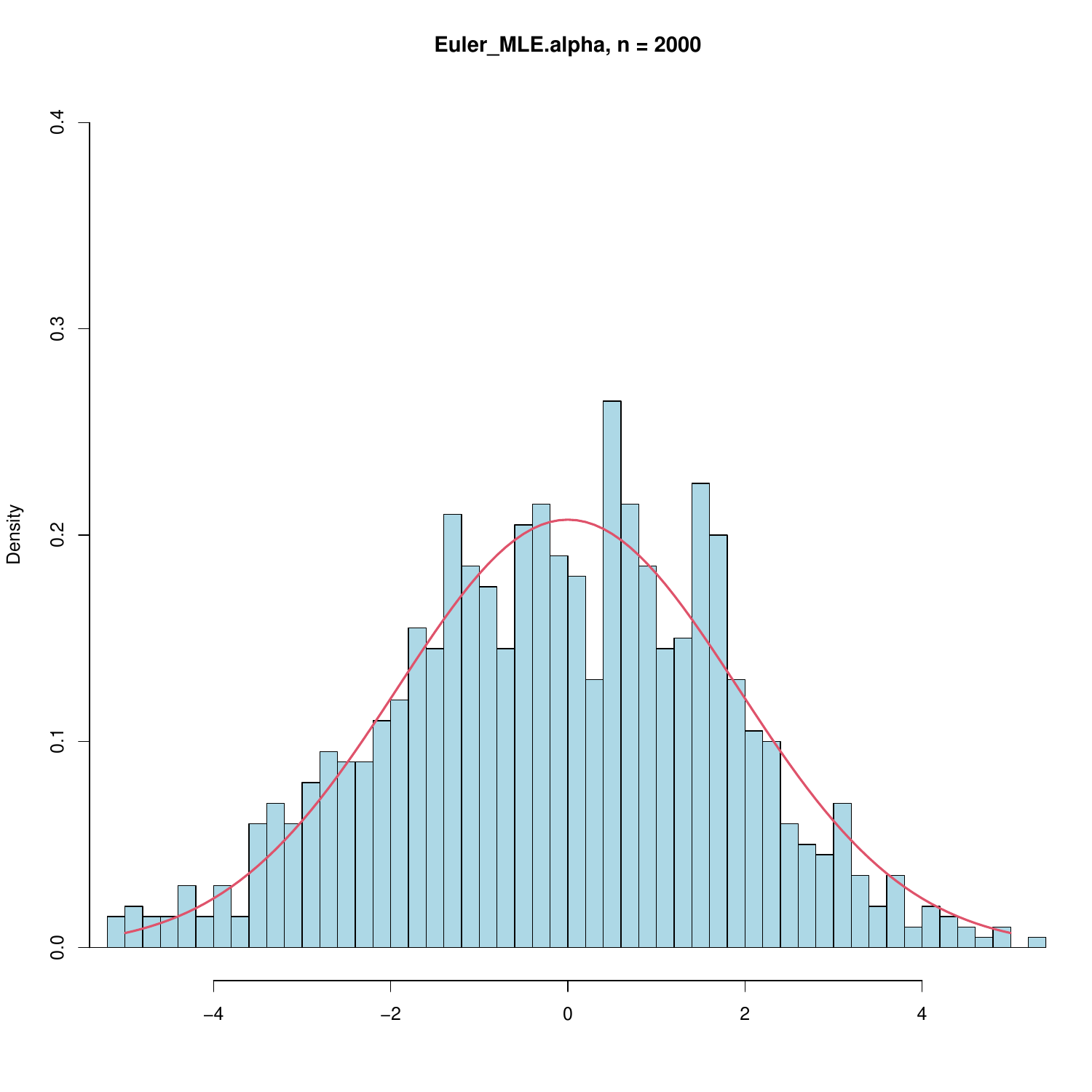}{$n=1000$}\hfill
  \CellPair{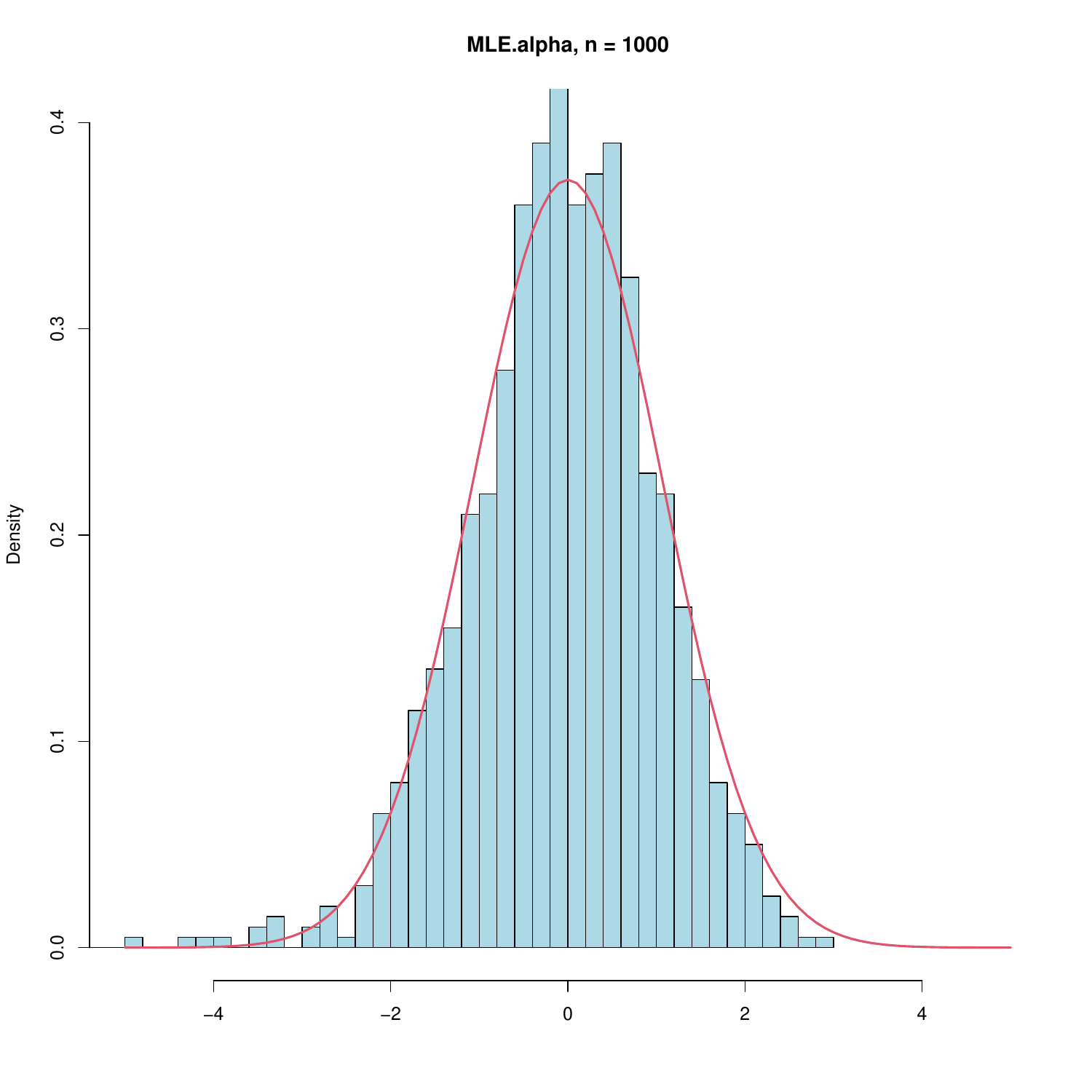}{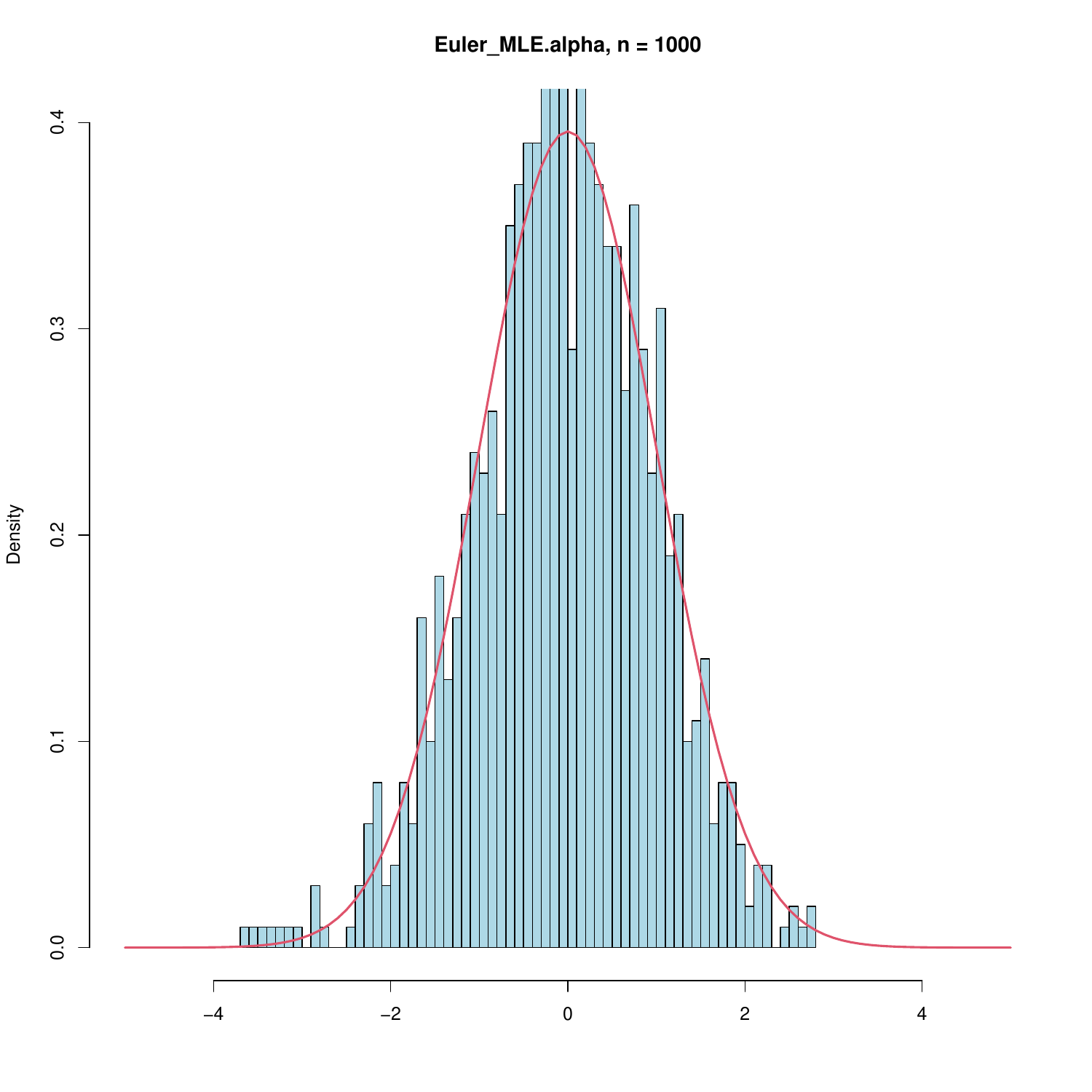}{$n=2000$}

  \vspace{5mm}
  
  \rowtitle{$\sig$}
  \CellPair{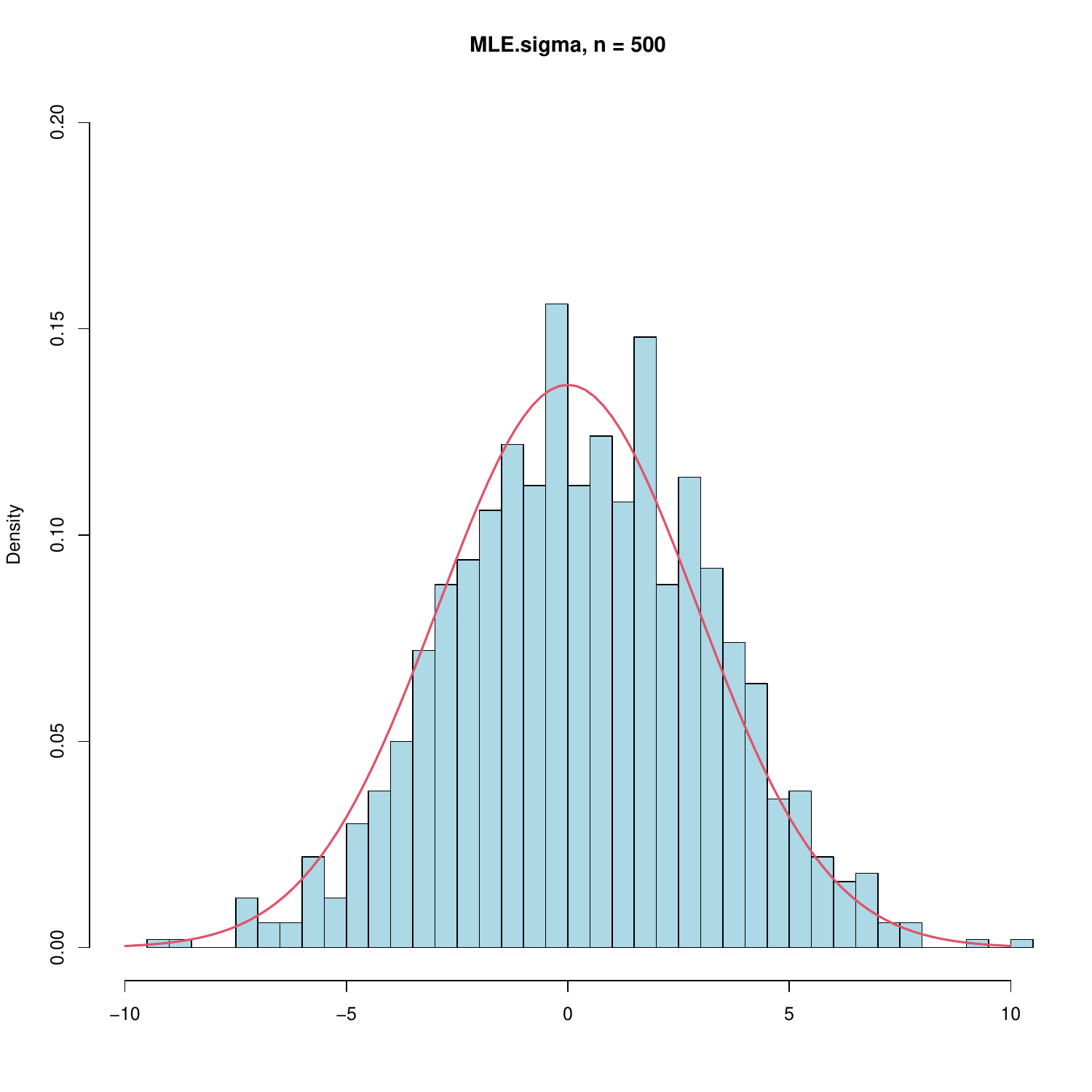}{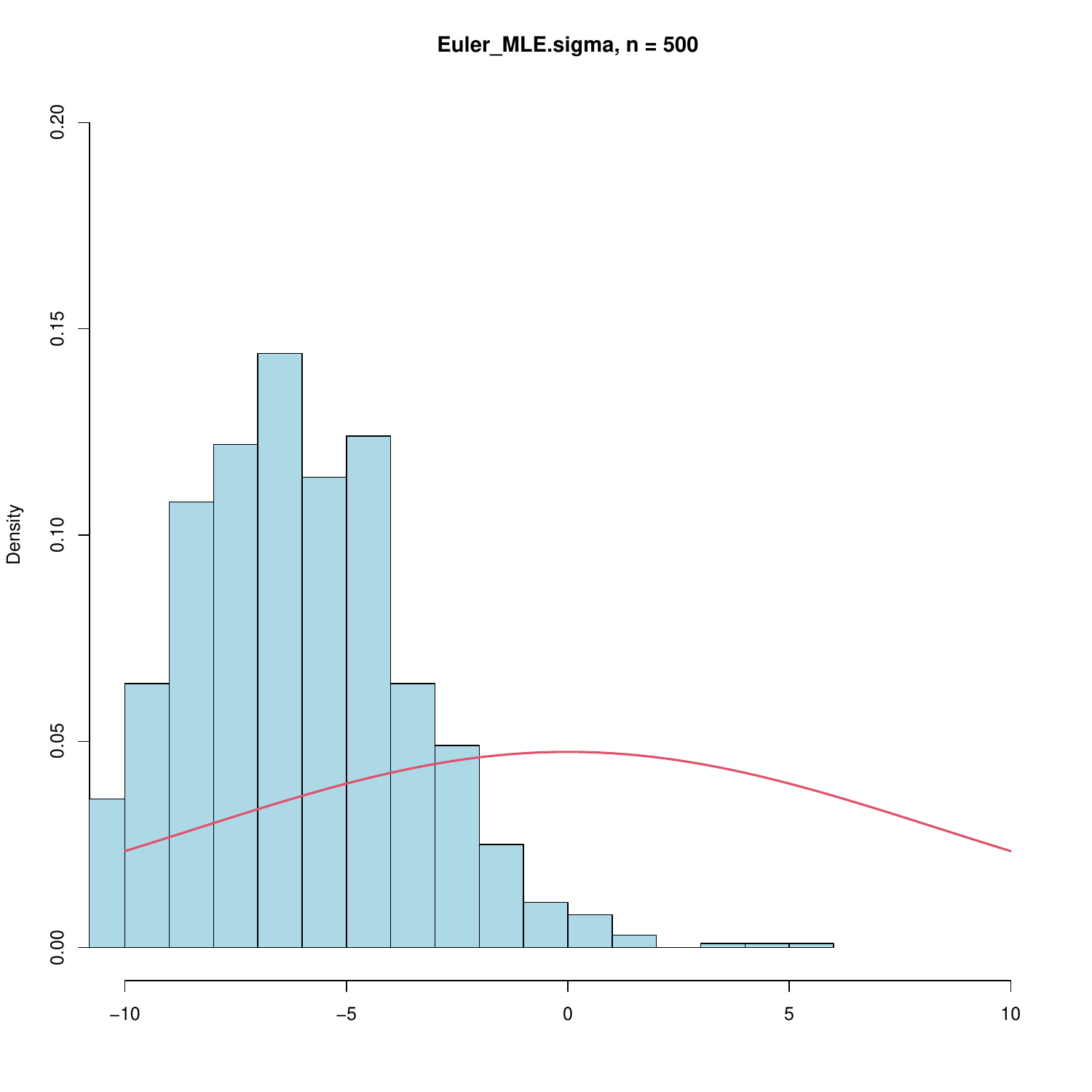}{$n=500$}\hfill
  \CellPair{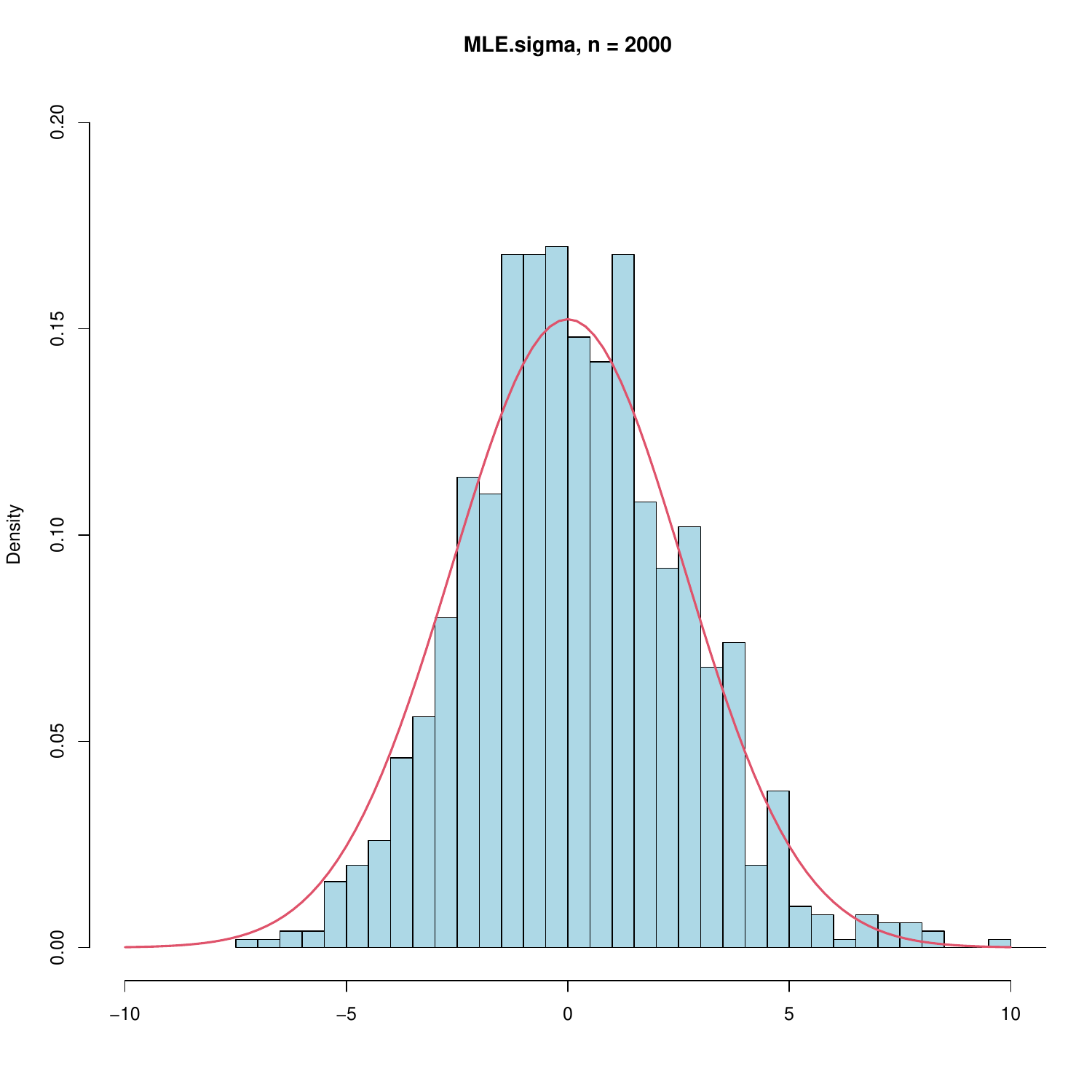}{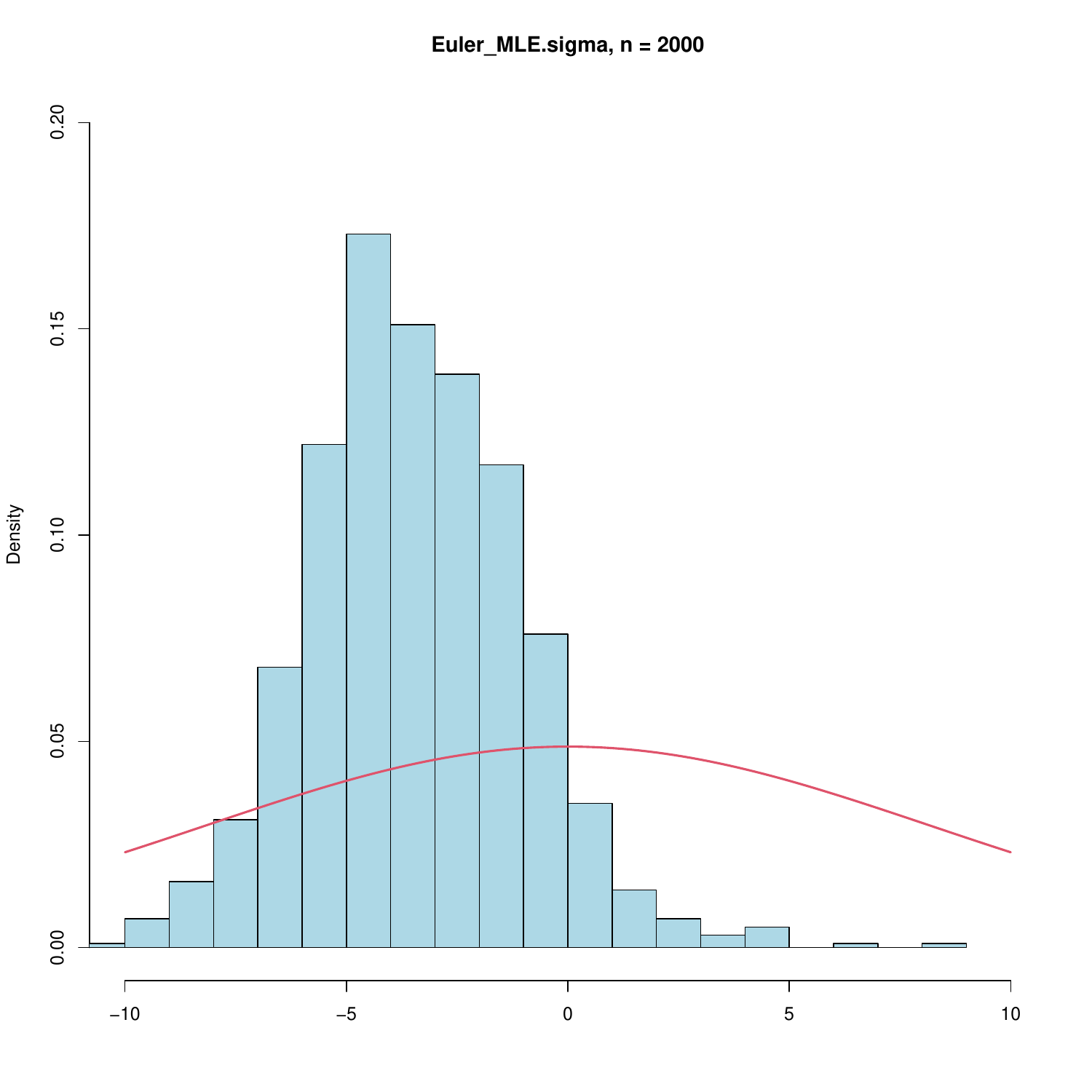}{$n=1000$}\hfill
  \CellPair{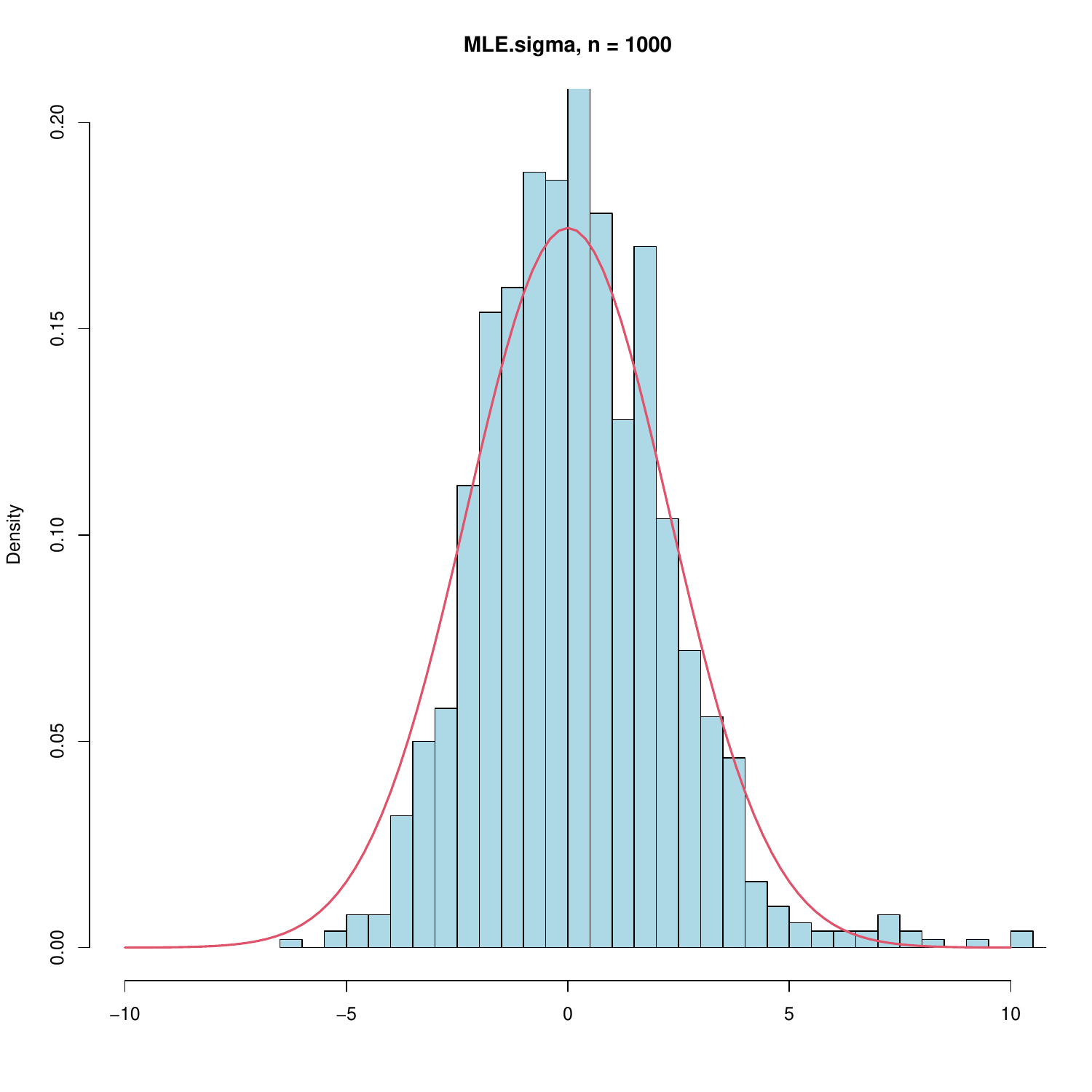}{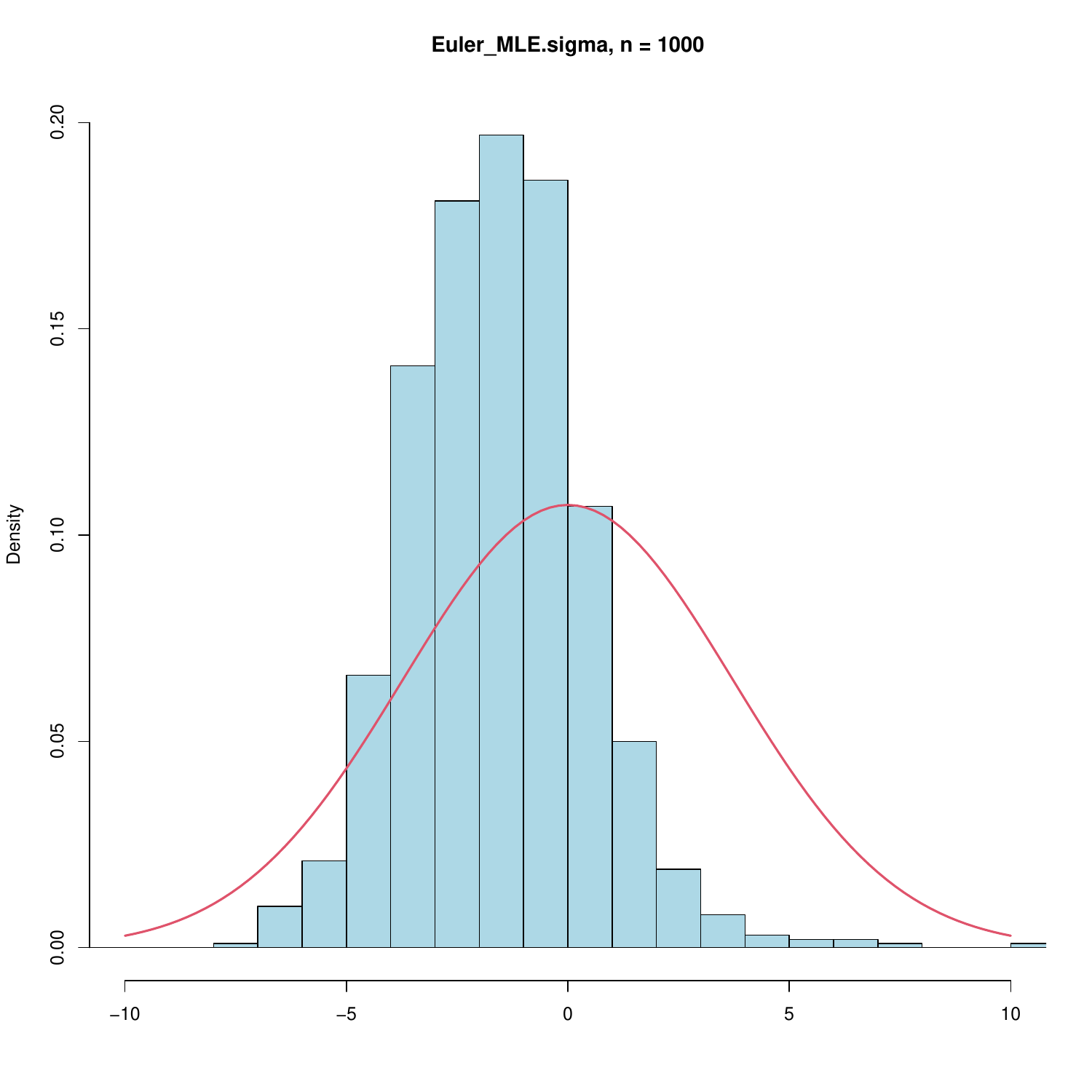}{$n=2000$}

  \vspace{5mm}

  \rowtitle{$\be$}
  \CellPair{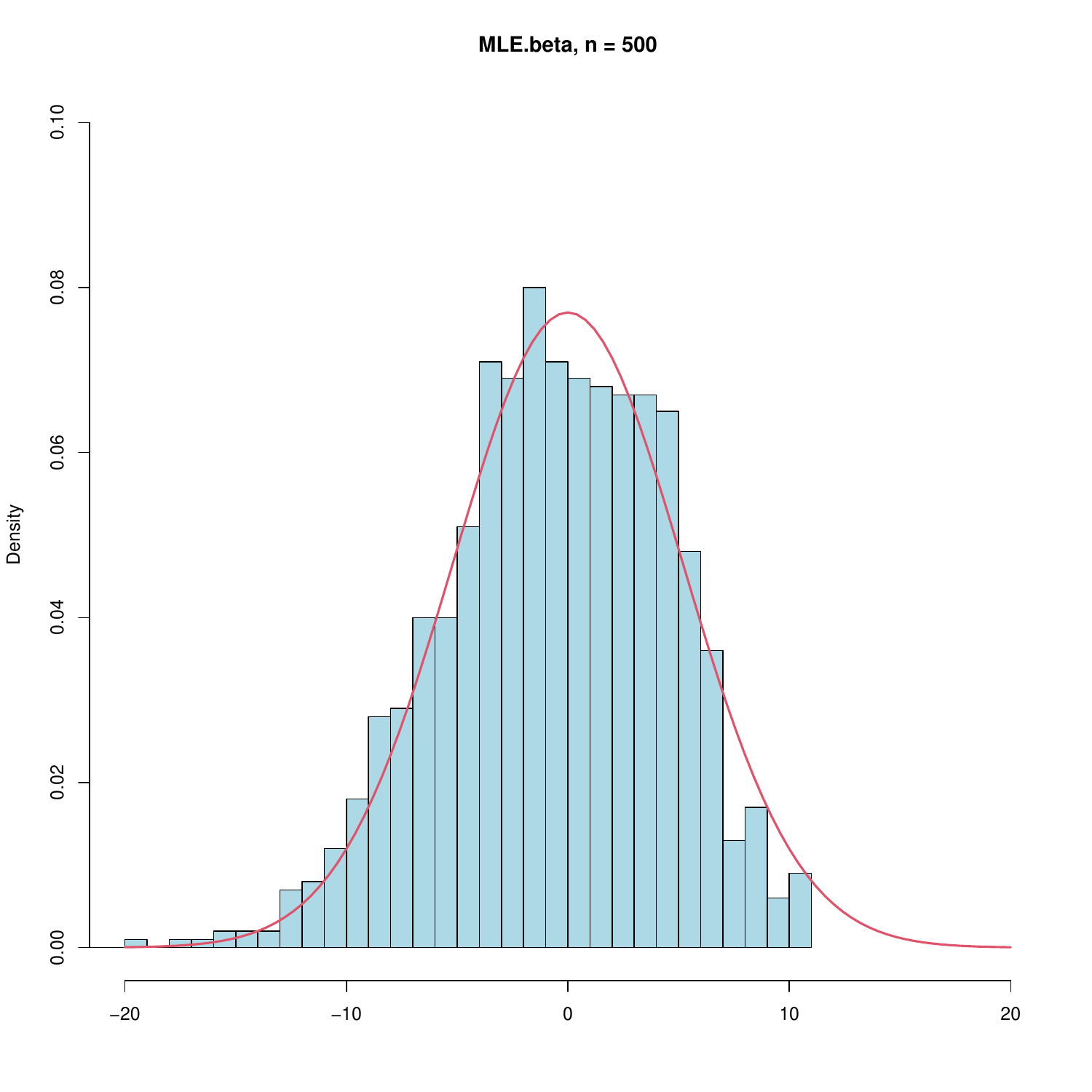}{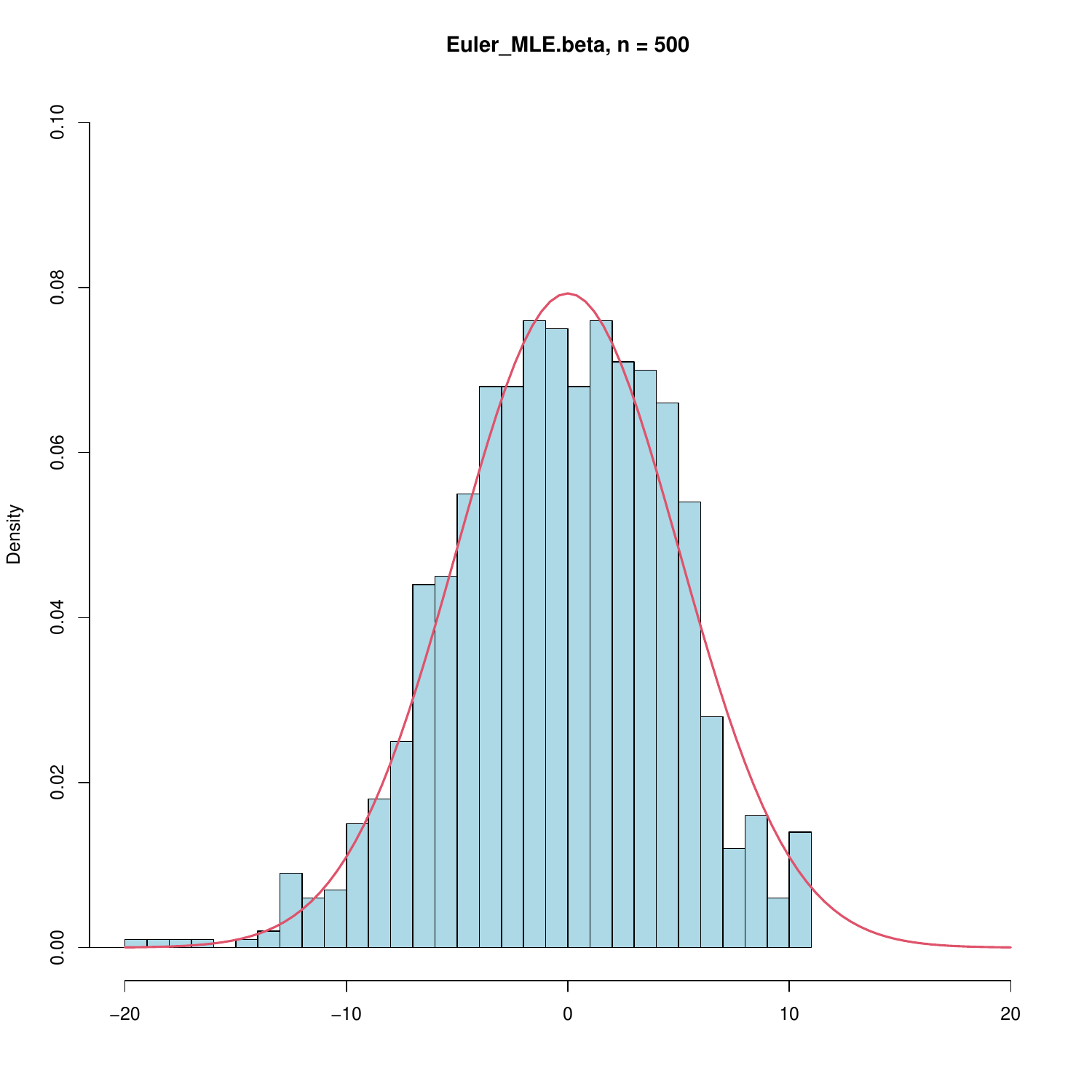}{$n=500$}\hfill
  \CellPair{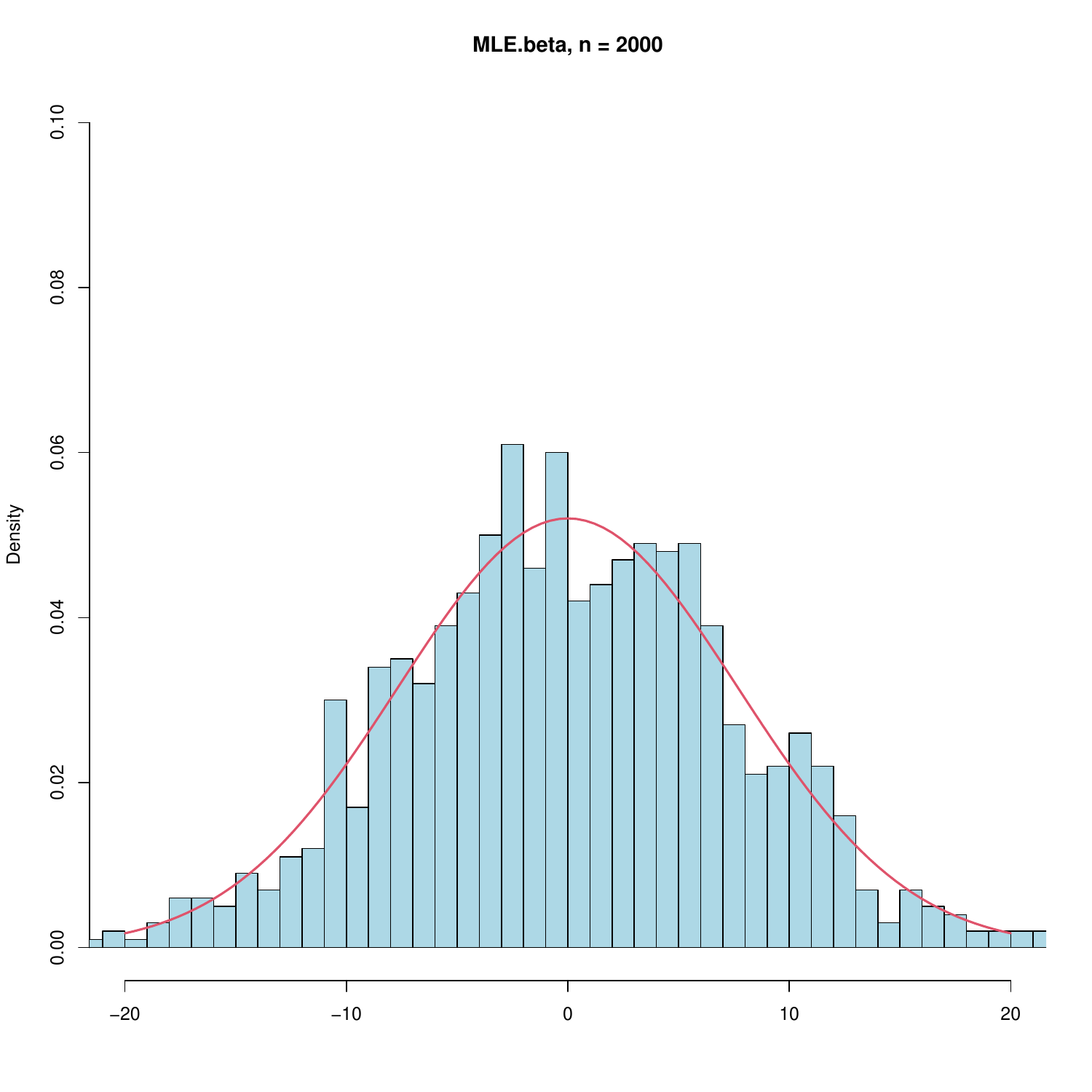}{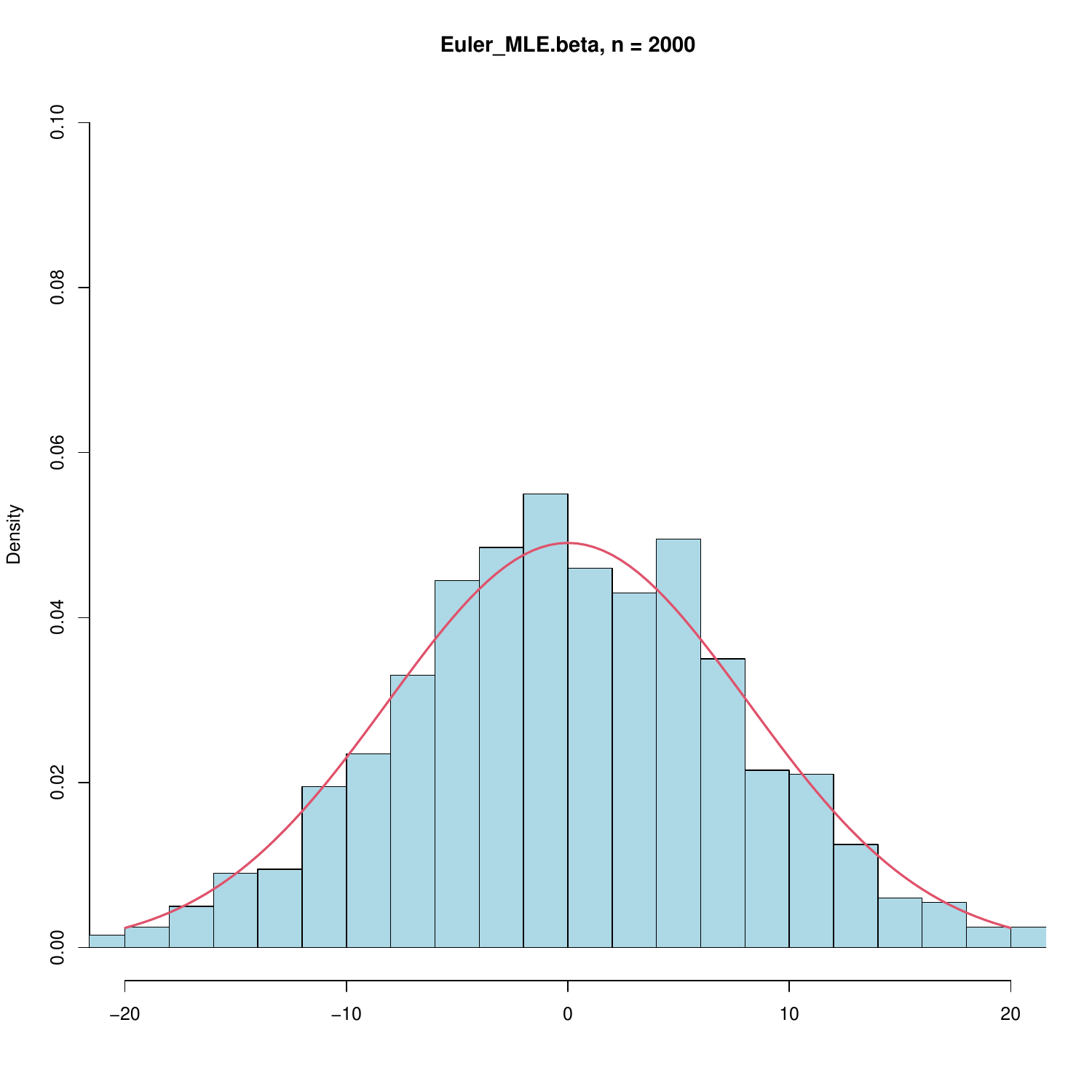}{$n=1000$}\hfill
  \CellPair{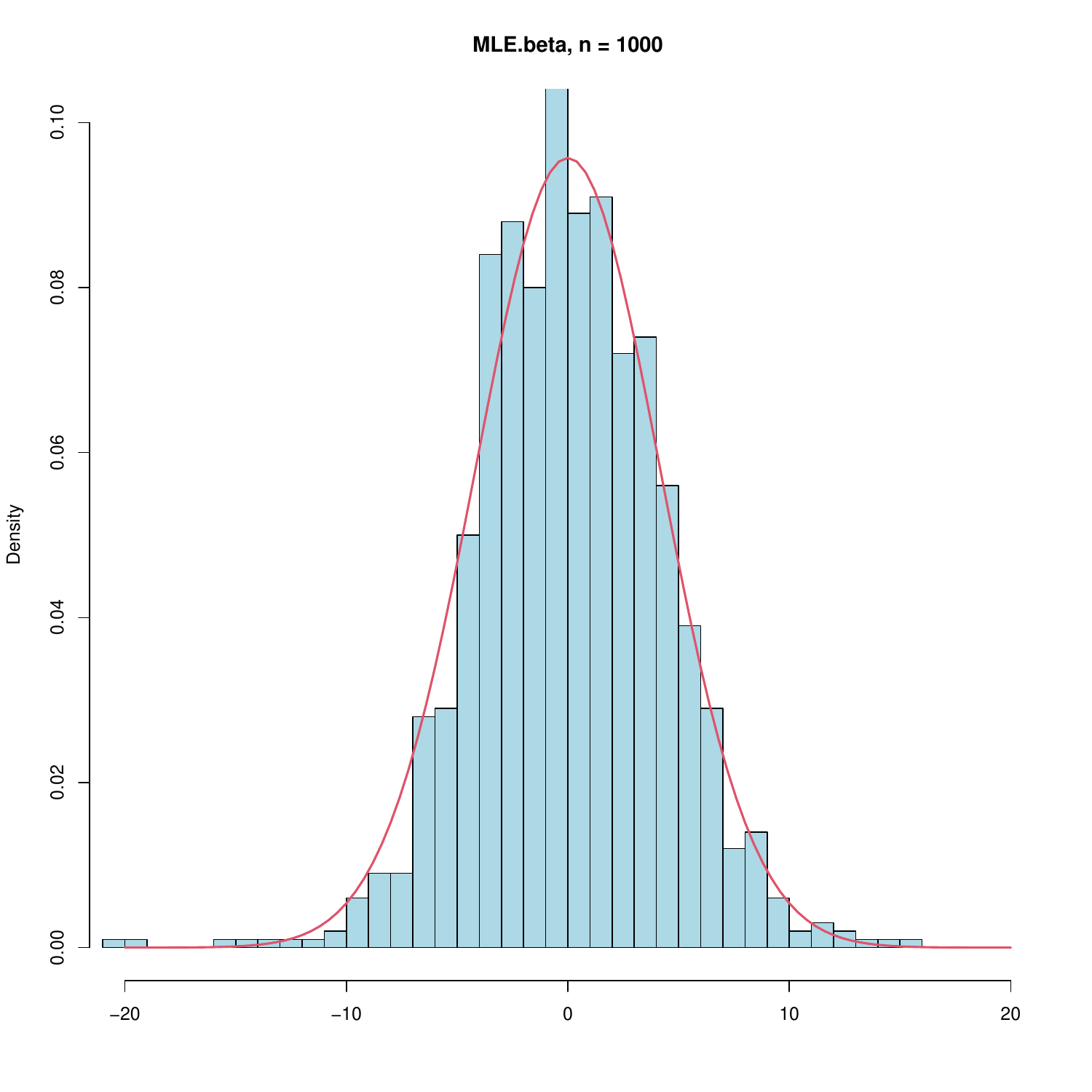}{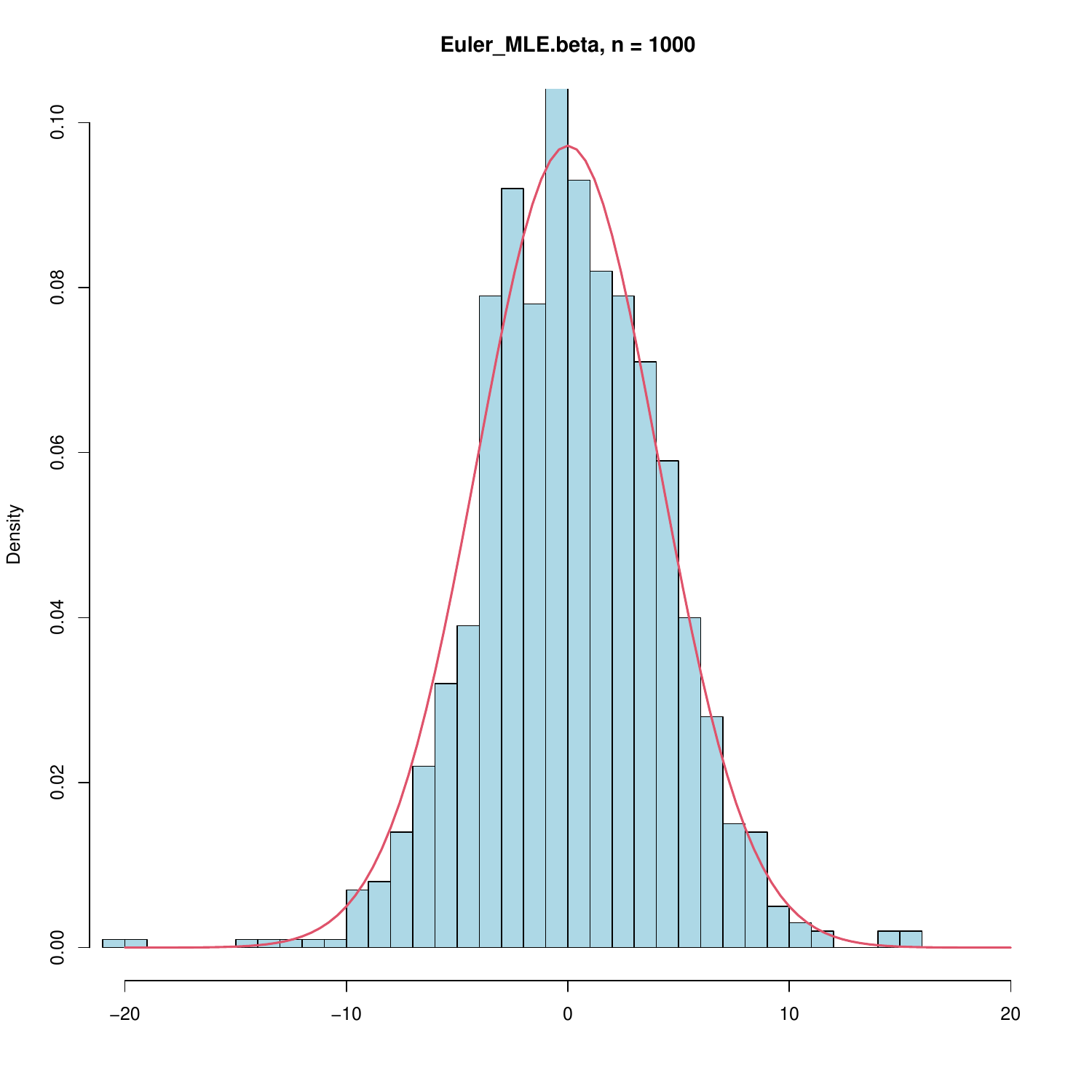}{$n=2000$}
  
  \caption{Comparison between histograms of the normalized MLE and Euler-QMLE with $\alpha=1.8$, $(\lambda,\mu,\sigma,\beta)=(1,2,5,0.5)$.}
  \label{fig:param-by-n_mle-vs-euler-1.8}
\end{figure}

\section{Modeling time scale}
\label{hm:sec_time.scale}


With a minor modification to the parametrization, we can incorporate the model time scale as a statistical parameter and derive the corresponding asymptotic properties.
The attempt to model the time scale has been previously made for a multidimensional ergodic diffusion process in \cite{EguMas19}. However, as the scaling factor depends on $\alpha$, the resulting asymptotic properties essentially differ from the framework in \cite{EguMas19}.

By Lemma \ref{hm:ssou_inc} and \eqref{hm:standard.Z}, we have for each $h,\tau>0$,
\begin{align}
    \tau^{-1/\al}J_{h \tau}
    &\sim S_\al^0\left(
    \beta, h^{1/\al}, \beta (h^{1/\al}-h) t_\al
    \right) \ast \del_{h\beta (1-\tau^{1-1/\al}) t_\al}
    \nn\\
    &\sim J_h + \left(\beta (1-\tau^{1-1/\al}) t_\al \right) h.
\end{align}
This shows that the process
\begin{equation}
    \tilde{J}_t := \tau^{-1/\al}J_{t \tau}
    -\left(\beta (1-\tau^{1-1/\al}) t_\al \right) t,\qquad t\in[0,1],
\end{equation}
defines a {\lp} which is distributionally equivalent to the original {\lp} $(J_t)_{t\in[0,1]}$; in particular, $\tilde{J}_t \sim S_\al^0(\beta,t^{1/\al},0)$.
Hence, by changing the time scale from $t$ to $t\tau$ with $\tau=T$, we rewrite the original model \eqref{hm:ssou_sde} as
\begin{align}\label{hm:ssou_sde_ts_pre1}
    dY_{t\tau} &= \left\{\left( \mu \tau
    + \beta\sig (\tau^{1/\al}-\tau) t_\al \right) - \lam \tau Y_{t\tau} \right\}dt
    + \sig \tau^{1/\al}d\tilde{J}_{t},\qquad t\in[0,1].
\end{align}
Let $\tilde{Y}_t:=Y_{t\tau}$.
To make the scale parameters identifiable, 
we set $\sig=1$:
\begin{align}
    d\tilde{Y}_{t} &= \left\{\left( \mu \tau
    + \beta (\tau^{1/\al}-\tau) t_\al \right) - \lam \tau \tilde{Y}_{t} \right\}dt
    + \tau^{1/\al}d\tilde{J}_{t},\qquad t\in[0,1].
\end{align}
Then, we formally get the same model as in \eqref{hm:ssou_sde}, but now with the terminal time $T=1$ and the (partial) parametrization change 
\begin{equation}
    (\lam,\mu,\al,\sig(=1),\beta) \mapsto \tilde{\theta}:=(\tilde{\lam},\tilde{\mu},\tilde{\al},\tilde{\sig},\tilde{\beta})
\end{equation}
caused by the time-scale change from $t$ to $t\tau$, where
\begin{equation}\label{hm:time.scale.para.trans}
(\tilde{\lam},\tilde{\mu},\tilde{\al},\tilde{\sig},\tilde{\beta}) := 
    \left(\lam \tau,\mu \tau
    + \beta (\tau^{1/\al}-\tau) t_\al,\al,\tau^{1/\al},\beta\right).
\end{equation}
Thus, we arrive at
\begin{align}\label{hm:ssou_sde_ts}
    d\tilde{Y}_{t} &= ( \tilde{\mu} - \tilde{\lam}\tilde{Y}_{t}) dt
    + \tilde{\sig} d\tilde{J}_{t},\qquad t\in[0,1],
\end{align}
with the available data being now $(\tilde{Y}_{s_j})_{j=0}^{n}$ for $s_j:=j/n$, hence $h=1/n$.
Note that $(\tilde{Y}_{s_j})_{j=0}^{n}$ is equivalent in law to $(Y_{t_j})_{j=0}^{n}$ with $t_j=j\tau/n$.

By the construction, estimation of $\tilde{\theta}$ of \eqref{hm:ssou_sde_ts} is equivalent to that of $\theta$ of \eqref{hm:ssou_sde}. Yet, in the former, we have the parameter $\tau$, which is interpretable as the model-time scale.
Although we set $\sig=1$, the time-scale parameter $\tau$ is instead unknown.
The above observations show that in the present case, where the noise is skewed, we need an additional drift adjustment compared with the symmetric noise case discussed in \cite[Remark 2.2(2)]{Mas23}.

For example, we could consider the following stepwise estimation procedure for \eqref{hm:ssou_sde_ts}:
\begin{itemize}
    \item First, we compute the moment estimator $(\hat{\al}_n,\hat{\sig}_n,\hat{\beta}_n)$ as before (with $T=1$), and then return $(\hat{\al}_n,\hat{\tau}_n:=\hat{\sig}_n^{\aes},\hat{\beta}_n)$ as an estimate of $(\al,\tau,\beta)$;
    \item Then, letting 
    \begin{equation}
        \hat{\mathfrak{m}}(\mu):=\mu \,\hat{\tau}_n
    + \bes (\hat{\tau}_n^{1/\aes}-\hat{\tau}_n)\tan\frac{\aes\pi}{2},
    \end{equation}
    we estimate $(\lam,\mu)$ as before by the Euler-QMLE:
    \begin{align}
        (\lam,\mu) &\mapsto 
        \sumj \log \left( \frac{1}{(\hat{\tau}_n h)^{1/\aes}}
        \phi_{\aes,\bes} \left(
        \frac{\D_j\tilde{Y} - (\hat{\mathfrak{m}}(\mu)-\lam\hat{\tau}_n)h}{(\hat{\tau}_n h)^{1/\aes}}
        \right) \right)
        \nn\\
        &= \hat{C}_n + \sumj \log \phi_{\aes,\bes} \left(
        \frac{\D_j\tilde{Y} - (\hat{\mathfrak{m}}(\mu)-\lam\hat{\tau}_n)h}{(\hat{\tau}_n h)^{1/\aes}}
        \right),
    \end{align}
    where $\hat{C}_n$ denotes a random variable not depending on the parameter.
\end{itemize}
In practice, we can also proceed with numerical joint optimization of $\tilde{\theta} \mapsto \mbbh_n(\tilde{\theta})$ with $h=1/n$ by using the estimate $(\hat{\al}_n,\hat{\tau}_n^{1/\aes},\hat{\beta}_n)$ as a partial initial value for numerical search of $(\tilde{\al},\tilde{\sig},\tilde{\beta})$. We can, in principle, apply the delta method based on \eqref{hm:time.scale.para.trans} to get the asymptotic behavior of $\tilde{\theta}_n$.

The simulation design is as follows:
\begin{itemize}
\item The terminal sampling time $T=1$.
\item Initial value of stochastic process $Y_0=0$.
\item True values: $\lam =1,\ \mu=2,\ \al \in \left\{1.0, 1.01, 1.5, 1.8\right\},\ \sig=5,\ \be=0.5$, with initial value $Y_0 = 0$; just for reference, we also ran the case $\al=1.0$.
\item Sample size: $n=500,1000,2000$.
\item The degree of the moment: $q=0.2$.
\item The number of the Monte Carlo simulations: $L=1000$.
\end{itemize}


The numerical results are given in Tables \ref{tab:euler-only-1.0-T1-time} to \ref{tab:euler-only-1.8-T1-time-case2}, and the associated histograms are given in Figures \ref{fig:hist-all-alpha-1.0} to \ref{fig:hist-all-alpha-1.8}. As in Section \ref{hm:sec_sim}, the proposed estimation procedure exhibits favorable performance, thereby validating our approach.

\begin{table}[H]
\centering
\scriptsize
\setlength{\tabcolsep}{4pt}
\resizebox{\linewidth}{!}{
\begin{tabular}{c|c|c|c|c|c|c}
\hline
$n$
 & $\hat\lambda_n$
 & $\hat\mu_n$
 & $\hat\alpha_n$
 & $\hat\sigma_n$
 & $\hat\beta_n$
 & Time (s) \\
\hline
500
 & 1.0343
 & 2.0074
 & 1.0004
 & 5.1805
 & 0.4971
 & 47.44 \\
 & (0.2604)
 & (0.3362)
 & (0.0627)
 & (2.1137)
 & (0.0792)
 &  \\
\hline
1000
 & 1.0183
 & 2.0006
 & 0.9985
 & 5.1580
 & 0.4991
 & 100.32 \\
 & (0.2065)
 & (0.2441)
 & (0.0411)
 & (1.9028)
 & (0.0553)
 &     \\
\hline
2000
 & 1.0186
 & 1.9966
 & 1.0004
 & 5.0628
 & 0.4984
 & 208.66 \\
 & (0.1651)
 & (0.2092)
 & (0.0289)
 & (1.1549)
 & (0.0448)
 &     \\
\hline
\end{tabular}}
\caption{
Euler-QMLE results.
True values: $\alpha=1.0$, $(\lambda,\mu,\sigma,\beta)=(1,2,5,0.5)$, $T=1$.
Numbers in parentheses denote standard deviations.
}
\label{tab:euler-only-1.0-T1-time}
\end{table}

\begin{table}[H]
\centering
\scriptsize
\setlength{\tabcolsep}{4pt}
\resizebox{\linewidth}{!}{
\begin{tabular}{c|c|c|c|c|c|c}
\hline
$n$
 & $\hat\lambda_n$
 & $\hat\mu_n$
 & $\hat\alpha_n$
 & $\hat\sigma_n$
 & $\hat\beta_n$
 & Time (s) \\
\hline
500
 & 1.0291
 & 1.7402
 & 1.0071
 & 5.1814
 & 0.4987
 & 21.79 \\
 & (0.2563)
 & (26.7190)
 & (0.0544)
 & (1.7035)
 & (0.0756)
 &     \\
\hline
1000
 & 1.0190
 & 1.3783
 & 1.0091
 & 5.1215
 & 0.4990
 & 51.57 \\
 & (0.2014)
 & (24.0290)
 & (0.0389)
 & (1.4539)
 & (0.0564)
 &     \\
\hline
2000
 & 1.0063
 & -3.1829
 & 1.0088
 & 5.0974
 & 0.4975
 & 100.12 \\
 & (0.1574)
 & (133.0130)
 & (0.0280)
 & (1.0914)
 & (0.0481)
 &     \\
\hline
\end{tabular}}
\caption{
Euler-QMLE results.
True values: $\alpha=1.01$, $(\lambda,\mu,\sigma,\beta)=(1,2,5,0.5)$, $T=1$.
Numbers in parentheses denote standard deviations.
}
\label{tab:euler-only-1.01-T1-time-case2}
\end{table}

\begin{table}[H]
\centering
\scriptsize
\setlength{\tabcolsep}{4pt}
\resizebox{\linewidth}{!}{
\begin{tabular}{c|c|c|c|c|c|c}
\hline
$n$
 & $\hat\lambda_n$
 & $\hat\mu_n$
 & $\hat\alpha_n$
 & $\hat\sigma_n$
 & $\hat\beta_n$
 & Time (s) \\
\hline
500
 & 1.1655
 & 2.2354
 & 1.5013
 & 5.0233
 & 0.5037
 & 47.44 \\
 & (0.4842)
 & (0.8531)
 & (0.0690)
 & (1.0114)
 & (0.1202)
 &     \\
\hline
1000
 & 1.1061
 & 2.1418
 & 1.5022
 & 5.0055
 & 0.5028
 & 45.99 \\
 & (0.3991)
 & (0.7769)
 & (0.0492)
 & (0.8098)
 & (0.0902)
 &     \\
\hline
2000
 & 1.0446
 & 2.0343
 & 1.5002
 & 5.0223
 & 0.5002
 & 110.51 \\
 & (0.3207)
 & (0.6998)
 & (0.0352)
 & (0.6666)
 & (0.0630)
 &     \\
\hline
\end{tabular}}
\caption{
Euler-QMLE results.
True values: $\alpha=1.5$, $(\lambda,\mu,\sigma,\beta)=(1,2,5,0.5)$, $T=1$.
Numbers in parentheses denote standard deviations.
}
\label{tab:euler-only-1.5-T1-time}
\end{table}

\begin{table}[H]
\centering
\scriptsize
\setlength{\tabcolsep}{4pt}
\resizebox{\linewidth}{!}{
\begin{tabular}{c|c|c|c|c|c|c}
\hline
$n$
 & $\hat\lambda_n$
 & $\hat\mu_n$
 & $\hat\alpha_n$
 & $\hat\sigma_n$
 & $\hat\beta_n$
 & Time (s) \\
\hline
500
 & 1.3679
 & 2.5346
 & 1.7908
 & 5.0233
 & 0.4743
 & 57.43 \\
 & (0.7387)
 & (1.3331)
 & (0.0594)
 & (1.0114)
 & (0.2393)
 &     \\
\hline
1000
 & 1.3119
 & 2.4109
 & 1.7976
 & 5.0055
 & 0.4920
 & 48.43 \\
 & (0.7232)
 & (1.2824)
 & (0.0431)
 & (0.8098)
 & (0.1780)
 &     \\
\hline
2000
 & 1.1912
 & 2.2093
 & 1.7997
 & 5.0223
 & 0.4969
 & 111.08 \\
 & (0.6568)
 & (1.2636)
 & (0.0318)
 & (0.6666)
 & (0.1280)
 &     \\
\hline
\end{tabular}}
\caption{
Euler-QMLE results.
True values: $\alpha=1.8$, $(\lambda,\mu,\sigma,\beta)=(1,2,5,0.5)$, $T=1$.
Numbers in parentheses denote standard deviations.
}
\label{tab:euler-only-1.8-T1-time-case2}
\end{table}


\AllHistFigure{1.0}
\AllHistFigure{1.01}
\AllHistFigure{1.5}
\AllHistFigure{1.8}



\bigskip

\noindent
\textbf{Acknowledgements.} 
The first author (EK) would like to thank JGMI of Kyushu University for their support. This work was partially supported by JST CREST Grant Number JPMJCR2115 and JSPS KAKENHI Grant Numbers 23K22410 and 24K21516, Japan (HM).

\bigskip


\appendix
\section{Proof of Theorem \ref{thm:main_result}}
\label{sec:main_result_proof}

\subsection{Outline}
\label{hm:sec_euler.outline}

To briefly capture the flow of the proof before going into details, we outline the strategy for deriving the local asymptotics for the log-likelihood function $\ell_n(\theta)$. The proof is based on a series of error estimates and approximations of Riemann integrals combined with the general results presented in Sweeting \cite[Theorems 1 and 2]{Swe80} and \cite{Swe83}.

Specifically, we will complete the proof of Theorem \ref{thm:main_result} through verifying the following key steps:
\begin{enumerate}[label=(\alph*)]
    \item Stochastic expansion of the normalized Euler-QMLE $\tes$;
    \item Estimating the gap between the quasi-log-likelihoods $\mbbh_n(\theta)$ and the genuine log-likelihoods $\ell_n(\theta)$;
    \item Joint convergence of the quai-score vector and quasi-observed-information matrix associated with $\mbbh_n(\theta)$;
    \item Asymptotically maximal concentration of the good local maximum $\tes$.
\end{enumerate}
In the rest of this subsection, we briefly describe the points of the above four steps.

\medskip

\subsubsection{Stochastic expansion of a normalized Euler-QMLE}~
\label{hm:sec_proof_se}

In this step, we will verify the conditions C1 and C2 in \cite{Swe80} for the Euler-type quasi-likelihood $\mbbh_n(\theta)$ with the normalizing matrix $\vp_n(\theta)$ given by \eqref{hm:def_vp}. 
We need to introduce some notation. Let 
\begin{align}
        \D_n(\theta) &\coloneqq \vp_n(\theta)^\top \p_\theta \mbbh_n(\theta), \nn\\
        \mci_n(\theta) &\coloneqq -\vp_n(\theta)^\top \p_\theta^2 \mbbh_n(\theta) \vp_n(\theta).
\end{align}
Let, for $c>0$ and $\theta\in\Theta$,
\begin{align}
\mathfrak{R}_n(c;\theta) &\coloneqq \big\{\theta' \in \Theta : |\varphi_n(\theta)^{-1}(\theta' - \theta)| \leq c\big\}.
\label{hm:R.set_def}
\end{align}
Given a matrix-valued function $\mathcal{G}=(\mathcal{G}_{kl})_{k,l} : \Theta \to \mbbr^{5} \otimes \mbbr^{5}$, we will write
\begin{equation}
    \mathcal{G}(\theta^{(1)},\dots,\theta^{(5)}) = \left[\mathcal{G}_{kl}(\theta^{(k)})\right]_{k,l=1}^{5}
\end{equation}
for $\theta^{(1)},\dots,\theta^{(5)}\in\Theta$; that is, its $k$th-row entries are evaluated at $\theta^{(k)}$. 
Let $I_k$ denote the $k \times k$-identity matrix.
Then, verification of the conditions C1 and C2 amounts to showing that
\begin{equation}\label{condi:obs_conv}
-\varphi_n(\theta)^{\top} \p_{\theta}^2 \mbbh_n(\theta) \varphi_n(\theta) \cip_u \mci(\theta)
\end{equation}
for some $P_{\theta}$-a.s. positive definite matrix $\mci(\theta)$, and that, for any $c>0$,
\begin{equation}
  \sup_{\theta' \in \mathfrak{R}_n(c;\theta)} \left|\varphi_n(\theta')^{-1} \varphi_n(\theta) - I_5\right| \rightarrow_{u} 0,
\label{condi:rate_mat}  
\end{equation}
\begin{equation}
\sup_{\theta^{(1)},\dots,\theta^{(5)}\in \mathfrak{R}_n(c;\theta)} \left|\varphi_n(\theta)^{\top} \left( \p_{\theta}^2\mbbh_n(\theta^{(1)},\dots,\theta^{(5)}) - \p_{\theta}^2 \mbbh_n(\theta) \right) \varphi_n(\theta)\right| \cip_{u}0.
\label{condi:obs_conti}    
\end{equation}
Under these convergences, by \cite[Theorem 2]{Swe80}, with $\pr$-probability tending to $1$, there exists a local maxima $\tes$ of $\mbbh_n(\theta)$ for which
    \begin{equation}\label{hm:outline-1}
        \vp_n(\theta)^{-1} (\tes-\theta) = \mci_n(\theta)^{-1} \D_n(\theta) + o_{u,p}(1).
    \end{equation} 
At this stage, we have not specified the asymptotic distribution of $(\mci_n(\theta),\D_n(\theta))$, hence that of $\vp_n(\theta)^{-1} (\tes-\theta)$ either.

\begin{rem}
    On the one hand, we can apply \cite[Theorem 2]{Swe80} even when $\mbbh_n(\theta)$ is not an exact log-likelihood as in the present case. On the other hand, we cannot apply \cite[Theorem 1]{Swe80}, which concludes the joint convergence \eqref{hm:joint.conv_Del-I}, since the proof of the theorem essentially requires that $\mbbh_n(\theta)$ is the genuine log-likelihood.
    It is worth mentioning that we could use an existing stable central limit theorem for $\D_n(\theta)$ to identify its asymptotic mixed-normal distribution; we refer to \cite{Jac97} for details.
\end{rem}

\medskip

\subsubsection{Estimating gap between quasi and true likelihoods}~


We define the gap between $\ell_n(\theta)$ and $\mbbh_n(\theta)$ as
\begin{equation}
    \del_n(\theta) := \ell_n(\theta) - \mbbh_n(\theta).
\end{equation}
We will show that the Hessian of $\del_n(\theta)$ is asymptotically negligible uniformly in $\theta$ in the following sense: for each $c>0$,
\begin{equation}
    \left| \vp_n(\theta)^\top \p_\theta^2 \del_n(\theta) \vp_n(\theta) \right| \cip_u 0,
\label{condi:gap_conv1}
\end{equation}
\begin{equation}
    \sup_{\theta^{(1)},\dots,\theta^{(5)}\in \mathfrak{R}_n(c;\theta)} \left| \vp_n(\theta)^\top \left( \p_\theta^2 \del_n(\theta^{(1)},\dots,\theta^{(5)}) - \p_\theta^2 \del_n(\theta)\right) \vp_n(\theta) \right| \cip_u 0.
    \label{condi:gap_conv2}
\end{equation}
In particular, it follows that
\begin{equation}\label{hm:outline-6}
\mci_n(\theta) = \mci_n^\star(\theta) + o_{u,p}(1).
\end{equation}

\medskip

\subsubsection{Joint convergence of score and observed information}~
    
    The uniform asymptotic negligibility \eqref{condi:gap_conv1} and \eqref{condi:gap_conv2} imply that C1 and C2 in \cite{Swe80} hold also for the exact log-likelihood $\ell_n(\theta)$. 
    Then, by \cite[Theorems 1 and 2]{Swe80},
    \begin{align}
        &\vp_n(\theta)^{-1} (\tes^\star-\theta) = \mci_n^\star(\theta)^{-1} \D_n^\star(\theta) + o_{u,p}(1),
        \label{hm:outline-4}\\
        &\left(\mci_n^\star(\theta), \D_n^\star(\theta)\right) 
        \cil_u \big(\mci(\theta), \mci(\theta)^{1/2}\eta\big)
    \end{align}
    for some local maximum $\tes^\star$ of $\ell_n(\theta)$, 
    where $\eta \sim N_5(0,I_5)$ is a random variable independent of $\mci(\theta)$ and defined on an extended filtered probability space of the original one; in particular, we have
    \begin{equation}
        \vp_n(\theta)^{-1} (\tes^\star-\theta) \cil_u \mci(\theta)^{-1/2}\eta
        \overset{\pr}{\sim} MN\left(0,\,\mci(\theta)^{-1}\right),
    \end{equation}
    with $\overset{\pr}{\sim}$ denoting the distributional equivalence under $\pr$. 
    In addition, we will show that
    \begin{equation}\label{condi:gap_partial}
    \left| \vp_n(\theta)^\top \p_\theta \del_n(\theta) \right| 
    = \left| \D_n(\theta) - \D_n^\star(\theta)\right| \cip_u 0.
    \end{equation}
    With this, together with \eqref{hm:outline-6}, we can deduce
    \begin{equation}\label{hm:outline-5}
        \left(\mci_n(\theta), \D_n(\theta)\right) \cil_u \big(\mci(\theta), \mci(\theta)^{1/2}\eta\big).
    \end{equation}
    Combining \eqref{hm:outline-1} and \eqref{hm:outline-5} gives
    \begin{equation}
        \hat{\mci}_n^{1/2}\vp_n(\theta)^{-1} (\tes-\theta) \cil_u N_5(0,I_5)
    \end{equation}
    for any statistics $\hat{\mci}_n$ satisfying that $\hat{\mci}_n \cip_u \mci(\theta)$.   

\medskip

\subsubsection{Asymptotically maximal concentration of good local maximum $\tes$}~
\label{hm:sec_proof_amc}

Having verified C1 and C2 in \cite{Swe80}, the claim automatically follows from \cite[Theorem 2.1]{Swe83}.

\medskip

Building on the observations in Section \ref{hm:sec_proof_se} to \ref{hm:sec_proof_amc}, 
the proof of Theorem \ref{thm:main_result} is complete if we prove \eqref{condi:obs_conv}, \eqref{condi:rate_mat}, \eqref{condi:obs_conti}, \eqref{condi:gap_conv1}, \eqref{condi:gap_conv2}, and \eqref{condi:gap_partial}.

\subsection{Convergence and positive definiteness of observed information}
\label{hm:conv.oim}

For convenience, we introduce the following abbreviations.
Let $\phi \coloneqq \phi_{\alpha,\beta}$ and, for $y\in\mathbb{R}$,
\begin{align}
& \phi_{(a)}(y) \coloneqq \partial_a \phi(y),\qquad
  \psi(y) \coloneqq \frac{\phi_{(y)}(y)}{\phi(y)}, \qquad f(y) \coloneqq \frac{\phi_{(\alpha)}(y)}{\phi(y)},\\
& g(y) \coloneqq \frac{\phi_{(\beta)}(y)}{\phi(y)},\qquad \zeta'(y) \coloneqq \partial_y \zeta(y) \quad\text{for} \quad \zeta =f,g,\psi.
\end{align}
Recall the notation $\ep_j(\theta)$ given in \eqref{hm:def_E.ep}, the residual term constructed from the Euler approximation. For each $j$, let
\begin{equation}
\ep_j \coloneqq \ep_j(\theta),\qquad
\ep_{(a),j} \coloneqq \partial_a \ep_j(\theta),\qquad
\ep_{(a,b),j} \coloneqq \partial_a\partial_b \ep_j(\theta),
\end{equation}
with $a,b\in\{\lambda,\mu,\alpha,\sigma,\beta\}$.
For any $\zeta\in\{f,g,\psi\}$, let
\begin{equation}
\zeta_j \coloneqq \zeta(\ep_j),\qquad
\zeta'_j \coloneqq \partial_y \zeta(\ep_j),\qquad
\zeta_{(a),j} \coloneqq \partial_a \zeta(\ep_j(\theta)),\qquad
\zeta_{(a,b),j} \coloneqq \partial_a\partial_b \zeta(\ep_j(\theta)).
\end{equation}

Below, Sections \ref{sec:part_deri} and \ref{sec:empirical_func} provide some preliminaries and technical lemmas, which will be used in the subsequent Section \ref{sec:conv_obs}. 

\subsubsection{Partial derivatives of $\mbbh_n(\theta)$}\label{sec:part_deri}

In this section, we summarize the explicit forms of the first- and second-order partial derivatives of $\mbbh_n(\theta)$ to be used in the subsequent discussion.

The first-order derivatives of $\mbbh_n(\theta)$ are given as follows:
\begin{align}
\p_{\lam} \mbbh_n(\theta)&=\sumj \psi_j\ep_{(\lam),j},\\
\p_{\mu} \mbbh_n(\theta)&=\sumj \psi_j\ep_{(\mu),j},\\
\p_{\al} \mbbh_n(\theta)&=\sumj ( -\al^{-2}\overline{l} + f_j + \psi_j \ep_{(\al),j}),\\
\p_{\sig} \mbbh_n(\theta)&=\sumj ( -\sig^{-1} + \psi_j \ep_{(\sig),j}),\\
\p_{\be} \mbbh_n(\theta)&=\sumj ( g_j + \psi_j \ep_{(\be),j}),
\end{align}
where
\begin{align}\label{ep_1dev}
\ep_{(\lam),j}&= \sig^{-1} h^{1-1/\al}Y_{t_{j-1}},\quad \ep_{(\mu),j}= -\sig^{-1} h^{1-1/\al},\quad \ep_{(\be),j}=-\xi_h(\al),\\
\ep_{(\al),j}&= \p_{\al} \left( \frac{\Delta_j Y -(-\lam Y_{t_{j-1}} + \mu)h -\be \sig (h^{1/\al} - h)t_{\al}}{\sig h^{1/\al}}\right)\\
&= \frac{\Delta_jY - (-\lam Y_{t_{j-1}} -\mu)h}{\sig}(\p_{\al}h^{-1/\al}) - \be \p_{\al} \left ( t_{\al}(1-h^{1-1/\al}) \right )\\
&= (\ep_j(\theta) + \be \xi_h(\al)) h^{1/\al} (\p_{\al} h^{-1/\al}) - \be (\p_{\al}\xi_h(\al))\\
&= -\al^{-2} \overline{l} (\ep_j(\theta) + \be \xi_h(\al)) - \be (\p_{\al} \xi_h(\al)),\\
\ep_{(\sig),j}&=-\frac{\Delta_jY - (-\lam Y_{t_{j-1}})h}{\sig^2 h^{1/\al}}\\
&= -\sig^{-1} ( \ep_j(\theta) + \be \xi(\al)).
\end{align}

Next, we note the following expressions:
\begin{align}
\ep_{(\lam,\lam),j}&= \ep_{(\lam,\mu),j} =  \ep_{(\lam,\be),j}=\ep_{(\mu,\mu),j}= \ep_{(\mu,\be),j}=\ep_{(\sig,\be)} =\ep_{(\be,\be),j}=0,\\
\ep_{(\lam,\al),j} &= \p_{\al} (\sig^{-1} h^{1-1/\al} Y_{t_{j-1}})= \sig^{-1} Y_{t_{j-1}} (\p_{\al}h^{1-1/\al})=-\al^{-2}\sig^{-1} Y_{t_{j-1}}h^{1-1/\al} \overline{l},\\
\ep_{(\lam,\sig),j} &= \p_{\sig}(\sig^{-1}h^{1-1/\al}Y_{t_{j-1}}) = -\sig^{-2}Y_{t_{j-1}}h^{1-1/\al},\\
\ep_{(\mu,\al),j}&= -\sig^{-1} (\p_{\al} h^{1-1/\al}) = \al^{-2} \sig^{-1} h^{1-1/\al} \overline{l},\\
\ep_{(\mu,\sig),j} &= -h^{1-1/\al} (\p_{\sig} \sig^{-1})= h^{1-1/\al} \sig^{-2},
\nn\\
\ep_{(\al,\al),j}
&= 2\al^{-3} \overline{l} (\ep_j + \be \xi_h(\al)) + \al^{-4} \overline{l}^2 (\ep_j + \be \xi_h(\al)) -\be \p_{\al}^2 \xi_h(\al),
\nn\\
\ep_{(\al,\sig),j}
&= \al^{-2} \sig^{-1} \overline{l} (\ep_j + \be \xi_h(\al)),\\
\ep_{(\al,\be),j}&= -\p_{\al} \left\{ (1-h^{1-1/\al}) t_{\al} \right \}= -\p_{\al}\xi_h(\al),\\
\ep_{(\sig,\sig),j}
&= 2 \sig^{-2}(\ep_j + \be \xi_h(\al)).
\end{align}
The second-order derivatives are given as follows:
\begin{align}
\p_{\lam}^2 \mbbh_n(\theta) &=  \sumj \ep_{(\lam),j}^2 \psi'_j,\label{deriv:lam_lam}\\
\p_{\mu}^2 \mbbh_n(\theta) &= \sumj \ep_{(\mu),j}^2 \psi'_j,\\
\p_{\al}^2 \mbbh_n(\theta) &= \sumj (2\al^{-3}\overline{l} + f_{(\al),j} + f'_j \ep_{(\al),j} + ((\p_{\al}\psi_j) + \psi'_j\ep_{(\al),j})\ep_{(\al),j}+ \psi_j \ep_{(\al,\al),j}),\\
\p_{\sig}^2 \mbbh_n(\theta) &= \sumj ( \sig^{-2} + \psi'_j\ep_{(\sig),j}^2 + \psi_j \ep_{(\sig,\sig)}),\\
\p_{\be}^2 \mbbh_n(\theta) &= \sumj ( g_{(\be),j} + g'_j\ep_{(\be),j} + (\psi_{(\be),j} + \psi'_j\ep_{(\be),j})\ep_{(\be),j} + \psi_j \ep_{(\be,\be),j} ),\\
\p_{\lam} \p_{\mu} \mbbh_n(\theta) &= \sumj \psi'_j \ep_{(\lam),j} \ep_{(\mu),j},\\
\p_{\lam} \p_{\al} \mbbh_n(\theta) &= \sumj \{ (\psi_{(\al),j} + \psi'_j\ep_{(\al),j})h^{1-1/\al} + \psi_j (\p_{\al}h^{1-1/\al}) \}\sig^{-1}Y_{t_{j-1}},\\
\p_{\lam} \p_{\sig} \mbbh_n(\theta) &= \sumj ( \psi'_j\ep_{(\lam),j} \ep_{(\sig),j} + \psi_j \ep_{(\lam,\sig),j}),\\
\p_{\lam} \p_{\be} \mbbh_n(\theta) &= \sumj \ep_{(\lam),j}( \psi_{(\be),j} + \psi'_j\ep_{(\be),j} ),\\
\p_{\mu} \p_{\al} \mbbh_n(\theta) &= \sumj \{ (\psi_{(\al),j} + \psi'_j \ep_{(\al),j})\ep_{(\mu),j} + \psi_j \ep_{(\mu,\al),j}\},\\
\p_{\mu} \p_{\sig} \mbbh_n(\theta) &= \sumj ( \psi'_j\ep_{(\mu),j} \ep_{(\sig),j} + \psi_j \ep_{(\mu,\sig),j}),\\
\p_{\mu} \p_{\be} \mbbh_n(\theta) &= \sumj \ep_{(\mu),j} (\psi_{(\be),j} + \psi'_j \ep_{(\be),j}),\\
\p_{\al} \p_{\sig} \mbbh_n(\theta) &= \sumj ( f'_j\ep_{(\sig),j} + \psi'_j\ep_{(\al),j}\ep_{(\sig),j} + \psi_j \ep_{(\al,\sig),j} ),\\
\p_{\al} \p_{\be} \mbbh_n(\theta) &= \sumj \{ (g'_j +\psi'_j \ep_{(\be),j})\ep_{(\al),j} + g_{(\al),j} + \psi_{(\al),j}\ep_{(\be),j} + \psi_j \ep_{(\al,\be),j} \},\\
\p_{\sig} \p_{\be} \mbbh_n(\theta) &= \sumj ( g'_j\ep_{(\sig),j} + \psi'_j \ep_{(\sig),j} \ep_{(\be),j} + \psi_j\ep_{(\sig,\be)}).
\end{align}

\subsubsection{Localization}
\label{hm:sec_localization}

We will use the localization to estimate the error terms.
Denote by $\mu_{\al,\beta}(dt,dz)$ the Poisson random measure associated with $J$ under $\pr$. 
Recall \eqref{hm:Theta_closure}. For any $\ve>0$, there is a constant $M=M(\ve)>0$ for which the event
\begin{equation}
    \Omega_{\ve,M}(\al,\beta)\coloneqq\left\{
    \omega\in\Omega :\, 
    \mu_{\al,\beta}\left((0,T],\{z\in\mbbr\setminus\{0\}:\,|z|>M\}\right)=0
    \right\} \in \mcf
\end{equation}
satisfies that $\inf_\theta\pr[\Omega_{\ve,M}(\al,\beta)]>1-\ep$.
On $\Omega_{\ve,M}(\al,\beta)$, $J$ equals the process with bounded jumps:
\begin{equation}
 J_t(M)\coloneqq J_t-\sum_{0<s\le t}\D J_s I(|\D J_s|>M).
\end{equation}
Then, we may and do suppose that
\begin{equation}
    \E\left[\sup_{t\le T}|J_t|^K\right]\lesssim_u 1 
    \quad\text{and}\quad \E\left[|J_h|^2\right]\lesssim_u h
\end{equation}
for every $K>0$.
Here and in what follows, the notation ``$|f_n(\theta)|\lesssim_u 1$'' for random function $f_n(\theta)$ means that ``$\sup_\theta |f_n(\theta)|\lesssim 1$'', where $a_n \lesssim b_n$ denotes that $a_n \le C b_n$ a.s. for a universal positive constant $C$. 
A standard application of the Gronwall inequality leads to the estimate $\E\left[\sup_{t\le T}|Y_t|^K\right]\lesssim_u 1$ for any $K>0$, in particular $\sup_{t\le T}|Y_t|=O_p(1)$.
We also have for $t\in[t_{j-1},t_j]$,
\begin{align}
    \E\left[|Y_t-Y_{t_{j-1}}|^2\right]
    &\lesssim h^2 + \E\left[|(e^{-\lam h}-1)Y_{t_{j-1}}|^2\right]
    + \E\left[\left| \int_{t_{j-1}}^{t_j}e^{-\lam(t_j-s)}\sig dJ_s \right|^2\right]
    \nn\\
    &\lesssim_u h^2 + h^2 + h \lesssim h.
\end{align}
We will apply the above properties to estimate the stochastic orders of the error (residual) terms.

\subsubsection{Asymptotic behavior of empirical functional}
\label{sec:empirical_func}

To establish the convergence of $\mci_n(\theta)$, we need to study the asymptotic behavior of
\begin{equation}\label{conv:info}
\frac{1}{n} \sum_{j=1}^n F(\ep_j(\theta)) Y_{t_{j-1}}^k,
\qquad k=0,1,2,
\end{equation}
where $F$ is a sufficiently smooth function such that for each $l\ge 1$ and $K>0$,
\begin{equation}\label{hm:F_log.growth}
    \limsup_{|y|\to\infty} \left(\frac{|F(y)|}{1+(\log(1+|y|))^K} + |\p_y^l F(y)| \right) < \infty.
\end{equation}
It will be given in Corollary \ref{hm:cor_aux.cip}.

Observe that (recall \eqref{hm:def_ep0} for the definition of $\ep_{0,j}(\theta)$)
\begin{align}
\ep_j(\theta) - \ep_{0,j}(\theta)
= - h^{-1/\al} \sig^{-1} \lambda \int_{t_{j-1}}^{t_{j}} (Y_s - Y_{t_{j-1}})ds.
\end{align}
Let
\begin{align}
\qquad \overline{Y^k} \coloneqq \frac{1}{T} \int_{0}^{T} Y_t^{k} dt
\end{align}
for $k\ge 0$. 

\begin{lem}\label{lem:conv_LLN}
Under \eqref{hm:F_log.growth}, for any $k\ge 0$,
\begin{equation}
\frac{1}{n} \sumj F(\ep_{0,j}(\theta)) Y_{t_{j-1}}^k \cip_u \E[F(\ep_{0,1}(\theta))] \, 
\overline{Y^k}.
\end{equation}
\end{lem}

\begin{proof}
We only need to consider $k\ge 1$, for the case of $k=0$ is trivial.
Let $\chi_j(\theta) \coloneqq (F(\ep_{0,j}(\theta))- \E[F(\ep_{0,j}(\theta))])Y_{t_{j-1}}^k$, which defines a martingale array with respect to $(\mcf_{t_j})_{j=0}^{n}$ for each $n$. 
Then,
\begin{align}
    \frac{1}{n} \sumj F(\ep_{0,j}(\theta)) Y_{t_{j-1}}^k
    &= \frac{1}{n} \sumj \chi_j(\theta) + 
    E_\theta[F(\ep_{0,1}(\theta))]\, \frac{1}{n} \sumj Y_{t_{j-1}}^k
\end{align}
Under the localization (Section \ref{hm:sec_localization}), the first term on the right-hand side is $O_{u,p}(n^{-1/2})$ thanks to the Burkholder inequality, hence negligible.
It remains to show that
\begin{equation}
\frac{1}{n} \sumj Y_{t_{j-1}}^k \cip_u \overline{Y^k}.
\end{equation}

For each $\omega\in\Omega$, $t\mapsto Y_t(\omega)^{k-1}$ is {\cadlag} and hence Riemann integrable. 
We have
\begin{align}
\left|\frac{1}{n} \sumj Y_{t_{j-1}}^k - \overline{Y^k} \right|
&\leq \frac{1}{n} \sumj \frac1h \int_{t_{j-1}}^{t_j} |Y_{t_{j-1}}^k - Y_t^k| dt\\
&\lesssim \frac{1}{n} \sumj \frac1h \int_{t_{j-1}}^{t_j} \left( 1 + |Y_t| + |Y_{t_{j-1}}|\right)^C |Y_t - Y_{t_{j-1}}|dt.
\end{align}
Taking the expectation and then using the Cauchy-Schwarz inequality under the localization, the rightmost side in the last display turns out to be $o_{u,p}(1)$. 
The proof is complete.
\end{proof}

\begin{cor}\label{hm:cor_aux.cip}
Under \eqref{hm:F_log.growth}, for any $k>0$,
\begin{equation}
\frac{1}{n} \sumj F(\ep_{j}(\theta)) Y_{t_{j-1}}^k \cip_u \E[F(\ep_{0,1}(\theta))] \, 
\overline{Y^k}.
\end{equation}
\end{cor}

\begin{proof}
By the expression \eqref{eq:ssOU} and since $\p_\ep F(\ep)$ is bounded, we have
\begin{align}
    & \left| \frac{1}{n} \sumj F(\ep_{j}(\theta)) Y_{t_{j-1}}^k - \frac{1}{n} \sumj F(\ep_{0,j}(\theta)) Y_{t_{j-1}}^k \right| \nn\\
    &\lesssim \frac1n \sumj |Y_{t_{j-1}}|^k h^{-1/\al} \int_{t_{j-1}}^{t_j} |Y_s - Y_{t_{j-1}}| ds \nn\\
    &\lesssim \frac1n \sumj |Y_{t_{j-1}}|^k \frac1h \int_{t_{j-1}}^{t_j} 
    \left( h^{2-1/\al}(1+|Y_{t_{j-1}}|) + h\,
    \left|h^{-1/\al}\int_{t_{j-1}}^s e^{-\lam (t_j-s)}dJ_s\right|
    \right) ds
\end{align}
Through the localization argument, the rightmost side turns out to equal
\begin{equation}
    O_{u,p}(h^{2-1/\al} \vee h)=O_{u,p}(h)=o_{u,p}(1).
\end{equation}
This, along with Lemmas \ref{lem:conv_LLN}, establishes the claim.
\end{proof}

\subsubsection{Convergence}
\label{sec:conv_obs}

The objective of this section is to prove \eqref{condi:obs_conv}. Let us recall the order of the components of the parameter $\theta = (\lam,\mu,\al,\sig, \be)$.
Write $\p_{\theta}^2 \mbbh_n(\theta)$ as
\begin{align}
\p_{\theta}^2 \mbbh_n(\theta) 
&= \begin{pmatrix}
H_{11,n}(\theta)& H_{12,n}(\theta) \\
H_{12,n}^{\top}(\theta) & H_{22,n}(\theta)
\end{pmatrix},
\end{align}
where $H_{11,n}(\theta) \in \mbbr^2 \otimes \mbbr^2, H_{12,n}(\theta) \in \mbbr^2 \otimes \mbbr^3$, and $H_{22,n}(\theta) \in \mbbr^3 \otimes \mbbr^3$. 
Recalling the form \eqref{hm:def_vp} of $\vp_n(\theta)$, we have
\begin{align}
\mathcal{I}_n(\theta)
&= -\begin{pmatrix}
r_n^{-2}H_{11,n}(\theta) & \text{Sym.} \\
\ds{\frac{1}{\sqrt{n}r_n \xi_h(\al)} \tilde{\varphi}_n(\theta)^{\top} H_{12,n}^{\top}(\theta)} & \ds{\frac{1}{n \xi_h(\al)^2} \tilde{\varphi}_n(\theta)^{\top} H_{22,n}(\theta)\tilde{\varphi}_n(\theta)}
\end{pmatrix}.
\end{align}
Below, we will demonstrate convergence and positive definiteness of the limit.


Let us recall the notation $\xi_h(\al) = (1-h^{1-1/\al})t_{\al}$. To handle the second derivatives of $\mbbh_n(\theta)$, we prove the following lemma.

\begin{lem}\label{lem:xi_asymp}
The following properties hold for $h\to 0$ (The primes therein stand for the partial derivatives with respect to $\al$).
\begin{enumerate}
  \item Asymptotic behavior of $\xi_h(\al)$:
  \begin{equation}
  \xi_h(\alpha) =
  \begin{cases}
    t_{\al} + O(h^{1 - 1/\alpha}) & (\alpha > 1) \\
    -h^{1 - 1/\alpha} t_{\al} + O(1) & (\alpha < 1) \\
    -\dfrac{2}{\pi} \, \overline{l} & (\alpha = 1)
  \end{cases}
  \end{equation}
  hence
  \begin{equation}
  \lim_{h \rightarrow 0} \xi_h(\alpha) =
  \begin{cases}
    t_{\al} & (\alpha > 1) \\
    -\infty & (\alpha \leq 1)
  \end{cases}
  \end{equation}

  \item First-order derivative:
\begin{equation}
  \partial_\alpha \xi_h(\alpha) =
  \begin{cases}
    t_{\al}' + O(h^{1 - 1/\alpha} \, \overline{l}) & (\alpha > 1) \\
    h^{1 - 1/\alpha} \, \overline{l} \, \alpha^{-2} t_{\al} + O(h^{1 - 1/\alpha}) & (\alpha < 1) \\
    \infty & (\alpha = 1)
  \end{cases}
  \end{equation}
  hence
  \begin{equation}
  \lim_{h \rightarrow 0} \partial_\alpha \xi_h(\alpha) =
  \begin{cases}
    \dfrac{\pi}{2} \left( \cos \dfrac{\alpha \pi}{2} \right)^{-2} & (\alpha > 1) \\
    \infty & (\alpha \leq 1)
  \end{cases}
  \end{equation}
  
  \item Logarithmic derivative:
  \begin{equation}
  \lim_{h \rightarrow 0} \frac{\partial_\alpha \xi_h(\alpha)}{\xi_h(\alpha)} =
  \begin{cases}
    \dfrac{\pi}{2} \left( \cos \dfrac{\alpha \pi}{2} \right)^{-2}  t_{\al}^{-1}, & (\alpha > 1) \\
    -\infty. & (\alpha < 1)
  \end{cases}
  \end{equation}
  \item Second derivative:
  \begin{equation}
  \partial_\alpha^2 \xi_h(\alpha) =
  \begin{cases}
     t_{\al}'' + O(h^{1 - 1/\alpha} \, \overline{l}^2), & (\alpha > 1) \\
    h^{1 - 1/\alpha} \, \overline{l}^2 \, \alpha^{-4} t_{\al} + O(h^{1 - 1/\alpha} \overline{l}), & (\alpha < 1) \\
    \infty, & (\alpha = 1)
  \end{cases}
  \end{equation}
  and
  \begin{equation}
  \lim_{h \rightarrow 0} \p_\alpha^2 \xi_h(\alpha) =
  \begin{cases}
    \dfrac{\pi^2}{2} t_{\al} \left( \cos \dfrac{\alpha \pi}{2} \right)^{-2}, & (\alpha > 1) \\
    \infty. & (\alpha \leq 1)
  \end{cases}
  \end{equation}
\end{enumerate}
\end{lem}

\begin{proof}
Each expansion follows from elementary calculus 
together with the Taylor expansions in $h^{1 - 1/\alpha}$. 
\end{proof}


Let $\zeta_1 \coloneqq \zeta(\ep_{0,1}(\theta))$ for $\zeta = \psi,f,g$ and 
\begin{equation}
\overline{Y}_n \coloneqq \frac{1}{n} \sumj Y_{t_{j-1}} 
\qquad 
\text{(Then, $\overline{Y}_n \cip \overline{Y^1}$)}.    
\end{equation}
By Corollary \ref{hm:cor_aux.cip}, 
Lemma \ref{lem:xi_asymp}, and the law of large numbers (for i.i.d. random variables), the following convergences hold:
\begin{align}
R_{1,n} &\coloneq \frac{1}{n \xi_h(\al)} \sumj \left( (\p_{\al} \xi_h(\al)) \be \psi_j - f_j \right)\psi_j\\
& \cip_u \frac{1}{t_{\al}} \E \left [ \left( \frac{\pi \be}{2 \left(\cos \frac{\al \pi}{2}\right)^2}\psi_1 - f_1 \right) \psi_1\right] \eqqcolon R_1, 
\\
R_{2,n} &\coloneq \frac{-1}{n \xi_h(\al)} \sumj (\ep_j(\theta) + \be \xi_h(\al)) \psi_j^2,\\
&\cip_u \frac{-1}{t_{\al}} \E \left[ \left(\ep_{0,1}(\theta) + \be t_{\al} \right) \psi_1^2 \right],\\
&= \frac{-1}{t_{\al}} \E \left[ \left( \be t_{\al} \psi_1 + \ep_{0,1}(\theta) \psi_1 + 1 \right) \psi_1\right] \eqqcolon R_2,
\\
R_{3,n} &\coloneq \frac{1}{n \xi_h(\al)} \sumj ( \xi_h(\al) \psi_j - g_j)\psi_j \cip_u \frac{1}{t_{\al}} \E\left[ \left( t_{\al} \psi_1 - g_1 \right)\psi_1\right] \eqqcolon R_3,
\\
Q_{1,n} &\coloneq \frac{1}{n \xi_h(\al)^2} \sumj (\be (\p_{\al} \xi_h(\al)) \psi_j - f_j)^2\\
&= \frac{1}{t_{\al}^2} \E \left[ \left( \frac{\pi \be}{2 \left( \cos \frac{\al \pi}{2} \right)^2} \psi_1 - f_1 \right)^2  \right] \eqqcolon Q_1,
\\
Q_{2,n} &\coloneq \frac{1}{n \xi_h(\al)^2} \sumj (\be ( \p_{\al} \xi_h(\al)) \psi_j - f_j)( (\ep_j(\theta) + \be \xi_h(\al))\psi_j + 1)\\
&\cip_u \frac{1}{ t_{\al}^2} \E \left[ \left( \frac{\pi \be}{2 \left(\cos \frac{\al \pi}{2}\right)^2}\psi_1 - f_1 \right) \left(\left( \ep_{0,1}(\theta) + \be t_{\al} \right)\psi_1 + 1 \right)   \right] \eqqcolon Q_2,
\\
Q_{3,n} &\coloneq \frac{1}{n \xi_h(\al)^2} \sumj ( (\ep_j(\theta) + \be \xi_h(\al))\psi_j + 1 )^2\\
&\cip_u \frac{1}{t_{\al}^2} \E\left[\left( \left( \ep_{0,1}(\theta) + \be t_{\al}\right) \psi_1  + 1 \right)^2  \right] 
\eqqcolon Q_3,
\\
Q_{4,n} &\coloneq \frac{1}{n \xi_h(\al)} \sumj (\xi_h(\al)\psi_j - g_j)(\be (\p_{\al} \xi_h(\al)) \psi_j - f_j)\\
&\cip_u \E \left[ \left( \left( \frac{\pi \be}{2 \left(\cos \frac{\al \pi}{2}\right)^2 } \right) \psi_1 -f_1 \right) \left(t_{\al}\psi_1 - g_1 \right) \right] 
\eqqcolon Q_4,
\\
Q_{5,n} &\coloneq \frac{1}{n \xi_h(\al)} \sumj ((\ep_j(\theta) + \be \xi_h(\al))\psi_j + 1)( \xi_h(\al) \psi_j - g_j)\\
&\cip_u \frac{1}{t_{\al}^2}\E \left[ \left(\left( \ep_{0,1}(\theta) + \be t_{\al}\right)\psi_1 + 1 \right) \left(t_{\al} \psi_1 - g_1 \right) \right] 
\eqqcolon Q_5,
\\
Q_{6,n} &\coloneq \frac{1}{n \xi_h(\al)^2} \sumj (\xi_h(\al) \psi_j - g_j)^2 \cip_u \frac{1}{t_{\al}^2} \E\left[ \left( t_{\al} \psi_1 - g_1 \right)^2 \right] 
\eqqcolon Q_6. 
\end{align}

\begin{lem}\label{lem:2deriv_conv}
We have the following asymptotic properties.
\begin{align}
-\frac{1}{r_n^2}\p_{\lam}^2 \mbbh_n(\theta) &= \sig^{-2} \frac{1}{n} \sumj Y_{t_{j-1}}^2 \psi_j^2 + O_{u,p}\left( \frac{1}{\sqrt {n}}\right),\\
-\frac{1}{r_n^2} \p_{\lam} \p_{\mu} \mbbh_n(\theta) &= -\sig^{-2} \frac{1}{n} \sumj Y_{t_{j-1}} \psi_j^2 + O_{u,p}\left( \frac{1}{\sqrt {n}}\right),\\
-\frac{1}{r_n^2} \sumj \p_{\mu}^2 \mbbh_n(\theta) &= \sig^{-2} \frac{1}{n} \sumj \psi_j^2 + O_{u,p}\left( \frac{1}{\sqrt {n}}\right),\\
\frac{-1}{\sqrt{n}r_n \xi_h(\al)} \p_{\lam}\p_{\al} \mbbh_n(\theta)
&= -\sig^{-1} R_1 \overline{Y}_n + \al^{-2} \sig^{-1} \overline{l} R_2 \overline{Y}_n + O_{u,p}\left(\frac{\overline{l}}{\sqrt{n}} \right),\\
\frac{-1}{\sqrt{n}r_n \xi_h(\al)} \p_{\al} \p_{\sig} \mbbh_n(\theta) &= \sig^{-2} R_2 \overline{Y}_n + O_{u,p} \left( \frac{1}{\sqrt{n}}\right),\\
\frac{-1}{\sqrt{n} r_n \xi_h(\al)} \p_{\lam} \p_{\be} \mbbh_n(\theta)&= -\sig^{-1} R_3 \overline{Y}_n + O_{u,p}\left(\frac{1}{\sqrt{n}}\right),\\
\frac{-1}{\sqrt{n}r_n \xi_h(\al)} \p_{\mu} \p_{\al} \mbbh_n(\theta)
&= \sig^{-1} R_{1,n} -\al^{-2} \sig^{-1} \overline{l}R_{2,n} + O_{u,p}\left(\frac{\overline{l}}{\sqrt{n}}\right),\\
\frac{-1}{\sqrt{n}r_n \xi_h(\al)} \p_{\mu} \p_{\sig} \mbbh_n(\theta)&= -\sig^{-2} R_{2,n} + O_{u,p}\left( \frac{1}{\sqrt{n}}\right),\\
\frac{-1}{\sqrt{n} r_n \xi_h(\al)} \p_{\mu} \p_{\be} \mbbh_n
&= \sig^{-1} R_{3,n},\\
-\frac{1}{n \xi_h(\al)^2} \p_{\al}^2 \mbbh_n(\theta) &= Q_{1,n} + 2\al^{-2} \overline{l} Q_{2,n} + \al^{-4} \overline{l}^2 Q_{3,n} + O_{u,p} \left( \frac{\overline{l}^2}{\sqrt{n}}\right),\\
-\frac{1}{n \xi_h(\al)^2} \p_{\al} \p_{\sig} \mbbh_n(\theta) 
&= \sig^{-1} Q_{2,n} + \al^{-2} \sig^{-1} Q_{3,n} \overline{l} + O_{u,p}\left( \frac{\overline{l}}{\sqrt{n}}\right),\\
-\frac{1}{n \xi_h(\al)^2} \p_{\sig}^2 \mbbh_n(\theta)
&=\sig^{-2} Q_{3,n} + O_{u,p}\left(\frac{1}{\sqrt{n}}\right),\\
-\frac{1}{n \xi_h(\al)} \p_{\al} \p_{\be} \mbbh_n(\theta)
&= Q_{4,n} + \al^{-2} \overline{l} Q_{5,n} + O_{u,p}\left( \frac{\overline{l}}{\sqrt{n}}\right),\\
-\frac{1}{n \xi_h(\al)^2} \p_{\sig} \p_{\be} \mbbh_n(\theta)
&=  \sig^{-1} Q_{5,n} + O_{u,p}\left(\frac{1}{\sqrt{n}}\right),\\
-\frac{1}{n \xi_h(\al)^2} \p_{\be}^2 \mbbh_n(\theta)
&= Q_{6,n} + O_{u,p}\left(\frac{1}{\sqrt{n}}\right).
\end{align}
\end{lem}

We note that some of the expressions given in Lemma \ref{lem:2deriv_conv} contain some divergent terms (including $\overline{l}$) on the right-hand side, such as $\p_\al^2\mbbh_n(\theta)$. They will be cancelled out in the process of computing the limit of $\mci_n(\theta)$

\begin{proof}[Proof of Lemma \ref{lem:2deriv_conv}.]
From 
Corollary \ref{hm:cor_aux.cip} and Lemma \ref{lem:xi_asymp},
\begin{align}
-\frac{1}{r_n^2}\p_{\lam}^2 \mbbh_n(\theta) &=-\sig^{-2} \frac{1}{n} \sumj Y_{t_{j-1}} \psi'_j
= \sig^{-2} \frac{1}{n} \sumj Y_{t_{j-1}} \psi_j^2 + O_{u,p}\left(\frac{1}{\sqrt{n}}\right),\\
-\frac{1}{r_n^2} \p_{\lam} \p_{\mu} \mbbh_n(\theta) &= \sig^{-2} \frac{1}{n} \sumj Y_{t_{j-1}} \psi'_j
= -\sig^{-2} \frac{1}{n} \sumj Y_{t_{j-1}} \psi_j^2 + O_{u,p}\left(\frac{1}{\sqrt{n}}\right),\\
-\frac{1}{r_n^2} \p_{\mu}^2 \mbbh_n(\theta) &= -\sig^{-2} \frac{1}{n} \sumj \psi'_j
= \sig^{-2} \frac{1}{n} \sumj \psi_j^2 + O_{u,p} \left(\frac{1}{\sqrt{n}}\right),\\
\frac{-1}{\sqrt{n} r_n \xi_h(\al)} \p_{\lam} \p_{\al} \mbbh_n(\theta)&=  \frac{-1}{n \xi_h(\al)} \sumj \bigg[ \sig^{-1} \left( \psi_{(\al),j} - (\p_{\al} \xi_h(\al)) \be \psi'_j \right) Y_{t_{j-1}}\\
&\qquad - \al^{-2} \sig^{-1} \overline{l} (\ep_j(\theta) + \be \xi_h(\al)) \psi'_j Y_{t_{j-1}}\bigg] + O_{u,p} \left( \frac{\overline{l}}{\sqrt{n}} \right),\\
&= \frac{-\sig^{-1}}{n \xi_h(\al)} \sumj ((\p_{\al} \xi_h(\al)) \be \psi_j - f_j)\psi_j Y_{t_{j-1}}\\
&\qquad + \al^{-2} \sig^{-1} \overline{l} \frac{1}{n \xi_h(\al)} \sumj (\ep_j(\theta) + \be \xi_h(\al)) \psi'_j Y_{t_{j-1}} \nn\\
&{}\qquad + O_{u,p}\left( \frac{\overline{l}}{\sqrt{n}}\right)\\
&= \frac{- \sig^{-1}}{n \xi_h(\al)} \sumj (\ep_j(\theta) + \be \xi_h(\al)) \psi_j^2 \overline{Y}_n + O_{u,p}\left(\frac{\overline{l}}{\sqrt{n}}\right),\\
&= -\sig^{-1} R_{1,n} \overline{Y}_n + \al^{-2} \sig^{-1} \overline{l} R_{2,n} \overline{Y}_n + O_{u,p}\left( \frac{\overline{l}}{\sqrt{n}}\right)\\
&= -\sig^{-1} R_1 \overline{Y}_n + \al^{-2} \sig^{-1} \overline{l} R_2 \overline{Y}_n + O_{u,p}\left( \frac{\overline{l}}{\sqrt{n}} \right),\\
\frac{-1}{\sqrt{n} r_n \xi_h(\al)} \p_{\lam} \p_{\sig} \mbbh_n(\theta)&= \frac{1}{n \xi_h(\al)} \sumj \left(\sig^{-2} (\ep_j(\theta) + \be \xi_h(\al))\psi'_j Y_{t_{j-1}}\right) + O_{u,p}\left(\frac{1}{\sqrt{n}}\right)\\
&=\frac{\sig^{-2} h^{1-1/\al}}{\sqrt{n} r_n \xi_h(\al)} \sumj (\ep_j(\theta) + \be \xi_h(\al)) \psi'_j Y_{t_{j-1}} + O_{u,p} \left( \frac{1}{\sqrt{n}}\right)\\
&=  \frac{-\sig^{-2}}{n \xi_h(\al)} \sumj (\ep_j(\theta) + \be \xi_h(\al)) \psi_j^2 Y_{t_{j-1}} + O_{u,p} \left(\frac{1}{\sqrt{n}}\right)\\
&= \sig^{-2} R_2 \overline{Y}_n + O_{u,p}\left(\frac{1}{\sqrt{n}}\right),\\
\frac{-1}{\sqrt{n} r_n \xi_h(\al)} \p_{\lam} \p_{\be} \mbbh_n(\theta)&= \frac{-1}{n \xi_h(\al)} \sumj \sig^{-1} \left( \psi_{(\be),j} - \psi'_j \xi_h(\al) \right) Y_{t_{j-1}}\\
&= \frac{-\sig^{-1}}{n \xi_h(\al)} \sumj (\psi_{(\be),j} - \psi'_j \xi_h(\al))Y_{t_{j-1}}\\
&= \frac{-\sig^{-1}}{n \xi_h(\al)} \sumj (g'_j - g_j \psi_j + \xi_h(\al) \psi_j^2)Y_{t_{j-1}} + O_{u,p} \left(\frac{1}{\sqrt{n}}\right)\\
&= \frac{-\sig^{-1}}{n \xi_h(\al)} \sumj (\xi_h(\al) \psi_j - g_j)\psi_j Y_{t_{j-1}} + O_{u,p}\left(\frac{1}{\sqrt{n}} \right)\\
&= \frac{-\sig^{-1}}{n \xi_h(\al)} \sumj (\xi_h(\al) \psi_j - g_j)\psi_j \overline{Y}_n + O_{u,p} \left(\frac{1}{\sqrt{n}}\right)\\
&= -\sig^{-1} R_{3,n} \overline{Y} + O_{u,p} \left( \frac{1}{\sqrt{n}}\right),\\
\frac{-1}{\sqrt{n} r_n \xi_h(\al)} \p_{\mu} \p_{\al} \mbbh_n(\theta)
&= \frac{-1}{n \xi_h(\al)} \sumj \bigg[ - \sig ^{-1} \left( \psi_{(\al),j} - \be \psi'_j (\p_{\al} \xi_h(\al ))\right)\\
&\qquad + \al^{-2} \sig^{-1} \overline{l} (\ep_j(\theta) + \be \xi_h(\al)) \psi'_j\bigg] + O_{u,p}\left( \frac{\overline{l}}{\sqrt{n}}\right)\\
&= \frac{\sig^{-1}}{n \xi_h(\al)} \sumj \bigg\{ \bigg(f'_j - f_j \psi_j - \be (\p_{\al} \xi_h(\al)) \frac{\phi''_j}{\phi_j} 
\nn\\
&{}\qquad \qquad 
+ \be (\p_{\al} \xi_h(\al))\psi_j^2 \bigg)\\
&\qquad + \al^{-2} \overline{l} (\ep_j(\theta) + \be \xi_h(\al))\psi_j^2 \bigg\} + O_{u,p} \left(\frac{\overline{l}}{\sqrt{n}}\right)\\
&= \frac{\sig^{-1}}{n \xi_h(\al)} \sumj \bigg\{ ( \be (\p_{\al} \xi_h(\al)) \psi_j - f_j)\\
&\qquad + \al^{-2} \overline{l} (\ep_j(\theta) + \be \xi_h(\al)) \psi_j^2 \bigg\} + O_{u,p} \left(\frac{\overline{l}}{\sqrt{n}}\right)\\
&= \sig^{-1}R_{1,n} - \al^{-2} \sig^{-1}\overline{l} R_{2,n} + O_{u,p}\left( \frac{\overline{l}}{\sqrt{n}} \right),\\
\frac{-1}{\sqrt{n} r_n \xi_h(\al)} \p_{\mu} \p_{\sig} \mbbh_n(\theta)&= \frac{-1}{n \xi_h(\al)} \sumj \sig^{-2} (\ep_j(\theta) + \be \xi_h(\al))\psi'_j + O_{u,p}\left(\frac{1}{\sqrt{n}}\right)\\
&= \frac{\sig^{-2}}{n \xi_h(\al)} \sumj (\ep_j(\theta) + \be \xi_h(\al)) \psi_j^2 + O_{u,p} \left(\frac{1}{\sqrt{n}}\right)\\
&=  -\sig^{-2} R_{2,n} + O_{u,p}\left(\frac{1}{\sqrt{n}}\right),\\
\frac{-1}{\sqrt{n} r_n \xi_h(\al)} \p_{\mu} \p_{\be} \mbbh_n(\theta)&= \frac{1}{n \xi_h(\al)} \sumj \left\{ \sig^{-1} (\psi_{(\be),j} - \phi'_j \xi_h(\al))\right\}\\
&= \frac{\sig^{-1}}{n \xi_h(\al)} \sumj (\xi_h(\al) \psi_j - g_j)\psi_j + O_{u,p} \left(\frac{1}{\sqrt{n}}\right)\\
&= \sig^{-1} R_{3,n},
\end{align}
\begin{align}
&\frac{-1}{n \xi_h(\al)^2} \p_{\al}^2 \mbbh_n(\theta)\\
&=- \frac{-1}{n \xi_h(\al)^2} \sumj \bigg[ f_{(\al),j} - 2 \be (\p_{\al} \xi_h(\al)) \psi_{(\al),j} -\be^2 (\p_{\al} \xi_h(\al))^2 \psi_j^2\\
&\qquad + 2 \al^{-2} \overline{l} \left\{ -\xi_h(\al) \be \psi_{(\al),j} + (\p_{\al} \xi_h(\al)) \be \ep_j(\theta) \psi'_j - (\p_{\al} \xi_h(\al))\be^2\psi_j^2 - \ep_j(\theta) \psi_{(\al),j} \right \}\\
&\qquad + \al^{-4} \overline{l}^2 \left\{\xi_h(\al)^2 \be^2\psi'_j  + 2\xi_h(\al)\be \ep_j(\theta) \psi'_j + \ep_j(\theta)^2 \psi'_j + \ep_j(\theta) \psi_j \right\}\bigg ] + O_{u,p}\left( \frac{\overline{l}^2}{\sqrt{n}}\right)\\
&= \frac{1}{n\xi_h(\al)^2} \sumj (\be (\p_{\al} \xi_h(\al))\psi_j - f_j)^2\\
&\qquad + 2 \al^{-2} \overline{l} \frac{1}{n \xi_h(\al)^2} \sumj (\be (\p_{\al} \xi_h(\al)) (\ep_j(\theta) + \be \xi_h(\al)) \psi_j^2 - f_j \psi_j (\ep_j(\theta) + \be \xi_h(\al)))\\
&\qquad + \al^{-4} \overline{l} \frac{1}{n \xi_h(\al)^2} \sumj ((\ep_j(\theta) + \be \xi_h(\al))^2 \psi_j^2 + 2(\ep_j(\theta) + \be \xi_h(\al)) + 1 ) + O_{u,p} \left(\frac{\overline{l}^2}{\sqrt{n}}\right)\\
&= \frac{1}{n \xi_h(\al)^2} \sumj (\be (\p_{\al} \xi_h(\al)) \psi_j - f_j)^2\\
&\qquad + 2 \al^{-2} \overline{l} \frac{1}{n \xi_h(\al)^2} \sumj ((\ep_j(\theta) + \be \xi_h(\al))\psi_j + 1)(\be (\p_{\al} \xi_h(\al))\psi_j - f_j )\\
&\qquad + \al^{-4} \overline{l} \frac{1}{n \xi_h(\al)^2} \sumj ((\ep_j(\theta) + \be \xi_h(\al)) \psi_j + 1)^2 + O_{u,p} \left(\frac{\overline{l}^2}{\sqrt{n}}\right)\\
&= Q_{1,n} + 2 \al^{-2} \overline{l} Q_{2,n} + \al^{-4} \overline{l}^2 Q_{3,n} + O_{u,p}\left( \frac{\overline{l}^2}{\sqrt{n}}\right),\\
&-\frac{1}{n \xi_h(\al)^2} \p_{\al}\p_{\sig} \mbbh_n(\theta)\\
&= -\frac{1}{n \xi_h(\al)^2} \sumj \bigg[ \sig^{-1} \bigg\{ \be^2 (\p_{\al} \xi_h(\al))\xi_h(\al) + \be (\p_{\al} \xi_h(\al))\ep_j(\theta) \psi'_j\\
&\qquad -\be \psi_{(\al),j}\xi_h(\al) - \ep_j(\theta) \psi_{(\al),j}\bigg\}\\
&\qquad + \al^{-2} \sig^{-1} \overline{l} \left\{-\be^2 \xi_h(\al)^2 \psi_j^2 + 2 \be \xi_h(\al) \ep_j(\theta) \psi'_j + \ep_j(\theta)^2 \psi'_j + \ep_j(\theta) \psi_j\right\}\bigg] \nn\\
&{}\qquad + O_{u,p} \left( \frac{\overline{l}}{\sqrt{n}}\right)\\
&= \frac{-\sig^{-1}}{n \xi_h(\al)^2} \sumj \bigg[ f_j \psi_j (\ep_j(\theta) + \be\xi_h(\al)) - \be (\p_{\al} \xi_h(\al)) (\ep_j(\theta) + \be\ xi_h(\al)) \psi_j^2\\
&\qquad + \al^{-2} \sig^{-1} \overline{l} (\ep_j(\theta) + \be \xi_h(\al)) ((\ep_j(\theta) + \be \xi_h(\al)) \psi_j^2 + \psi_j)\bigg] + O_{u,p}\left(\frac{\overline{l}}{\sqrt{n}}\right)\\
&= \frac{\sig^{-1}}{n \xi_h(\al)^2} \sumj (\ep_j(\theta) + \be \xi_h(\al)) (\be (\p_{\al}\xi_h(\al))\psi_j^2 - f_j \psi_j)\\
&\qquad + \frac{\al^{-2} \sig^{-1} \overline{l}}{n \xi_h(\al)^2} \sumj (\ep_j(\theta) + \be \xi_h(\al))((\ep_j(\theta) + \be \xi_h(\al))\psi_j^2 + \psi_j) + O_{u,p} \left(\frac{\overline{l}}{\sqrt{n}}\right)\\
&= \frac{\sig^{-1}}{n \xi_h(\al)^2} \sumj (\be (\p_{\al} \xi_h(\al)) \psi_j - f_j)((\ep_j(\theta) + \be \xi_h(\al))\psi_j + 1 )\\
&\qquad + \frac{\al^{-2} \sig^{-1} \overline{l}}{n \xi_h(\al)^2} \sumj ((\ep_j(\theta) + \be \xi_h(\al))\psi_j + 1)^2 + O_{u,p} \left(\frac{\overline{l}}{\sqrt{n}}\right)\\
&= \sig^{-1} Q_{2,n} + \al^{-2} \sig^{-1} Q_{3,n} \overline{l} + O_{u,p}\left(\frac{\overline{l}}{\sqrt{n}}\right),\\
&-\frac{1}{n \xi_h(\al)^2} \p_{\al}\p_{\be} \mbbh_n(\theta)\\
&= -\frac{1}{n \xi_h(\al)^2} \sumj \bigg[- (\p_{\al} \xi_h(\al)) \left\{ \be(g'_j - \psi'_j \xi_h(\al)) + \psi_j \right\} + g_{(\al),j} - \psi_{(\al),j} \xi_h(\al)\\
&\qquad - \al^{-2} \overline{l} \left\{-\be \psi'_j \xi_h(\al)^2 + (\be g'_j - \ep_j(\theta) \psi'_j)\xi_h(\al) + \ep_j(\theta) g'_j \right\} \bigg]\\
&= -\frac{1}{n \xi_h(\al)} \sumj [\be (\p_{\al}\xi_h(\al))(g_j \psi_j -\psi_j^2) - f_j g_j + \xi_h(\al) f_j \psi_j\\
&\qquad - \al^{-2} \overline{l} (\be \xi_h(\al)^2 \psi_j^2 + (\ep_j(\theta) \psi_j^2 - \be g_j \psi_j)\xi_h(\al) -\ep_j(\theta) g_j \psi_j )] + O_{u,p} \left(\frac{\overline{l}}{\sqrt{n}}\right)\\
&= \frac{1}{n\xi_h(\al)} \sumj (\be(\p_{\al} \xi_h(\al)) \xi_h(\al)\psi_j^2 - (\be (\p_{\al} \xi_h(\al))g_j + \xi_h(\al)f_j)\psi_j + f_j g_j)\\
&\qquad + \frac{\al^{-2}\overline{l}}{n \xi_h(\al)} \sumj (\xi_h(\al)(\ep_j(\theta) + \be \xi_h(\al)) \psi_j^2 - (\ep_j(\theta) + \be \xi_h(\al))g_j \psi_j ) + O_{u,p} \left( \frac{\overline{l}}{\sqrt{n}}\right)\\
&= \frac{1}{n \xi_h(\al)} \sumj (\xi_h(\al) \psi_j - g_j)(\be (\p_{\al}\xi_h(\al))\psi_j - f_j)\\
&\qquad + \frac{\al^{-2}\overline{l}}{n \xi_h(\al)} \sumj ((\ep_j(\theta) + \be \xi_h(\al)) \psi_j + 1)(\xi_h(\al) \psi_j -g_j) + O_{u,p} \left(\frac{\overline{l}}{\sqrt{n}}\right)\\
&=Q_{4,n} + \al^{-2} \overline{l} Q_{5,n} + O_{u,p} \left(\frac{\overline{l}}{n}\right),\\
&-\frac{1}{n \xi_h(\al)^2} \p_{\sig}^2 \mbbh_n(\theta)\\
&= -\frac{1}{n \xi_h(\al)^2} \sumj \big\{ -\sig^{-2} \be^2 \psi_j^2 \xi_h(\al)^2 + 2 \sig^{-2} \be \ep_j(\theta) \psi'_j \xi_h(\al) \nn\\
&{}\qquad 
+ \sig^{-2} (\ep_j(\theta)^2 \psi'_j + \ep_j(\theta) \psi_j)\big\} + O_{u,p} \left(\frac{1}{\sqrt{n}}\right)\\
&= \frac{\sig^{-2}}{n \xi_h(\al)^2} \sumj (\be^2 \xi_h(\al)^2 \psi_j^2 + 2\be \xi_h(\al) \ep_j(\theta) \psi_j^2 + \ep_j(\theta)^2 \psi_j^2 + 1) + O_{u,p} \left(\frac{1}{\sqrt{n}}\right)\\
&= \frac{\sig^{-2}}{n \xi_h(\al)^2} \sumj ((\ep_j(\theta) + \be \xi_h(\al))\psi_j + 1)^2 + O_{u,p} \left(\frac{1}{\sqrt{n}}\right)\\
&= \sig^{-2} Q_{3,n} + O_{u,p}\left(\frac{1}{\sqrt{n}}\right),\\
&-\frac{1}{n \xi_h(\al)^2} \p_{\sig}\p_{\be} \mbbh_n(\theta)\\
&= -\frac{1}{n \xi_h(\al)^2} \sumj \left\{ \sig^{-1} \be \psi'_j\xi_h(\al)^2 + \sig^{-1} (\ep_j(\theta) \psi'_j - \be g'_j)\xi_h(\al) - \sig^{-1} \ep_j(\theta) g'_j \right\}\\
&= \frac{\sig^{-1}}{n \xi_h(\al^{2})} \sumj (\be \xi_h(\al)^2 \psi^2 + (\ep_j(\theta) \psi_j^2 - \be g_j \psi_j)\xi_h(\al) - \ep_j(\theta) g_j \psi_j ) + O_{u,p}\left( \frac{1}{\sqrt{n}}\right)\\
&=  \frac{\sig^{-1}}{n \xi_h(\al)^2} \sumj (\xi_h(\al) (\ep_j(\theta) + \be \xi_h(\al))\psi_j^2 - g(\ep_j(\theta) + \be \xi_h(\al))\psi_j ) + O_{u,p}\left( \frac{1}{\sqrt{n}}\right)\\
&= \frac{\sig^{-1}}{n \xi_h(\al)^2} \sumj ((\ep_j(\theta) + \be \xi_h(\al)) \psi_j + 1)(\xi_h(\al) \psi_j - g_j) + O_{u,p}\left( \frac{1}{\sqrt{n}}\right)\\
&=  \sig^{-1} Q_{5,n} + O_{u,p}\left( \frac{1}{\sqrt{n}}\right),\\
&-\frac{1}{n \xi_h(\al)^2} \p_{\be}^2 \mbbh_n(\theta)\\
&= -\frac{1}{n \xi_h(\al)^2} \sumj \left\{ \psi'_j \xi_h(\al)^2 - 2 \psi_{(\be),j} \xi_h(\al) + g_{(\be),j}\right\}\\
&= \frac{1}{n \xi_h(\al)^2} \sumj (\xi_h(\al)^2 \psi_j^2 - 2 \xi_h(\al) g_j \psi_j + g_j^2 ) + O_{u,p} \left(\frac{1}{\sqrt{n}}\right)\\
&= \frac{1}{n \xi_h(\al)^2} \sumj (\xi_h(\al) \psi_j - g_j)^2 + O_{u,p} \left(\frac{1}{\sqrt{n}}\right)\\
&= Q_{6,n} + O_{u,p} \left( \frac{1}{\sqrt{n}}\right).
\end{align}
\end{proof}

We now identify the limit in probability of the normalized observed information matrix $\mci_n(\theta)$.
Let
\begin{equation}
    \overline{Q} := 
    \begin{pmatrix}
    Q_{1} & Q_{2} & Q_{4} \\
    Q_{2} & Q_{3} & Q_{5}\\
    Q_{4} & Q_{5} & Q_{6}
    \end{pmatrix}
\end{equation}
where the components were defined after the proof of Lemma \ref{lem:xi_asymp} as follows:
\begin{align}
Q_1 &= t_\al^{-2}\, \E\!\left[\left(f_1 -\frac{\pi\beta}{2\cos^2(\al\pi/2)}\right)^2\right],\\
Q_2 &= t_\alpha^{-2}\, \E\!\left[(\beta t_\alpha \psi_1+\epsilon_{0,1}\psi_1+1)
\left( \frac{\pi\beta}{2\cos^2(\alpha\pi/2)}\psi_1-f_1\right )\right],\\
Q_3 &= t_\alpha^{-2} \E \left[(\beta t_\alpha \psi_1+1+\epsilon_{0,1}\psi_1)^2\right],\\
Q_4 &= t_\alpha^{-2}\, \E \left[(t_\alpha \psi_1-g_1)
\left (\frac{\pi\beta}{2\cos^2(\alpha\pi/2)}\psi_1-f_1\right )\right],\\
Q_5 &= t_\alpha^{-2}\, \E \left[(\beta t_\alpha \psi_1+\epsilon_{0,1}\psi_1+1 ) (t_\alpha \psi_1-g_1 )\right],\\
Q_6 &= t_\alpha^{-2}\, \E \left[(t_\alpha \psi_1-g_1)^2\right].
\end{align}
Let
\begin{equation}
    \Phi(\theta) \coloneqq \begin{pmatrix}
    \overline{\varphi}_{33}(\theta) & \overline{\varphi}_{34}(\theta) & 0\\
    \overline{\varphi}_{43}(\theta) & \overline{\varphi}_{44}(\theta) & 0\\
    0 & 0 & 1
    \end{pmatrix}.
\end{equation}

\begin{lem}
\label{lem:block_2deriv_conv}
We have the following convergences:
\begin{align}
\textup{(i)} \quad &-r_n^{-2} H_{11,n}(\theta)
\cip_u \sig^{-2} \E[\psi_1^2] \begin{pmatrix}
\overline{Y^2} & -\overline{Y^1} \\
-\overline{Y^1} & 1 \\
\end{pmatrix}\eqqcolon U,\\
\textup{(ii)} \quad & -\frac{1}{\sqrt{n}r_n \xi_h(\al)} H_{12,n}(\theta) \tilde{\varphi}_n(\theta)
\cip_u -\sig^{-1} 
\begin{pmatrix}
\overline{Y^1}R_1 & -\overline{Y^1}R_2 & \overline{Y^1}R_3\\
-R_1 & R_2 & -R_3
\end{pmatrix}\Phi(\theta) \eqqcolon V,\\
\textup{(iii)} \quad &-\frac{1}{n \xi_h(\al)^2} \tilde{\varphi}_n(\theta)^{\top} H_{22,n}(\theta)\tilde{\varphi}_n(\theta)\cip_u \Phi(\theta)^{\top}
\,\overline{Q}\,
\Phi(\theta) \eqqcolon W.
\end{align}
In particular,
\begin{equation}
\label{hm:def_FIm}
-\varphi_n(\theta)^{\top} \p_{\theta}^2 \mbbh_n(\theta) \varphi_n(\theta) \cip_u \begin{pmatrix}
U & V \\
V^{\top} & W \\
\end{pmatrix} \eqqcolon\mathcal{I}(\theta).
\end{equation}
\end{lem}

\begin{proof}
From Lemma \ref{lem:2deriv_conv}, we can deduce that
\begin{align}
-r_n^{-2} H_{11,n}
&=-r_n^{-2}\begin{pmatrix}
\p_{\lam}^2 \mbbh_n(\theta) & \p_{\lam}\p_{\mu}\mbbh_n(\theta) \\
\p_{\lam}\p_{\mu}\mbbh_n(\theta) & \p_{\mu}^2\mbbh_n(\theta)\\
\end{pmatrix}\\
&= -\sig^{-2}\begin{pmatrix}
\frac{1}{n} \sumj Y_{t_{j-1}}^2 \psi'_j & -\frac{1}{n} \sumj Y_{t_{j-1}} \psi'_j \\
-\frac{1}{n} \sumj Y_{t_{j-1}} \psi'_j & \frac{1}{n} \sumj \psi'_j\\
\end{pmatrix} + O_{u,p}\left( \frac{1}{\sqrt{n}}\right)\\
&\cip_u \sig^{-2} \E[\psi_1^2] \begin{pmatrix}
\overline{Y^2} & -\overline{Y^1} \\
-\overline{Y^1} & 1
\end{pmatrix},
\end{align}
\begin{align}
&-\frac{1}{\sqrt{n}r_n \xi_h(\al)} H_{12,n} \tilde{\varphi}_n(\theta)\\
&= \frac{-1}{\sqrt{n} r_n \xi_h(\al)} \begin{pmatrix}
\p_{\lam}\p_{\al} \mbbh_n(\theta) & \p_{\lam}\p_{\sig}\mbbh_n(\theta) & \p_{\lam}\p_{\be}\mbbh_n(\theta)\\
\p_{\mu}\p_{\al}\mbbh_n(\theta) & \p_{\mu}\p_{\sig}\mbbh_n(\theta) & \p_{\mu}\p_{\be}\mbbh_n(\theta)\\
\end{pmatrix} \tilde{\varphi}_n(\theta)\\
&= -\sig^{-1} \begin{pmatrix}
R_1 \overline{Y}_n - \al^{-2}\overline{l}R_2 \overline{Y}_n & -\sig^{-1} R_2 \overline{Y}_n & R_3 \overline{Y}_n\\
-R_{1,n} + \al^{-2}\overline{l}R_{2,n} & \sig^{-1} R_{2,n} & -R_{3,n}\\
\end{pmatrix}
\begin{pmatrix}
\varphi_{33,n}(\theta) & \varphi_{34,n}(\theta) & 0\\
\varphi_{43,n}(\theta) & \varphi_{44,n}(\theta) & 0\\
0 & 0 & 1
\end{pmatrix}\\
& \qquad + O_{u,p}\left(\frac{\overline{l}^2}{n} \right)\\
&\resizebox{\textwidth}{!}{$= -\sig^{-1} \begin{pmatrix}
\varphi_{33,n}(\theta) \overline{Y}_n R_{1} - s_{43,n}(\theta)\overline{Y}_n R_{2} & \varphi_{34,n}(\theta) \overline{Y}_n R_{1} - s_{44,n}(\theta)\overline{Y}_n R_{2} & \overline{Y}_n R_{3,n}\\
-\varphi_{33,n}(\theta) R_{1,n} + s_{43,n}(\theta) R_{2,n} & -\varphi_{34,n}(\theta) R_{1,n} + s_{44,n}(\theta)R_{2,n} & - R_{3,n} \\
\end{pmatrix}$}\\
&\qquad + O_{u,p}\left(\frac{\overline{l}^2}{n} \right)\\
&\cip_u -\sig^{-1} \begin{pmatrix}
\overline{\varphi}_{33}(\theta) \overline{Y^1} R_{1} - \overline{\varphi}_{43}(\theta)\overline{Y^1} R_{2} & \overline{\varphi}_{34}(\theta) \overline{Y^1} R_{1} - \overline{\varphi}_{44}(\theta) \overline{Y^1} R_{2} & \overline{Y^1} R_{3}\\
-\overline{\varphi}_{33}(\theta) R_{1} + \overline{\varphi}_{43}(\theta) R_{2} & -\overline{\varphi}_{34}(\theta) R_{1} + \overline{\varphi}_{44}(\theta) R_{2} & -R_{3} \\
\end{pmatrix}\\
&= -\sig^{-1} \begin{pmatrix}
\overline{Y^1}R_1 & -\overline{Y^1}R_2 & \overline{Y^1}R_3\\
-R_1 & R_2 & -R_3
\end{pmatrix}
\Phi(\theta),
\end{align}
\begin{align}
&-\frac{1}{n \xi_h(\al)^2} \tilde{\varphi}_n(\theta)^{\top} H_{22,n}\tilde{\varphi}_n(\theta)\\
&\resizebox{\textwidth}{!}{$= \tilde{\varphi}_n(\theta)^{\top} \begin{pmatrix}
Q_{1,n} + 2\al^{-2} \overline{l}Q_{2,n} + \al^{-4} \overline{l}^2 Q_{3,n} & \sig^{-1} Q_{2,n} + \al^{-2} \sig^{-1} Q_{3,n} \overline{l} & Q_{4,n} + \al^{-2} \overline{l} Q_{5,n}\\
\sig^{-1} Q_{2,n} + \al^{-2} \sig^{-1}Q_{3,n} \overline{l} & \sig^{-2} Q_{3,n} & \sig^{-1} Q_{5,n}\\
Q_{4,n} + \al^{-2} \overline{l} Q_{5,n} & \sig^{-1} Q_{5,n} & Q_{6,n}
\end{pmatrix}\tilde{\varphi}_n(\theta)$}\\
&\quad + O_{u,p}\left(\frac{\overline{l}^4}{\sqrt{n}}\right)\\
&= \tilde{\varphi}_n(\theta)^{\top}\begin{pmatrix}
1 & 0 & 0\\
\al^{-2} \overline{l} & \sig^{-1} & 0\\
0 & 0 & 1
\end{pmatrix}^{\top} \begin{pmatrix}
Q_{1,n} & Q_{2,n} & Q_{4,n} \\
Q_{2,n} & Q_{3,n} & Q_{5,n}\\
Q_{4,n} & Q_{5,n} & Q_{6,n}
\end{pmatrix}\begin{pmatrix}
1 & 0 & 0\\
\al^{-2} \overline{l} & \sig^{-1} & 0\\
0 & 0 & 1
\end{pmatrix}\tilde{\varphi}_n(\theta)\\
&= \begin{pmatrix}
\varphi_{33,n}(\theta) & \varphi_{34,n}(\theta) & 0\\
s_{43,n}(\theta) & s_{44,n}(\theta) & 0 \\
0 & 0 & 1
\end{pmatrix}^{\top}\begin{pmatrix}
Q_{1,n} & Q_{2,n} & Q_{4,n} \\
Q_{2,n} & Q_{3,n} & Q_{5,n}\\
Q_{4,n} & Q_{5,n} & Q_{6,n}
\end{pmatrix}
\begin{pmatrix}
\varphi_{33,n}(\theta) & \varphi_{34,n}(\theta) & 0\\
s_{43,n}(\theta) & s_{44,n}(\theta) & 0 \\
0 & 0 & 1
\end{pmatrix}\\
&\quad + O_{u,p}\left(\frac{\overline{l}^4}{\sqrt{n}}\right)\\
&\cip_u \Phi(\theta)^{\top}
\,\overline{Q}\,
\Phi(\theta).
\end{align}
\end{proof}

\subsubsection{Positive definiteness}
\label{hm:sec_proof.FI.pd}

We write $A \succ 0$ if $A \in \mbbr^{p} \otimes \mbbr^{p}$ is positive definite.
We will use the following two elementary lemmas; the first one is well-known as the Schur complement lemma.

\begin{lem}\label{lem:block_pos_definite}
Let $A \in \mathbb{R}^{p} \otimes \mathbb{R}^{p}$ and $C \in \mathbb{R}^{q} \otimes \mathbb{R}^{q}$ be symmetric matrices, and let $B \in \mathbb{R}^{p} \otimes \mathbb{R}^{p}$.
Suppose that $A \succ 0 $ and $S_A := C - B^\top A^{-1} B \succ 0$. Then, 
\begin{equation}
\begin{pmatrix}
A & B \\
B^\top & C
\end{pmatrix}
\succ 0.
\end{equation}
\end{lem}

\begin{lem}\label{lem:mat_pos_definite}
Let $A  \in \mbbr^{p} \otimes \mbbr^{p}$ and $ B \in \mbbr^{p} \otimes \mbbr^{p}$ be symmetric matrix.
If $A$ is non-singular and $B \succ 0$, then $A^{\top}B A \succ 0$.
\end{lem}


\begin{lem}
\label{lem:I-pd}
$\mci(\theta)$ defined in \eqref{hm:def_FIm} is a.s. positive definite.
\end{lem}

\begin{proof}
By Lemma \ref{lem:block_pos_definite}, it suffices to show that both the $U$ and $S_U\coloneqq W- V^{\top} U^{-1} V$ are positive definite. 

We begin with $U$.
We have $\inf_\theta \E[\psi_1^2]>0$. Since the process $t\mapsto Y_t$ is not a constant, we have $\overline{Y^2} > (\overline{Y^1})^2$ a.s. by the Cauchy-Schwarz inequality. Further, note that the event $\{ x_1 \overline{Y^1} = x_2 \}$ has probability zero for any $\boldsymbol{x}= (x_1,x_2)^\top \in \mbbr^{2}\backslash \{\boldsymbol{0}\}$ because the distribution of the random variable $\overline{Y}$ obeys a $S_\al^0$ distribution (with specific parameters), hence is continuous.
With these observations, we can conclude that $U\succ 0$:
\begin{align}
\boldsymbol{x}^{\top} U \boldsymbol{x} &= \sig^{-2} \E[\psi_1^2] \boldsymbol{x}^{\top}
\begin{pmatrix}
\overline{Y^2} & -\overline{Y^1} \\
-\overline{Y^1} & 1 \\
\end{pmatrix}
\boldsymbol{x}\\
&=  \sig^{-2}\E[\psi_1^2]  \left(x_1^2 \overline{Y^2} -2 x_1 x_2 \overline{Y^1} + x_2^2 \right) \\
&> \sig^{-2} \E[\psi_1^2] (x_1 \overline{Y^1} -x_2)^2 >0 \quad \text{a.s.}
\end{align}

We now turn to $S_U$. By direct calculations,
\begin{align}
S_U 
&= \Phi(\theta)^{\top}
\,\overline{Q}\,
\Phi(\theta)\\
&\quad - \frac{1}{\E[\psi_1^2] (\overline{Y^2} - \overline{Y^1}^2)} \Phi(\theta)^{\top} 
\begin{pmatrix}
\overline{Y^1} R_1 & -\overline{Y^1} R_2 & \overline{Y^1} R_3 \\
-R_1 & R_2 & -R_3 \\
\end{pmatrix}^{\top}\\
&{}\qquad \times 
\begin{pmatrix}
1 & \overline{Y^1} \\
\overline{Y^1} & \overline{Y^2} \\
\end{pmatrix}
\begin{pmatrix}
\overline{Y^1} R_1 & -\overline{Y^1} R_2 & \overline{Y^1} R_3 \\
-R_1 & R_2 & -R_3 \\
\end{pmatrix}
\Phi(\theta)\\
&= \Phi(\theta)^{\top} 
\Bigg\{ 
\overline{Q}
- \frac{1}{\E[\psi_1^2] (\overline{Y^2} - \overline{Y^1}^2 )}
\\
&{}\quad \times 
\begin{pmatrix}
\overline{Y^1} R_1 & -\overline{Y^1} R_2 & \overline{Y^1} R_3 \\
-R_1 & R_2 & -R_3 \\
\end{pmatrix}^{\top}
\!\!
\begin{pmatrix}
1 & \overline{Y^1} \\
\overline{Y^1} & \overline{Y^2} \\
\end{pmatrix}
\begin{pmatrix}
\overline{Y^1} R_1 & -\overline{Y^1} R_2 & \overline{Y^1} R_3 \\
-R_1 & R_2 & -R_3 \\
\end{pmatrix}\Bigg\}\Phi(\theta)\\
&\eqqcolon \Phi(\theta)^{\top} M \Phi(\theta).
\end{align}
From Lemma \ref{lem:mat_pos_definite}, it suffices to show the non-singularity of $\Phi(\theta)$ and the positive definiteness of $M$. But the former holds by the last item in \eqref{assump:rate_mat}, and it remains to look at $M$.

Let $\chi \coloneqq (\chi_1,\chi_2,\chi_3)^\top$, where
\begin{align}
& \chi_1 \coloneqq \frac{\pi \be}{2 \left(\cos \frac{\al \pi}{2}\right)^2} \psi_1 -f_1,
\qquad \chi_2 \coloneqq \be t_{\al} \psi_1 + \ep_{0,1} \psi_1 + 1, 
\qquad \chi_3 \coloneqq t_{\al} \psi_1 - g_1.
\end{align}
Then, direct computations give
\begin{align}
M 
&= \overline{Q} - \frac{1}{\E[\psi_1^2]} \begin{pmatrix}
R_1^2 & -R_1R_2 & R_1R_3 \\
-R_1R_2 & R_2^2 & -R_2R_3 \\
R_1R_3 & -R_2R_3 & R_3^2
\end{pmatrix}\\
&= t_{\al}^{-2} \Bigg\{ \begin{pmatrix}
\E[\chi_1^2] & \E[\chi_1 \chi_2] & \E[\chi_1 \chi_3] \\
\E[\chi_1 \chi_2] & \E[\chi_2^2] & \E[\chi_2 \chi_3] \\
\E[\chi_1 \chi_3] & \E[\chi_2 \chi_3] & \E[\chi_3^2]
\end{pmatrix}\\
&{}\qquad - \begin{pmatrix}
\E[\chi_1 \psi_1]^2 & \E[\chi_1 \psi_1]\E[\chi_2 \psi_1] & \E[\chi_1 \psi_1]\E[\chi_3 \psi_1]\\
\E[\chi_1 \psi_1]\E[\chi_2 \psi_1] & \E[\chi_2 \psi_1]^2 & \E[\chi_2 \psi_1]\E[\chi_3 \psi_1]\\
\E[\chi_1 \psi_1]\E[\chi_3 \psi_1] & \E[\chi_2 \psi_1]\E[\chi_3 \psi_1]& \E[\chi_3 \psi_1]^2
\end{pmatrix}
\Bigg \} \\
&= t_{\al}^{-2} \left( \E[\chi \chi^{\top}] - \frac{\E[\chi \psi_1] \E[\chi \psi_1]^{\top}}{\E[\psi_1^2]} \right).
\end{align}
We note that $M$ is positive semidefinite: 
for any $x \in \mbbr^{3}$,
\begin{align}
x^{\top} M x&= t_{\al}^{-2} \left( \E [(x^{\top} \chi)^2] - \frac{\E[x^{\top}\chi \psi_1]^2}{\E[\psi_1^2]}\right)\\
&= \frac{1}{t_{\al}^2 \E[\psi_1^2]} \left(\E [(x^{\top} \chi)^2]\E[\psi_1^2] - \E[x^{\top}\chi \psi_1]^2 \right)
\geq 0
\end{align}
by the Cauchy-Schwarz inequality.
Now it suffices to show the non-singularity of $M$.
From the equality condition of the Cauchy-Schwarz inequality, if $x^{\top} M x =0$ holds, then there exists a constant $a >0$ for which $x^{\top} \chi = a \psi_1$ a.s. 
From the definition of $\chi$ and the continuity of $\psi_1, f_1, g_1$, this equation does not hold.
Hence, $M$ is non-singular and we are done.
\end{proof}

\subsection{Continuity of rate matrix}

Before proving the continuity of the rate matrix \eqref{condi:rate_mat}, we estimate the order of the gap between the true parameter and its local perturbations.

\begin{lem}\label{lem:rate_est}
\begin{equation}
\sup_{\theta' \in \mathfrak{R}_n(c;\theta)} \left( \left| \sqrt{n} (\al' - \al)  \right| \lor \left|\sqrt{n}\,\overline{l}^{-1}(\sig' - \sig)\right| \lor \left| \sqrt{n} (\be' - \be)\right| \right) \lesssim_u 1.
\end{equation}
\end{lem}

\begin{proof}
Recall the definition \eqref{hm:R.set_def}:
\begin{equation}
\mathfrak{R}_n(c;\theta) = \big\{\theta' \in \Theta : |\varphi_n(\theta)^{-1}(\theta' - \theta)| \leq c\big\}.
\end{equation}
The definition of $\mathfrak{R}_n(c;\theta)$ specifies the appropriate order of $(\al',\sig', \be')$, hence we analyze $\varphi(\theta)^{-1}(\theta' - \theta)$.
To this end, we first derive $\varphi(\theta)^{-1}$. 
By definition \eqref{hm:def_vp}, 
it suffices to calculate $\tilde{\varphi}_n(\theta)^{-1}$
to obtain $\varphi_n(\theta)^{-1}$:

\begin{align}
\tilde{\varphi}_n(\theta)^{-1}= \frac{1}{D_n(\theta)} \begin{pmatrix}
\varphi_{44,n}(\theta) & -\varphi_{34,n}(\theta) & 0\\
-\varphi_{43,n}(\theta) & \varphi_{33,n}(\theta) & 0\\
0 & 0 & D_n(\theta)
\end{pmatrix},
\end{align}
where $D_n(\theta)\coloneqq \det \tilde{\varphi}_n(\theta)$.
The non-singularity of $\tilde{\varphi}_n(\theta)$ for every sufficiently large $n$ is guaranteed by the assumption \eqref{assump:rate_mat}:
\begin{align}
\inf_{\theta \in \overline{\Theta}} |D_n(\theta)| &= \inf_{\theta \in \overline{\Theta}} |\varphi_{33,n}(\theta)\varphi_{44,n}(\theta) - \varphi_{34,n}(\theta)\varphi_{43,n}(\theta)|\\
&= \inf_{\theta \in \overline{\Theta}}
|\sig(\varphi_{33,n}(\theta)s_{44,n}(\theta) - \varphi_{34,n}(\theta)s_{43,n}(\theta))|\\
&= \inf_{\theta \in \overline{\Theta}} | \sig (\overline{\varphi}_{33} (\theta) \overline{\varphi}_{44}(\theta) - \overline{\varphi}_{34} (\theta) \overline\varphi_{43}(\theta))| + o(1)\\
&\eqqcolon \inf_{\theta \in \overline{\Theta}} |\overline{D}(\theta)| + o(1).
\end{align}

Since we are concerned only with the orders associated with $(\al', \sig', \be')$, we normalize these three parameters. Let $\rho \coloneqq (\al, \sig, \be)$ and $\rho' = (\al', \sig', \be')$. From simple algebraic calculations, we have

\begin{equation}\label{eq:normalized_stable_parameter}
\begin{aligned}
&\sqrt{n}\xi_h(\al) \tilde{\varphi}_n(\theta)^{-1} (\rho' - \rho)\\
& = \frac{\sqrt{n} \xi_h(\al)}{D_n(\theta)} 
\begin{pmatrix}
y_{33,n}(\theta) & -y_{34,n}(\theta) & 0 \\
-y_{43,n}(\theta) & y_{44,n}(\theta) & 0\\
0 & 0 & y_{55,n}(\theta)
\end{pmatrix} 
\begin{pmatrix}
\al' -\al  \\
\sig' - \sig  \\
\be' - \be
\end{pmatrix}
\\
& = \frac{\sqrt{n} \xi_h(\al)}{\overline{D}(\theta) + o_u(1)} 
\begin{pmatrix}
y_{33,n}'(\theta) & -y_{34,n}(\theta) & 0 \\
-y_{43,n}'(\theta) & y_{44,n}(\theta) & 0\\
0 & 0 & y_{55,n}(\theta)
\end{pmatrix} 
\begin{pmatrix}
\sig(\al' -\al)  \\
\sig' - \sig + \al^{-2} \sig \overline{l}(\al' - \al) \\
\be' - \be
\end{pmatrix}\\
& = A_n(\theta) b_n(\theta,\theta'),
\end{aligned}
\end{equation}
where
\begin{equation}
\begin{cases}
y_{33,n}'(\theta) \coloneqq \al^{-2} \overline{l} y_{34,n}(\theta) + \sig y_{33,n}(\theta) \rightarrow_u \overline y_{33}(\theta),\\
y_{34,n}(\theta) \coloneqq \varphi_{34,n}(\theta) \rightarrow_u \overline{y}_{34}(\theta),\\
y_{43,n}'(\theta) \coloneqq \al^{-2} \overline{l} y_{44,n}(\theta) + \sig y_{43,n}(\theta) \rightarrow_u \overline y_{43}(\theta),\\
y_{44,n}(\theta) \coloneqq \varphi_{33,n}(\theta)\rightarrow_u \overline{y}_{44}(\theta),\\
y_{55,n}(\theta) \coloneqq \varphi_{33,n}(\theta)\varphi_{44,n}(\theta)-\varphi_{34,n}(\theta)\varphi_{43,n}(\theta) \rightarrow_u \overline{y}_{55}(\theta),
\end{cases}
\end{equation}
\begin{equation}
A_n(\theta) \coloneqq \frac{1}{\overline{D}(\theta) + o_{u}(1)} \begin{pmatrix}
y_{33,n}'(\theta) & -y_{34,n}(\theta) & 0 \\
-y_{43,n}'(\theta) & y_{44,n}(\theta) & 0\\
0 & 0 & y_{55,n}(\theta)
\end{pmatrix},
\end{equation}
and
\begin{equation}
b_n(\theta, \theta') \coloneqq \sqrt{n} \,\xi_h(\al) \begin{pmatrix}
\sig(\al' -\al) \\
\sig' - \sig + \al^{-2}\sig \overline{l}(\al' -\al) \\
\be' - \be
\end{pmatrix}.
\end{equation}
In view of \eqref{eq:normalized_stable_parameter}, to bound $|b_n(\theta,\theta')|$ we need an upper bound for$|A_n(\theta)^{-1}|$.

Let us compute the minimum eigenvalue of $A_n(\theta,\theta')$.
By the continuity in $\theta$, 
\begin{align}
&A_n(\theta)^{\top} A_n(\theta)\\
&= \frac{1}{\overline{D}(\theta)^2}
\begin{pmatrix}
y_{33,n}'(\theta) & -y_{34,n}(\theta) & 0 \\
-y_{43,n}'(\theta) & y_{44,n}(\theta) & 0\\
0 & 0 & y_{55,n}(\theta)
\end{pmatrix}^{\top}\\
&\qquad\times \begin{pmatrix}
y_{33,n}'(\theta) & -y_{34,n}(\theta) & 0 \\
-y_{43,n}'(\theta) & y_{44,n}(\theta) & 0\\
0 & 0 & y_{55,n}(\theta)
\end{pmatrix} 
+ o_{u}(1)
\\
&= \frac{1}{\overline{D}(\theta)^2} 
\bigg(
\begin{array}{l}
y_{33,n}(\theta)^2 + y_{43,n}'(\theta)^2 \\
-(y_{33,n}'(\theta) y_{34,n}(\theta)+ y_{43,n}'(\theta) y_{44,n}(\theta) ) \\
0
\end{array}\\
&\begin{matrix}
-(y_{33,n}'(\theta) y_{34,n}(\theta)+ y_{43,n}'(\theta) y_{44,n}(\theta)) & 0 \\
y_{34,n}(\theta)^2 + y_{44,n}(\theta)^2 & 0 \\
0 & y_{55,n}(\theta)^2
\end{matrix}
\bigg) + o_{u}(1)\\
&= \frac{1}{\overline{D}(\theta)^2} \begin{pmatrix}
\overline{y}_{33}(\theta)^2+ \overline{y}_{43}(\theta)^2 & -(\overline{y}_{33}(\theta)\overline{y}_{43}(\theta)+\overline{y}_{43}(\theta)\overline{y}_{44}(\theta)) & 0 \\
-(\overline{y}_{33}(\theta)\overline{y}_{43}(\theta)+\overline{y}_{43}(\theta)\overline{y}_{44}(\theta)) & \overline{y}_{34}(\theta)^2+ \overline{y}_{44}(\theta)^2 & 0\\
0 & 0 & \overline{y}_{55}(\theta)^2
\end{pmatrix} 
\nn\\
&{}\qquad + o_{u}(1). \\
\label{eq:min_eigen}
\end{align}
The characteristic polynomial is given as follows:
\begin{align}
&\left| 
\begin{pmatrix}
t - (\overline{y}_{33}(\theta)^2 + \overline{y}_{43}(\theta)^2) & \overline{y}_{33}(\theta)\overline{y}_{34}(\theta) + \overline{y}_{43}(\theta)\overline{y}_{44}(\theta) & 0\\
\overline{y}_{33}(\theta)\overline{y}_{34}(\theta) + \overline{y}_{43}(\theta)\overline{y}_{44}(\theta) & t - (\overline{y}_{34}(\theta)^2 + \overline{y}_{44}(\theta)^2) & 0\\
0 & 0 & t - \overline{y}_{55}(\theta)^2
\end{pmatrix}\right |\\
&= (t - \overline{y}_{55}(\theta)^2) \big\{ t^2 - (\overline{y}_{33}(\theta)^2 + \overline{y}_{34}(\theta)^2 + \overline{y}_{43}(\theta)^2 + \overline{y}_{44}(\theta)^2) t \nn\\
&{}\qquad + (\overline{y}_{33}(\theta)\overline{y}_{44}(\theta)-\overline{y}_{34}(\theta)\overline{y}_{43}(\theta))^2 \big\}
\\
&= (t - \overline{y}_{55}(\theta)^2) (t^2 - \mathsf{f}(\theta)t + \mathsf{g}(\theta)),
\end{align}
where $\mathsf{f}(\theta) \coloneqq \overline{y}_{33}(\theta)^2 + \overline{y}_{34}(\theta)^2 + \overline{y}_{43}(\theta)^2 + \overline{y}_{44}(\theta)^2$ and $\mathsf{g}(\theta) \coloneqq  (\overline{y}_{33}(\theta)\overline{y}_{44}(\theta) - \overline{y}_{34}(\theta)\overline{y}_{43}(\theta))^2$.

Here we note that $\mathsf{f}(\theta)^2 - 4\mathsf{g}(\theta)\ge 0$. 
Indeed, by using the Cauchy--Schwarz inequality $(ad - bc)^2 \leq (a^2 + b^2)(c^2 + d^2)$, we have
\begin{align}\label{eq:g}
\mathsf{g}(\theta)
&\leq(\overline{y}_{33}(\theta)^2 + \overline{y}_{34}(\theta)^2 )(\overline{y}_{43}(\theta)^2 + \overline{y}_{44}(\theta)^2 ).
\end{align}
Also, by the arithmetic-geometric mean inequality,
\begin{equation}\label{eq:f}
    (\overline{y}_{33}(\theta)^2 + \overline{y}_{34}(\theta)^2)(\overline{y}_{43}(\theta)^2 + \overline{y}_{44}(\theta)^2) \leq \frac{\mathsf{f}(\theta)^2}{4}.
\end{equation}
From \eqref{eq:f} and \eqref{eq:g}, we get $\mathsf{f}(\theta)^2 - 4\mathsf{g}(\theta)\ge 0$. 

The minimum eigenvalue of the first term in \eqref{eq:min_eigen} is
\begin{equation}
\underline{t}_n(\theta) \coloneqq \min \left\{\overline{y}_{55}(\theta)^2, \, \frac{\mathsf{f}(\theta) - \sqrt{\mathsf{f}(\theta)^2 -4\mathsf{g}(\theta)}}{2}\right\}.
\end{equation}
Suppose that $\underline{t}_n(\theta)=\frac{\mathsf{f}(\theta)-\sqrt{\mathsf{f}(\theta)^2-4\mathsf{g}(\theta)}}{2}$. We have
\begin{equation}
    \mathsf{f}(\theta)- \sqrt{\mathsf{f}(\theta)^2-4\mathsf{g}(\theta)}
    = \frac{4\mathsf{g}(\theta)}{\mathsf{f}(\theta) + \sqrt{\mathsf{f}(\theta)^2 + 4\mathsf{g}(\theta)}} \gtrsim_u 1
\end{equation}
by \eqref{assump:rate_mat}, concluding that $\underline{t}_n(\theta)^{-1} \lesssim_u 1$.
If in turn $\underline{t}_n(\theta)=\overline{y}_{55}(\theta)^2$, then $\inf_{\theta \in \Theta} \underline{t}_n(\theta) >0$ by \eqref{assump:rate_mat}. Hence, $\underline{t}_n(\theta)^{-1} \lesssim_u 1$ also in this case.

Thus, by \eqref{eq:normalized_stable_parameter} we have
\begin{align}
&\sup_{\rho: \left | \sqrt{n}\xi_h(\al) \tilde{\varphi}_n(\theta)^{-1}(\rho'-\rho)\right | \leq c} |b(\theta, \theta')|\\
&= \sup_{\rho: |\sqrt{n}\xi_h(\al) \tilde{\varphi}_n(\theta)^{-1}(\rho'-\rho)| \leq c} \left |A_n(\theta,\theta')^{-1} \sqrt{n}\xi_h(\al) \tilde{\varphi}_n(\theta)^{-1}(\rho'-\rho) \right |\\
&\leq |\underline{t}_n(\theta)|^{-1/2} \sup_{\rho: \left | \sqrt{n}\xi_h(\al) \tilde{\varphi}_n(\theta)^{-1}(\rho'-\rho)\right | \leq c} \left | \sqrt{n}\xi_h(\al) \tilde{\varphi}_n(\theta)^{-1}(\rho'-\rho)\right |\\
&\leq c |\underline{t}_n(\theta)|^{-1/2} \lesssim_u 1,
\end{align}
so that
\begin{equation}
\sup_{\theta' \in \mathfrak{R}_n(c;\theta)} \left\{|\sqrt{n}(\al' - \al)| \lor |\sqrt{n}\overline{l}^{-1} (\sig' - \sig)| \lor |\sqrt{n}(\be' - \be)| \right\}= O_{u}(1).
\end{equation}
The proof is complete.
\end{proof}

\medskip

We are now in the position of proving \eqref{condi:rate_mat}. 
We have
\begin{align}
&\sup_{\theta' \in \mathfrak{R}_n(c;\theta)} \left| \varphi_n(\theta')^{-1} \varphi_n(\theta) - I_5 \right|\\
&= \sup_{\theta' \in \mathfrak{R}_n(c;\theta)} \left| \begin{pmatrix}
r_n(\theta') r_n(\theta)^{-1} & O \\
O & \frac{\xi_h(\al')}{\xi_h(\al)} \tilde{\varphi}_n(\theta')^{-1} \tilde{\varphi}_n (\theta)\\
\end{pmatrix} - I_5 \right|\\
&\le \sup_{\theta' \in \mathfrak{R}_n(c;\theta)}
\left| (r_n(\theta')^{-1} r_n(\theta)-1) I_2 \right|
\left| \frac{\xi_h(\al')}{\xi_h(\al)} \tilde{\varphi}_n(\theta')^{-1} \tilde{\varphi}_n(\theta) - I_3\right| \\
&\leq \sup_{\theta' \in \mathfrak{R}_n(c;\theta)} |r_n(\theta')^{-1} r_n(\theta) -1| 
\left(
\left|\frac{\xi_h(\al')}{\xi_h(\al)}-1 \right| +  \left|\frac{\xi_h(\al')}{\xi_h(\al)} \tilde{\varphi}_n(\theta')^{-1} \tilde{\varphi}_n(\theta) - I_3\right|\right).
\end{align}
It suffices to show that
\begin{equation}\label{sufficient:rate}
\sup_{\theta' \in \mathfrak{R}_n(c;\theta)}
\left(|r_n(\theta')^{-1} r_n(\theta) - 1| \lor \left| \frac{\xi_h(\al')}{\xi_h(\al)} - 1 \right|\lor \left|\tilde{\varphi}_n(\theta')^{-1} \tilde{\varphi}_n(\theta) - I_3 \right|\right) \rightarrow_u 0.
\end{equation}
We use the (ordinary) order symbols $o_u(\cdot)$ and $O_u(\cdot)$ when they are valid uniformly in $\theta\in\overline{\Theta}$.
From the mean value theorem and Lemma \ref{lem:rate_est}, there exists $\al''$ between $\al$ and $\al'$,
\begin{align}
\sup_{\theta' \in \mathfrak{R}_n(c;\theta)} |r_n(\theta')^{-1} r_n(\theta) -1|
&= \sup_{\theta' \in \mathfrak{R}_n(c;\theta)} |h^{1/\al - 1/\al'} - 1|\\
&= \sup_{\theta' \in \mathfrak{R}_n(c;\theta)} \left|\frac{1}{\al''} (\log h) h^{1/\al - 1/\al''} (\al'' - \al) \right|\\
&\lesssim  \sup_{\theta' \in \mathfrak{R}_n(c;\theta)} h^{1/\al - 1/\al''}  \sup_{\theta' \in \mathfrak{R}_n(c;\theta)} |\overline{l}^{-1} (\al' -\al)|\\
&\lesssim \sup_{\theta' \in \mathfrak{R}_n(c;\theta)} h^{\frac{\al''-\al}{\al \al''}} o_u(1)\\
&\leq \sup_{\theta' \in \mathfrak{R}_n(c;\theta)} \left(h^{-|\al'' - \al|}\right)^{\frac{1}{\al \al''}}o_u(1)\\
&\leq \sup_{\theta' \in \mathfrak{R}_n(c;\theta)} \left(h^{-|\al' - \al|}\right)^{\frac{1}{\al \al''}}o_u(1)\\
&= \sup_{\theta' \in \mathfrak{R}_n(c;\theta)} \left(\exp\left( \overline{l} |\al' - \al| \right) \right)^{\frac{1}{\al \al''}}o_u(1)\\
&=o_u(1).
\end{align}

Next, since $\xi_h$ is $C^1$ in $\alpha$,
by Lemma \ref{lem:rate_est} and the mean value theorem there exists $\alpha''$ between $\alpha$ and $\alpha'$ such that
\begin{align}
\sup_{\theta' \in \mathfrak{R}_n(c;\theta)} \left|\frac{\xi_h(\al')}{\xi_h(\al)} -1\right|
&= \sup_{\theta' \in \mathfrak{R}_n(c;\theta)} \frac{|\xi_h(\al') - \xi_h(\al)|}{\xi_h(\al)}\\
&\leq \sup_{\theta' \in \mathfrak{R}_n(c;\theta)} 
\frac{|\p_{\al}\xi_h(\al'') (\al' - \al) |}{\xi_h(\al)}\\
&\leq \frac{\sup_{\theta \in \Theta} |\p_{\al} \xi_h(\al)|}{\inf_{\theta \in \Theta} \xi_h(\al)} \sup_{\theta' \in \mathfrak{R}_n(c;\theta)} |\al' - \al|\\
&=o_u(1).
\end{align}

Finally,
\begin{align}
\sup_{\theta' \in \mathfrak{R}_n(c;\theta)} | \tilde \varphi_n(\theta')^{-1} \tilde \varphi_n(\theta) - I_3|
&= \sup_{\theta' \in \mathfrak{R}_n(c;\theta)} |\tilde{\varphi}_n(\theta')^{-1} (\tilde \varphi_n(\theta) - \tilde \varphi_n(\theta')) |\\
&=  \sup_{\theta' \in \mathfrak{R}_n(c;\theta)}|\tilde \varphi_n(\theta')^{-1}| |\tilde \varphi_n(\theta) - \tilde \varphi_n(\theta')|.
\end{align}
By the assumption \ref{assump:rate_mat},
\begin{align}
\sup_{\theta' \in \mathfrak{R}_n(c;\theta)} \left | \tilde{\varphi}_n (\theta')^{-1}\right | 
&= \sup_{\theta' \in \mathfrak{R}_n(c;\theta)}  \left |\frac{1}{\delta_n(\theta')} \begin{pmatrix}
\varphi_{44,n}(\theta') & -\varphi_{34,n}(\theta') \\
-\varphi_{43,n}(\theta') & \varphi_{33,n}(\theta') \\
\end{pmatrix}\right |\\
&\leq \sup_{\theta' \in \mathfrak{R}_n(c;\theta)} \frac{1}{|\delta_n(\theta')|} \left |\begin{pmatrix}
\varphi_{44,n}(\theta') & -\varphi_{34,n}(\theta') \\
-\varphi_{43,n}(\theta') & \varphi_{33,n}(\theta') \\
\end{pmatrix}\right |\\
&\leq O_u(\overline{l}).
\end{align}
Since $\tilde \varphi_n(\theta)$ is continuously differentiable and by Taylor’s theorem together with Lemma \ref{lem:rate_est}, we obtain the following estimates:
\begin{align}
\sup_{\theta' \in \mathfrak{R}_n(c;\theta)} | 
\tilde{\varphi}_n(\theta') - \tilde\varphi_n(\theta)|
&= \sup_{\theta' \in \mathfrak{R}_n(c;\theta)} \left|\int_{0}^{1} \tilde{\varphi}'_n (\theta + t(\theta' - \theta))dt (\theta' - \theta)\right|\\
&\leq  \sup_{\theta \in \Theta} |\tilde\varphi_n' (\theta)| \sup_{\theta' \in \mathfrak{R}_n(c;\theta)} |\theta' - \theta|\\
& \lesssim O_u(\overline{l} r_n^{-1}).
\end{align}
The above three estimates combined with \eqref{sufficient:rate} yield \eqref{condi:rate_mat}.
\subsection{Negligibility of remainder term}

Here we prove \eqref{condi:obs_conti}.
From Taylor's theorem,
\begin{align}
&\sup_{\theta^{(1)},\dots,\theta^{(5)} \in \mathfrak{R}_n(c;\theta)} \left | \varphi_n(\theta)^{\top} \left( \partial_{\theta}^2 \mbbh_n(\theta^{(1)}, \dots, \theta^{(5)}) - \partial_{\theta}^2 \mbbh_n(\theta) \right) \varphi_n(\theta) \right |\\
&=\sup_{\theta^{(1)},\dots,\theta^{(5)} \in \mathfrak{R}_n(c;\theta)} \left | \varphi_n(\theta)^{\top} \left[ \p_{l} \p_{m} \mbbh_n(\theta^{(l)}) - \p_{l} \p_{m} \mbbh_n(\theta)\right]_{l,m=1,...5} \varphi_n(\theta)\right |\\
&=\sup_{\theta^{'}, \theta^{''} \in \mathfrak{R}_n(c;\theta)}\left | \varphi_n(\theta)^{\top} \left[ \p_{k} \p_{l} \p_{m} \mbbh_n(\theta'') (\theta' - \theta) \right]_{k,l,m=1,...,5} \varphi_n(\theta)\right |\\
&\eqqcolon \sup_{\theta^{'}, \theta^{''} \in \mathfrak{R}_n(c;\theta)} \left | \varphi_n(\theta)^{\top} [e_{kl}(\theta, \theta', \theta'')]_{k,l=1,...,5} \varphi_n(\theta)\right |\\
&= \sup_{\theta^{'}, \theta^{''} \in \mathfrak{R}_n(c;\theta)} \left| \left[\sum_{i=1}^{5}\sum_{j=1}^{5} \varphi_{ki}(\theta) e_{ij}(\theta, \theta', \theta'') \varphi_{jl}(\theta) \right]_{k,l=1,...,5} \right|\\
&\lesssim_u \sup_{\theta', \theta'' \in \mathfrak{R}_n(c;\theta)} \max_{k,l=1,...,5} \max_{i,j=1,...,5} |\varphi_{ki}(\theta) e_{ij}(\theta, \theta', \theta'') \varphi_{jl}(\theta)|.
\end{align}
We will evaluate the order of each term on the right-hand side.

For $k=1$, since $\varphi_{1i}(\theta)=0$ for all $i\neq 1$, only the case $i=1$ needs to be considered:
\begin{equation}
\varphi_{11}(\theta) e_{1j}(\theta, \theta', \theta'')\varphi_{jl}(\theta),\qquad j,l=1,\dots,5.
\end{equation}
Each of the terms $\varphi_{ki}(\theta) e_{ij}(\theta, \theta', \theta'') \varphi_{jl}(\theta)$ is bounded by either one of
\begin{align}
&r_n^{-2}(\theta)|e_{11}(\theta, \theta', \theta'')|,\quad r_n^{-2}(\theta)|e_{12}(\theta, \theta', \theta'')|,\quad
\tfrac{1}{\sqrt{n}r_n(\theta)}|e_{13}(\theta, \theta', \theta'')|,\\
&\tfrac{\overline{l}}{\sqrt{n}r_n(\theta)}|e_{14}(\theta, \theta', \theta'')|,\quad
\tfrac{1}{\sqrt{n}r_n (\theta)}|e_{15}(\theta, \theta', \theta'')|.
\end{align}
By Lemma~\ref{lem:rate_est}, all of these are
\(O_{u,p}(\overline{l}/\sqrt{n})\), hence
\begin{equation}
\sup_{\theta',\theta'' \in \mathfrak{R}_n(c;\theta)}
\max_{l,j}
|\varphi_{11}(\theta)e_{1j}(\theta,\theta',\theta'')\varphi_{jl}(\theta)|
=O_{u,p}\!\left(\frac{\overline{l}}{\sqrt{n}}\right).
\end{equation}
As for the other cases, analogous calculations for \(k=2,\dots,5\) give
\begin{align}
\sup_{\theta',\theta'' \in \mathfrak{R}_n(c;\theta)}
\max_{i,l,j}
|\varphi_{2i}(\theta)e_{ij}(\theta,\theta',\theta'')\varphi_{jl}(\theta)|
&=
O_{u,p}\!\Bigl(\tfrac{\overline{l}}{\sqrt{n}}\Bigr),
\nn\\
\sup_{\theta',\theta'' \in \mathfrak{R}_n(c;\theta)}
\max_{i,l,j}
|\varphi_{3i}(\theta)e_{ij}(\theta,\theta',\theta'')\varphi_{jl}(\theta)|
&=O_{u,p}\!\Bigl(\tfrac{h^{1-1/\alpha}\overline{l}^2}{\sqrt{n}}\Bigr),
\nn\\
\sup_{\theta',\theta'' \in \mathfrak{R}_n(c;\theta)}
\max_{i,l,j}
|\varphi_{4i}(\theta)e_{ij}(\theta,\theta',\theta'')\varphi_{jl}(\theta)|
&=O_{u,p}\!\Bigl(\tfrac{h^{1-1/\alpha}\overline{l}^2}{\sqrt{n}}\Bigr),
\nn\\
\sup_{\theta',\theta'' \in \mathfrak{R}_n(c;\theta)}
\max_{i,l,j}
|\varphi_{5i}(\theta)e_{ij}(\theta,\theta',\theta'')\varphi_{jl}(\theta)|
&=O_{u,p}\!\Bigl(\tfrac{h^{1-1/\alpha}\overline{l}}{\sqrt{n}}\Bigr).
\end{align}
Since all of these bounds are $o_{u,p}(1)$, the entire expression tends to zero uniformly in $\theta$. Therefore,
\begin{equation}
\sup_{\theta^{(1)},\dots,\theta^{(5)}\in\mathfrak{R}_n(c;\theta)}
\bigl|
\varphi_n(\theta)^{\top}
(\p_{\theta}^2\mbbh_n(\theta^{(1)},\dots,\theta^{(5)})
- \p_{\theta}^2\mbbh_n(\theta))
\varphi_n(\theta)
\bigr|
\cip_u 0,
\end{equation}
which completes the proof of \eqref{condi:obs_conti}.

\subsection{Estimating gap}

Here, we will show the remaining \eqref{condi:gap_partial}, \eqref{condi:gap_conv1}, and \eqref{condi:gap_conv2}.
To prove them, we need to determine the stochastic order of $|\p_{\theta}^{k}\delta_n(\theta)|$. 
We write
\begin{equation}
    \delta_n(\theta)= \delta_n^{0}(\theta) + \check{\delta}_n(\theta),
\end{equation}
where $\delta_n^{0}(\theta) \coloneqq n(\log c)$ and $\check{\delta}_n(\theta) \coloneqq \sumj G_j(\theta)d_j(\theta)$.
First, we prove the following lemma.

\begin{lem}\label{lem:general_CLT}
Let $k \ge 2$ be an even number. Then, we have
    \begin{align}
    \frac{1}{n} \sumj |\ep_j'(\theta)| &= O_{u,p}(1),
    \label{lem:general_CLT+1}\\
    \frac{1}{n} \sumj |\ep_j(\theta)| &= O_{u,p}(1),
    \label{lem:general_CLT+2}\\
    \frac{1}{n} \sumj \ep_j'(\theta)^k &= O_{u,p}(h^{1-k/\al}),
    \label{lem:general_CLT-1}\\
    \frac{1}{n} \sumj \ep_j(\theta)^k &= O_{u,p}(h^{1-k/\al}).
    \label{lem:general_CLT-2}
    \end{align}
\end{lem}
\begin{proof}
Since $\ep_1'(\theta),\dots,\ep_n'(\theta)$ are i.i.d. with common law $S_\al^0(\beta,1,0)$, \eqref{lem:general_CLT+1} is trivial by the law of large numbers and by the standing assumption \eqref{hm:Theta_closure} together with the smoothness of $(\al,\beta)\mapsto \phi_{\al,\beta}(y)$ for each $y\in\mbbr$. 
Also, \eqref{lem:general_CLT+2} is easily seen as follows:
\begin{align}
    \frac{1}{n} \sumj |\ep_j(\theta)| &\lesssim  \frac{1}{n} \sumj |\ep_j'(\theta)| + \frac{1}{n} \sumj |d_j(\theta)|\\
    &\lesssim O_p(1) + h \left(\frac{1}{n} \sumj |\ep_j(\theta)| + h^{1-1/\al} \frac{1}{n} \sumj |Y_{t_{j-1}}|  + 1\right)
    \nn\\
    &\lesssim O_{u,p}(1) + O_{u,p}(h) = O_{u,p}(1).
\end{align}
From the above inequality, we have $n^{-1} \sumj |\ep_j(\theta)| = O_{u,p}(1)$.

Let $k\ge 2$ be an even number. By \cite[p.151 and Proposition A.2]{DavKniLiu92} we have
\begin{equation}\label{hm:DavRes85_order}
    \sumj \left(\frac{\ep'_j(\theta)}{n^{1/\al}}\right)^k = O_{u,p}(1),
\end{equation}
from which \eqref{lem:general_CLT-1} immediately follows.
By using \eqref{hm:DavRes85_order} and \eqref{hm:p-add1}, we obtain
\begin{align}
\frac{1}{n} \sumj |\ep_j(\theta)|^k &\lesssim\frac{1}{n} \sumj |\ep_j'(\theta)|^k + \frac{1}{n} \sumj |d_j(\theta)|^k\\
&\lesssim O_{u,p}(h^{1-k/\al}) + \frac{1}{n} \sumj\left(h|\ep_j(\theta)| + h^{2-1/\al}|Y_{t_{j-1}}| + h\right)^k\\
&\lesssim O_{u,p}(h^{1-k/\al}) + \frac{1}{n} \sumj \left(h^k |\ep_j(\theta)|^k + h^{(2-1/\al)k}|Y_{t_{j-1}}|^k +h^k\right)\\
&= O_{u,p}(h^{1-k/\al}) + \frac{1}{n} \sumj |\ep_j(\theta)|^k h^{k} + O_{u,p}(h^{(2-1/\al)k}) + h^{k}
\nn\\
&= O_{u,p}(h^{1-k/\al}) + \frac{1}{n} \sumj |\ep_j(\theta)|^k h^{k}.
\end{align}
Hence, for sufficiently small $h>0$ we get
\begin{equation}
    \frac{1}{n} \sumj |\ep_j(\theta)|^k = O_{u,p}(h^{1-k/\al}),
\end{equation}
establishing \eqref{lem:general_CLT-2}.
\end{proof}

Then, it suffices to determine the order of $|\p_{\theta}^k \delta_n^{0}(\theta)|$ and $|\p_{\theta}^k\check{\delta}_n(\theta)|$. 

\begin{lem}\label{lem:delta0_order}
For any integer $k \geq 1$,
\begin{equation}
|\p_{\theta}^{k} \delta_n^{0}(\theta)| \lesssim O(h^{k}).
\end{equation}
\end{lem}

\begin{proof}
We have $\log c= -\al^{-1} \log \eta(\lam \al h)$ by the definition. For integer $i_1,i_2 \geq 1$,
\begin{align}\label{eq:log_deriv}
&|\p_{\lam}^{i_1} \p_{\al}^{i_2}(\log c)| = O(h^{i_1 + 1}),
\end{align}
hence the claim.
\end{proof}

\begin{lem}\label{lem:delta1_order}
For any integer $k \geq 0$, we have
\begin{align}
|\p_{\theta}^k d_j(\theta)| &\lesssim h \overline{l}^k (|\ep_j(\theta)| + |Y_{t_{j-1}}| + 1),\label{ineq:dj_deriv}\\
|\p_{\theta}^k G_j(\theta)| &\lesssim \overline{l}^k (|h\ep_j(\theta)| + |Y_{t_{j-1}}| + 1),\label{ineq:Gj_deriv}\\
|\p_{\theta}^k \check{\delta}_n(\theta)| &\lesssim \overline{l}^{2k}.\label{ineq:delta_deriv}
\end{align}
\end{lem}

\begin{proof}
First, \eqref{ineq:dj_deriv} can be proved as follows:
\begin{align}
|\p_{\theta}^{k} d_j(\theta)| &\lesssim |\p_{\theta}^k \{ (\rho_h(\lam,\al) \ep_j(\theta) + \rho_h(\lam,\al) (g_j(\theta) + C(\theta)))\}|\\
&\lesssim \max_{0 \leq i \leq k} |\p_{\theta}^i \rho_h(\lam,\al)|\max_{0 \leq i \leq k} |\p_{\theta}^i \ep_j(\theta)| \nn\\
&{}\qquad + \max_{0 \leq i \leq k} |\p_{\theta}^i \rho_h(\lam,\al)| \max_{0 \leq i \leq k} |\p_{\theta}^i(g_j(\theta )+ C(\theta))|\\
&\lesssim  h(\overline{l}^k |\ep_j(\theta)| + h^{1-1/\al}|Y_{t_{j-1}}| + \overline{l}^k) + h(h^{2-1/\al} \overline{l}^k |Y_{t_{j-1}}| + h) \\
&\lesssim h \overline{l}^k (|\ep_j(\theta)| + h^{1-1/\al}|Y_{t_{j-1}}| + 1) \label{hm:p-add1} \\
&\lesssim h \overline{l}^k (|\ep_j(\theta)| + |Y_{t_{j-1}}| + 1).
\end{align}

Next, we prove \eqref{ineq:Gj_deriv}. Since $\p_{\theta}^k G_j(\theta) = \int_{0}^{1} \p_{\theta}^k \psi(\ep_j(\theta) + u d_j(\theta))du$, we need to calculate the upper bound for $|\p_{\theta}^k \psi(\ep_j(\theta) + ud_j(\theta))|$. We have
\begin{align}
&|\p_{\lam}^{i_1}\p_{\mu}^{i_2}\p_{\al}^{i_3}\p_{\sig}^{i_4}\p_{\be}^{i_5} \psi(\ep_j(\theta) + u d_j(\theta))|\\
&\lesssim |\p_{x}^k \psi (\ep_j(\theta) + u d_j(\theta)) (\ep_j(\theta) + u d_j(\theta))^k \,\p_{\lam}^{i_1} \cdots\p_{\be}^{i_5}(\ep_j(\theta) + u d_j(\theta))|
.
\end{align}
To estimate this, we prove that $\p_{\lam}^{i_1} \cdots \p_{\be}^{i_5}(\ep_j(\theta) + ud_j(\theta))$ can be written as a non-random-constant multiple of $\ep_j(\theta) + u d_j(\theta)$ plus a negligible term. First, we calculate the derivative of $\ep_j(\theta)$.
From the equation \eqref{ep_1dev}, we have
\begin{align}
& \p_{\lam} \ep_j(\theta) = \sig^{-1} h^{1-1/\al} Y_{t_{j-1}},\qquad \p_{\lam}^k \ep_j(\theta)=0 \quad (k \geq 2),\\
& \p_{\mu} \ep_j(\theta) = -\sig^{-1} h^{1-1/\al}, \qquad \p_{\mu}^k \ep_j(\theta)=0 \quad (k \geq 2),\\
& \p_{\al}^k \ep_j(\theta) = (-\al^{-2} \overline{l})^k \ep_j(\theta)  + O(\overline{l}^k) \qquad (k \geq 1),\\
& \p_{\sig}^k \ep_j(\theta) = O(1) \ep_j(\theta) + O(1) \qquad 
(k \geq 1),\\
&\p_{\be} \ep_j(\theta) = - \xi_h(\al), \qquad \p_{\be}^k \ep_j(\theta) = 0 \quad (k \geq 2).
\end{align}
Based on the above formulae and by further calculations, we can show that for integers $i_1,\dots,i_5 \geq 1$ such that $i_1 + \cdots+i_5=k$,
\begin{equation}\label{eq:deriv_ep}
\p_{\lam}^{i_1}\p_{\mu}^{i_2}\p_{\al}^{i_3}\p_{\sig}^{i_4}\p_{\be}^{i_5}\ep_j(\theta)= O(\overline{l}^{i_{3}}) \ep_j(\theta) + O(h^{1-1/\al}) Y_{t_{j-1}} + O(\overline{l}^{i_3}).
\end{equation}
Next, we calculate the derivative of $d_j(\theta)$. 
\begin{align}
&\p_{\lam}^k d_j(\theta) = O(h^k)\ep_j(\theta) + O(h^{k-1/\al}) Y_{t_{j-1}} + O(h) \quad (k \geq 1),\\
&\p_{\mu} d_j(\theta) = O(h^{2-1/\al})Y_{t_{j-1}} + O(h^{2-1/\al}),\quad \p_{\mu}^k d_j(\theta) =0\quad (k \geq 2)\\
&\p_{\al}^k d_j(\theta) = (-\al \overline{l})^k d_j(\theta) + O(h \overline{l}^k) \ep_j(\theta) + O(h^{2-1/\al} \overline{l}^k) Y_{t_{j-1}} + O(h \overline{l})\quad (k \geq 1),\\
&\p_{\sig}^{k} d_j(\theta) = O(1) d_j(\theta) + O(h^{2-1/\al}) + O(h) \quad (k \geq 1),\\
&\p_{\be}^k d_j(\theta) = O(h),\quad \p_{\be}^k d_j(\theta) = 0 \quad (k \geq 2). 
\end{align}
As in \eqref{eq:deriv_ep}, we can derive the following:
\begin{align}\label{eq:deriv_dj}
&\p_{\lam}^{i_1}\p_{\mu}^{i_2}\p_{\al}^{i_3}\p_{\sig}^{i_4}\p_{\be}^{i_5} d_j(\theta)\\
&= O(\overline{l}^{i_3})d_j(\theta) + O(h\overline{l}^{i_3})\ep_j(\theta) + O(h^{2-1/\al} \overline{l}^{i_3})Y_{t_{j-1}} + O(h\overline{l}^{i_3}).
\end{align}
From the equation \eqref{eq:deriv_ep} and \eqref{eq:deriv_dj}, we have
\begin{align}
&\p_{\lam}^{i_1}\p_{\mu}^{i_2}\p_{\al}^{i_3}\p_{\sig}^{i_4}\p_{\be}^{i_5} (\ep_j(\theta) + ud_j(\theta))\\
&= O(\overline{l}^{i_3})(\ep_j(\theta) + u d_j(\theta)) + O(h\overline{l}^{i_3}) \ep_j(\theta) + O(h^{2-1/\al}\overline{l}^{i_3})Y_{t_{j-1}} + O(\overline{l}^{i_3}).
\end{align}
Then, we have
\begin{align}
    &|\p_{\lam}^{i_1}\p_{\mu}^{i_2}\p_{\al}^{i_3}\p_{\sig}^{i_4}\p_{\be}^{i_5} \psi(\ep_j(\theta) + u d_j(\theta))|\\
    &\lesssim \overline{l}^{i_3} \left| \left( \p_{x}^k \psi(\ep_j(\theta) + u d_j(\theta)\right) (\ep_j(\theta) + u d_j(\theta))^{k+1}\right|\\
    &\quad + |(\p_{x}^k \psi(\ep_j(\theta) + u d_j(\theta))) (\ep_j(\theta) + ud_j(\theta))^k|\\
    & \quad \times |O(\overline{l}^{i_3}) h\ep_j(\theta) + O(h^{2-1/\al}) Y_{t_{j-1}} + O\big(\overline{l}^{i_3}\big)|\\
    &\lesssim \overline{l}^{i_3} (|h\ep_j(\theta)| + |Y_{t_{j-1}}| + 1).
\end{align}
This establishes \eqref{ineq:Gj_deriv}.

Finally, we prove \eqref{ineq:delta_deriv}. From inequality \eqref{ineq:dj_deriv} and \eqref{ineq:Gj_deriv}, for any integer $k \geq 1$, we have
\begin{align}
    |\p_{\theta}^k \check{\delta}_n(\theta)| &\lesssim\sumj \max_{i}|\p_{\theta}^iG_j(\theta)||\p_{\theta}^id_j(\theta)|\\
    &\lesssim \overline{l}^{2k} \frac{1}{n} \sumj \left(h|\ep_j(\theta)|^2 + |\ep_j(\theta)| + |\ep_j(\theta)||Y_{t_{j-1}}| + |Y_{t_{j-1}}|^2 + |Y_{t_{j-1}}| + 1\right)\\
    &\lesssim O_p\big(\overline{l}^{2k}\big).
\end{align}
\end{proof}


We are now in a position of proving \eqref{condi:gap_conv1}, \eqref{condi:gap_conv2}, and \eqref{condi:gap_partial}. By using Lemmas \ref{lem:delta0_order} and \ref{lem:delta1_order}, and by the definition of $\delta_n(\theta)$, we get
\begin{equation}\label{ineq:deriv_delta_order}
    |\p_{\theta}^k \delta_n(\theta)| \leq |\p_{\theta}^k\delta_n^0(\theta)| + |\p_{\theta}^k\check{\delta}_n(\theta)| \lesssim O_p(\overline{l}^{2k}). 
\end{equation}
By Lemmas \ref{lem:rate_est} and \eqref{ineq:deriv_delta_order}, and the submultiplicativity of the norm, we can deduce that
\begin{align}
    & |\varphi_n(\theta)^{\top} \p_{\theta} \delta_n(\theta)| \lesssim r_n(\theta)^{-1} \overline{l}^{2},\\
    & |\varphi_n(\theta)^{\top} \p_{\theta}^2 \delta_n(\theta) \varphi_n(\theta)| \lesssim r_n(\theta)^{-2} \overline{l}^4,\\
    & \sup_{\theta^{(1)},\dots,\theta^{(5)}\in \mathfrak{R}_n(c;\theta)} \left| \vp_n(\theta)^\top \left( \p_\theta^2 \del_n(\theta^{(1)},\dots,\theta^{(5)}) - \p_\theta^2 \del_n(\theta)\right) \vp_n(\theta) \right|
    \nn\\
    &\lesssim \sup_{\theta', \theta'' \in\mathfrak{R}_n(c;\theta)}|\varphi_n(\theta)^{\top} \p_{\theta}^3 \delta_n(\theta'') (\theta' - \theta) \varphi_n(\theta)| \lesssim r_n(\theta)^{-2}\frac{\overline{l}^7}{\sqrt{n}}.
\end{align}
The upper bounds are all $o_{p}(1)$ uniformly in $\theta$ since $r_n(\theta)^{-1} \overline{l}^m \cip_u 0$ for any $m\ge 1$, thereby establishing the conditions \eqref{condi:gap_partial}, \eqref{condi:gap_conv1}, and \eqref{condi:gap_conv2}.

\bigskip

\end{document}